\documentclass[12pt]{article}
\usepackage[utf8]{inputenc}
\usepackage[english]{babel}
\usepackage[T1]{fontenc}
\usepackage{hyperref}
\usepackage{amsmath}
\usepackage{amsthm}
\usepackage{amsfonts}
\usepackage{amssymb}
\usepackage{graphicx}
\numberwithin{equation}{section}
\usepackage{graphicx}
\usepackage{tikz}
\usetikzlibrary{calc}
\usepackage{framed}
\usepackage{caption}
\usepackage{gensymb}
\usetikzlibrary{arrows}
\tikzset{>=latex}
\usetikzlibrary{patterns}
\usepackage{authblk}
\usepackage[top=2.5cm, bottom=2.5cm, left=2.5cm, right=2.5cm]{geometry}
\usepackage{fancyhdr}
\usepackage{cuted}
\usepackage{appendix}
\usepackage{chngcntr}

\newtheorem{theorem}{Theorem}[section]
\newtheorem{proposition}[theorem]{Proposition}

\newtheorem{corollary}[theorem]{Corollary}
\newtheorem{definition}[theorem]{Definition}
\newtheorem{lemma}[theorem]{Lemma}

\newtheorem{remark}[theorem]{Remark}
\newtheorem{assumption}[theorem]{Assumption}

\newcommand{\N}{\mathbb{N}}
\newcommand{\E}{\mathbb{E}}
\newcommand{\R}{\mathbb{R}}
\newcommand{\C}{\mathbb{C}}
\newcommand{\Z}{\mathbb{Z}}
\renewcommand{\P}{\mathbb{P}}
\newcommand{\Prob}{\mathbb{P}}

\newcommand{\T}{\mathbb{T}}
\newcommand{\cB}{\mathcal{B}}
\newcommand{\cD}{\mathcal{D}}
\newcommand{\cF}{\mathcal{F}}

\newcommand{\cS}{\mathcal{S}}
\newcommand{\cY}{\mathcal{Y}}
\newcommand{\cI}{\mathcal{I}}
\newcommand{\cJ}{\mathcal{J}}

\newcommand{\cR}{\mathcal{R}}
\newcommand{\cC}{\mathcal{C}}
\newcommand{\cE}{\mathcal{E}}
\newcommand{\bbL}{\mathbb{L}}
\newcommand{\bbW}{\mathbb{W}}

\renewcommand{\epsilon}{\varepsilon}

\DeclareMathOperator{\He}{Hess}
\newcommand{\nnb}{\nonumber\\}
\DeclareMathOperator{\ent}{Ent}
\DeclareMathOperator{\diag}{diag}
\DeclareMathOperator{\id}{id}
\DeclareMathOperator{\tr}{Tr}

\newcommand\xqed[1]{%
  \leavevmode\unskip\penalty9999 \hbox{}\nobreak\hfill
  \quad\hbox{#1}}
\newcommand\demo{\xqed{$\blacksquare$}}

\usepackage[mathscr]{euscript}

\newcommand{\mc}[1]{{\mathcal #1}}
\newcommand{\mb}[1]{{\mathbf #1}}
\newcommand{\mf}[1]{{\mathfrak #1}}

\newcommand{\bb}[1]{{\mathbb #1}}
\newcommand{\ms}[1]{{\mathscr #1}}
\newcommand{\mg}[1]{{\textsf #1}}

\newcommand{\mtt}[1]{{\mathtt #1}}

\newcommand{\<}{\langle}
\renewcommand{\>}{\rangle}

\title{Critical dynamical fluctuations in reaction-diffusion processes}
\author{Benoit Dagallier\footnote{CEREMADE, Université Paris-Dauphine, Université PSL. E-mail: {\tt dagallier@ceremade.dauphine.fr}.}\hspace{0.2cm} and Claudio Landim\footnote{IMPA, Estrada Dona Castorina 110, CEP 22460 Rio de Janeiro, Brasil and CNRS
UMR 6085, Universit\'e de Rouen, France. E-mail: {\tt landim@impa.br}.}}
\date{May 23, 2025}

\begin{document}
\maketitle

\begin{abstract}
We consider a one-dimensional microscopic reaction-diffusion process
obtained as a superposition of a Glauber and a Kawasaki dynamics.  The
reaction term is tuned so that a dynamical phase transition occurs in
the model as a suitable parameter is varied.  We study dynamical
fluctuations of the density field at the critical point.

We characterise the slowdown of the dynamics at criticality, and prove 
that this slowdown is induced by a single observable, the global density (or magnetisation). 
We show that magnetisation fluctuations are non-Gaussian and characterise
their limit as the solution of a non-linear SDE. We  prove,
furthermore, 
that other observables remain fast: 
the density field acting on the fast modes (i.e. on mean-0 test functions) and with Gaussian scaling converges, in the sense of finite dimensional distributions, to a Gaussian field with space-time covariance that we compute explicitly.

The proof relies on a decoupling of slow and fast modes  relying in
particular on a relative entropy argument.  Major technical
difficulties include the fact that local equilibrium does not hold due
to the non-linearity, and proving replacement estimates on diverging
time intervals due to critical slowdown.
\end{abstract}

\setcounter{tocdepth}{1}
\tableofcontents

\section{Introduction}
The goal of this article is the study of an interacting particle system in the vicinity of a phase transition. 
General heuristics, based on dynamical renormalisation group arguments, identify possible universal behaviours at a critical point depending on conserved quantities and order parameters~\cite{HH}. 
Near a phase transition the dynamics is expected to slow down and, typically, evolve in a time-scale polynomial in the distance to the critical point or in the size of the model if exactly at the transition. 
Detailed predictions on how various observables slow down have been made, with formulas for the corresponding critical exponents. 
These predictions typically assume a decoupling of fast and slow observables to reduce the full microscopic dynamics to just a few effective equations for the slow objects.\\

Making these predictions rigorous starting from a microscopic dynamics has proven to be challenging.  
A central difficulty lies in the justification of the decoupling assumption. 
Indeed, in the microscopic dynamics observables of interest typically do not satisfy closed equations.  
Understanding how to close these equations is a well-known problem even away from critical points. 
At criticality a second difficulty comes from the presence of non-linear effects.

One class of models for which critical dynamics are well understood involves particles with long range, mean-field interactions. 
In the case of real particles with individual Brownian noise, 
Dawson~\cite{Dawson} identified, in the limit of a large number of particles, 
the time-scale on which the dynamics evolves close to the critical point, 
as well as described the dynamics of the mean particle position in this time-scale. 
Similar results could be obtained for discrete (Glauber) dynamics in a mean-field Ising model. 
In both cases, the magnetisation satisfies a closed equation at the microscopic level, 
which means that both the form of the equation and the correct time-scale close to the critical point are immediately apparent. 
The magnetisation (or mean particle position) is also the only slow observable, the limiting dynamics of its fluctuations being given by a finite-dimensional, non-linear SDE. 

Going beyond mean-field interactions introduces significant difficulties. The much more involved case of Glauber dynamics for Ising models (and variants) with intermediate-range interactions of Kac type was studied first in dimension one~\cite{Kac1d}; 
then in the harder two-dimensional~\cite{Mourrat_Weber_Kac,Shen_Weber_Kac,Iberti_Kac} and three dimensional cases~\cite{Kac3d}, using methods from Hairer's theory of regularity structures~\cite{Hairer}. 
In these models there is a diverging number of  observables that become slow at the transition.  
Contrary to the mean-field case, 
these observables do not individually satisfy closed equations. 
The limiting equations are correspondingly formulated for the magnetisation fluctuation field, 
which is in general a random distribution satisfying an infinite-dimensional (singular) SPDE related to models from field theory. 
Although the complexity of the limiting SPDE makes the study of the microscopic dynamics very involved, 
there is again no difficulty in identifying the correct time-scale and form of the equation. 
Indeed, as in the mean-field case, the range of the Kac interaction is sufficiently long that the fluctuation field satisfy (almost) closed equations at the microscopic level. 
\medskip

In contrast, results for microscopic models with short-range interactions remain
far and few.  Two types of dynamics have been studied extensively:
conservative Kawasaki dynamics (with a single conservation law), 
and non-conservative Glauber dynamics.
In Ising models with Glauber dynamics (meaning no conserved quantities as in previous examples), 
polynomial slowdown around the critical point has been established in many cases in the last ten years~\cite{LuSly_Ising,BauDag_Ising}.  A description of the dynamics close to
the critical point however remains out of reach.

For microscopic models with Kawasaki dynamics, say for definiteness models where particles can jump at finite range subject to an exclusion rule so that the number of particle is conserved, understanding of the dynamics at the critical point is
also limited but general results relating certain critical
exponents exist~\cite{SY_diffusivity}.  The presence of the conserved
quantity however makes conservative dynamics more tractable, at least
at high temperature.  Indeed, due to the global conservation law, the
density of particles evolves on much longer time-scales than other
observables, 
making it possible to obtain closed macroscopic equations on its evolution.  This observation was made rigorous by Guo, Papanicolaou
and Varadhan~\cite{gvp}, Rost~\cite{rost} and Brox and
Rost~\cite{BR84}, who set forward general arguments, the so-called
entropy method in \cite{gvp} and the Boltzmann-Gibbs principle in
\cite{rost, BR84}, to obtain closed equations for the density and its equilibrium fluctuations.  
These general arguments rely on the explicit knowledge of the invariant measure of the dynamics, 
although duality arguments can sometimes be used in special cases where the invariant measures are not known~\cite{dfl}.  
The general arguments were then extensively used and
improved (see the book~\cite{kl} for a review), 
with variations such
as the relative entropy method of Yau~\cite{yau_relative_ent},
recently improved to gain in-depth understanding of fluctuations of the
density at (very) high temperature in general settings, with no requirement on knowledge of invariant measures~\cite{Jara2018}. 
\\

In this article, we consider an interacting particle with a conserved
quantity at a critical point.  Our main result,
Theorem~\ref{theo_convergence_magnetisation}, establishes critical
slowdown and shows that the density field on this time-scale acts by projection on the global
density. 
Similarly to mean-field models,  
this global density of particles is thus shown to be the sole observable responsible for critical slowdown, 
and Theorem~\ref{theo_convergence_magnetisation} also characterises its non-linear
dynamics. 
Our second main result, Theorem~\ref{theo_fastmodes}, 
gives a precise description of fast modes, 
a term we loosely employ to refer to all other observables. 
We show these fast modes remain Gaussian with an explicit covariance.

Precisely, we
work with a so-called Glauber + Kawasaki dynamics in one space
dimension, where the Kawasaki dynamics conserves the number of
particles while the Glauber dynamics is associated with a short-range
Ising model.  The typical evolution of the macroscopic density $\rho$ of particles
in such a model is well known to be described by a reaction-diffusion
equation in the scaling limit~\cite{dfl}:
\begin{equation}
\partial_t\rho 
=
\partial^2_x \rho - V'(\rho)
,\qquad 
V:[0,1]\to\R
.
\label{eq_reacdiff_intro}
\end{equation}
Taken separately, 
neither the Kawasaki dynamics nor the Glauber part are associated with models having static phase transitions. 
Indeed, 
the Kawasaki dynamics giving rise to the Laplacian term in~\eqref{eq_reacdiff_intro} microscopically corresponds to a model at infinite temperature, see~\eqref{10} below, 
while the Glauber part is reversible with respect to a one dimensional short-range Ising model (see below~\eqref{eq_Glauberjump_rates_as_Ising}).   
The superposition of the two dynamics, 
however, 
features a \emph{dynamical} phase transition with associated critical slowdown. 
Occurrence of a phase transition can be read off properties of the reaction term $V$.  
To wit, by analogy with a static model, situations where the reaction term $V$ is strictly convex can be viewed as the underlying particle system being at high temperature. 
The presence of two or more distinct minimisers for $V$ corresponds on the other hand to a low temperature situation, 
and a dynamical phase transition occurs when the unique global minimiser for $V$ becomes unstable. 
At low temperature, in the simplest case of a reaction term with two minima $\rho_-\neq\rho_+$ of equal depth and in the limit where the strength of the reaction term goes to infinity, 
one expects interfaces between regions at densities $\rho_\pm$ to form rapidly, then evolve in a diffusive time-scale according to motion by curvature; an instance of the so-called Lifshitz law~\cite{Lifshitz1962}.

Making the above heuristics rigorous has proved to be challenging. 
Indeed, both the entropy method and the duality techniques of~\cite{dfl} only easily give access to \emph{qualitative} information on the dynamics.  
This makes it difficult to control the dynamics on different time-scales, 
relevant at low temperature to describe interface formation in the strong reaction limit, 
as well as at the critical point due to critical slowdown. 
A few years ago Jara and Menezes gave a  roadmap to providing \emph{quantitative} control on the dynamics~\cite{Jara2018}, 
improving on the relative entropy method of Yau and building on the earlier work~\cite{gjs}. 
These quantitative results were initially used to study density fluctuations at high temperature in contexts where the invariant measure is not known~\cite{jm_example,jl}. 
Improvements on these quantitative estimates however also opened the way to studying interface formation and evolution at low temperature, 
i.e. to obtaining law of large number-type results on the low temperature dynamics in the strong reaction limit. 
This however turned out to require substantially different ideas and was recently achieved in a series of work~\cite{ft_early,Funaki_nongradient,ft}. 
The corresponding controls were shown to be sufficient to also describe fluctuations around a low temperature interface~\cite{fls}. 
This left open the delicate study of fluctuations at the critical point, addressed in the present paper through yet improved relative entropy estimate. \\

Obtaining relative entropy bounds precise enough to study density fluctuations of Glauber + Kawasaki dynamics close to a critical point involves substantial additional difficulties. 
On the one hand, the density fluctuation field, 
seen as a random distribution (see~\eqref{eq_def_ttY}), 
does not evolve on the same time-scale depending on the test function it acts on. 
As the correct time-scale for the slow modes corresponds to very long times for the fast modes, 
precise quantitative controls on the homogenisation of the latter are required to obtain the correct limiting equations.

On the other hand, 
traditional local equilibrium arguments, 
at the heart of both entropy and relative entropy methods, 
break down at the critical point of Glauber + Kawasaki dynamics. 
The local equilibrium heuristics asserts that the law of the dynamics at each time is close to a so-called local equilibrium state, 
in the present case given by a product measure which prescribes the correct macroscopic behaviour for the density, 
but discards all other microscopic information. 
Local equilibrium has been shown to hold, with quantitative bounds, in very general high temperature settings - this is a major achievement of~\cite{Jara2018} - as well as at low temperature~\cite{fls}; density fluctuations being linear (i.e. with Gaussian scaling) in both cases. 
Local equilibrium also turns out to provide a good enough approximation to give access to non-linear fluctuations in KPZ-type models \cite{GJ_fluct_WASEP,gjs,Yang_KPZ_longrange, Yang_BG}, 
where the non-linearity is \emph{local} (it involves the square of the gradient of the field). 
In contrast, 
the non-linearity in our case is \emph{global}, 
as it involves powers of the density field itself rather than its gradient. 
This non-linearity comes from long-distance correlations arising at the critical point. Those are felt even at microscopic level, 
making the uncorrelated local equilibrium state a poor approximation of the dynamics.  
Let us however mention the very recent~\cite{HMW25}, 
where convergence of weakly asymmetric simple exclusion to the KPZ equation is established. 
The argument does not rely on local equilibrium, 
instead using methods of regularity structures adapted to the discrete setting. It would be very interesting to see what such techniques yield for our type of non-linearity.

\medskip
To set up the relative entropy estimate that will enable us to obtain closed equations on the density fluctuations, we start by identifying the slow observables and 
decompose observables into slow and fast contributions at the microscopic level.  
In the present case the Kawasaki part of the dynamics is infinite-temperature with the consequence that there is a single slow mode, the
global number of particles. 
This greatly simplifies the decomposition, 
which could however be carried out in more general settings.  
The key observation is then that, 
conditioned on the slow mode, the dynamics is
effectively high temperature. 
This is established rigorously using suitable log-Sobolev inequalities. 
With this observation in hand we can show that a form of local equilibrium is indeed valid, 
uniformly in time and with quantitative bounds, 
for the dynamics projected on configurations with fixed value of the slow mode.  
The uniformity in time is necessary because the long time-scale of the
critical slowdown corresponds to longer-than-hydrodynamic times for
the fast observables. 
At this point the contribution of the fast modes to the relative entropy has been taken into account, 
and we are only left with bounding non-linear functions of the slow mode. 
This is done by energy-type estimates.\\

The article is organised as follows.  In
Section~\ref{sec_model_result} we present the model, 
the main results and introduce the relative entropy estimate.  
The first main result,
Theorem~\ref{theo_convergence_magnetisation}, asserts that the density
field acts by projecting on the magnetisation, 
which evolves according to a
non-linear stochastic ODE. It is proven in
Section~\ref{sec3}. 
The main ingredient in the
proof of Theorem~\ref{theo_convergence_magnetisation} is the key
technical input of this work, the relative entropy estimate that spans
Sections~\ref{n-sec1} to~\ref{sec_free_energy_bounds}.
Section~\ref{n-sec1} has pedagogical aims: general background on the
relative entropy method is presented and the failure of local
equilibrium is illustrated.  The proof of the relative entropy estimate used in
Theorem~\ref{theo_convergence_magnetisation} is then split into two
parts.
 
In Section~\ref{n-sec2}, 
we explain how to include non-linear information into the local equilibrium measure in order to control the slow mode. 
A bound on relative entropy of the law of the dynamics with respect to this modified local equilibrium is established, 
assuming suitable bounds on the fast modes. 
The technical Sections~\ref{n-sec3} and~\ref{sec_free_energy_bounds} then provide the control on fast modes under various conditions.  

In Section~\ref{n-sec3}, we work in a perturbative regime, where a suitable parameter in the dynamics is small.  
Still with a pedagogical aim, the section is as elementary as possible, 
proving local equilibrium for fast modes through classical equivalence of ensembles arguments  and an integration by parts formula stated in Appendix~\ref{sec_IBP}. 

The restriction on the size of the parameter is then removed in Section~\ref{sec_free_energy_bounds}, 
under the assumption that a suitable log-Sobolev inequality hold. 
While this log-Sobolev inequality is conjectured to be true, 
it is not known at the moment. 
Using the log-Sobolev inequality, 
the restriction to small parameters is lifted by yet again improving on our modified local equilibrium, 
this time adding more precise information on fast modes. 
To make the resulting very technical estimates as accessible as possible, 
we work with a modified local equilibrium providing worse control on the non-linearity compared to Section~\ref{n-sec3}. 
It is then not possible to control directly the relative entropy, 
and instead we bound a kind of free energy.  
The estimate of this free energy makes use of a combination of concentration estimates and large deviation bounds replacing equivalence of ensembles arguments, 
established in Sections~\ref{app_concentration} and \ref{app_LD} respectively. 
 
The second main result,
Theorem~\ref{theo_fastmodes} proven in Section~\ref{sec8}, 
states that the fast modes remain Gaussian, with a covariance that is white in time and coloured in space. 
The lack of correlation in time reflects the fact that the fast modes equilibrate instantaneously on the time-scale of the critical slowdown. 
Theorem~\ref{theo_fastmodes} is proved using a combination of the same tools (relative entropy/free energy estimate and concentration bounds) together with Duhamel's formula.\\

\section{Model and results}\label{sec_model_result}
\subsection{A reaction-diffusion model}
Denote by $\color{blue} \bb T_n = \bb Z / n \bb Z$,
$n \in \bb N =\{1, 2, \dots \}$, the one-dimensional discrete torus
with $n$ points. Let $\color{blue} \Omega_n = \{0,1\}^{\bb T_n}$
be the state space. Elements of $\Omega_n$ are represented by the
Greek letters $\eta = (\eta_i : i\in \bb T_n)$ and $\xi$, with $\eta_i =1$ if the configuration $\eta$ is occupied at site $i$ and
$\eta_i =0$ otherwise.

Denote by $L_n^{\rm ex}$ the generator of the symmetric simple exclusion
process on $\bb T_n$ given by
\begin{equation}
\label{10}
(L_n^{\rm ex}f)\, (\eta)\;=\; \sum_{i\in\bb T_n} 
\, \big\{\, f(\eta^{i,i+1})-f(\eta)\, \big\}
\end{equation}
for all $f\colon \Omega_n \to \bb R$.  In this formula, $\eta^{ i,j}$
represents the configuration of particles obtained from $\eta$ by
exchanging the occupation variables $\eta_i$ and $\eta_j$:
\begin{equation}
{\color{blue} (\eta^{i,j})_{k} } \;:=\;
\begin{cases}
\eta_{k}  \,, & k\not =i \,,\, j\;,\\
\eta_{i} \,, & k=j\;, \\
\eta_{j} \,, & k=i\;.
\end{cases}
\label{eq_def_eta_i_iplus1}
\end{equation}
Here and below, summations are to be understood modulo $n$.

Denote by $\color{blue} \nu^n_\rho$, $0\le \rho\le 1$, the Bernoulli
product measure on $\Omega_n$ with density $\rho$. This is the product
measure with marginals given by Bernoulli distributions with parameter
$\rho$. A straightforward computation shows that these measures
satisfy the detailed balance conditions for the symmetric exclusion
process. In particular, they are stationary for this dynamics.

Fix a positive cylinder function
$h \colon \{0,1\}^{\bb Z} \to \bb R$, and let $L^G_n$ be the
generator of the Glauber dynamics on $\Omega_n$:
\begin{equation}
(L^G_n f)(\eta)  \;=\; \sum_{i\in \bb T_n}  h (\tau_i \eta)
\,\big[\,  f(\eta^i) \,-\, f(\eta)\, \big]
\end{equation}
for all $f\colon \Omega_n \to \bb R$.  In this formula, $\eta^i$
represents the configuration obtained from $\eta$ by flipping the
value of $\eta_i$:
\begin{equation}
{\color{blue}(\eta^i)_{k}} \;:=\;
\begin{cases}
\eta_{k}\;, & k\not =i\;,\\
1\,-\, \eta_{i}\;, & k=i\;,
\end{cases}
\label{eq_def_etai}
\end{equation}
and $\{\tau_i : i\in \bb T_n\}$ are the translations acting on
$\Omega_n$: 
\begin{equation}
\label{12}
(\tau_i \eta)_k \;=\;\eta_{i+k}\;, \quad i\,,\, k\,\in\, \bb T_n\;,
\;\; \eta\,\in\, \Omega_n  \;.
\end{equation}

In this article we consider a particular family of rates, but the
arguments do not rely on the specific form of the rates and can be
applied to other non-conservative dynamics which exhibit a phase
transition.  For $-1\le \gamma \le 1$, denote by
$c \colon \{0,1\}^{\bb Z} \to \bb R$ the positive cylinder function
given by
\begin{equation}
\label{01}
{\color{blue} c (\eta)} \,=\, c_\gamma(\eta)
\,:=\, 
1\,-\, \gamma \, \sigma_0\, [\, \sigma_{-1} + \sigma_{1}
\,] \,+\, \gamma^2 \, \sigma_{-1}\, \sigma_{1}  \;,
\end{equation}
where $\color{blue} \sigma_i = 2\eta_i -1 \in \{-1,1\}$. Mind that
$c_\gamma (\cdot)$ takes only three values:
\begin{equation}
c_\gamma (\sigma) \,=\,
\left\{
\begin{aligned}
& (1-\gamma)^2 \quad  \text{if $\sigma_{0} = \sigma_{-1} = \sigma_{1}$ } \,,
\\
&1- \gamma^2 \quad\;\;  \text{if $\sigma_{-1} \neq \sigma_{1}$ } \,,
\\
&(1+\gamma)^2 \quad  \text{if $\sigma_{0} \neq \sigma_{-1} = \sigma_{1}$  } \,.
\end{aligned}\label{eq_value_jump_rates}
\right.
\end{equation}
Mind that the rates are increasing for $0<\gamma<1$, as the dynamics
tries to align spins.

For $0\le \gamma\le 1$, the Glauber dynamics induced by the generator
$L^G_n$ is reversible for the Gibbs measure of an Ising model with
nearest neighbour interaction at inverse temperature $\beta$
satisfying $\gamma = \tanh \beta$.

Indeed, consider the Hamiltonian
$H_n\colon \{-1,1\}^{\bb T_n} \to \bb R$ given by
\begin{equation}
H_n(\sigma) \,:=\, -\, \sum_{j\in\bb T_n} \sigma_j\,
\sigma_{j+1}\,. 
\end{equation}
Denote by $\mu^n_{\beta}$ the Gibbs measure induced by 
$H_n$ at inverse temperature $\beta$:
\begin{equation}
{\color{blue} \mu^n_\beta (\sigma) } \,:=\,
\frac{1}{Z^n_\beta}\, e^{-\beta H_n(\sigma)}\;,
\label{eq_Ising}
\end{equation}
where $Z^n_\beta$ represents the normalization which turns
$\mu^n_\beta$ into a probability measure.  The Glauber dynamics is the name
given to the $\Omega_n$-valued, continuous-time Markov chain in which
a spin $\sigma_i$ is flipped at the rate
\begin{equation*}
 e^{-(\beta/2) \, [H_n(\sigma^i) - H_n(\sigma)]}  \;.
 \end{equation*}
 
An elementary computation yields that the Gibbs measure $\mu^n_\beta$
satisfies the detailed balance conditions for the Glauber dynamics, and that
\begin{equation}
A\, e^{-(\beta/2) \, [H_n(\sigma^i) - H_n(\sigma)]}  \,=\,
c(\tau_i\eta)
\label{eq_Glauberjump_rates_as_Ising}
\end{equation}
provided $A = 1-\gamma^2$, $\gamma = \tanh \beta$. In particular, the
Gibbs measure $\mu^n_\beta$ is stationary for the Markov chain induced
by the generator $L^G_n$ with the rates $c(\cdot)$ given by \eqref{01}.

Fix $a>0$ and let $\color{blue} (\xi^n(t); t\geq 0)$ be the
$\Omega_n$-valued, continuous-time Markov chain whose generator,
denoted by $L_n$, is given by
\begin{equation}
{\color{blue} L_n}  \;: =\; n^2\, L_n^{\rm ex} \;+\; a \, L^G_n\;.
\label{eq_def_generator}
\end{equation}

Denote by $\color{blue} \ms M (\bb T)$ the space of positive measures
on $\bb T$ with total mass bounded by $1$, endowed with the weak
topology.  Introduce the empirical measure
$\pi^n(\eta) \in \ms M(\bb T)$ associated with an element
$\eta \in \Omega_n$ as the measure on $\bb T$ defined by
\begin{equation}
{\color{blue} \pi^n} \,=\, \pi^n(\eta) \,:=\, 
\frac{1}{n}\sum_{j\in \bb T_n}\eta_j \, \delta_{j/N}\;,
\label{eq_def_pin_intro}
\end{equation}
where $\color{blue}\delta_{x}$ stands for the Dirac mass at point $x$.  

Fix a density profile $\rho_0 \colon \bb T \to [0,1]$. Consider a
sequence of initial configurations $\eta^n\in \Omega_n$, and suppose
that the associated empirical measure $\pi^n(\eta^n)$ converges to
$\rho_0(x)\, dx$.  De Masi, Ferrari and Lebowitz \cite{dfl} proved that
for all $t\ge 0$, the random measure $\pi^n(\xi^n_t)$ converges to the
absolutely continuous measure $u(t,x)\, dx$ whose density $u$ solves the
reaction-diffusion equation
\begin{equation}
\label{eq_LLN}
\left\{
\begin{aligned}
& \partial_t u = \Delta u \,-\,  \, a\, V'(u) \;,
\\
& u(0, x) = \rho_0 (x) \;,
\end{aligned}
\right.
\end{equation}
where $V\colon [0,1] \to \bb R$ is the potential given by
\begin{align}
{\color{blue} V'(\rho)}
\;&=\; -\, E_{\nu^n_\rho} \big[\, c(\eta) \, [1- 2\eta_0]\,\big]  
\nnb
\;&=\; -\, (2\gamma-1)(2\rho-1) + \gamma^2 (2\rho-1)^3\;
,
\qquad 
\rho\in[0,1]
.
\end{align}
Here and below, given a function $g\colon \Omega_n \to \bb R$ and a
probability measure $\mu$ on $\Omega_n$, we denote by
$\color{blue} E_\mu[g]$ or $\color{blue}\mu(g)$ the expectation of $g$
with respect to $\mu$.  
If $\gamma\leq 1/2$, $V(\cdot)$ is strictly convex, while for
$1/2<\gamma\le 1$ $V(\cdot)$ is a double-well potential with a local
maximum at $\rho=1/2$ and local minima at $0<\rho_- < 1/2< \rho_+<1$.
By symmetry, $\rho_- = 1-\rho_+$.  For $\gamma=1/2$ the global minimum at $1/2$ becomes degenerate as $V''(1/2)=0$.


We investigate in this article the fluctuations of the density field~\eqref{eq_def_pin_intro}, 
in particular of the fluctuations of the total number of particles (which we will also call fluctuations of the magnetisation in view of the connection~\eqref{01} to~\eqref{eq_Glauberjump_rates_as_Ising} to an Ising model). 
We are interested in a
neighborhood of the local minimum $\rho=1/2$ when this local minimum
becomes degenerate (that is for $\gamma \sim 1/2$).  Fix
\begin{equation}
{\color{blue} \gamma \,=\, \gamma_n } \, :=\,
\frac{1}{2}\, \Big(\, 1 \,-\, \frac{\theta}{ \sqrt{n}}\,\Big)\,,
\qquad
\theta\in\R \;.
\label{eq_def_gamma}
\end{equation}

Let us present heuristic computations to identify the size of magnetisation fluctuations and the time-scale on which they evolve. 
Assume below that $\theta=0$ so that
$\gamma=1/2$. Denote by $\mc Y^n_t$ the normalized magnetisation:
\begin{equation}
{\color{blue} \mc Y^n_t} \,:=\, \frac{1}{b_n} \, \sum_{i\in\bb T_n}
\{\, \xi^n_i(\upsilon_n t) - (1/2)\,\} \;, 
\end{equation}
where $b_n$, $\upsilon_n$ are sequences of positive real numbers which
diverge as $n\to\infty$.  The semi-martingale decomposition of
$\mc Y^n_t$ reads as
\begin{equation}
\label{n-16}
\mathcal Y^n_t \,=\, \mathcal Y^n_0
+ \int_0^t  \upsilon_n \, L_n\mathcal Y^n_s\, ds \,+\, M^n_t 
\;,
\end{equation}
where $M^n_t$ is a martingale. The generator has been multiplied by
$\upsilon_n$ because the dynamics is accelerated by this amount. Since
the exclusion part of the dynamics does not modify the magnetisation,
\begin{equation}
\label{n-24}
\upsilon_n \, L_n\mathcal Y^n \,=\,
a\, \upsilon_n \, L^G_n\mathcal Y^n
\,=\,
-\frac{a \upsilon_n }{b_n}\sum_{i\in\T_n}c(\tau_i \eta)
\, \sigma_i \;.
\end{equation}
As the exclusion part of the dynamics is speeded-up by
$\upsilon_n n^2$, it is plausible that it provides enough ergodicity
to replace the local function $c(\tau_i \eta) \, \sigma_i$ by its
expectation with respect to the Bernoulli product measure with density
given by the empirical density $ m/n = n^{-1} \sum_i \eta_i$:
\begin{equation}
c(\tau_i \eta) \, \sigma_i \,\sim\,
\nu^n_{m/n} \big(\, c(\eta) \, \sigma_0 \,\big) \,=\,
V'(m/n)\,.
\label{eq_heuristic_replacement_intro}
\end{equation}
Since the first three derivatives of $V$ vanish at $1/2$, a third
order Taylor expansion of $V'$ yields that
$V'(m/n) \sim (1/6) V^{(4)}(1/2) [(m/n) - 1/2]^3$, where we omitted the
higher order terms.  Here, $V^{(4)}(\cdot)$ stands for the fourth
derivative of $V(\cdot)$.  Since $(m/n) - 1/2 = (b_n/n) \mc Y_n$, we
conclude that
\begin{equation}
\upsilon_n \, L_n\mathcal Y^n \,\sim\, -\frac{a \upsilon_n }{b_n}\,
n\, \frac{1}{6}\, V^{(4)}(1/2) \, \big(\,\frac{b_n}{n}\,\big)^3 \, \mc Y_n^3\;.
\end{equation}
The factor $n$ represents the sum over $i$.  For this expression to be
of order $1$ gives:
\begin{equation}
\upsilon_n\, (b_n/n)^2 =1
.
\label{eq_scaling1}
\end{equation}
Denote by $\<M^n\>_t$ the quadratic variation of the martingale
$M^n_t$ introduced in \eqref{n-16}. A straightforward computation
yields that
\begin{equation}
\label{n-17}
\<M^n\>_t \,=\, \int_0^t a\, \frac{\upsilon_n }{b^2_n}
\sum_{i\in\T_n} c(\tau_i \xi^n(\upsilon_n\, s)) \, ds \, .
\end{equation}
This expression is of order one only if:
\begin{equation}
 \upsilon_n = b^2_n/n
 .
\end{equation}
 Together with~\eqref{eq_scaling1} this imposes $b_n = n^{3/4}$ and $\upsilon_n =
\sqrt{n}$.

\begin{remark}
\label{rm1}
The previous argument can be carried out in any dimension. In dimension $d$ the appropriate parameters are $b_n = n^{3d/4}$,
$\upsilon_n = n^{d/2}$.
\end{remark}

\begin{remark}
\label{rm2}
The exclusion dynamics, which does not modify the magnetisation, only
entered in the argument to replace local quantities by macroscopic averages (a so-called two
blocks estimate for blocks of size $n$).  We are initially speeding-up
the exclusion part diffusively because it is in this scale that one
observes a smooth evolution of the density profile. Of course,
speeding it up more only helps.
\end{remark}

\subsection{Results}
In view of the previous discussion, denote by
$\color{blue} (\eta^n(t); t\ge 0)$ the $\Omega_n$--valued
continuous-time Markov chain induced by the generator
$\color{blue} \ms L_n := \sqrt{n}\, L_n$.  Let 
$\mc Y^n_t$, $\mtt Y^n_t$ respectively denote the centred, re-scaled
magnetisation and density fluctuation field, given by
\begin{equation}
\begin{gathered}
{\color{blue} \mc Y^n_t} \,:=\, \frac{1}{n^{3/4}} \, \sum_{i\in\bb T_n}
\{\, \xi^n_i(\sqrt{n}\, t) - (1/2)\,\}
\,:=\, \frac{1}{n^{3/4}} \, \sum_{i\in\bb T_n}
\{\, \eta^n_i(t) - (1/2)\,\}\;,
\\
{\color{blue} \mtt Y^n_t(H)} \,:=\, \frac{1}{n^{3/4}} \,
\sum_{i\in\bb T_n} H(i/n)\, 
\{\, \eta^n_i(t) - (1/2)\,\}\;, \quad H\,\in\, C^\infty(\bb T)
\;. 
\end{gathered}
\label{eq_def_ttY}
\end{equation}
The heuristic argument presented in the last subsection with the appropriate scalings
$b_n = n^{3/4}$, $\upsilon_n = \sqrt{n}$ and with $\theta\in \bb R$
yields that the re-scaled magnetisation should evolve according to the
non-linear stochastic ODE
\begin{equation}
\label{eq_SDE_magnetisation_in_thm}
\mathrm{d}\mathcal Y_t = 
-2\, a\,  \theta\, \mathcal Y_t \, \mathrm{d}t - 2\, a \, \mathcal Y_t^3 \,
\mathrm{d}t + \sqrt{a}\, \mathrm{d}W_t \,,
\end{equation}
where $(W_s)_{s\geq 0}$ is a standard Brownian motion.
This is the first main result of this article. A precise statement requires
some notation.

Denote by $\color{blue} \cD([0,T], \bb R)$, $T>0$, the space of
trajectories from $[0,T]$ to $\R$ which are right-continuous with
left-limits, endowed with the Skorohod topology and its associated
Borel $\sigma$-algebra. For a probability measure $\mu$ on $\Omega_n$,
denote by $\color{blue} \bb P^n_{\mu}$ the probability measure on
$\cD([0,T], \bb R)$ induced by the process $\mc Y^n_t$ when $\eta^n_0$
is distributed according to $\mu$.  Let $\color{blue} \E^n_\mu$ denote
the associated expectation.  When the dependence on $n$ is clear
(e.g. if the initial condition is denoted by $\mu_n$), we will drop
the superscript $n$ from the notation.

For $p\geq 1$, introduce also the set $\color{blue} \bb L^p([0,T],\R)$
of functions $f\colon [0,T] \to \bb R$ with finite $p^{\text{th}}$
moment: $
\int_{[0,T]} |f(t)|^p \, dt<\infty$.

To describe distributions on $\Omega_n$ the dynamics will start from, introduce the relative entropy $H_n(\mu|\nu)$ between
two probability measures $\mu,\nu$ on $\Omega_n$:
\begin{equation}
\label{2-34c}
{\color{blue} H_n(\mu |\nu)}\,:=\, \int
\frac{d\mu}{d\nu}\, \log \frac{d\mu}{d\nu} \;  d\nu\quad
\text{if}\;\; \mu\ll\nu \;,
\end{equation}
and $H_n(\mu |\nu) =+\infty$ otherwise.  Recall that we denote by
$\nu^n_{1/2}$ the Bernoulli product measure with density $1/2$ (which
corresponds to the uniform measure over all configurations).  For a
probability measure $\alpha$ on $\R$, let $\color{blue} \bb Q_\alpha$
denote the law of the solution to the
SDE~\eqref{eq_SDE_magnetisation_in_thm} with initial condition
$\alpha$ (uniquely defined, see e.g. \cite[Theorem 5.5.15]{ks}).  The
main result of the article reads as follows.  Below and throughout the
article, we use the notation $a_n=O(b_n)$ for two sequences
$a_n\in\R$, $b_n\geq 0$ ($n\in\N$) to say that there is $M>0$ with
$|a_n|\leq Mb_n$ for each $n$.

\begin{theorem}
\label{theo_convergence_magnetisation}
Let $\mu_n$ be a probability measure on $\Omega_n$, $n \ge 1$, and
let $\alpha$ be a probability measure on $\R$. Suppose that
\begin{equation}
H_n(\mu_n\, |\, \nu^n_{1/2}) 
=
O(\sqrt{n}\, ),\qquad 
\lim_{n\to\infty}\mu_n(\cY^n \in \cdot) = \alpha(\cdot)
\quad \text{ weakly}
.
\label{eq_assumption_CI_theo_magnetisation}
\end{equation}
Then the following convergence holds. 
\begin{itemize}
\item[(i)] (Small $a$).  There exists $\mf a_0>0$ independent of
$(\mu_n)_n$, $\alpha$ such that, for all $0<a<\mf a_0$
\begin{equation*}
\lim_{n\to\infty} 
\bb E_{\mu_n} \Big[\, \Big| \int_0^t \big\{ \mtt Y^N_s(H)
\, - \,\<H\> \,\mc Y^n_s \, \big\}\; ds \,\Big|\, \Big]\,=\, 0
\end{equation*}
for every $t>0$ and smooth function $H\colon \bb T\to \bb R$.  In this
formula, $\color{blue} \<H\> = \int_{\bb T} H(x)\, dx$.  Moreover, for
all $p\in(1,4/3)$, the law of $(\cY^n_t)_{t\in[0,T]}$ under
$\bb P^n_{\mu_n}$, seen as a probability measure on
$\bb L^p([0,T],\R)$, converges weakly to the measure $\bb Q_\alpha$ in
the topology of measures on $\bb L^p([0,T],\R)$.

\item[(ii)] (Conditional result for all $a$).  The same convergence
result holds true for any $a>0$ under the log-Sobolev Assumption~\ref{ass_LSI}.
\end{itemize}
\end{theorem}
\begin{remark}
Convergence in Theorem \ref{theo_convergence_magnetisation} 
holds in the stronger topology of probability measures on the
fractional Sobolev space $\bb W^{(p-1)/p,p}([0,T],\R)$ defined
Section~\ref{sec_tightness}.  We do not state this improvement as one
would expect convergence to directly hold for measures on
$\cD([0,T],\R)$.  We could however only control an integrated modulus
of continuity due to the non-linearity in the dynamics.
\end{remark}

\begin{remark}[Initial condition]\label{rm_IC}
The measure
$(Z^n_h)^{-1} \exp\{\frac{4h}{n^{1/4}}\sum_{i\in\bb T_n}
(\eta_i-1/2)\}\, \nu^n_{1/2}(\eta)$ on $\Omega_n$ ($h\in\R$) can be
checked to satisfy the assumptions of
Theorem~\ref{theo_convergence_magnetisation} with $\alpha = \delta_h$.
More interestingly,
Assumption~\eqref{eq_assumption_CI_theo_magnetisation} on initial
conditions is the natural relative entropy bound for critical
mean-field Ising measures.  Indeed, fix $b$, $c\in \bb R$, and denote
by $\mu^n_{b,c}$ the measure on $\Omega_n$ given by
\begin{equation}
\label{n-30}
{\color{blue}  \mu^n_{b,c}  (\eta) } \, := \,  \frac{1}{Z^n_{b,c}} \,
e^{ c n (\mf m/n)^2 + b \sqrt{n}\, (\mf m/n)^2}
\nu^n_{1/2} (\eta)\;, \quad \eta\in\Omega_n\;,
\end{equation}
where $Z^n_{b,c}$ is a normalizing constant, and
$\color{blue} \mf m \,:=\, \sum_{i\in \bb T_n} [\eta_i - (1/2)]$.
Denote by $\mu_{b}$  the probability measures on $\bb R$ given by
\begin{equation}
\label{n-32}
{\color{blue} \mu_{b} (dx) } \,:=\,  \frac{1}{Z_b}\, e^{-(4/3) x^4 + b
x^2} \, dx \,, 
\end{equation}
where $Z_b$ is a normalizing constant.  In
Appendix~\ref{app_refmeas_nuU} we show that the measures
$\mu^n_{b,c} $, $c\le 2$, satisfy the hypotheses of Theorem
\ref{theo_convergence_magnetisation} with $\alpha = \delta_0$ if $c<2$
and $\alpha = \mu_b$ if $c=2$.  The link between
Assumption~\ref{eq_assumption_CI_theo_magnetisation} and bounds for
critical mean-field Ising measures is not a coincidence: as we state
in Section~\ref{sec_sketch_proof_sec_results}, the law of the
magnetisation at each time is well approximated by a closely related
measure.
\end{remark}

\begin{remark}
By modifying the rates $c(\cdot)$ introduced in \eqref{01}, the
potential $V(\cdot)$ appearing in \eqref{eq_LLN} could be turned into
a sixth or higher order polynomial. Modifying appropriately the
time-scale and the magnetisation scaling, $\upsilon_n$ and $b_n$ in
equation \eqref{n-24}, one could derive a non-linear stochastic ODE of
fifth or higher order. Moreover, the fact that the Gibbs measure is
stationary for the Glauber dynamics plays no role. It was just
mentioned to stress that each piece of the dynamics has a distinct
stationary state.
\end{remark}

Theorem~\ref{theo_convergence_magnetisation} shows that fluctuations
of the magnetisation are non-linear and non-Gaussian, and that the
fluctuation field $\mtt Y^n_t$ projects over the magnetisation. 
It is however also possible to capture the much smaller fluctuations of other modes, which remain Gaussian and evolve on a faster time-scale. This is our second maint result, stated next.

Denote by $y^n_t$ the density field with
Gaussian scaling (rather than $n^{-3/4}$ prefactor as for $\mtt Y^n$: for any smooth $H\colon \T \to\R$,
\begin{equation}
{\color{blue} y^n_t(H)} \,:=\,
\frac{1}{\sqrt{n}\,}\sum_{i\in\T_n}\bar\eta_i (t)\, H(i/n)\;.
\end{equation}
\begin{theorem}
\label{theo_fastmodes}
Let $(\mu_n)_{n\ge 1}$ be a sequence of initial conditions satisfying
the hypotheses of Theorem~\ref{theo_convergence_magnetisation}. 
Assume either $a\leq \mf a_0$, 
or that Assumption~\ref{ass_LSI} holds as in Theorem~\ref{theo_convergence_magnetisation}.  

The
field $(y^n_t)_{t\in(0,T]}$ then converges in the sense of finite
dimensional distributions to a Gaussian field $y$ on
$(0,T]\times\{H\in \bb L^2(\T) : \<H\>=0\}$, with covariance:
\begin{equation}
\E \big[\,(y_t,H)\, (y_s,G) \, \big]
=
{\bf 1}_{t=s} \,
\Big\{\frac{1}{4}(H,G) + \frac{a}{2}\big(H,(-\Delta)^{-1} G\big)\Big\}\,
.
\end{equation}
In particular the limiting field is the same regardless of the initial condition.
\end{theorem}

\subsection{Closing equations: relative entropy bounds}\label{sec_sketch_proof_sec_results}
The main work in the proof of Theorems~\ref{theo_convergence_magnetisation}--\ref{theo_fastmodes} 
 consists in constructing a suitably precise approximation of the law of the microscopic dynamics at each time to allow for replacing microscopic observables by functions of the slow mode, the global magnetisation. 
The quality of this approximation is measured through relative entropy bounds in the spirit of the argument of Jara and
Menezes~\cite{Jara2018}. 
In this section we explain how to construct this approximation.

For two probability measures $\mu$, $\nu$ on
$\Omega_n$, 
recall that their relative entropy $H_n(\mu |\nu)$ is defined in~\eqref{2-34c}. 
Frequently we write
$\color{blue} H_n(f |\nu)$ instead of $H_n(\mu |\nu)$ if
$f = d\mu/d\nu$.

Fix a reference measure $\color{blue} \nu^n_{\rm ref}$ on $\Omega_n$ with $\inf_{\eta\in\Omega_n}\nu^{n}_{\rm ref}(\eta)>0$,
and denote by $\mu_n$ the initial distribution of particles in
$\Omega_n$.  Let $\color{blue} (S^n(t) : t\ge 0)$ be the semigroup
associated with the generator $\ms L_n = \sqrt{n}\, L_n$, and let
$f_t$ be the Radon-Nikodym derivative of the distribution of process
at time $t$ with respect to the reference measure $\nu^n_{\rm ref}$:
\begin{equation}
\label{2-34}
{\color{blue} f_t} \,=\, f_t^n
\;:=\; \frac{d\, \mu_n \,  S^n(t)}{d\, \nu^n_{\rm ref}}\; \cdot
\end{equation}
Thus, $H_n(f_t|\nu^n_{\rm ref})$ represents the relative entropy of
$\mu_n\, S^n(t)$ with respect to $\nu^n_{\rm ref}$:
\begin{equation}
\label{2-34b}
{\color{blue} H_n(f_t|\nu^n_{\rm ref})}\,:=\, 
\int f_t\, \log f_t \, d\nu^n_{\rm ref}\;. 
\end{equation}

The entropy inequality reads
\begin{equation}
\label{n-35}
\int F \; d\mu^n \,\le\, \frac{1}{\beta} \, H_n(\mu^n | \nu^n_{\rm ref})
\,+\, \frac{1}{\beta} \, \log \int e^{\beta F} \; d \nu^n_{\rm ref}
\end{equation}
for all $\beta>0$, probability measure $\mu^n$ and bounded continuous
function $F$. We refer to Section A.1.8 of \cite{kl} for a proof of this classical result.

The proof of Theorem \ref{theo_convergence_magnetisation} relies on a
a good estimate of $H_n(f_t|\nu^n_{\rm ref})$ (or more precisely of
the time-integrated carr\'e du champ $\ms G_n$ defined
in~\eqref{13b}). 
Once such an estimate is available, one can use the entropy inequality~\eqref{n-35} to bound expectations of observables under the law of the dynamics in terms of expectations under the reference measure.

If the reference measure is the invariant measure of the dynamics, 
then the relative entropy decreases in time by
\cite[Propostion A1.9.1]{kl}, 
so that $H_n(f^n_t|\nu^n_{\rm ref}) \le H_n(\mu^n|\nu^n_{\rm
ref})$. 
Unfortunately, the stationary state of the Markov process induced by
the generator $\ms L_n = \sqrt{n}\, L_n$ is not known. 
The dynamics is a superposition of
the exclusion dynamics, for which all Bernoulli product measures are
invariant, and the Glauber dynamics, with invariant measure given by
the nearest-neighbour Ising model described in~\eqref{eq_Ising}. 
The whole game is therefore to choose a reference measure as close as possible to the stationary state, and simple
enough to allow explicit computations. The better the approximation,
the more difficult the proof of a relative entropy estimate becomes. 

Evolution of the relative entropy with respect to a non-stationary reference measure is given by the
Jara--Yau lemma (see e.g.~\cite[Lemma A.1]{Jara2018})
\begin{equation}
\label{2-22b}
\frac{d}{dt}\, H_n(f_t | \nu^n_{\rm ref}) \;\leq\; -\,2\,
\ms G_n(f_t; \nu^n_{\rm ref})
\;+\; \int  f_t \, \ms L^*_{n} \, \mb 1   \; d \nu^n_{\rm ref} 
\end{equation}
for all $t\ge 0$, where $\color{blue} \mb 1 \colon \Omega_n \to \bb R$
is the constant function equal to $1$, $\color{blue} \ms L^*_{n}$
represents the adjoint of $\ms L_n$ in $\bbL^2(\nu^n_{\rm ref})$, and
$\ms G_n(h;\nu^n_{\rm ref})$ is the integral of the carr\'e du champ
associated with the generator $\ms L_n$, evaluated at a density $h$
with respect to $\nu^n_{\rm ref}$:
\begin{equation}
\label{13b}
{\color{blue}  {\ms G}_n(h; \nu^n_{\rm ref})}
\;:=\;  \int \frac{1}{2}\,  \big\{ \ms L_n h
\,-\, 2\, (\, \ms L_n \sqrt{h} \,) \,
\sqrt{h} \,\big\} \; d \nu^n_{\rm ref} \,\ge\, 0 \;.
\end{equation}

At this point, 
one would like to estimate the second term on the right-hand
side of \eqref{2-22b} in terms of the carr\'e du champ and the
entropy in order to set up a Gronwall estimate for $H_n(f^n_t|\nu^n_{\rm ref})$. More precisely the goal is to show that there exist finite constants $c_0\le
2$, $c_1$, and a sequence $(\alpha_n:n\ge 1)$ such that
\begin{equation}
\label{n-18}
\int  h \, \ms L^*_{n} \, \mb 1   \; d \nu^n_{\rm ref}
\,\le\, c_0\,
\ms G_n(h; \nu^n_{\rm ref}) \,+\, c_1 \, \big\{\,  H_n(f_t | \nu^n_{\rm
ref}) \,+\, \alpha_n \, \big\}
\end{equation}
for all density $h$ with respect to $\nu^n_{\rm ref}$. We call this
bound the {\it energy-entropy estimate}. Recollect the previous
estimates and apply the Gronwall inequality to obtain that
\begin{equation}
H_n(f_t | \nu^n_{\rm ref}) \,+\, (2-c_0)\,
\int_0^t \ms G_n(f_s; \nu^n_{\rm ref})\, ds \,\le\,
\big\{\,  H_n(f_0 | \nu^n_{\rm ref}) \,+\, \alpha_n \, \big\}
\, e^{c_1t} \quad\text{for all}\;\; t \,\ge\, 0\;.
\end{equation}
We next explain how to choose the reference measure.

\medskip\noindent{\bf Product measure.}  Recall the
definition~\eqref{eq_def_generator} of the generator.  The exclusion
dynamics in $L_n$ is speeded-up by $n^2$, in contrast with the Glauber
dynamics which is not speeded-up.  The exclusion part of the dynamics
restricted to a small enough neighbourhood of each macroscopic point
should therefore equilibrate much faster than the Glauber part.  A
first natural choice of reference measure is thus the Bernoulli
product measure - this is the local equilibrium measure discussed in the introduction.  Since we are interested in longer time-scales
corresponding to $\sqrt{n}\, L_n$, one should take density $1/2$ since the
density profile $\rho(x) = 1/2$ is the unique stationary solution of
the hydrodynamic equation \eqref{eq_LLN} when $\gamma=1/2$.  This is
the choice of reference measure made in \cite{jlv, bl10, LandimTsunoda,
flt} when deriving the hydrodynamic limit and the large deviations
principle for the empirical measure, and in \cite{Jara2018,jm_example, fls} to examine
nonequilibrium fluctuations of the model away from the critical point.

In the present case we claim that the Bernoulli
product measure with density $1/2$, $\nu^n_{1/2}$, 
is not a good reference measure to prove the entropy-energy estimate~\eqref{n-18}. 
This is perhaps not surprising as the density under this measure has Gaussian fluctuations (i.e. $\sum_i [\eta_i-1/2]\approx \sqrt{n}\,$), 
while the heuristic argument~\eqref{n-16}--\eqref{n-17} predicts $\sum_i [\eta_i-1/2]\approx n^{3/4}$. 

In Section \ref{n-sec1}, for pedagogical reasons, we compute the adjoint~\eqref{2-22b} with respect to the product measure. 
The following term appears in the computation,
see equation \eqref{n-19}:
\begin{equation}
\sqrt{n} \, \int  f_t \, 
\sum_{i\in \bb T_n} \bar\eta_i \bar\eta_{i+1}\; d \nu^n_{1/2}
,\qquad
\bar\eta_\cdot := \eta_\cdot - \frac{1}{2}
.
\label{eq_term_to_estimate_sec3_in_secresults}
\end{equation}
To control this term we separate its contribution to slow and fast modes, 
writing:
\begin{equation}
\bar\eta_i 
:=
\bar\eta_i^m + \ms M^n,\qquad 
\ms M^n := 
\frac{1}{n}\sum_{i\in\T_n}\bar\eta_i,\qquad 
i\in\T_n
.
\label{eq_mode_decomp_secresults}
\end{equation}
The magnetisation $\ms M^n$ is the slow mode, 
while $\bar\eta^m_\cdot$ converges much faster to its steady-state value, in a way that we make precise using the energy term $\ms G_n$ in~\eqref{2-22b} and a log-Sobolev inequality for the projection of the dynamics on the fast modes, 
see~\eqref{eq_estimate_mean0_modes_sec_product_measure}. 
Using the above decomposition, the term~\eqref{eq_term_to_estimate_sec3_in_secresults} contributes as follows to the magnetisation: 
\begin{equation}
\label{n-20}
n^{3/2} \, \int  f_t \, (\ms M^n)^2 
\; d \nu^n_{1/2}
=
\sqrt{n} \, \int  f_t \, (\cY^n)^2 
\; d \nu^n_{1/2}
. 
\end{equation}
This term depends only on the total magnetisation, 
which the energy $\ms G_n(f_s^n; \nu^n_{\rm ref})$ is
not helpful to estimate. This prevents us from obtaining suitable bounds on relative entropy with respect to $\nu^n_{1/2}$.  
To better control the magnetisation we therefore choose a different reference measure relying on the Glauber part of the dynamics. 

\medskip\noindent{\bf Product measure tilted by the magnetisation.}
In order to handle the term \eqref{n-20}, 
we present two corrections
to the Bernoulli product measure. In Section \ref{n-sec2}, we choose 
as reference measure 
\begin{equation}
\label{n-21}
{\color{blue} \nu^n_U (\eta)} \,:=\, \frac{1}{Z_n}\, e^{n U(\mf m/n)}
\,=\, \frac{1}{Z^n_U}\, e^{n U(\mf m/n)} \nu^n_{1/2} (\eta)\;, 
\end{equation}
where, recall, $\mf m = \sum_{i\in \bb T_n} [\eta_i - (1/2)]$, $Z_n$,
$Z^n_U$ are normalisation constants, and $U\colon [-1/2,1/2]\to \bb R$
is an appropriate potential. The second equality follows from the fact
that $\nu^n_{1/2}$ is the uniform measure.

The potential $U$ is chosen by examining a birth and death process
corresponding to the evolution of the magnetisation, assuming that the
exclusion part of the dynamics is strong enough to average out local
quantities and transform them into functions of the magnetisation. For
instance, the birth rate at $k = \sum_{i\in\bb T_n} \eta_i$ is given
by the expectation under the Bernoulli product measure with density
$k/n$ of $c(\eta) [1-\eta_0]$, 
see Section~\ref{n-sec2} for details. 
This approximation is very good to
control the contribution of the magnetisation to the relative entropy,
and cancels the term \eqref{n-20}.

The main result of Sections \ref{n-sec2}--\ref{n-sec3}, stated next, 
says that the entropy-energy estimate~\eqref{n-18} can be established using $\nu^n_U$ as reference measure, 
provided the reaction part is weak enough. 
\begin{theorem}
\label{n-s05}
There exists $\mf a_0>0$ and a finite constant $\mf c_0$ such that
\begin{equation*}
H_n(f_t | \nu^n_U) \,+\,  \int_0^t
\ms G_n(f_s; \nu^n_U) \, ds
\,\le\,  H_n(f_0 | \nu^n_U) \,+\,  a\, \mf c_0 \, \sqrt{n} \, t
\end{equation*}
for all $a\le \mf a_0$,  $t\ge 0$.
\end{theorem}

The constant $\mf a_0$ appearing in
Theorems~\ref{theo_convergence_magnetisation}--\ref{theo_fastmodes} is the one given in
Theorem~\ref{n-s05}.  This result is the main ingredient in the proofs
of Theorems~\ref{theo_convergence_magnetisation}--\ref{theo_fastmodes} in the $a\leq \mf a_0$
case.

Theorem~\ref{n-s05} is proven in Section~\ref{n-sec2}. To focus on the
key parts of the proof we have chosen to state some general estimates,
such as the concentration inequalities used to estimate~\eqref{n-20}, 
in the specific case of the
measure $\nu^n_U$.  Some of these results are needed in more general
form to treat the case where $a$ is large. 
They are therefore reproven in a more abstract and technical set-up in
Appendix~\ref{app_concentration}.

\medskip\noindent{\bf Product measure tilted by two point
correlation functions.}  
The argument presented in Section~\ref{n-sec2} compares the law of the dynamics to $\nu^n_U$ where all dependence between the magnetisation and other modes is neglected.  
This turns out to only be a good approximation when $a$ is small, 
a restriction needed to estimate the $a\, \sqrt{n}\, \sum_i\bar\eta^m_i\bar\eta^m_{i+1}$ term coming from~\eqref{eq_term_to_estimate_sec3_in_secresults} when using the mode decomposition~\eqref{eq_mode_decomp_secresults}. 
To overcome the restriction on $a$, 
we introduce
a second perturbation of the reference measure $\nu^n_{1/2}$, which
takes into account dynamical correlations between magnetisation and
other modes.  
Previous works carried out at high temperature~\cite{Correlations2022,Courant} suggest: 
\begin{enumerate}
\item[(a)] that it is enough to tune two-point correlations, all
higher-order correlations being negligible;

\item[(b)] that an effective way to do so is to tilt $\nu^n_{1/2}$ by a
discrete Gaussian term, considering a tilted measure of the form:
\begin{equation}
{\color{blue} \nu^n_g(\eta)} \,:=\, \frac{1}{Z^n_g}\, 
\exp\Big[\frac{1}{2n}\sum_{i\neq j}
\bar\eta_i\bar\eta_j g_{i,j}\Big] \, \nu^n_{1/2}(\eta) \, ,
\label{n-22}
\end{equation}
where $Z^n_g$ is the normalisation constant, $g\colon \bb T^2 \to \bb R$
is a smooth function ($g_{i,j}:= g(i/n,j/n)$) and $\color{blue} \bar\eta_i = \eta_i - (1/2)$.
\end{enumerate}
When $a=0$ there is no
reaction part and the dynamics is invariant with respect to the
product measure $\nu^n_{1/2}$.  We thus should have $g=0$ in this
case.  If on the other hand $\gamma=0$, then the reaction part of the
generator defines a dynamics that is reversible with respect to
$\nu^n_{1/2}$.  Thus we again expect $g=0$ in that case.  Enforcing
these two conditions and continuity in the parameters, we can uniquely
define $g$ as stated next. 
\begin{proposition}[Definition of $g$]\label{prop_g}
There is a unique family of functions
${\color{blue} g_{\delta,b}} \in C^\infty([0,1])\cap C(\T)$,
$\delta\in(-1,1)$, $b>0$, with the following properties.  The function
$g_{\delta,b}$ solves:
\begin{equation}
\label{eq_ODE_g}
\left\{
\begin{aligned}
&\displaystyle{
g''(x)  - c_{\delta, b} \, g(x) 
-\frac{1}{4} \int_{\T} g'(x-z)g'(z)\, dz}
+
\displaystyle{\frac{b}{2}\int_{\T} g(x-z)g(z)\, dz\,
=0, \quad x \in(0,1),}
\\ \\
&\displaystyle{g'(0_+)-g'(1_-)
= -16\delta b},
\end{aligned}
\right.
\end{equation}
where $c_{\delta, b} = 2b(1+2\delta)$, and
$b\mapsto \|g^0_{\delta, b}\|_2$ and
$\delta\mapsto \|g_{\delta, b}\|_2$ are continuous functions vanishing
at $0$.  Here, we defined
${\color{blue} g^0_{\delta,b}} := g_{\delta,b} -\int_{\T}
g_{\delta,b}$.

The function $g_{\delta,b}$ has the explicit Fourier decomposition:
\begin{equation}
g_{\delta,b}(x)
=
\min\{4,8\gamma\}{\mb 1} + g^0_{\delta,b},\quad
g^0_{\delta,b}(x)
=
\sum_{\ell\geq 1}\lambda^{-}_\ell 2\cos(2\pi\ell x)
,
\label{eq_Fourier_g}
\end{equation}
with $\lambda^-_\ell$ for $\ell\neq 0$ given by:
\begin{equation}
\lambda^-_\ell 
=
\frac{4\pi^2\ell^2+2b(1+2\delta)-\big|4\pi^2\ell^2 - 2b(1-2\delta)\big|}{b+2\pi^2\ell^2}
\geq 
0\;.
\label{eq_fourier_g_sec2}
\end{equation}
\end{proposition}
We just write {\color{blue} $g$ for $g_{\gamma,a}$} where $\gamma,a$
are the parameters appearing in the
definition~\eqref{eq_def_generator} of the dynamics.
Proposition~\ref{prop_g} is proven in Appendix~\ref{app_prop_g}.

In \eqref{n-22}, let $\color{blue} g_{i,j} = g((i-j)/n)$, where
$g(\cdot)$ is the function introduced in Proposition \ref{prop_g}.
Denote by $\cF(h | \nu^n_g , \kappa)$, $\kappa\ge 0$, the free energy
of a density $h$ with respect to the measure $ \nu^n_g$:
\begin{equation} 
{\color{blue} \cF( h \,|\, \nu^n_g , \kappa)}
\, := \, H_n(h|\nu^n_g)
\,+\, 
\kappa \, \sqrt{n}\, \int (\cY^n)^2 \, h\; d \nu^n_g
\,.
\label{eq_def_free_energy}
\end{equation}
The non-perturbative counterpart of Theorem~\ref{n-s05} for the free
energy~\eqref{eq_def_free_energy} is stated next and proven in Section~\ref{sec_free_energy_bounds}.
\begin{theorem}[Free energy bound]
\label{theo_free_energy}
Let $a>0$ and $\theta\in\R$.  Suppose the log-Sobolev inequality of
Assumption~\ref{ass_LSI} holds.  There is then $\kappa>0$ such that,
for each $T>0$, there is $C(T)>0$ with:
\begin{equation}
\cF( f_t \,|\, \nu^n_g , \kappa)
\, +\, \int_0^t
\ms G_n(f_s; \nu^n_g) \, ds
\,\le\,  
C(T)\, \cF( f_0 \,|\, \nu^n_g , \kappa)
\,+\,
C(T)\, \sqrt{n}
\end{equation}
for all $n\ge 1$ and $t\in[0,T]$.
\end{theorem}
\begin{remark}
We could also consider the probability measure
$\nu^n_{U,g^0}\propto \exp[\frac{1}{2n}\sum_{i\neq
j}g^0_{i,j}\bar\eta_i\bar\eta_j]\, \nu^n_{U}$ rather than $\nu^n_g$. 
This is discussed at the beginning of Section~\ref{sec_free_energy_bounds} where Theorem~\ref{theo_free_energy} is proven. 
\end{remark}
\begin{remark}
The control on the free energy stated in Theorem~\ref{theo_free_energy} appears to be stronger than the bound on relative entropy of Theorem~\ref{n-s05} as it also provides a bound on $\E_{\mu_n}[(\mc Y^n_t)^2]$. 
 
This is in fact not the case: the measure $\nu^n_U$ provides much better control on large values of the magnetisation than $\nu^n_g$, 
and an estimate on $\E_{\mu_n}[(\mc Y^n_t)^2]$ is a consequence of Theorem~\ref{n-s05} rather than requiring a separate argument. 
We shall indeed prove that there exists a finite constant $\mf c_0$ and
an integer $\mf n_0$ such that
\begin{equation}
\label{n-40}
\begin{aligned}
& \bb E^n \big[\,
\big(\mc Y^n_t\big)^2\, \big]
\, \le\; e^{4 a |\theta|  t}\, \Big\{\, 
E^n \big[\,
\big(\mc Y^n_0 \big)^2\, \big]
\,+\, \mf c_0 \, a\,
\Big( \frac{1}{\sqrt{n}}\, H_n (f_0 \,|\, \nu^{n}_{U})  \,+\, t\,\Big)
\, \Big\}
\end{aligned}
\end{equation}
for all $n\ge \mf n_0$, $t\ge 0$.
\end{remark}
\begin{remark}
The function $g$ and the correlation structure of the fast modes are related and can be deduced from one another. 
Let us give a heuristic, which has been made rigorous in high temperature settings, see e.g.~\cite[Proposition 2.2]{Courant}.  
It is proven in~\cite[Proposition 1.2]{Courant_suppl} that a measure of the form $\nu^n_{\tilde g}$ has covariance $(\sigma^{-1}\id-\tilde g)^{-1}$ ($\sigma=1/4$) with $\tilde g\in C(\T)$ having operator norm strictly below $1$ seen as an operator on $\bb L^2(\T)$.  
In the case $\tilde g=g$ the above operator is degenerate and only makes sense restricted to mean-0 test functions. 
The function $g$ is designed so that the measure $\nu^n_{g}$ approximates the covariance of fast modes under the dynamics, in the sense:
\begin{equation}
\E[y_t(G)y_t(H)]
=
\int_{\T} G(x)(4-g_0)^{-1}H(x)\, dx
,\qquad 
t>0, G,H\in\C_0
.
\end{equation}
Thus one expects the identity $\frac{\id}{4}+\frac{a}{2}(-\Delta^{-1})= (4\id-g_0)^{-1}$ relating covariance and $g$, 
an identity readily checked thanks to the formula~\eqref{eq_fourier_g_sec2} for the Fourier coefficients of $g_0$. 
\end{remark}
We conclude this section by stating that all initial distributions mentioned in Theorem~\ref{theo_convergence_magnetisation} have relative entropy and free energy bounded by $O(\sqrt{n}\, )$, 
therefore Theorems~\ref{n-s05}--\ref{theo_free_energy} provide $O(\sqrt{n}\, )$ bounds on relative entropy and free energy uniformly on $[0,T]$. 
\begin{proposition}\label{prop_free_energy_time0}
Let $\mu_n$ denote a probability measure on $\Omega_n$ ($n\in\N$) and assume:
\begin{equation}
\sup_{n\geq 1}\frac{H_n(\mu_n\, |\, \nu^n_{1/2})}{\sqrt{n}}
<
\infty
.
\end{equation}
Then:
\begin{equation}
\sup_{n\geq 1}\Big\{\frac{H_n(\mu_n\, |\,\nu^{n}_{\rm ref})}{\sqrt{n}}
\,+\, E_{\mu_n}[(\cY^n)^2]\Big\}
<
\infty
.
\end{equation}
\end{proposition}
Proposition~\ref{prop_free_energy_time0} is proven in Appendix~\ref{app_relative_entropy}.

\subsection{Discussion and perspectives}\label{sec_discussion}
\paragraph{Steady state fluctuations.}
The present analysis can be used to characterise magnetisation fluctuations under the invariant measure of the dynamics~\cite{Tsunoda_invmeas}. 

\paragraph{Higher dimensions.} 
By accelerating further the exclusion dynamics Theorem
\ref{theo_convergence_magnetisation} can be extended to higher
dimensions, at least when $a$ is small enough. Set
$\gamma = (1/2) [ 1 - (\theta/n^{d/2})]$, and recall the parameters
introduced in Remark \ref{rm1}.  Speed-up further the exclusion
dynamics by $\mg s_n := \sqrt{\log n}$ in dimension $2$ and by
${\sf s}_n := n^{(d-2)/2}$ in dimension $d\ge 3$.  The generator $L_n$
is thus
\begin{equation*}
L_n \,=\, n^{d/2} \, \big\{ {\sf s}_n \, n^2 \, L^{\rm ex}_n
\,+\, a \, L^{G}_n \,\big\} \;.
\end{equation*}
Applying the Jara-Menezes flow lemma \cite[Lemma 3.2]{Jara2018}, we
can adapt the proof of the energy estimate presented in Theorems
\ref{n-s02} and Lemma \ref{lemm_estimate_WphiJp} to higher dimensions,
as well as the proofs of the other results needed to establish Theorem
\ref{theo_convergence_magnetisation}.  Compare with the scaling in
Remark~\ref{rm1}: the extra acceleration of the diffusive part is an
artefact of the proof that we currently do not know how to remove.

For general $a$, our results rely on comparing the dynamics to the measure built from the kernel $g$ of  Proposition~\ref{prop_g}. 
In particular the relative entropy method makes use of regularity properties of $g$, 
since we replace discrete sums by integrals and discrete derivatives by their continuous counterparts. 
In higher dimension we do not expect $g$ to be more regular than the inverse of the Laplacian, as observed e.g. in~\cite{Correlations2022}. 
The fact that this inverse is not a function in dimension $d\geq 2$ would then introduce further difficulties to a generalisation of the proof to higher dimensions.

\paragraph{Beyond mean-field.}
In the dynamics~\eqref{eq_def_generator} the Kawasaki part of the dynamics is at infinite temperature. This implies that all modes except the magnetisation equilibrate fast at the critical point. 
A natural generalisation would be to consider a Kawasaki part with itself a (static or dynamical) phase transition. 
If at the transition only finitely many modes slow down in addition to the magnetisation, 
then one may hope to generalise the present results by similarly projecting microscopic observables on slow and fast modes and controlling the relaxation of the fast modes through appropriate log-Sobolev inequalities for projections of the dynamics on the fast modes. 
Although computations would be significantly more involved and, for instance, 
log-Sobolev inequalities for measures conditioned on more than the magnetisation (which are not known), 
would have to be established, 
there should be no qualitative difference to what is done in this article.

Settings where there is a diverging number of slow modes at the transition, 
such as one would expect if the Kawasaki part is e.g. reversible with respect to a model with short-range interactions, 
are beyond the scope of this article. 
In such situations there are many different time-scales for the dynamics, 
with faster modes having non-trivial influence on the slower ones. 
It would be very interesting to see how to adapt the methods of the present papers to settings where a good control of relaxation of the dynamics at each scale is available, 
such as hierarchical $\varphi^4$-type models~\cite{BauBo_hierarchical}.

\section{Entropy production for product reference measure}
\label{n-sec1}
\subsection{Generalities on entropy production}
Fix a reference probability measure $\nu^n_{\rm ref}$ on $\Omega_n$
with $\inf_{\eta\in \Omega_n}\nu^n_{\rm ref} (\eta)>0$.  Recall from
\eqref{13b} the definition of the integrated carr\'e du champ
$\ms G_n(h;\nu^n_{\rm ref})$ of $h$ with respect to
$\nu^n_{\rm ref}$ for the dynamics with generator $\ms L_n$.  
It is convenient in the following to explicitly keep track of the time acceleration factor $\sqrt{n}$, 
so we write the generator as $\sqrt{n} \, L_n$ and define the associated carr\'e du champ as:
\begin{equation}
{\color{blue} \Gamma_n(h; \nu^n_{\rm ref})}
\;: =\; n^{2} \, \Gamma^{\rm ex}_n(h; \nu^n_{\rm ref})
\;+\; a \, \Gamma^G_n(h; \nu^n_{\rm ref})\;
=\; 
n^{-1/2}\, \ms G_n(h;\nu^n_{\rm ref})\;,
\end{equation}
where
\begin{equation}
\begin{gathered}
{\color{blue} \Gamma^{\rm ex}_n(h; \nu^n_{\rm ref})}
\;:=\; \frac{1}{2}\, \sum_{i\in\bb T_n} 
\int  \big[\,  \sqrt{h (\eta^{i,i+1})} \,-\,
\sqrt{h(\eta)}\, \big]^2 \; d \nu^n_{\rm ref}\;, \\
{\color{blue} \Gamma^G_n(h; \nu^n_{\rm ref})}
\;:=\;  \frac{1}{2}\, \sum_{i\in\bb T_n} 
\int c(\tau_i \eta)\, \big[\,  \sqrt{h (\eta^{i})} \,-\,
\sqrt{h(\eta)}\, \big]^2 \; d \nu^n_{\rm ref} \;.
\end{gathered}
\label{eq_def_carreduchamp}
\end{equation}
The proof of the next result is similar to the one of
\cite[Proposition 5.1]{jl}. 
Note that in this computation the carr\'e
du champ corresponding to the Glauber part of the dynamics
appears,  but will not be used in the rest of the argument.

\begin{proposition}
\label{l06-1}
For all $t\ge 0$,
\begin{equation}
\label{2-22}
H'_n(f_t | \nu^n_{\rm ref}) \;\leq\; -\,2\,  \sqrt{n}\,
\Gamma_n(f_t ; \nu^n_{\rm ref})
\;+\;    \sqrt{n}\,  \int  f_t \, L^*_{n} \, \mb 1   \; d \nu^n_{\rm ref} \;,
\end{equation}
where, recall, $\mb 1 \colon \Omega_n \to \bb R$ is the constant
function equal to $1$, and $\color{blue} L^*_{n}$ represents the
adjoint of $L_n$ in $\bbL^2(\nu^n_{\rm ref})$.
\end{proposition}

Denote by $\color{blue} L^{*,\rm ex}_n$, $\color{blue} L^{*,G}_n$ the
adjoint in $\bbL^2(\nu^n_{\rm ref})$ of the generators $L^{\rm ex}_n$,
$L^{G}_n$, respectively, so that for any $F:\Omega_n\to\R$:
\begin{equation}
\label{eq_adjoint_general_measure}
L_n^*F \,=\, n^2 L^{*,\rm ex}_nF \,+\, a \, L^{*,G}_nF \;.
\end{equation}
An elementary computation yields (recall
definitions~\eqref{eq_def_eta_i_iplus1}--\eqref{eq_def_etai} of
$\eta^{i,i+1},\eta^i$ for $i\in\T_n$):
\begin{equation}
\begin{gathered}
L_n^{*,\rm ex} F(\eta)
\,=\, n^2\sum_{i\in\T_n}
\Big[F(\eta^{i,i+1})\frac{\nu^n_{\rm ref}(\eta^{i,i+1})}
{\nu^n_{\rm ref}(\eta)} - F(\eta)\Big] \;,
\\
L^{*,G}_nF (\eta)
\,=\, \sum_{i\in\T_n}
\Big[c(\tau_i(\eta^i))F(\eta^i)\frac{\nu^n_{\rm ref}(\eta^{i})}
{\nu^n_{\rm ref}(\eta)}-
c(\tau_i\eta)F(\eta)\Big] \;.
\label{eq_general_formula_adjoint}
\end{gathered}
\end{equation}
\subsection{The product case}
To present some of the key ideas and difficulties in the proof of Theorems~\ref{n-s05}--\ref{theo_free_energy}, 
we first consider the case where the reference measure is the Bernoulli product measure with density
$1/2$: $\nu^n_{\rm ref} = \nu^n_{1/2}$.  
This corresponds to comparing the dynamics to the local equilibrium, which away from a critical point is known to be a good approximation~\cite{Jara2018}. 
As discussed in the introduction and shown below, 
this will not be the case in our present near-critical setting.

As product measures are
reversible for the exclusion dynamics, if $L^{{\rm ex},*}_n$
represents the adjoint of $L^{{\rm ex}}_n$ in $\bbL^2(\nu^n_{1/2})$,
$L^{{\rm ex},*}_n = L^{{\rm ex}}_n$ so that
$L^{{\rm ex},*}_n \mb 1=0$. On the other hand, an elementary
computation, based on the explicit formula for the Glauber jump rates
\eqref{01}, yields that
\begin{equation}
\label{04}
a\, \sqrt{n}\, (L^{G,*}_n \mb 1)(\eta)
\; =\; a\, \sqrt{n}\, \sum_{i\in \bb T_n}
\big\{\,c( (\tau_i \eta)^i) \,-\, c(\tau_i\eta) \, \big\}
\;=\;
16\,a\,\gamma \, \sqrt{n}\,
\sum_{i\in \bb T_n} \bar\eta_i\bar\eta_{i+1}
\;,
\end{equation}
where $\bar\eta_i := \eta_i- (1/2)$.  Therefore, by
Proposition \ref{l06-1},
\begin{equation}
\label{n-19}
\begin{aligned}
H'_n(f_t | \nu^n_{1/2}) \; \leq\; & -\,2\,  \sqrt{n}\, \Gamma_n(f_t; \nu^n_{1/2})
\;+\;    16\, a\,  \gamma\, \int  f_t \, \sqrt{n} \,
\sum_{i\in \bb T_n} \bar\eta_i \, \bar\eta_{i+1}  \; d \nu^n_{1/2}
.
\end{aligned}
\end{equation}
To prove an $O(\sqrt{n})$ bound on
$H_n(f_t \,|\,\nu^n_{1/2}) + \int_0^t \sqrt{n}\, \Gamma_n(f_s;
\nu^n_{1/2})\, ds$ at each time $t\geq 0$, we would like to bound the
second term above as follows: for some $C(a)>0$,
\begin{equation}
16\, a\,  \gamma\,  \int  f_t \, \sqrt{n} \,
\sum_{i\in \bb T_n} \bar\eta_i \, \bar\eta_{i+1}  \; d \nu^n_{1/2}
\leq 
\,  \sqrt{n} \ \Gamma_n(f_t; \nu^n_{1/2})
+ 
C(a)\, \sqrt{n}
.
\label{eq_bound_eta_i_eta_i+1_by_carre_du_champ}
\end{equation}
At time $t\geq 0$, 
the Kawasaki part of the dynamics has run a time of order $n^{5/2}$, 
longer than the $n^2$ timescale at which relaxation to equilibrium would take place at fixed magnetisation. 
On the other hand, 
the $\sqrt{n}$ acceleration of the Glauber part is expected to be the time-scale at which fluctuations of the magnetisation $\sum_{i\in\T_n}\bar\eta_i$ evolve. 
We therefore consider the relaxation of $\sum_{i\in \bb T_n} \bar\eta_i \, \bar\eta_{i+1}$ due to the Glauber and Kawasaki parts separately as mentioned in Section~\ref{sec_sketch_proof_sec_results}, 
through the following projection. 
Define:
\begin{equation}
{\color{blue} \bar\eta^m}
\,:=\,
\bar\eta_i - \ms M^n,\qquad 
{\color{blue} \ms M^n }
:=
\frac{1}{n}\sum_{i\in\T_n}\bar\eta_i = \frac{\cY^n}{n^{1/4}}
\in \Big[-\frac{1}{2},\frac{1}{2}\Big]
.
\label{eq_recentering_sec_product}
\end{equation}
Then:
\begin{align}
16\, a\,  \gamma\,  \int  f_t \, \sqrt{n} \,
\sum_{i\in \bb T_n} \bar\eta_i \, \bar\eta_{i+1}  \; d \nu^n_{1/2} 
&=
16\, a\,  \gamma\,  \int  f_t \, \sqrt{n} \,
\sum_{i\in \bb T_n} \bar\eta^m_i \, \bar\eta^m_{i+1}  \; d \nu^n_{1/2}
\nnb
&\quad + 
16\, a\,  \gamma\, n \int f_t\,  (\cY^n)^2\, d\nu^n_{1/2}
.
\label{eq_mode_decomposition_sec_product}
\end{align}
The relaxation of the first term is governed by the Kawasaki part of the dynamics, 
which means in particular that $\bar\eta^m_\cdot$ has Gaussian fluctuations. 
We will indeed prove in the next section that, 
if $a$ is smaller than a universal constant:
\begin{equation}
\bigg|16\, a\,  \gamma\,  \int  f_t \, \sqrt{n} \,
\sum_{i\in \bb T_n} \bar\eta^m_i \, \bar\eta^m_{i+1}  \; d \nu^n_{1/2}\, \bigg|
\leq 
\frac{\sqrt{n}}{2}\, \Gamma_n^{\rm ex}(f_t;\nu^n_{1/2}) 
+ 
C(a)\sqrt{n}
.
\label{eq_estimate_mean0_modes_sec_product_measure}
\end{equation}
On the other hand, the magnetisation term
in~\eqref{eq_mode_decomposition_sec_product} is problematic.  To see
it, apply the entropy inequality to bound it for each $\lambda>0$ by:
\begin{align}
16\, a\,  \gamma\, n\,  \int f_t\,  (\cY^n)^2\, d\nu^n_{1/2}
\leq 
\frac{16\, a\,  \gamma\, \sqrt{n}\, H_n(f_t|\nu^n_{1/2})}{\lambda}
+ \frac{16\, a\,  \gamma\,\sqrt{n} }{\lambda}\log \nu^n_{1/2}
\Big( e^{\lambda \sqrt{n} (\cY^n)^2}\Big)
.
\end{align}
We need $\lambda$ small enough to make the exponential moment bounded
with $n$ (in fact $\lambda<2$ as follows from a similar proof as
Lemma~\ref{n-29}).  However, this restriction $\lambda<2$ is
incompatible with a bound on the entropy term.  Indeed, note first
that the factor $\sqrt{n}$ in front of the entropy prevents any
$O(\sqrt{n})$ bound on $H_n(f_t|\nu^n_{1/2})$ through Gronwall
inequality.  We therefore need to bound $H_n(f_t|\nu^n_{1/2})$ in
terms of the carr\'e du champ appearing
in~\eqref{eq_bound_eta_i_eta_i+1_by_carre_du_champ}.  This can be done
through the following log-Sobolev inequality for the product measure
$\nu^n_{1/2}$ with optimal constant:
\begin{equation}
H_n(f|\nu^n_{1/2})
\leq 
C_0\, \Gamma^G_n (f;\nu^n_{1/2})
,\qquad 
\text{for all }f:\Omega_n\to\R_+\text{ with }\nu^n_{1/2}(f)=1
.
\end{equation}
The constant $C_0$ can be computed explicitly from the jump
rates~\eqref{eq_value_jump_rates} and is at least $4+o_n(1)$ in the
regime $\gamma=1/2 +o_n(1)$ we are interested in.  One then has, for
any $\lambda\in(0,2)$ and some $C(\lambda)>0$:
\begin{align}
16 \, a\,  \gamma\, \sqrt{n} \int f_t\,  (\cY^n)^2\, d\nu^n_{1/2}
&\leq 
\frac{16\, \gamma\, C_0\, }{\lambda} \,
\sqrt{n}\, a\, \Gamma^G_n(f_t;\nu^n_{1/2})
+ a\, \sqrt{n}\, C(\lambda)
\nnb
&\leq 
\frac{16\, \gamma\, C_0\, }{\lambda} \,\sqrt{n}\,
\Gamma_n (f_t;\nu^n_{1/2})
+ a\, \sqrt{n}\, C(\lambda)
.
\end{align}
The choice $\lambda<2$ is thus incompatible with the fact that we would need $16\, \gamma \, C_0 /\lambda \leq 1$ to
prove~\eqref{eq_bound_eta_i_eta_i+1_by_carre_du_champ}.  Although we can
afford a slightly better factor in front of the carr\'e du champ
in~\eqref{eq_bound_eta_i_eta_i+1_by_carre_du_champ}, the gap between
the two restrictions on $\lambda$ cannot be bridged and
we cannot obtain the desired $O(\sqrt{n})$ bound on the relative entropy with respect to the product measure at each time.\\

Section~\ref{n-sec2} solves the above problem by considering a reference measure tailored to make the problematic magnetisation term disappear. 
This is enough to prove the energy--entropy estimate of Theorem~\ref{n-s05} for small $a$, 
where the restriction on $a$ this time comes from the estimate of the mean-0 modes~\eqref{eq_estimate_mean0_modes_sec_product_measure}. 
The measure is yet again refined in Section~\ref{sec_free_energy_bounds} to better estimate the mean-0 modes and obtain Theorem~\ref{theo_free_energy}, valid for all $a$.

\section{Entropy production for a density tilted product measure}
\label{n-sec2}

In this section, we show that correcting the local equilibrium
(Bernoulli product) measure with an appropriate function of the
magnetisation can improve the entropy estimate.  We start with a
heuristic argument to determine the correction.

Denote by $\color{blue} \nu^{n,k}_{1/2}$, $0\le k\le n$, the canonical
stationary state of the exclusion dynamics with $k$ particles. The
measure $\nu^{n,k}_{1/2}$ corresponds to conditioning the product
measure $\nu^n_{1/2}$ to the hyperplane of configurations with $k$
particles. In the context of the simple exclusion process, the measure
$\nu^{n,k}_{1/2}$ is the uniform measure of the configurations of
$\Omega_n$ with $k$ particles.

Recall the expression~\eqref{01} of the jump rates. 
In terms of $\bar\eta_i:=\eta_i-1/2$  ($i\in\T_n$), 
they read:
\begin{equation}
c(\tau_i\eta) 
=
1-4\, \gamma\, \bar\eta_i\, [\, \bar\eta_{i-1}+\bar\eta_{i+1}\, ] +4\, \gamma^2\, \bar\eta_{i-1}\,\bar \eta_{i+1}
.
\label{eq_jp_rates_as_bareta}
\end{equation}
Denote by $B_n(\cdot)$, $D_n(\cdot)$ the average birth and death jump
rates, respectively:
\begin{equation}
B_n(k) \;:=\, \nu^{n,k}_{1/2} \Big(\,
\sum_{i\in \bb T_n} (1-\eta_i) \, c(\tau_i\eta) \Big) \;,
\quad
D_n(k) \;:=\, \nu^{n,k}_{1/2} \Big(\,
\sum_{i\in \bb T_n} \eta_i \, c(\tau_i\eta) \Big)
\;.
\end{equation}
If the exclusion dynamics is sufficiently speeded-up, one expects the
creation and destruction rates in the full dynamics to approximately
coincide with their averages under the canonical measure with the
corresponding magnetisation.  The magnetisation would then evolve as a
birth and death process on $\{0, \dots, n\}$ with rates given by
$B_n$, $D_n$.

The stationary measure of this birth and death dynamics, denoted by
$\pi_n$, satisfies the detailed balance conditions
\begin{equation}
\label{n-05}
\pi_n (k)\, B_n(k) \,=\, \pi_n(k+1)\, D_n(k+1)\,, \quad 0\le k<n\;.
\end{equation}
A direct computation 
provides an 
estimate for these rates:
\begin{equation}
\begin{gathered}
B_n(k) \,=\,
n\,\big\{  \, B(\rho)  \,+\,
O\ \big (\, \frac{1}{n}\,\big) \, \big\}\;,
\quad B(\rho) \,=\, (1-\rho)\,  [\, 1 + \gamma\,  (2\rho -1 )\, ]^2
\,, 
\\
D_n(k) \;=\, 
n\,\big\{  \,  D(\rho) \,+\,
O\big (\, \frac{1}{n}\,\big) \, \big\} \;, \quad
D(\rho) \,=\,  \rho\,  [\, 1 - \gamma\,  (2\rho -1 )\, ]^2\;, 
\end{gathered}
\end{equation}
where $\rho=k/n$. 

It is then easily checked that the detailed balance condition
\eqref{n-05} holds with $\pi_n$ satisfying:
\begin{equation}
\pi_n (k) \,=\,  e^{- n \, \{ E (k/n) - U_0(k/n) \} + O(\log n) }
\, \sim \,
e^{n U_0 (k/n) } \, {n \choose k }\, \Big( \frac{1}{2}\Big)^n 
,
\end{equation}
where the $O(\log n)$ is uniform in $k$ and $E(\cdot)$ stands for the
entropy weight
$\color{blue} E(\rho) := \rho\, \log \rho + (1-\rho) \log (1-\rho) +
\log 2$.  On the other hand, the energy $U_0$ is defined for
$\rho\in[0,1]$ as:
\begin{equation}
\label{n-06} 
{\color{blue} U_0(\rho) \,:=\, U (\varrho)} \,:=\, \frac{1}{\gamma}\,
\Big\{ \, \big [ 1 + 2 \gamma  \varrho\, \big] \log \big[ \, 1 +
2 \gamma \varrho\, \big]
\,+\, \big [ 1 - 2 \gamma \varrho\, \big] \log \big[ \, 1 -
2 \gamma \varrho\, \big] \, \Big\} \;,
\end{equation}
where ${\color{blue} \varrho := \rho - (1/2)} \in [-1/2, 1/2]$. Note
that $U(0) \,=\, U'(0) \,=\, U^{(3)} (0) \,=\, 0$,
\begin{equation}
\label{n-34}
\begin{gathered}
U''(0) \,=\, 8\gamma\,=\, 4 \Big(\, 1 \,-\, \frac{\theta}{
\sqrt{n}}\,\Big)
\,, \quad 
U^{(4)} (0) \,=\, 64 \gamma^3 \,=\,
8 \Big(\, 1 \,-\, \frac{\theta}{
\sqrt{n}}\,\Big)^3\;,
\end{gathered}
\end{equation}
where $U^{(j)}$ stands for the $j$-th derivative of $U$. Moreover, $U$
is convex and symmetric (so that $U_0$ is symmetric around $1/2$).
Similarly, the entropy weight $E(\cdot)$ satisfies
$E(1/2) = E'(1/2) = E^{(3)}(1/2) =0$ and
\begin{equation}
\label{n-38}
E''(1/2) \,=\, 4 \,, \quad E^{(4)} (1/2) \,=\, 32 \;.
\end{equation}
Thus, by \eqref{n-34}, for $\gamma =1/2$, $\color{blue} W : = E - U_0$
is a non-negative convex function which vanishes at $\rho=1/2$ where
it is quartic.  
For future reference, note also 
\begin{equation}
\label{n-07}
U_0'(\rho) \,=\, U'(\varrho) \,=\, 2  \, \log \frac{ 1 + 2 \gamma \varrho}
{1 - 2 \gamma \varrho}\;
, \qquad
U_0''(\rho)
=
\frac{4\gamma}{1-4\gamma^2\varrho^2}   
\;\cdot
\end{equation}

In view of the previous discussion, let $U\colon [-1/2,1/2]\to \bb R$ be
given by \eqref{n-06}, and set
\begin{equation}
\label{eq_def_nu_V}
{\color{blue} \nu^n_U (\eta)}  
\,=\, \frac{1}{Z^n_U}\, e^{n U(\mf m/n)}\, \nu^n_{1/2} (\eta)
\,=\, \frac{1}{Z^n_U}\, e^{n U(\mf m/n)}\,
\, \Big( \frac{1}{2}\Big)^n  
\end{equation}
to be the reference measure.  Here 
$Z^n_U$ is a normalising constant, and, recall,
$\mf m = \sum_i \bar \eta_i$ so that $\mf m/n = \varrho$ in the
previous notation. The second identity follows from the fact that
$\nu^n_{1/2}$ is the uniform measure.

\subsection*{Computation of $L^*_n  \mb 1$}

As the exclusion dynamics does not change the magnetisation, if
$L^{{\rm ex},*}_n$ represents the adjoint of $L^{{\rm ex}}_n$ in
$\bbL^2(\nu^n_U)$, $L^{{\rm ex},*}_n \mb 1=0$. On the other hand, an
elementary computation yields that
\begin{equation}
\label{04b}
a\, \sqrt{n}\, (L^{G, *}_n \mb 1)(\eta)
\; =\; a\, \sqrt{n} \, \sum_{i\in \bb T_n}
\big\{\,c(\tau_i \eta^i) \, e^{ n[ U(\mf m(\eta^i)/n) - U(\mf m/n)]} 
\,-\, c(\tau_i\eta) \, \big\}\;.
\end{equation}
Using the expression~\eqref{eq_jp_rates_as_bareta} of the jump rates, rewrite~\eqref{04b} as
\begin{equation}
\label{n-01}
\begin{aligned}
& a \, \sqrt{n} \, 
e^{ (\nabla^+_n U)(\mf m/n) }
\sum_{i\in \bb T_n} (1-\eta_i) \,
\Big\{\, 1\, +\, 4\, \gamma \, \bar\eta_i\, [\, \bar\eta_{i-1} + \bar\eta_{i+1}
\,] \,+\,4\, \gamma^2 \, \bar\eta_{i-1}\, \bar\eta_{i+1} \,\Big\} 
\\
&  \;+\; a \, \sqrt{n} \, 
e^{ (\nabla^-_n U)(\mf m/n)} 
\sum_{i\in \bb T_n} \eta_i\,
\Big\{\, 1\, + 4\, \, \gamma \, \bar\eta_i\, [\, \bar\eta_{i-1} + \bar\eta_{i+1}
\,] \,+\,4\, \gamma^2 \, \bar\eta_{i-1}\, \bar\eta_{i+1} \,\Big\} 
\\
& -\, a\, \sqrt{n} \, \sum_{i\in \bb T_n}
\Big\{\, 1\,-\, 4\, \gamma \, \bar\eta_i\, [\, \bar\eta_{i-1} + \bar\eta_{i+1}
\,] \,+\,4\, \gamma^2 \, \bar\eta_{i-1}\, \bar\eta_{i+1} \,\Big\} \;,
\end{aligned}
\end{equation}
where $(\nabla^\pm _n U)(\mf m/n) = n\, [ U((\mf m \pm 1)/n) - U(\mf m/n)]$.

Recall that $\bar\eta^m_i = \eta_i - (m/n) = \bar \eta_i - (\mf m/n)$,
so that
\begin{equation*}
\begin{gathered}
\sum_{i\in \bb T_n}  \bar\eta_i\, \bar\eta_{i+k} \;=\;
 n\, (\mf m/n)^2 \,+\, G^{2,k}_{n,m} (\eta) \;,\;\;
k\,=\, 1\,,\, 2\,, 
\\
\sum_{i\in \bb T_n} \bar\eta_{i-1}\, \bar\eta_i\, \bar\eta_{i+1} \;=\;
n\,  (\mf m/n)^3 \,+\, F^{(3)}_{n,m} (\eta)
\end{gathered}
\end{equation*}
where
\begin{equation}
\label{n-15}
\begin{gathered}
G^{2,k}_{n,m}  (\eta) \,=\, \sum_{i\in \bb T_n}
\bar \eta^m_i\, \bar \eta^m_{i+k}\;,
\quad
G^{(3)}_{n,m}  (\eta) \,=\,
\sum_{i\in \bb T_n}
\bar \eta^m_{i-1}\, \bar \eta^m_i\, \bar \eta^m_{i+1}\;,
\\
F^{(3)}_{n,m}  (\eta) \,=\, G^{(3)}_{n,m}  (\eta)
\,+\, 2\, (\mf m/n)\, G^{2,1}_{n,m}
\,+\,  \, (\mf m/n)\, G^{2,2}_{n,m}\;.
\end{gathered}
\end{equation}
With this notation and noting that $\eta_i\bar\eta_i=1/2$,
$(1-\eta_i)\bar\eta_i=-1/2$,~\eqref{n-01} becomes:
\begin{equation}
\label{n-02}
\begin{aligned}
a\, \sqrt{n}\, (L^{G, *}_n \mb 1)(\eta)
\; & =\; a\, n^{3/2} \,  
e^{ (\nabla^+_n U)(\mf m/n) }\,  [(1/2) - (\mf m/n)] \, [\, 1 -
2\gamma (\mf m /n) ]^2
\\
& \,+\,  a\, n^{3/2} \,   e^{ (\nabla^-_n U)(\mf m/n) }
\, [(1/2) + (\mf m/n)] [\, 1 + 2\gamma (\mf m/n) ]^2
\\
& -\, a\, n^{3/2} \, \big\{ \,1 \,+\, 4\, (\gamma^2 - 2 \gamma) 
\, (\mf m/n)^2 \, \big\}
\,+\, a \, \sqrt{n}\,\sum_{k=1}^3 R^{(k)}_{n,m}\;,
\end{aligned}
\end{equation}
where
\begin{equation}
\label{n-14}
\begin{gathered}
R^{(1)}_{n,m} \,=\,   e^{ (\nabla^+_n U)(\mf m/n) }\,
\big\{\, 4\,\gamma\, G^{2,1}_{n,m} \,+\, 2\,\gamma^2 \, G^{2,2}_{n,m}
\,-\, 4\, \gamma^2 \, F^{(3)}_{n,m} \,\big\} \;, 
\\
R^{(2)}_{n,m} \,=\,   e^{ (\nabla^-_n U)(\mf m/n) }\,
\big\{\, 4\,\gamma\, G^{2,1}_{n,m} \,+\, 2\,\gamma^2 \, G^{2,2}_{n,m}
\,+\, 4\, \gamma^2 \, F^{(3)}_{n,m} \,\big\} \;,
\\
R^{(3)}_{n,m} \,=\,  
\big\{\, 8\,\gamma\, G^{2,1}_{n,m} \,-\, 4\,\gamma^2 \, G^{2,2}_{n,m}
\,\big\} \;.
\end{gathered}
\end{equation}

By \eqref{n-07}, the absolute value of $U''$ is bounded in
$[-1/2, 1/2]$. Therefore,
$(\nabla^\pm_n U)(\mf m/n) = \pm U' (\mf m/n) + r_n/n$ for some $r_n$
uniformly bounded in $\mf m$ and $n$. Thus, by \eqref{n-07}, expanding
$(\nabla^\pm_n U)(\mf m/n)$ yields to an exact cancellation of the terms with prefactor $n^{3/2}$ so that 
\begin{equation}
\label{n-09}
a\, \sqrt{n}\, (L^{G, *}_n \mb 1)(\eta)
\; =\; a\, \sqrt{n} \, h_n (\eta) \, \,+\,a \, \sqrt{n}\, \sum_{k=1}^3 R^{(k)}_{n,m}(\eta)
\end{equation}
where $\sup_{n\ge 1} \max_{\eta\in\Omega_n} |h_n(\eta)|\le \mf c_0$
for some finite constant $\mf c_0$. 
Assuming the bounds on the $R^{(k)}_{m,n}$ proven below in Theorem~\ref{n-s02}, 
we may now prove the entropy
estimate. 

\begin{proof}[Proof of Theorem \ref{n-s05}]
By Proposition \ref{l06-1} and Equation \eqref{n-09},
\begin{equation}
H'_n(f_t | \nu^n_U) \;\leq\; -\,2\,  \sqrt{n}\, \Gamma_n(f_t; \nu^n_U)
\;+\;    a\, \mf c_0 \, \sqrt{n}
\,+\, a\, \sqrt{n}\, \sum_{k=1}^3  \int  f_t \, R^{(k)}_{n,m}(\eta)  \; d \nu^n_U \;,
\label{eq_der_entropy_with_error_terms_sec4}
\end{equation}
for some finite constant $\mf c_0$. By \eqref{n-14}, \eqref{n-15}, and
since the absolute value of $U''$ is bounded in $[-1/2, 1/2]$, by
Theorem \ref{n-s02} and Remark \ref{n-s04}, the previous expression is
bounded by
\begin{equation*}
H'_n(f_t | \nu^n_U) \;\leq\; (a\, \mf c_0 \,
-\,2) \,  \sqrt{n}\, \Gamma_n(f_t \,|\, \nu^n_U)
\;+\;    a\, \mf c_0 \, \sqrt{n} \;,
\end{equation*}
where the value of the constant $\mf c_0$ has changed. Set $\mf a_0 =
1/\mf c_0$ to complete the proof of the theorem.
\end{proof}

	\section{An estimate on the average of local functions}
\label{n-sec3}

We prove in this section the bounds on the terms $R^{(k)}_{m,n}$ ($1\leq k\leq 3$) appearing in~\eqref{eq_der_entropy_with_error_terms_sec4} used in the last section. 

\begin{theorem}
\label{n-s02}
Let $\mc Z_n \colon \Omega_n \to \bb R$ be a sequence of uniformly
bounded random variables which are invariant by the exclusion
dynamics:
$\mf c_Z := \sup_{n\ge 1} \Vert \mc Z_n\Vert_\infty <\infty $;
$\mc Z_n(\eta^{i,i+1}) = \mc Z_n(\eta)$ for all $i\in \bb T_n$,
$n\ge 1$.  There exist finite constants $\mf c_1$, $\mf c_2$ 
depending only on $\mf c_Z$ and on the variable $\eta_0\eta_1$, and an
integer $\mf n_0$ such that:
\begin{equation*}
\Big| \, E_{\nu^{n}_{U}} \Big[\, f\, \mc Z_n\, 
\sum_{i\in \bb T_n} \bar\eta^m_i \, \bar\eta^m_{i+1}
\, \Big] \, \Big|
\, \le\; \mf c_1 \,+\, \mf c_2\,  n^{2}\,
\Gamma_n^{\rm ex}(f; \nu^{n}_{U})
\end{equation*}
for all $n\ge \mf n_0$, and density $f$ with respect to $\nu^{n}_{U}$.
\end{theorem}

The requirement that $n\ge \mf n_0$ comes from the local central limit
theorem for i.i.d.\! random variables, see Lemma \ref{n-l02}.

\begin{remark}
\label{n-s04}
The proof presented in this section applies to the weighted measure
$\nu^n_U$, to the product measure $\nu^n_{1/2}$ and to the canonical
measure $\nu^{n,k}_{1/2}$, $0\le k\le n$. We present the proof for the
measure $\nu^n_{1/2}$ to focus on the main steps.  The argument also
applies to the cylinder functions $\bar\eta^m_{-1} \, \bar\eta^m_{1}$
and $\bar\eta^m_{-1}\, \bar\eta^m_{0}\, \bar\eta^m_{1}$ or to any
other finite product of this type. We indicated below the
modifications needed in the proof to cover these other cases.
\end{remark}

The next result is a consequence of Theorem \ref{n-s02}.

\begin{corollary}
\label{n-s09}
There exist finite constants $\mf c_1$, $\mf c_2$, and an integer
$\mf n_0$ such that
\begin{equation*}
\begin{aligned}
& \bb E_{\mu_n} \big[\,
\big(\mc Y^n_t\big)^2\, \big]
\,+\, 16 \, a\, \gamma^2 \,   \int_0^t  \bb E_{\mu_n} \big[\,
(\mc Y^n_s )^4 \, \big]\; ds  
\\
&\;\;
\, \le\;  E_{\mu_n} \big[\,
\big(\mc Y^n \big)^2\, \big]
\,-\, 4 \, a\, \theta \,   \int_0^t  \bb E_{\mu_n} \big[\,
(\mc Y^n_s )^2 \, \big]\; ds  
\,+\, \mf c_1 \, t \, +\,
\mf c_2\,  n^{2}\, \int_0^t \Gamma_n^{\rm ex}(f_s ; \nu^{n}_{U})
\; ds 
\end{aligned}
\end{equation*}
for all $n\ge \mf n_0$, probability measures $\mu_n$ and $t\ge 0$. In this formula,
$f_s = d\mu_n S^n(s)/ d\nu^n_{U}$ as usual.
\end{corollary}

\begin{proof}
By the semi-martingale decomposition of the square of the
magnetisation, 
\begin{equation}
\bb E_{\mu_n} \Big[\,
\big(\mc Y^n_t\big)^2\, \Big] \;=\;
\bb E_{\mu_n} \Big[\,
\big(\mc Y^n_0\big)^2\, \Big]
\,+\, \bb E_{\mu_n} \Big[\, \int_0^t 
\sqrt{n}\, L_n \big(\mc Y^n_s\big)^2\; ds \, \Big] \;.
\end{equation}
We estimate the second term. Since the exclusion dynamics does not
modify the magnetisation and $1-2 \gamma = \theta/ \sqrt{n}$, 
one has recalling the expression~\eqref{eq_jp_rates_as_bareta} of the jump rates:
\begin{equation}
\begin{aligned}
\sqrt{n}\, L_n \big(\mc Y^n\big)^2 \, & =\,
-\, \frac{4a}{n^{1/4}} \, \mc Y^n \sum_{i\in\bb T_n} c(\tau_i \eta) \,
\bar\eta_i\;+\; \frac{a}{n} \sum_{i\in\bb T_n} c(\tau_i \eta)
\\
\, & =\,
-\, 4\,a\, \theta \, (\mc Y^n)^2
\;-\; 16 \, a\, \gamma^2 \, \ms M^n\, \sum_{i\in\bb T_n}
\bar\eta_{i-1}\bar\eta_i\bar\eta_{i+1}
\;+\;
\frac{a}{n} \sum_{i\in\bb T_n} c(\tau_i \eta) \;, 
\end{aligned}
\label{eq_th51_0}
\end{equation}
where $\ms M^n = \mc Y^n/n^{1/4}=\sum_{i\in\T_n}\bar\eta_i = \mf m/n$,
introduced in \eqref{eq_recentering_sec_product}, is bounded in
absolute value by $1/2$.

The last term on the right-hand side of~\eqref{eq_th51_0} is
bounded by $(1+\gamma)^2a$. The second one can be written as
\begin{equation}
\begin{gathered}
\;-\; 16 \, a\, \gamma^2 \, \ms M^n\, \sum_{i\in\bb T_n}
\bar\eta_{i-1} \bar\eta_i   \bar\eta_{i+1} \, =\,
\; -\; 16 \, a\, \gamma^2 \,  (\mc Y^n )^4
\,-\,16 \, a\, \gamma^2 \,  \mc A_n
\\
\text{where}\quad
\ms A_n \,=\, 
\ms M^n\, \sum_{i\in\bb T_n}
\bar\eta^m_{i-1} \bar\eta^m_i   \bar\eta^m_{i+1}
\, +\,  (\ms M^n)^2 \,
\Big\{ 2\, \sum_{i\in\bb T_n} \bar\eta^m_{i} \bar\eta^m_{i+1}
\,+\, \sum_{i\in\bb T_n} \bar\eta^m_{i-1} \bar\eta^m_{i+1} \,\Big\}\;.
\end{gathered}
\end{equation}
Since $\ms M^n$ is bounded and invariant by the exclusion dynamics, by
Theorem \ref{n-s02}, 
\begin{equation}
\begin{aligned}
\bb E_{\mu_n} \Big[\, \int_0^t 
\mc A_n (s)  \; ds \, \Big]
\;=\; \int_0^t  E_{\nu^n_{U}} \big[\, f_s \;
\mc A_n \,\big] \; ds
\,\le\,  \mf c_1 \, t \, +\,
\,  n^{2}\, \int_0^t \Gamma_n^{\rm ex}(f_s ; \nu^{n}_{U})
\; ds \;, 
\end{aligned}
\end{equation}
for all sufficiently large $n$.  In this formula,
$f_s = d\, \mu_n\, S^n(s)/ d\, \nu^n_{U}$ and the values of the constant
$\mf c_1$ may have changed. To complete the proof it
remains to recollect all previous estimates.
\end{proof}

The next corollary is a consequence of the previous result and Theorem
\ref{n-s05}. 

\begin{corollary}
\label{n-s11}
There exists a finite constant $\mf c_3$ and an integer $\mf n_0$ such
that
\begin{equation*}
\begin{aligned}
& \bb E_{\mu_n} \big[\,
\big(\mc Y^n_t\big)^2\, \big]
\,+\, 16 \, a\, \gamma^2 \,   \int_0^t  \bb E_{\mu_n} \big[\,
(\mc Y^n_s  )^4 \, \big]\; ds  
\\
&\quad 
\, \le\; e^{4 a |\theta|  t}\, \Big\{\,  E_{\mu_n} \big[\,
\big(\mc Y^n \big)^2\, \big]
\,+\, \mf c_3 \, a\,
\big\{ \, n^{-1/2} \, H_n (f_0 \,|\, \nu^{n}_{U})  \,+\, t\,\big\}\,
\Big\}
\end{aligned}
\end{equation*}
for all $n\ge \mf n_0$, $t\ge 0$. In this formula, $f_0 = d\mu_n / d\nu^n_U$.
\end{corollary}

\begin{proof}
By Theorem \ref{n-s05}, the sum of the last two terms appearing on the
right-hand side of the statement of Corollary \ref{n-s09} is less than or
equal to
\begin{equation}
\mf c_3 \, \big\{\, t \, +\, n^{-1/2}\, H_n (f_0 \,|\, \nu^{n}_{U})
\,\big\} 
\end{equation}
for some finite constant $\mf c_3$. Therefore, by Corollary
\ref{n-s09}, there exist a finite constant $\mf c_3$ and an integer
$\mf n_0$ such that
\begin{align}
\mb e_n(t)  \,&+\,
16 \, a\, \gamma^2 \,  \bb E_{\mu_n} \Big[\, \int_0^t 
(\mc Y^n (s) )^4   \; ds \, \Big]
\nnb
&\qquad 
\le\; \mb e_n(0)
\;-\; 4\,a\, \theta\, \int_0^t \mb e_n(s) \; ds 
\,+\, \mf c_3 \, a\, \big\{ \, n^{-1/2} \, H_n (f_0 \,|\, \nu^{n}_{U})
\,+\, t\,\big\}
\end{align}
for all $n\ge \mf n_0$, $t\ge 0$.  Here $\mb e_n(t) = \bb
E_{\mu_n} [\, (\mc Y^n_t)^2\, ] $.
To complete the proof of the corollary, it remains to apply Gronwall's
estimate. 
\end{proof}

We turn to proof of Theorem \ref{n-s02} which is divided in several
steps.  Recall that we present the proof taking as reference measure
the product measure $\nu^n_{U}$, but all arguments apply to
$\nu^n_{1/2}$ or $\nu^{n,k}_{1/2}$. 
The starting point is an energy estimate.

\subsection*{Step 1: Energy estimate}

We start with an integration by part formula. This result is well
known, 
and proven in Appendix~\ref{sec_IBP} in a more general framework useful in Section~\ref{sec_free_energy_bounds}. 
For a probability measure $\nu$ on $\Omega_n$, $i\in \bb T_n$, and a
density $f$ with respect to $\nu$, let 
\begin{equation*}
\Gamma^{{\rm ex},i,i+1}_{n}(f;\nu)
\;:=\; \frac{1}{2}\, 
\int  \big[\,  \sqrt{f (\eta^{i,i+1})} \,-\,
\sqrt{f(\eta)}\, \big]^2 \; d \nu\;.
\end{equation*}

\begin{lemma}
\label{n-s03}
For all $n\ge 2$,
\begin{equation*}
\int h \, [\bar\eta_i^m - \bar\eta^m_{i+1}]\, f\; d\nu^n_{U}
\;\leq\; \beta  \, \Gamma^{{\rm ex},i,i+1}_{n}(f; \nu^n_{U})
\;+\;  \frac{1}{2\beta}  \, \int h^2 \, f  \; d\nu^n_{U} 
\end{equation*}
for all $\beta >0$, $i \in \bb T_n$, $h\colon \Omega_n \to \bb R$ such that
$h(\eta^{i,i+1}) = h(\eta)$ for all $\eta \in \Omega_n$, and
density $f\colon \Omega_n \to [0,\infty)$ with respect to $\nu^n_{U}$.
\end{lemma}

\begin{lemma}
\label{n-l01}
Let $\mc Z_n \colon \Omega_n \to \bb R$ be a sequence of random
variables satisfying the hypotheses of Theorem \ref{n-s02}.  There
exist finite constants $C_0$, $C_1$ which depend only on the local
function $\eta_0\eta_1$ and on
$\sup_{n\ge 1} \Vert \mc Z_n\Vert_\infty$ such that
\begin{equation}
\label{n-13}
\begin{aligned}
& \Big| \, \nu^{n}_{U} \Big( f\, \mc Z_n\, 
\sum_{i\in \bb T_n} \bar\eta^m_i \, \bar\eta^m_{i+1}
\Big) \, \Big|
\\
&\quad
\le\; C_0  \,+\,
\delta \,  n^{2}\, \Gamma_n^{\rm ex}(f; \nu^{n}_{U})
\,+\,
\frac{C_1} {\delta\, n} \,
\sum_{j\in\bb T_n} \nu^{n}_{U}
\Big[f \,\Big( \mc Z_n\, \frac{1}{\sqrt{n} }
\sum_{i\neq j+1-n} I_n\big(j-(i+1)\big)\, \bar\eta^m_{i}  \Big)^2\Big]\,,
\end{aligned}
\end{equation}
for all $\delta>0$ and density $f$ with respect to $\nu^{n}_{U}$.
\end{lemma}

\begin{proof}
As $\sum_{1\le j\le n} \bar\eta^m_{j} =0$, 
\begin{equation*}
\sum_{i\in\T_n}\bar\eta^m_i \bar\eta^m_{i+1} \,=\, 
\frac{1}{n}\sum_{i\in\T_n}\bar\eta^m_i
\sum_{j=1}^{n} [\, \bar\eta^m_{i+1} - \bar\eta^m_{i+j} \,]
=
\sum_{i\in\T_n}\bar\eta^m_i
\sum_{k=1}^{n-1} \, [1-(k/n)]\, (\bar\eta^m_{i+k}-\bar\eta^m_{i+k+1})
,
\end{equation*}
Performing the change of variables $j=k+i$ and exchanging sums permits
to rewrite this expression as
\begin{align}
\label{03}
\sum_{j\in\T_n}& (\bar\eta^m_{j}-\bar\eta^m_{j+1}) 
\sum_{i\in\T_n} I_n(j-(i+1))\, \bar\eta^m_i
\nnb
&\hspace{2cm}=
\sum_{j\in\T_n} (\bar\eta^m_{j}-\bar\eta^m_{j+1}) 
\sum_{i\neq j+1} I_n\big(j-(i+1)\big)\, \bar\eta^m_i
+\frac{1}{n}\, \sum_{j\in\T_n}(\bar\eta^m_{j}-\bar\eta^m_{j+1})\bar\eta^m_{j+1}
\; ,
\end{align}
where $I_n$ is given by:
\begin{equation}
I_n : p\in\Z
\longmapsto 
\frac{n-1-p}{n}\, {\bf 1}_{[0,n-1]}(p)
.
\end{equation}
In~\eqref{03} we disassociate the term $i=j+1-n$ from the sum because it is the only
term in the sum over $i$ which depends on the variables $\eta_j$,
$\eta_{j+1}$.  The absolute value of the second sum is bounded by
$1$. 
It corresponds to the first contribution on the right-hand side of
the inequality~\eqref{n-13} stated in the lemma:
\begin{equation}
\Big|\, \nu^n_U\Big( f\, \mc Z_n\, \frac{1}{n}\,
\sum_{j\in\T_n}(\bar\eta^m_{j}-\bar\eta^m_{j+1})\bar\eta^m_j\Big)\, \Big|
\leq 
C_0\, ,
\qquad 
C_0 := \sup_{n\geq 1}\Vert \mc Z_n\Vert_\infty
. 
\end{equation}
 We turn to the first term in~\eqref{03}.

Since $\mc Z_n(\eta^{i,i+1}) = \mc Z_n(\eta)$ for all $i\in\bb T_n$,
$n\ge 1$, by the integration by parts formula stated in
Lemma~\ref{n-s03}, 
with $h = \mc Z_n\, \sum_{i\neq j+1}I_n\big(j-(i+1)\big)\,\bar\eta^m_j$, 
\begin{equation}
\label{n-39}
\begin{aligned}
& \Big| \nu^{n}_{U} \Big( f\, \mc Z_n\, 
\sum_{j\in\T_n} (\bar\eta^m_{j}-\bar\eta^m_{j+1}) 
\sum_{i\neq j+1-n} I_n\big(j-(i+1)\big)\, \bar\eta^m_{i}  \Big)\Big|
\nnb
&\quad\leq 
\delta\, n^{2} \,  \Gamma^{\rm ex}_n(f ; \nu^{n}_{U})
\,+\, \frac{1} {2\, \delta\, n} \, \sum_{j\in\bb T_n} \nu^{n}_{U}
\Big[f \,\Big( \mc Z_n\, \frac{1}{\sqrt{n} }
\sum_{i\neq j+1-n} I_n\big(j-(i+1)\big)\, \bar\eta^m_{i}  \Big)^2\Big]
\end{aligned}
\end{equation}
for all $\delta>0$. On the second term of the right-hand side, we may
bound $\mc Z_n^2$ by $C_1 = \sup_n \Vert \mc Z_n\Vert^2$. This
completes the proof of the lemma.

\end{proof}

\subsection*{Step 2: Logarithmic Sobolev inequality}

The renormalised terms in the right-hand side of~\eqref{n-13} are now
sufficiently averaged in space for the entropy and log-Sobolev
inequality to provide a bound as we now explain.  The next lemma gives
a first bound using the log-Sobolev inequality for the fast part of
the dynamics (i.e. for Kawasaki dynamics at each fixed value of the
magnetisation).

\begin{lemma}
\label{n-s01}
There exists a finite constant $\mf c_{\rm LS}$ such that
\begin{align*}
& \nu^{n}_{U}
\Big[f \,\Big(\, \frac{1}{\sqrt{n} }
\sum_{i\neq j+1-n} I_n\big(j-(i+1)\big) \, \bar\eta^m_{i} \Big)^2\Big]
\le\,
\frac{1}{\lambda} \, \mf c_{\rm LS} \, n^2\,
\Gamma^{\rm ex}_n (f; \nu^{n}_{U})
\\
& \qquad +\, \frac{1}{\lambda}
\sum_{m=0}^n \nu^{n}_{U}(f\, {\bf 1}_m)\, \log \,
\nu^{n,m}_{1/2} \Big( \exp  \Big\{ \lambda\, \Big(\frac{1}{\sqrt{n} }
\sum_{i\neq j+1-n} I_n\big(j-(i+1)\big) \, \bar\eta^m_{i}\Big)^2
\Big\} \, \Big) 
\end{align*}
for all $n\ge 1$, $\lambda>0$ and densities $f$ with respect to $\nu^{n}_{U}$.
\end{lemma}

\begin{proof}
The proof does dot depend on the particular choice of the function
$\frac{1}{\sqrt{n} } \sum_{i\neq j+1-n} I_n\big(j-(i+1)\big) \,
\bar\eta^m_{i}$, so we give a proof for a general
$X_n\colon \Omega_n\to\R$.  Decomposing the expectation according to
the total number of particles yields that
\begin{align*}
& \nu^{n}_{U}
\big[f\,  X_n\big]
\; =\; 
\sum_{m=0}^n \nu^{n}_{U}(f\, {\bf 1}_m)\,
\nu^{n,m}_{1/2} \big(f^m \, X_n\, \big)
\;,
\end{align*}
where
${\color{blue} f^m: \Omega_{n,m}} := \{\eta\in\Omega_n : \sum_i \eta_i
= m\} \to \bb R_+$ represents the density $f$ projected on the
hyperplane of configurations with $m$ particles:
$f^m (\eta) = f(\eta) / \nu^{n,m}_{1/2} (f)$, $\eta\in
\Omega_{n,m}$. Mind that $f^m$ is a density with respect to the
measure $\nu^{n,m}_{1/2}$.  By the entropy inequality \eqref{n-35},
the expectation is bounded for all $\lambda>0$ by
\begin{equation*}
\frac{1}{\lambda} \, H_n (f^m |\nu^{n,m}_{1/2})
\,+\, \frac{1}{\lambda} \log \,
\nu^{n,m}_{1/2} \Big( e^{\lambda\, X_n}\, \Big)
\end{equation*}
By the logarithmic Sobolev inequality
\cite[Theorem 4]{yau}, the first term is bounded by
\begin{equation}
\frac{1}{\lambda} \, \mf c_{\rm LS} \, n^2\,
\Gamma^{\rm ex}_n (f^m; \nu^{n,m}_{1/2}) 
\label{eq_LSI_Yau_sec_5}
\end{equation}
for some finite constant $\mf c_{\rm LS}$ independent of $m,n$. 
To complete the proof of
the lemma, observe that
\begin{equation*}
\sum_{m=0}^n \nu^{n}_{U}(f {\bf 1}_m)\,
\Gamma^{\rm ex}_n (f^m; \nu^{n,m}_{1/2})
\;=\;
\Gamma^{\rm ex}_n (f; \nu^n_{U})\;.
\end{equation*}
\end{proof}

\subsection*{Step 3: Concentration inequality}

In this subsection, we estimate the exponential moment appearing in the statement of Lemma~\ref{n-s01}, 
presenting a general concentration inequality for 
canonical measures. 
Given a sequence of functions
$a_n: \bb T_n \to \bb R$, denote by $\Vert a\Vert_\infty$ the supremum
of the sup norms:
$\color{blue} \Vert a\Vert_\infty = \sup_{n\ge 1} \max_{i\in\T_n}
|a_{n,i}|$, where $a_{n,i} := a_n(i)$.

\begin{proposition}
\label{n-l03}
There exists an integer $\mf n_1$ such that
\begin{equation*}
\nu^{n,m}_{1/2} \Big( \exp  \Big\{ \alpha \, \Big(\frac{1}{\sqrt{n}}
\sum_{i\in\T_n} \, a_{n,i} \, \bar\eta^m_{i} \Big)^2
\Big\} \, \Big) \,\le \, 16 \, e^{8 \alpha \Vert a\Vert^2_\infty}
\end{equation*}
for all $n\ge \mf n_1$, $0\le m\le n$, sequence of functions
$a_n: \bb T_n \to \bb R$, and
$\alpha \le (4 \Vert a\Vert^2_\infty)^{-1}$.
\end{proposition}
\begin{remark} 
The result is also valid in the case where the local function
$\bar\eta^m_{0}$ in the sum
$\sum_{1\le k\le n} \, a_{n,k} \, \bar\eta^m_{-k}$ is replaced by
$h(\eta) = \prod_{j\in B} \bar\eta^m_{j}$, $B$ a finite subset of
$\bb Z$.  There is however an additional decoupling step because the
$(h(\tau_i\eta))_{i\in\T_n}$ are no longer independent. This step is
explained in the proof of \cite[Corollary 4.5]{jl}, see
also~\cite[Appendix F]{Jara2018} for a general treatment.  The
constants $16$ and $8$ appearing on the right-hand side of the
proposition then have to be replaced by larger ones which depend on
the set $B$.
\end{remark}

The proof of this result is based on the local central limit
theorem. Recall from Theorems VII.4--6 in \cite{petrov} the statement
of this result. The proof of Theorem VII.4 yields that there exists a
finite constant $\mf c_{CL}$ and a finite integer $\mf n_0$, such that
\begin{equation}
\label{n-11}
\max_{0\le j\le n}\, 
\Big|\, \sqrt{ n \rho (1-\rho)} \,
\nu^n_\rho (\sum_{i=1}^n \eta_i = j) \,-\,
\frac{1}{\sqrt{2\pi}}\, e^{-x^2/2}\, \Big|
\,\le\, \frac{\mf c_{CL} }{\sqrt{ n \rho (1-\rho)}}
\end{equation}
for all $0<\rho<1$, $n\ge \mf n_0$. In this formula,
$x= (j - n\rho)/\sqrt{ n \rho (1-\rho)}$.

The next result compares the expectation with respect to the canonical
measure of a nonnegative function to its expectation with respect to
the grand canonical one.  This result is similar to \cite[Corollary
5.5]{lpy}, but there the ratio between the truncated moments and the
corresponding powers of the variance are uniformly bounded (see
equation (5.1)) which is not the case in the present context.  For a
positive real number $r$ denote by $[r]$ its integer part:
$\color{blue} [r] = \max\{k\in \bb N : k\le r\}$.

\begin{lemma}
\label{n-l02}
Fix $0<\vartheta<1$. Then,
\begin{equation*}
\nu^{n,m}_{1/2} (F) \,\le\, \frac{4}{1-\vartheta} \, \, \nu^n_{\rho}  (F)
\;, \;\;\text{where}\;\; \rho = m/n\;, 
\end{equation*}
for all $n\ge \mf n_0$, $m$ such that $\sqrt{ n \rho
(1-\rho)} \ge 2 \sqrt{ 2 \pi} \mf c_{CL}$, and
nonnegative functions $F\colon \Omega_n\to \bb R_+$ which
depend only on the variables, $\eta_1, \dots, \eta_{[n\vartheta]}$. 
\end{lemma}
\begin{proof}
Let $k= [n\vartheta]$. Since $F$ depends only on the variables
$\eta_1, \dots , \eta_k$,
\begin{equation*}
\nu^{n,m}_{1/2} (F)  \,=\, \sum_{\xi\in \Omega_k} F(\xi)\,
\nu^{n,m}_{1/2}\big(\, \{\eta \in \Omega_n : \eta_i = \xi_i\,,\, 1\le i\le
k\}\, \big)\;. 
\end{equation*}
By definition of the canonical measure, the previous sum is equal to
\begin{equation}
\label{n-10}
\sum_{\xi\in \Omega_k} F(\xi)\, \nu^n_\alpha(\xi)\,
\frac{\nu^n_\beta\big(\sum_{k+1\le i\le n} \eta_i = m - |\xi|\, \big)}
{\nu^n_\beta \big(\sum_{1\le i\le n} \eta_i = m\, \big)} 
\end{equation}
for every $0<\beta<1$, where $|\xi| = \sum_{1\le i\le k} \xi$. Choose
$\beta = \rho= m/n$. It remains to show that the last ratio is
bounded uniformly in $\xi$, $\rho$ and $n$.

In the ratio appearing in the formula \eqref{n-10}, multiply the
numerator and the denominator by $\sqrt{ n \rho (1-\rho)}$. By
\eqref{n-11}, since $x=0$, the denominator is larger than or equal to
$(1/2) (1/\sqrt{ 2 \pi})$ provided $n\ge \mf n_0$ and $m$ is such that
$\sqrt{ n \rho (1-\rho)} \ge 2 \sqrt{ 2 \pi} \mf c_{CL}$ [to guarantee
that
$(1/\sqrt{ 2 \pi}) - (\mf c_{CL} / \sqrt{ n \rho (1-\rho)} )\ge
(1/2\sqrt{ 2 \pi})$].

On the other hand, by \eqref{n-11} and since $k\le \vartheta n$, the
numerator multiplied by $\sqrt{ n \rho (1-\rho)}$ is bounded by
\begin{equation*}
\frac{\sqrt{ n }}{\sqrt{ n - k }} \,\Big\{ \frac{1}{\sqrt{ 2 \pi}}
\,+\, \frac{\mf c_{CL}}{\sqrt{ (n - k)  \rho(1-\rho)} }\,\Big\}
\,\le\, \frac{1}{1-\vartheta}\,
\frac{2}{\sqrt{ 2\pi }}\;\cdot
\end{equation*}
Adding the previous estimates completes the proof of the lemma.
\end{proof}

The previous result plays a central role in the proof of the
concentration inequality.

\begin{proof}[Proof of Proposition \ref{n-l03}]
Fix $0<m<n$. 
By Schwarz inequality, the exponential appearing in the
statement of Proposition~\ref{n-l03} is bounded by
\begin{equation}
\label{n-12}
\nu^{n,m}_{1/2} \Big( \exp  \Big\{ 2 \alpha \, \Big(\frac{1}{\sqrt{n}}
\sum_{i=0}^{[n/2]-1} \, a_{n,i} \, \bar\eta^m_{i} \Big)^2
\Big\} \, \Big)^{1/2}  \,
\nu^{n,m}_{1/2} \Big( \exp  \Big\{ 2\alpha \, \Big(\frac{1}{\sqrt{n}}
\sum_{i=[n/2]}^n \, a_{n,i}\, \bar\eta^m_{i} \Big)^2
\Big\} \, \Big)^{1/2} 
.
\end{equation}

We focus on the first term, the second one is handled similarly.  For
$n\ge \mf n_0$, and $m=\rho n$ such that
$\sqrt{ n \rho (1-\rho)} \ge 2 \sqrt{ 2 \pi} \mf c_{CL}$, by Lemma
\ref{n-l02}, the first term is bounded by
\begin{equation*}
4\,  \nu^{n}_\rho  \Big( \exp  \Big\{ 2 \alpha \, \Big(\frac{1}{\sqrt{n}}
\sum_{i=0}^{[n/2]-1} \, a_{n,i} \, \bar\eta^m_{i} \Big)^2
\Big\} \, \Big)^{1/2} 
\end{equation*}
By \cite[Lemma 4.4]{jl}, which relies on Hoeffding's inequality
\cite[Lemma 2.2]{BLM13},
$\sum_{0\le i< [n/2]} \, a_{n,i} \, \bar\eta^m_{i}$ is
$A$-subgaussian, where $A = (n/2) \Vert a\Vert^2_\infty$. Thus,
by equation (4.8) in \cite[Lemma 4.4]{jl}, the previous expression is
bounded by $4 e^{4 \alpha \Vert a\Vert^2_\infty}$ provided $\alpha \le
(4 \Vert a\Vert^2_\infty)^{-1}$. 

Therefore, \eqref{n-12} is bounded by
$16 e^{8 \alpha \Vert a\Vert^2_\infty}$ if
$\alpha \le (4 \Vert a\Vert^2_\infty)^{-1}$, $n\ge \mf n_0$, and $m$ is
such that $\sqrt{ n \rho (1-\rho)} \ge 2 \sqrt{ 2 \pi} \mf c_{CL}$.

Suppose now that $\sqrt{ n \rho (1-\rho)} \le 2 \sqrt{ 2 \pi} \mf c_{CL}$ and start again from the exponential appearing in the
statement of Proposition~\ref{n-l03}. 
Take $m\le n/2$ (the other case is handled similarly). 
Then
$m\le 16 \pi \mf c_{CL}$, and:
\begin{equation*}
\alpha \, \Big(\frac{1}{\sqrt{n}}
\sum_{i\in\T_n} \, a_{n,i} \, \bar\eta^m_{i} \Big)^2
\,\le\, 4\alpha \Vert a\Vert^2_\infty (m^2/n) 
\,\le\, 2^{10} \alpha \Vert a\Vert^2_\infty \pi^2 \mf c_{CL}^2 /n\;.
\end{equation*}
This expression is bounded by $8 \alpha \Vert a\Vert^2_\infty$
provided $n$ is large enough depending only on $\mf c_{CL}$. This completes the proof of the proposition. 
\end{proof}

\subsection*{Conclusion}

\begin{proof}[Proof of Theorem \ref{n-s02}]
Recall the estimate presented in Lemma \ref{n-l01}.  Lemma \ref{n-s01}
asserts that the last term on the right-hand side of \eqref{n-13} is
bounded by the sum of two expressions. The first one is the carr\'e du
champ
$( \mf c_{\rm LS}\, C_1/ 2\delta \lambda) \, n^2\, \Gamma^{\rm
ex}_n(f;\nu^n_U)$ ($\delta,\lambda>0$).  Choose $\lambda=1/4$. The
second one is at most
\begin{equation}
\frac{C_1\, }{2\, \delta \lambda}\, \max_{0\leq m\leq n} \log \nu^{n,m}_{1/2} \Big( \exp  \Big\{ \lambda\, \Big(\frac{1}{\sqrt{n}}
\sum_{i\neq j+1-n}I_n\big(j-(i+1)\big) \, \bar\eta^m_{i} \Big)^2
\Big\} \, \Big) 
\, ,
\end{equation}
and by Proposition \ref{n-l03} the exponential moment is bounded by
$e^{C_2} := 16 e^2$.

In conclusion, since $\lambda=1/4$,
\begin{equation}
\begin{aligned}
& \Big| \, \nu^{n}_{U} \Big( f \, \mc Z_n\, 
\sum_{i\in \bb T_n} \bar\eta^m_i \, \bar\eta^m_{i+1}
\Big) \, \Big|
\\
&\quad
\le\; C_0\, +\,
\big\{\, 2\, \frac{\mf c_{\rm LS}}{\delta}\, C_1\, 
\,+\, \delta\, \big\} \,
n^{2}\, \Gamma_n^{\rm ex}(f ; \nu^{n}_U)\,+\,
\frac{2\, C_1\, C_2}{\delta} \, 
\end{aligned}
\end{equation}
for all $\delta>0$. Choose $\delta=1$ to complete the proof of the theorem.
\end{proof}

\section{Free energy estimate}\label{sec_free_energy_bounds}
In this section, we prove Theorem~\ref{theo_free_energy}, 
generalising the entropy estimate to all $a>0$.  
As in Section~\ref{n-sec2}, 
the proof involves a decoupling of slow and fast modes. 
This decoupling is established under the log-Sobolev inequality of Assumption~\ref{ass_LSI}, 
in force throughout the section.  \\

Recall from Sections~\ref{n-sec2}--\ref{n-sec3} that the restriction on the size of $a$ comes from the estimate of the following term in the adjoint:
\begin{equation}
a\sqrt{n}\, \E\Big[  \sum_{i\in \bb T_n} \bar\eta^m_i(t) \, \bar\eta^m_{i+1}(t)\Big]
.
\label{eq_term_to_cancel_nonpert}
\end{equation}
This term involves the fast, mean-0 modes, i.e. the $\bar\eta^m$. 
The aim of this section is to cancel~\eqref{eq_term_to_cancel_nonpert} by adding information on the fast modes into the reference measure. 
Recall that the measure $\nu^n_U$ of~\eqref{eq_def_nu_V} offered a correction to the behaviour of the magnetisation, 
but still neglected the influence of the magnetisation on mean-0 modes (that is, conditioned on the magnetisation $\nu^n_U$ is the uniform measure, just like the product measure $\nu^n_{1/2}$).  
In order to cancel~\eqref{eq_term_to_cancel_nonpert} we include this influence to leading order in $n$, considering a measure of the form:
\begin{equation}
\nu^n_g(\eta)
\propto \exp\Big[\frac{1}{2n}\sum_{i\neq j}\bar\eta^m_i\bar\eta^m_j g_{i,j}\Big] \nu^n_{1/2}(\eta)
,
\label{eq_def_nung_sec_free_energy}
\end{equation}
where $g$ is the function appearing in Proposition~\ref{prop_g},
identified with the symmetric function
$g\colon (x,y)\in\T^2\mapsto g(x-y)$.  Note that we could also consider the
probability measure
$\nu^n_{U,g^0}\propto \exp[\frac{1}{2n}\sum_{i\neq
j}g^0_{i,j}\bar\eta_i\bar\eta_j]\, \nu^n_{U}$.
However~\eqref{eq_def_nung_sec_free_energy}, which corresponds to
approximating $U$ by its Hessian at $0$, turns out to be technically
simpler to work with.  The drawback of using $\nu^n_g$ instead of
$\nu^n_{U,g^0}$ is that we lose the birth and death structure of Section~\ref{n-sec2}.  As a result,
when computing the adjoint magnetisation terms will not automatically
cancel as in Section~\ref{n-sec2}, 
and an additional estimate is
needed to control possible large values of the magnetisation.  
This is
done in Theorem~\ref{theo_free_energy} by estimating, instead of the
relative entropy, the free energy
$\cF( f_t \,|\, \nu^n_g , \kappa)$ introduced in \eqref{eq_def_free_energy}.

Throughout the section $a>0$, $\theta\in\R$ and
$\gamma=\frac{1}{2}(1-\theta n^{-1/2})$ are fixed and such that $\gamma\in(0,1)$. Constants implicitly depend on $a,\theta$.  
Note that $g$
implicitly depends on $n$ through $\gamma$, see
Proposition~\ref{prop_g}.  Nonetheless the bounds on $g$ proven in
Appendix~\ref{app_prop_g} imply that uniform norms of $g$ and its derivatives that appear below are uniform in $n$. \\

Let $f_t\, \nu^n_g$ stand for the law of the dynamics at time $t$.
Recall from Proposition~\ref{l06-1} that the derivative of the
relative entropy satisfies:
\begin{align}
H'_n(f_t|\nu^n_g)
\leq 
-2\, n^{5/2}\, \Gamma^{\rm ex}_n(f_t;\nu^n_g)
\,-\, 2\,a\, \sqrt{n}\, \Gamma^G_n(f_t;\nu^n_g)
+\nu^n_g\big(f_t\, \sqrt{n}\, L^*_n{\bf 1}\big)
,
\label{eq_der_entropy}
\end{align}
with the carré du champ defined in~\eqref{eq_def_carreduchamp} and $L^*_n$ the adjoint in $\bb L^2(\nu^n_g)$ of the generator $L_n$ (recall~\eqref{eq_def_generator}).

In view of the formula~\eqref{eq_der_entropy} for the entropy
dissipation, Theorem~\ref{theo_free_energy} is a consequence of Propositions~\ref{prop_bound_adjoint}--\ref{prop_bound_moment} below bounding the adjoint and the expectations of $(\mathcal Y^n_t)^2$.
\begin{proposition}[Expression of the adjoint and bounds]\label{prop_bound_adjoint}
Let $n\geq 1$.   
The adjoint in $\mathbb L^2(\nu^n_g)$ has the following properties. 
\begin{itemize}
	\item[(i)] (Renormalisation step).  
	There is $\alpha_0>0$ and, 
	for each $\delta>0$, 
	there is a non-negative function $Q^{4+,\delta}_n\colon \Omega_n\to\R_+$ such that:
\begin{equation}
\nu^n_g\big(f \sqrt{n}\, L^*_n{\bf 1}\,\big)
\leq 
\alpha_0\,  \nu^n_g\big(f \sqrt{n}\, (\mathcal Y^n)^4\big) 
+ \frac{\delta\, n^{5/2}}{2}\, \Gamma^{\rm ex}_n(f;\nu^n_g) 
+ \nu^n_g\big( f \sqrt{n}\, Q^{4+,\delta}_n\big)
,
\end{equation}
for any $f:\Omega_n\to\R_+$ with $f\nu^n_g$ a probability measure. 
In addition, 
$Q^{4+,\delta}_n$ satisfies, for some $C>0$:
\begin{equation}
\nu^{n}_g\big(f\, \sqrt{n}\,  Q^{4+,\delta}_n\, \big)
\leq 
C\, \sqrt{n}\, \big( H_n(f \, |\, \nu^n_g) + 1\big)
.
\label{eq_small_time_bound_Prop61}
\end{equation}
	\item[(ii)] (Exponential moment bound). 
	For all $T>0$ and $\delta>0$, 
	there is $C_1=C_1(\delta,T)>0$ and a time $\tau_1>0$ such that:
\begin{equation}
\forall t\in[\tau_1\, n^{-1/2},T],\qquad 
\nu^n_g\big(f_t \, \sqrt{n}\, Q^{4+,\delta}_n \, \big)
\leq 
\frac{\delta\, n^{5/2}}{2}\, \Gamma^{\rm ex}_n(f;\nu^n_g) +
C_1 \sqrt{n}\, 
. 
\label{eq_bounds_error_prop_free_energy}
\end{equation}
\end{itemize} 
\end{proposition}
\begin{remark}
\begin{itemize}
	\item For the proof of Theorem~\ref{theo_free_energy} it would be enough to state item (i) when $f=f_t$ is the density of the law of the dynamics at time $t\in[0,T]$. 
The general expression will be useful in the proof of tightness in Section~\ref{sec_tightness}.
	\item The time $\tau\, n^{-1/2}$ in~\eqref{eq_bounds_error_prop_free_energy} corresponds to the time after which large deviation bounds become available uniformly on the initial condition. 
	Before time $\tau\, n^{-1/2}$ the dynamics is controlled by the coarse bound~\eqref{eq_small_time_bound_Prop61}, which would be very bad for long time due to the $\sqrt{n}\, $ prefactor in front of the entropy. 
	For suitably nice initial conditions the better bound~\eqref{eq_bounds_error_prop_free_energy} would be valid for any $t\in[0,T]$. 
	\item The exponent $4+$ in $Q^{4+,\delta}_n$ indicates that only terms involving four-point correlations and above remain in the adjoint (up to corrections that are smaller in $n$). 
	This is a rigorous statement of the claim that computing the adjoint with respect to $\nu^n_g$ cancels the remaining two-point correlation term~\eqref{eq_term_to_cancel_nonpert}. 
\end{itemize}
\end{remark}
\begin{proposition}[Bound on second moment]\label{prop_bound_moment}
For each $\delta,T>0$, there are $C_2,C'_2>0$ and a time $\tau_2>0$ such that:
\begin{align}
\forall t\in[0,T],\qquad
\partial_t \nu^n_g\big(f_t\, \sqrt{n}\, (\mathcal Y^n)^2\big)
&\leq 
\delta\,  n^{5/2}\, \Gamma^{\rm ex}_n(f_t;\nu^n_g) 
- 8\, a\, \gamma^2\,   \nu^n_g\big(f_t\, \sqrt{n}\, (\mathcal Y^n)^4\big)
\nnb
&\quad
+C_2 \,\sqrt{n}\, + {\bf 1}_{[0,\tau_2\, n^{-1/2}]}(t)\, C'_2\, \sqrt{n}\, H_n(f_t\, |\, \nu^n_g)
. 
\end{align}
\end{proposition}
Recall that the free energy reads:
\begin{equation}
\cF( f_t \,|\, \nu^n_g , \kappa)
=
H_n(f_t|\nu^n_g) + \kappa\,\sqrt{n}\, \nu^n_g\big(f_t (\cY^n)^2\big)
.
\end{equation}
Putting Propositions~\ref{prop_bound_adjoint}--\ref{prop_bound_moment} together, 
we see that it is enough to choose $\kappa=\alpha_0/(8a\gamma^2)$ to get rid of the fourth moment of the magnetisation when computing $\partial_t \cF( f_t \,|\, \nu^n_g , \kappa)$. 
Indeed, 
for such a $\kappa$ and any $\delta<\min\{1/2,(2\kappa)^{-1}\}$, 
letting $\tau:=\max\{\tau_1,\tau_2\}$:
\begin{align}
&\partial_t\cF( f_t \,|\, \nu^n_g , \kappa)
+n^{5/2}\, \Gamma^{\rm ex}_n(f_t;\nu^n_g)
\nnb
&\qquad 
\leq 
{\bf 1}_{[0,\tau\, n^{-1/2}]}(t)\sqrt{n}\, (C+\kappa C'_2)\, H_n(f_t\, |\, \nu^n_g)
+ (C_1+\kappa C_2)\, \sqrt{n} 
\nnb
&\qquad\leq 
{\bf 1}_{[0,\tau\, n^{-1/2}]}(t)\sqrt{n}\, (C+\kappa C'_2)\,\, \cF( f_t \,|\, \nu^n_g , \kappa)
+(C_1+\kappa C_2)\, \sqrt{n}\, 
.
\end{align}
A Gronwall bound concludes the proof of
Theorem~\ref{theo_free_energy}. \smallskip

We start with a sketch of the proof of
Proposition~\ref{prop_bound_adjoint} in
Section~\ref{sec_sketch_free_energy}, while the proof itself takes up
the following three subsections.  Proposition~\ref{prop_bound_moment} is
proven in Section~\ref{sec_bound_moment}.  
\subsection{Sketch of the argument}\label{sec_sketch_free_energy}
The proof of Proposition~\ref{prop_bound_adjoint} involves the same
steps as the proof presented in Section~\ref{n-sec2} for the new reference measure $\nu^n_g$:
\begin{itemize}
	\item[(i)] A separation of slow and fast modes through $\bar\eta = \bar\eta^m +n^{-1/4}\cY^n$.
	\item[(ii)] An estimate of fast modes through a log-Sobolev inequality for the dynamics restricted to configurations with fixed magnetisation.   
	\item[(iii)] A separate ingredient, not related to Kawasaki dynamics, 
	to bound magnetisation terms. 
\end{itemize}
The proof of items (ii) and (iii) is however quite different. 
Item (iii) has already been discussed above: 
while for $\nu^n_U$ the bound was provided by construction of the measure, 
here the reference measure $\nu^n_g$ does not describe the law of the magnetisation as precisely and an additional control on moments of the magnetisation is needed, 
provided by Proposition~\ref{prop_bound_moment}.

The more crucial difference lies in the proof of item (ii), 
as we now wish to bound the contribution of the fast modes without assuming $a$ small.  
Let us explain how this is made possible, starting from  where the problem previously was. \\

With $\nu^n_U$ as reference measure, one had to estimate the expectation of the term:
\begin{equation}
a\sqrt{n}\, \sum_{i\in \bb T_n} \bar\eta^m_i(t) \, \bar\eta^m_{i+1}(t)
,
\label{eq_term_eta_i_eta_iplus1_to_cancel_sketch}
\end{equation}
or, after a renormalisation step making use of the carr\'e du champ,
of a term of the form:
\begin{equation}
a\sqrt{n}\, \sum_{i\in \bb T_n} \bar\eta^m_i(t)\bar\eta^m_j(t) \,
\phi^n(i,j) 
,
\end{equation}
with $\phi^n\colon \T^2_n\to\R$ satisfying
$\sup_{n\geq 1}\|\phi^n\|_\infty<\infty$.  To explain how we produce a
better estimate in the present section, let us work in a simplified
context, assuming that the above term is replaced by:
\begin{equation}
a\sqrt{n}\, X_n(\eta(t))^2\;,\quad 
X_n(\eta)
:=
\frac{1}{\sqrt{n}}\sum_{i\in \bb T_n} \bar\eta^m_i\psi^n(i)\;,\quad 
\sup_{n\geq 1}\|\psi^n\|_{\infty}
<
\infty
\;.
\end{equation}
We expect $X_n(\eta(t))$ to be approximately Gaussian at each time, so
that $a\, \sqrt{n}\, X_n^2(\eta(t)) \approx \sqrt{n}$.  Our only way
to make this estimate rigorous is through the entropy
inequality~\eqref{n-35} (applied at each fixed magnetisation), which
indeed gives a $O(\sqrt{n}\, )$ bound on
$a\sqrt{n}\, \E\big[ X_n(\eta(t))^2\big]$ provided:
\begin{equation}
\sup_{n,m}\nu^{n,m}_{1/2}( e^{a CX_n(\eta)^2}) 
<
\infty
,
\end{equation}
for some $C>0$ independent of $a$, determined by the log-Sobolev
constant of $\nu^{n,m}_{1/2}$ with respect to Kawasaki dynamics
(see~\eqref{eq_LSI_Yau_sec_5}).  Subgaussian concentration under the
canonical measure gives the following bounds: for some
$c_0 = c_0(\sup_n \|\psi^n\|_\infty)>0$,
\begin{equation}
\nu^{n,m}_{1/2}
\big( X_n(\eta)^2 > \lambda \big)
\leq 
2e^{-c_0 \lambda}
,\qquad 
\lambda>0
. 
\end{equation}
The dependence in $\lambda$ of the above bound cannot be improved uniformly in $n$ (since in the limit $X_n$ is Gaussian under $\nu^{n,m}_{1/2}$). 
We therefore have no choice: for $a\, \sqrt{n}\, \E[X_n(\eta(t))^2]$
to be bounded by $O(\sqrt{n})$, 
we need $a$ to be smaller than some
value $a_0$.

The reference measure $\nu^n_g$ is introduced to cancel the quadratic
term $a\, \sqrt{n}\, \E[X_n(\eta(t))^2]$ (more precisely the quadratic
term~\eqref{eq_term_eta_i_eta_iplus1_to_cancel_sketch} but let us keep
our simplified setting).  The adjoint in $\bb L^2(\nu^n_g)$ only contains
four-point correlations and higher (up to sub-leading order,
negligible corrections when $n\gg 1$), that is, product of at least four
$\bar\eta$'s or more.  After decomposing on slow and fast modes
$\bar\eta=\bar\eta^m+\ms M^n$ writing
$\ms M^n=n^{-1/4}\cY^n\in[-1/2,1/2]$, we get an expression of the
following form: for some $\alpha_0>0$,
\begin{align}
\sqrt{n}\, L^*_n{\mb 1}(\eta)
&=
\alpha_0 \sqrt{n} (\cY^n)^4 
+
a\, \sqrt{n}\, (\ms M^n)^{2} X_n^2 
\nnb
&\quad 
+ a\, \sqrt{n}\, \ms M^n \cdot  \frac{1}{\sqrt{n}}X_n^3 +
a\, \sqrt{n}\cdot \frac{1}{n}X_n^4 
\quad+\quad \text{higher order terms}
.
\end{align}
(Here again this is a simplification, the correct expression involves $2$-, $3$- and $4$-tensors instead of powers of $X_n$ and a renormalisation involving the carr\'e du champ is necessary).  
The term $(\cY^n)^4$ does not involve the mean-0 modes and is estimated through Proposition~\ref{prop_bound_moment}, 
so let us focus on the rest. 
We claim that all other terms should typically be much smaller than $a\sqrt{n}\, X_n^2$ regardless of the size of $a$. 
Indeed, notice first that all terms in the right-hand side are all of comparable size, of order $n^{3/2}$. 
However, 
we expect the fast modes to have Gaussian concentration: $X_n(t)\approx 1$. 
On the other hand the magnetisation should satisfy $\ms M^n_t\approx n^{-1/4}$,   
so that:
\begin{equation}
a\,\sqrt{n}\, (\ms M^n_t)^{2}\cdot X_n^2(t) 
\approx
1
\qquad
a\,\sqrt{n}\, \ms M^n_t \cdot  \frac{1}{\sqrt{n}}X_n^3(t)
\approx 
\frac{1}{n^{1/4}}
,\qquad
a \sqrt{n}\,\cdot\frac{1}{n}X_n^4(t)
\approx 
\frac{1}{\sqrt{n}}
.
\label{eq_heuristic_size_quartic_sec6}
\end{equation}
This is a considerable improvement compared to $a\, \sqrt{n}\, X_n^2(t) \approx \sqrt{n}$, 
and the reason why considering $\nu^n_g$ instead of $\nu^n_U$ as reference measure allows us to control the dynamics for arbitrary values of $a$.

Making (a version of) this heuristics rigorous is the key technical work of this section, 
corresponding to Lemma~\ref{lemm_projection_Q4}. 
The difficulty is that we can only estimate expectations at time $t$ through the entropy inequality, 
which means we require bounds on log moment generating functions under $\nu^n_g$ of the terms in~\eqref{eq_heuristic_size_quartic_sec6}.  
These may be much bigger than typical values. 
Indeed, consider e.g. $a\, (\ms M^n)^2 X^2_n$ which is of order $n$ for worst-case configurations. 
Its exponential moment therefore blows up exponentially in $n$ under $\nu^n_g$ if $a$ is larger than some threshold $a_\psi$ depending only on $\psi,g$, 
even though all its moments under $\nu^n_g$ are bounded with $n$, 
simply because $\nu^n_{1/2}(\eta)=2^{-n}$ cannot compensate $\exp[a\, (\ms M^n)^2 X^2_n]$.

We avoid these pathological configurations through large deviation bounds. 
This is done differently depending on the terms as explained next.
\begin{itemize}
	\item For $\ms M^n\cdot X_n^2$ and $\ms M^n\cdot \frac{1}{\sqrt{n}}\, X_n^3$, 
	there is a prefactor $\ms M^n= \frac{1}{n}\sum_{i\in\T_n}\bar\eta_i$ that we expect to be $\approx n^{-1/4}$. 
	While such a precise bound is part of what we want to prove, 
	we do already know that $\ms M^n$ cannot be of order $1$ since this would correspond to a large deviation event. 
	Using large deviation bounds to argue that $\sup_{t\in[0,T]}|\ms M^n(t)|\leq \epsilon$ with probability $1-c(\epsilon,T)^{-1}e^{-nc(\epsilon,T)}$ for any $\epsilon>0$ and some $c(\epsilon,T)$, 
	established in Appendix~\ref{app_LD}, 
	we are left with bounding:
	\begin{equation}
		\E\Big[a\, (\ms M^n_t)^2\cdot X_n^2(t) {\bf 1}_{|\ms M^n_t|\leq \epsilon}\Big],
		\qquad 
	\E\Big[a\, \ms M^n_t \cdot  \frac{1}{\sqrt{n}}X_n^3(t){\bf 1}_{|\ms M^n_t|\leq \epsilon}\Big]
	.
	\end{equation}
	Using entropy and log-Sobolev inequalities this translates as a bound on exponential moments under the canonical measure $\nu^{n,m}_g$, for $|m-n/2|\leq \epsilon n$, of:
	\begin{equation}
	Ka\, \epsilon^2\cdot X_n^2 ,\qquad 
	Ka\, \epsilon \cdot  \frac{1}{\sqrt{n}}|X_n|^3
	,
	\end{equation}
	with $K=K(g)>0$ an explicit constant. 
	Since $\epsilon$ is arbitrary we can ensure $K a\epsilon$, $K a\epsilon^2$ are small enough for these exponential moments to be bounded with $n$. 
	\item The last term $\frac{1}{n}X_n^4$ is not multiplied by the magnetisation. 
	A similar argument nonetheless applies, relying on the boundedness of $X_n$ and the small prefactor $1/n$. 	Indeed, 
	Gaussian concentration bounds in Appendix~\ref{app_concentration} give, for some $c(g)>0$:
	\begin{equation}
	\nu^{n,m}_g\Big(\frac{K a}{n} X_n^4 >\lambda \Big) 
	\leq 
	\exp\Big[ -c(g) \, \sqrt{\frac{n \lambda\,}{aK}}\Big],\qquad 
	\lambda>0
	,
	\end{equation}
	where the $\sqrt{n}$ factor comes from the prefactor $1/n$ in $X_n^4/n$. 
	Since $\frac{a K}{n} X_n^4$ is deterministically bounded by $Ka\|\psi^n\|^4_\infty\, n$, 
	its exponential moment satisfies:
	\begin{equation}
	\nu^{n,m}_{g}\Big[ e^{\frac{a}{n} X_n^4} \Big]
	\leq 
	1 + \int_0^{Ka \|\psi^n\|_\infty\, n } 
	e^{\lambda - c(g) \sqrt{n\lambda /aK}}\, d\lambda
	.
	\end{equation}
	Observe that the negative exponential beats the $e^\lambda$ for all $\lambda$'s smaller than some $c(a,K)n$, i.e. up to \emph{macroscopic} $\lambda$'s. 
	To bound the above exponential moment, 
	it is therefore enough to restrict to regions of the state space where $|X_n|\ll \sqrt{n}$ is small at the macroscopic level, 
	which can again be done by large deviation bounds.  
	Note that the above argument is valid for terms of the form $X_n^{p}$ for any $p\geq 3$, 
	since due to the $O(n^{3/2})$ bound arising in the computation of the adjoint these terms must arise with a prefactor $n^{(2-p)/p}$, 
	which becomes small when $p\geq 3$.  
\end{itemize}
\subsection{Computation of $n^{5/2}L^{*,\rm ex}_{n}{\bf 1}$}\label{subsec_comput_adjoint_exclusion}
We first compute the contribution of Kawasaki dynamics to the adjoint. 
For $p\geq 1$ and a function $\phi:\T^p\rightarrow\R$, 
write:
\begin{equation}
\phi_{i_1,...,i_p} 
:=
\phi\Big(\frac{i_1}{n},...,\frac{i_p}{n}\Big)
,\qquad 
i_1,...,i_p\in\T_n
.
\label{eq_notation_indices}
\end{equation}
Define also:
\begin{equation}
\partial^n_{1} g_{i,j}
:=
n[g_{i+1,j}-g_{i,j}]
,
\qquad
\Delta^n_1 g_{i,j}
:= 
n^2\big[g_{i+1,j}+g_{i-1,j}-2g_{i,j}\big]
,
\qquad i,j\in\T_n
.
\end{equation}
Let $i\in\T_n$, $\eta\in\Omega_n$ and write:
\begin{equation}
\Pi^n(g)
=
\Pi^n(g)(\eta)
:= 
\frac{1}{4n}\sum_{i\neq j}\bar \eta_i\bar\eta_j g_{i,j}
,\qquad 
\bar\eta_\cdot 
:= 
\eta_\cdot-\frac{1}{2}
,
\end{equation}
so that $\nu^n_g(\eta)\propto e^{2\Pi^n(g)}\nu^n_{1/2}(\eta)$. 
Then:
\begin{equation}
2\Pi^n(g)(\eta^{i,i+1})
-
2\Pi^n(g)(\eta)
=
\frac{(\eta_i-\eta_{i+1})}{n}\sum_{j\notin\{i,i+1\}} \partial^n_1 g_{i,j}\bar\eta_j
.
\end{equation}
From this fact and the formula~\eqref{eq_general_formula_adjoint} for
the adjoint, $n^{5/2}L^{*,\rm ex}_{n}{\bf 1}$ reads:
\begin{align}
n^{5/2}L^{*,\rm ex}_{n}{\bf 1}(\eta)
&=
n^{5/2}\sum_{i\in\T_n} \eta_i(1-\eta_{i+1})
\bigg( \exp\Big[\frac{1}{n}\sum_{j\notin\{i,i+1\}}
\bar\eta_j \big[g_{i+1,j}-g_{i,j}\big]\Big] -1\bigg) 
\nonumber\\
&\quad +n^{5/2}\sum_{i\in\T_n} \eta_{i+1}(1-\eta_{i})
\bigg( \exp\Big[-\frac{1}{n}\sum_{j\notin\{i,i+1\}}
\bar\eta_j \big[g_{i+1,j}-g_{i,j}\big]\Big] -1\bigg) 
\end{align}
Since $\sup_n\|\partial_1^n g\|_\infty<\infty$, 
the argument of the exponentials is of order $1/n$ uniformly in the configuration. 
Expanding the integrals using:
\begin{equation}
e^{x}
=
1+x+\frac{x^2}{2}+\frac{x^3}{6} + \frac{x^4}{6}\int_0^1(1-t)^3e^{tx}\, dt
,\qquad 
x\in\R
,
\end{equation}
one therefore finds:
\begin{align}
n^{5/2}L^{*,\rm ex}_n{\bf 1}(\eta)
&=
\sqrt{n}\, \sum_{i\in \T_n} \big(\eta_i-\eta_{i+1}\big)\sum_{j\notin\{i,i+1\}} \bar\eta_j \partial^n_1 g_{i,j}
\nonumber\\
&\quad
+\frac{\sqrt{n}\, }{2}\sum_{i\in \T_n} \big(\eta_{i+1}(1-\eta_{i}) + \eta_i(1-\eta_{i+1})\big)\Big(\frac{1}{n}\sum_{j\notin\{i,i+1\}} \bar\eta_j \partial^n_1 g_{i,j}\Big)^2
\nonumber\\
&\quad
+\frac{1}{6\sqrt{n}\, }\sum_{i\in\T_n}(\eta_i-\eta_{i+1}) \Big(\frac{1}{n}\sum_{j\notin\{i,i+1\}}\bar\eta_j \partial^n_1 g_{i,j}\Big)^3
+O(n^{-1/2})
,
\label{eq_adjoint_exclusion_part}
\end{align}
with the $O(n^{-1/2})$ uniform on the configuration. 
Below we single out the two-point correlation terms since these will be used to cancel~\eqref{eq_term_to_cancel_nonpert}. 
It will also be convenient to introduce the following definition, 
identifying terms with the largest scaling in terms of $n$. 
\begin{definition}[Leading order terms]\label{def_leading_order}
We say that $X_n:\Omega_n\to\R$ ($n\in\N$) is of \emph{leading order}
if there are $c>0$ and a sequence $\eta_n\in\Omega_n$ such that
$X_n(\eta_n)\geq c\, n^{3/2}$ for all $n\ge 1$.  We conversely say that
$X_n$ is of \emph{sub-leading order} if
$\sup_{\eta\in\Omega_n}|X_n|=O(\sqrt{n}\, )$.
\end{definition}
As an example $\sqrt{n}\ \Pi^n(\psi) = \frac{1}{4\sqrt{n}\, }\sum_{i\neq j}\bar\eta_i\bar\eta_j \psi_{i,j}$ is of leading order for a continuous, non-identically vanishing $\psi:\T^2\to\R$; while $n^{-1/2}\, \Pi^n(\psi)$ is of sub-leading order. 

\begin{remark}
Note that sub-leading order is not the opposite of leading order in Definition~\ref{def_leading_order}. 
However, in our context, terms that are not of leading order will always be $O(\sqrt{n})$, 
making the above definition convenient.  
\end{remark}

\subsubsection{First and third lines of~\eqref{eq_adjoint_exclusion_part}} 
The fact that $\eta_{i+1}-\eta_i = \bar\eta_{i+1}-\bar\eta_i$ and an
integration by parts give:
\begin{align}
\sqrt{n}\, \sum_{i\in \bb T_n}
&\big(\eta_i-\eta_{i+1}\big)\sum_{j\notin\{i,i+1\}} \bar\eta_j
\partial^n_1 g_{i,j} 
=
\sqrt{n}\, \sum_{i\in \bb T_n} \bar\eta_i\Big(\sum_{j\notin\{i,i+1\}}
\bar\eta_j \partial^n_1 g_{i,j} - \sum_{j\notin\{i-1,i\}} \bar\eta_j
\partial^n_1 g_{i-1,j}\Big) 
\nonumber\\
&\quad =
\frac{1}{\sqrt{n}}\sum_{i\in\T_n}\sum_{\substack{j\in\T_n \\ |j-i|>1}}\bar\eta_i\bar\eta_j \Delta^n_1 g_{i,j}
+\sqrt{n}\, \sum_{i\in\T_n}\bar\eta_i\big[\bar\eta_{i-1} \partial^n g_{i,i-1} - \bar\eta_{i+1}\partial^n g_{i-1,i+1}\big]
\nonumber\\
&\quad =
\frac{1}{\sqrt{n}}\sum_{i\in\T_n}\sum_{\substack{j\in\T_n \\ |j-i|>1}}\bar\eta_i\bar\eta_j \Delta^n_1 g_{i,j}
+\sqrt{n}\, \sum_{i\in\T_n}\bar\eta_i\bar\eta_{i+1}\big[ \partial^n_1 g_{i+1,i} - \partial^n_1 g_{i-1,i+1}\big]
.
\end{align}
A similar computation shows that the third order term in the third line of~\eqref{eq_adjoint_exclusion_part} is bounded by $O(n^{-1/2})$ uniformly on the configuration:
\begin{equation}
\frac{1}{6\, \sqrt{n} }\sum_{i\in\T_n}(\eta_i-\eta_{i+1}) \Big(\frac{1}{n}\sum_{j\notin\{i,i+1\}}\bar\eta_j \partial^n_1 g_{i,j}\Big)^3
=
O(n^{-1/2})
.
\end{equation}
\subsubsection{Second line in~\eqref{eq_adjoint_exclusion_part}}
Developping the square and exchanging the order of summation, one has:
\begin{align}
\frac{\sqrt{n}\, }{2}&\sum_{i\in T_n} \big(\eta_{i+1}(1-\eta_{i}) + \eta_i(1-\eta_{i+1})\big)\Big(\frac{1}{n}\sum_{j\notin\{i,i+1\}} \bar\eta_j \partial^n_1 g_{i,j}\Big)^2
\nonumber\\
&\quad =
\frac{1}{2n^{3/2}}\sum_{i\in \T_n} \big(\eta_{i+1}(1-\eta_{i}) + \eta_i(1-\eta_{i+1})\big)\sum_{j,\ell \notin\{i,i+1\}} \bar\eta_j\bar\eta_\ell \partial^n_1 g_{i,j}\partial^n_1 g_{i,\ell}
\nonumber\\
&\quad=
\frac{1}{2n^{3/2}}\sum_{j,\ell\in\T_n} \bar\eta_j\bar\eta_\ell \Big(\sum_{i\notin\{j-1,j,\ell-1,\ell\}} \big(\eta_{i+1}(1-\eta_{i}) + \eta_i(1-\eta_{i+1})\big) \partial^n_1 g_{i,j}\partial^n_1 g_{i,\ell}\Big)
.
\label{eq_2nd_order_exclusion_0}
\end{align}
Note the following elementary identities:
\begin{align}
\eta_i(1-\eta_{i+1}) + \eta_{i+1}(1-\eta_i)
&=
2\sigma - 2\bar\eta_i\bar\eta_{i+1}
,\quad 
(\bar\eta_i)^2
=
\sigma
,\qquad
\sigma=1/4
.
\end{align}
It follows that~\eqref{eq_2nd_order_exclusion_0} becomes:  
\begin{align}
\frac{1}{2n^{3/2}}\sum_{j,\ell\in\T_n} \bar\eta_j\bar\eta_\ell \Big(&\sum_{i\notin\{j-1,j,\ell-1,\ell\}} \big(\eta_{i+1}(1-\eta_{i}) + \eta_i(1-\eta_{i+1})\big) \partial^n_1 g_{i,j}\partial^n_1 g_{i,\ell}\Big)
\nonumber\\
&\quad =
\frac{\sigma}{\sqrt{n}}\sum_{j\neq \ell } \bar\eta_j\bar\eta_\ell \Big(\frac{1}{n}\sum_{i\notin\{j-1,j,\ell-1,\ell\}}\partial^n_1 g_{i,j}\partial^n_1 g_{i,\ell}\Big)
\nnb
&\qquad +
\frac{\sigma^2}{n^{3/2}}\sum_{j\in\T_n} \sum_{i\notin\{j-1,j\}}\big(\partial^n_1 g_{i,j}\big)^2
+\sqrt{n}\, C_n^{\rm ex}(\eta)
,
\end{align}
where we isolated the configuration-independent contribution of the $j=\ell$ term and 
where $C^{\rm ex}_{n}$ is an error term containing sub-leading order two-point correlations and higher-order correlations:
\begin{align}
\sqrt{n}\, C_n^{\rm ex}(\eta)
&=
-\frac{1}{n^{3/2}}\sum_{j,\ell\in\T_n}\sum_{i\notin\{j-1,j,\ell-1,\ell\}}\bar\eta_i\bar\eta_{i+1}\bar\eta_j\bar\eta_\ell\partial^n_1 g_{i,j}\partial^n_1 g_{i,\ell}
\nonumber\\
&=
\sqrt{n}\, C_n^{\rm ex,4}(\eta) -\frac{\sigma}{n^{3/2}}\sum_{i}\bar\eta_i\bar\eta_{i+1}\sum_{j\notin\{i,i+1\}}[\partial^n_1 g_{i,j}]^2
,
\label{eq_four_point_exclusion_part}
\end{align}
where the last term is bounded by $O(\sqrt{n})$ uniformly on the configuration, and:
\begin{equation}
\sqrt{n}\, C_n^{\rm ex,4}(\eta)
:=
-\frac{1}{n^{3/2}}\sum_{j\neq \ell}\sum_{i\notin\{j-1,j,\ell-1,\ell\}}\bar\eta_i\bar\eta_{i+1}\bar\eta_j\bar\eta_\ell\partial^n_1 g_{i,j}\partial^n_1 g_{i,\ell}
.
\label{eq_def_epsilon_ex4}
\end{equation}
Overall, the exclusion part contributes as follows:
\begin{align}
n^{5/2}L^{*,\rm ex}_n{\bf 1}(\eta)
&=
\frac{1}{\sqrt{n} }\sum_{i\in\T_n}\sum_{\substack{j\in\T_n \\ |j-i|>1}}\bar\eta_i\bar\eta_j \Delta^n_1 g_{i,j}
+\sqrt{n}\, \sum_{i\in\T_n}\bar\eta_i\bar\eta_{i+1}\big[ \partial^n_1 g_{i+1,i} - \partial^n_1 g_{i-1,i+1}\big]
\nonumber\\
&\quad 
+\frac{\sigma}{\sqrt{n} }\sum_{j\neq \ell } \bar\eta_j\bar\eta_\ell \Big(\frac{1}{n}\sum_{i\notin\{j-1,j,\ell-1,\ell\}}\partial^n_1 g_{i,j}\partial^n_1 g_{i,\ell}\Big)
+\sqrt{n}\, Q_n^{\rm ex,4} + O(\sqrt{n}\, )
,
\label{eq_adjoint_exlcusion_nearfinal}
\end{align}
with the $O(\sqrt{n}\, )$ uniform in the configuration. 
In order to use the special choice of $(x,y)\mapsto g(x-y)$ as the solution of the differential equation~\eqref{eq_ODE_g}, 
let us turn all discrete derivatives into continuous ones. 
Recall from Proposition~\ref{prop_g} that $g\in C^\infty([0,1])$, thus:
\begin{equation}
\max_{i,j}|\Delta^n_1 g_{i,j} - g''_{i-j}| 
= 
O(n^{-1})
,\qquad 
\max_i |\partial^n_1 g_{i,j} - g'_{i-j}|
=
O(n^{-1})
.
\label{eq_discrete_to_continuous_der_g}
\end{equation}
Note also that, as $g$ is even, $g'$ is odd and therefore:
\begin{align}
\sigma\int_{\T}\partial_1 g(x,z)\partial_1 g(y,z)\, dz
&=
\sigma\int_{\T} g'(x-z)g'(y-z)\, dz
\nnb
&=
-\sigma\int_{\T}g'(x-z)g'(z-y)\, dz
.
\end{align}
In view of~\eqref{eq_discrete_to_continuous_der_g}, 
the replacement of discrete sums/derivative is responsible for a $O(\sqrt{n}\,)$ error uniform in the configuration, 
and overall the adjoint reads, recalling~\eqref{eq_def_epsilon_ex4}:
\begin{align}
n^{5/2}L^{*,\rm ex}_n{\bf 1}(\eta)
&=
4\sqrt{n}\, \Pi^n\Big( (x,y)\mapsto g''- \sigma\int_{\T}g'(x-z)g'(z-y)\, dz \Big)
\nnb
&\quad
+\sqrt{n}\, \sum_{i\in\T_n} \bar\eta_i\bar\eta_{i+1}[g'(0_+)-g'(1_-)]
+\sqrt{n}\, C^{\rm ex, 4}_n + O(\sqrt{n}\, )
.
\label{eq_adjoint_exclusion_final}
\end{align}

\subsection{Computation of $a\, \sqrt{n}\, L^{*,G}_{n}{\bf 1}$}\label{subsec_comput_adjoint_reaction}
Recall definition~\eqref{01} of the jump rates ($\sigma_i:= 2\bar\eta_i$ for $i\in\T_n$) 
and write it as:
\begin{align}
c(\tau_i\eta)
&=
\eta_i 
\prod_{j\sim i}(1-2\gamma\bar\eta_j)
+(1-\eta_i)
\prod_{j\sim i}(1+2\gamma\bar\eta_j)
\nnb
&=: 
\eta_i r^-(\eta)
+(1-\eta_i)r^+(\eta)
,
\label{eq_c_i_sec_free_energy}
\end{align}
noting that $r^{\pm}(\eta^i)=r^\pm(\eta)$. 
From the definition~\eqref{eq_adjoint_general_measure} of the adjoint, 
$a\, \sqrt{n}\, L^{*,G}_{n}{\bf 1}$ reads:
\begin{align}
a\, \sqrt{n}\, L^{*,G}_{n}{\bf 1}(\eta) 
&=
a\, \sqrt{n}\, \sum_{i\in\T_n} r^+_i(\eta)\bigg[\eta_i\exp\Big[\frac{1-2\eta_i}{n}\sum_{j\neq i}\bar\eta_j g_{i,j}\Big]-(1-\eta_i)\bigg]
\nonumber\\
&\quad 
+ a\sum_{i\in\T_n} r^-_i(\eta)\bigg[(1-\eta_i)\exp\Big[\frac{1-2\eta_i}{n}\sum_{j\neq i}\bar\eta_j g_{i,j}\Big]-\eta_i\bigg]
.
\label{eq_L_r_star_0}
\end{align}
The term inside the exponential above is of order $1$ in $n$ in the
worst case, but should typically be small as we expect fluctuations of
the constant mode $\sum_{i}\bar\eta_i$ to be of order $n^{3/4}$ and
that of the other modes to be Gaussian:
\begin{equation}
\frac{1}{n}\sum_{j\neq i} \bar\eta_j g_{i,j}
=
\frac{\max\{8\gamma,4\}}{n}\sum_{j\neq i} \bar\eta_j
+\frac{1}{n}\sum_{j\neq i} \bar\eta_j g^0_{i,j}
\approx
n^{-1/4} + n^{-1/2}
.
\end{equation}
It thus makes sense to expand the exponentials in~\eqref{eq_L_r_star_0}, 
which we do to order $4$ using the elementary identity:
\begin{equation}
e^{x}-1-x-\frac{x^2}{2}-\frac{x^3}{6}
= \int_0^1 \frac{x^4(1-t)^3}{6}e^{tx}\, dt 
,\qquad 
x\in\R
.
\end{equation}
The adjoint then reads:
\begin{align}
a\, \sqrt{n}\, L^{*,G}_{n}{\bf 1} 
&=
2\, a\, \sqrt{n}\, \sum_{i\in\T_n} \bar\eta_i\big[r^+_i(\eta)- r^-_i(\eta)\big]
\nonumber\\
&\quad 
+
a\, \sqrt{n}\, \sum_{i\in \T_n} \Big(-\eta_i r^+_i(\eta) +(1-\eta_i)r^-_i(\eta)\Big)\Big[\frac{1}{n}\sum_{j\neq i}\bar\eta_j g_{i,j}\Big]
\nonumber\\
&\quad 
+ \frac{a\, \sqrt{n}\, }{2}\sum_{i\in \T_n} \Big(\eta_ir^+_i(\eta)+ (1-\eta_i) r^-_i(\eta)\Big)\Big[\frac{1}{n}\sum_{j\neq i}\bar\eta_j g_{i,j}\Big]^2
\nonumber\\
&\quad
+\frac{a\, \sqrt{n}\, }{6}\sum_{i\in \T_n} \Big(-\eta_i r^+_i(\eta) +(1-\eta_i)r^-_i(\eta)\Big)\Big[\frac{1}{n}\sum_{j\neq i}\bar\eta_j g_{i,j}\Big]^3
+\sqrt{n}\, R^G_n(\eta)
,
\label{eq_adjoint_expansion_rd_part}
\end{align}
with, for some constant $C$ depending only on $\|g\|_\infty$ and recalling $g=g^0 + \min\{8\gamma,4\}$:
\begin{equation}
|R^G_n(\eta)|
\leq
\frac{Ca}{n}\sum_{i\in\T_n} \frac{1}{n^3}\Big[\sum_{j\neq i}\bar\eta_j g_{i,j}\Big]^4
\leq 
Ca\bigg[\frac{1}{n^3}+(\cY^n)^4 + \frac{1}{n}\sum_{i\in\T_n} \frac{1}{n^3}\Big(\sum_{j\neq i}\bar\eta_j g^0_{i,j}\Big)^4\bigg]
.
\label{eq_bound_remainder}
\end{equation}

For future reference, note that:
\begin{align}
-\eta_i r^+_i(\eta) +(1-\eta_i)r^-_i(\eta) 
&=
-\bar \eta_i \big[r^+_i(\eta) +r^-_i(\eta) \big]
+\frac{1}{2}\big[r^-_i(\eta) -r^+_i(\eta) \big]
\nonumber\\
&=
-2\bar\eta_i -8\gamma^2\bar\eta_i\prod_{j\sim i}\bar \eta_j
-2\gamma\sum_{j\sim i}\bar\eta_j
\end{align}
and:
\begin{align}
\eta_ir^+_i(\eta)+ (1-\eta_i) r^-_i(\eta)
&=
\bar\eta_i \big[r^+_i(\eta) -r^-_i(\eta) \big]
+\frac{1}{2}\big[r^+_i(\eta) +r^-_i(\eta) \big]
\nonumber\\
&= 
4\, \gamma\, \bar\eta_i\, \sum_{j\sim i}\bar\eta_j 
+ 1 + 4\gamma^2\prod_{j\sim i}\bar\eta_j
.
\end{align}
In the following we sort all terms in the right-hand
side~\eqref{eq_bound_remainder} in terms of how many factors of
$\bar\eta$ they involve and their overall scaling with $n$, using in
particular Definition~\ref{def_leading_order} of sub-leading order
terms.

Let us write $a\, \sqrt{n}\, L^{*,G}{\bf 1}$ as:
\begin{equation}
a\, \sqrt{n}\, L^{*,G}_{n}{\bf 1} (\eta)
=
\sqrt{n}\, \sum_{k=0}^6 C^{\, \text{k point}}_n(\eta)
+\sqrt{n}\, R^G_n(\eta)
,
\end{equation}
where $C^{\text{k point}}_n(\eta)$ contains products $\bar\eta_{i_1}...\bar\eta_{i_k}$ with $i_1,...,i_k$ all different. 
That the sum goes only up to six is due to the explicit formula~\eqref{eq_c_i_sec_free_energy} for the jump rates and the fact that we expand exponentials in the adjoint only up to order $3$. 
The next lemma thins out the number of terms to compute. 
\begin{lemma}\label{lemm_subleading}
Let $p\geq 1$ be an integer, $\psi\colon \T^2\to\R$ be a bounded
function and consider functions $X^i_n:\Omega_n\to\R$ ($i\in\T_n$) of
the form:
\begin{equation}
X^i_n
=
n^{3/2}\, \Big[\frac{1}{n}\sum_{j\neq i}\bar\eta_j \psi_{i,j}\Big]^p
.
\end{equation}
Then:
\begin{equation}
X^i_n
=
\frac{n^{3/2}}{n^p}\, \sum_{\substack{j_1,...,j_p \neq  i \\\text{all different}}} \prod_{k=1}^p\bar\eta_{j_k} \psi_{i,j_k}
+ 
\tilde X^i_n
,
\end{equation}
where $\tilde X^i_n$ is of sub-leading order. 
\end{lemma}
\begin{proof}
The function $X^i_n$ is bounded by $O(n^{3/2})$ uniformly in
$i\in\T_n$ and the configuration.  Due to the $n^{-1}$ prefactor in
the sum on $\psi$, as soon as two $j$ indices are the same when
expanding the sum to the $p^{\text{th}}$ power the resulting object is
bounded by $O(\sqrt{n}\, )$ uniformly on the configuration, thus of
sub-leading order.
\end{proof}
Direct computations for the one-point term and
Lemma~\ref{lemm_subleading} for odd correlations between three or more
$\bar\eta$'s show that all $C^{k \text{ point}}_n$ terms with $k$ odd
are of sub-leading order.  Similarly, the constant, $k=0$, term is of
sub-leading order.  We can thus
rewrite~\eqref{eq_adjoint_expansion_rd_part} as:
\begin{equation}
a\, \sqrt{n}\, L^{*,G}_{n}{\bf 1} 
=
\sqrt{n}\, \text{Corr}^G_{n} + \sqrt{n}\, C^{G,4+}_{n} + \sqrt{n}\, \epsilon^G_{n}
,
\end{equation}
where $\text{Corr}^G_{n}$ involves leading order two point correlations and 
$C^{G,4+}_{n}$ contains leading order four-point correlations and higher.
We compute the right-hand side above in the next sections and show that $\sqrt{n}\, \epsilon^G_n$ is of sub-leading order: uniformly on the configuration,
\begin{equation}
\sqrt{n}\, \epsilon^{G}_n
=
O(\sqrt{n}\, )
.
\end{equation}
\begin{remark}
The adjoint does not contain leading order one-point term.  This is
due to the fact that the measure $\nu^n_g$ prescribes the correct
macroscopic density profile, i.e. the constant profile equal to $1/2$
which is invariant for the macroscopic dynamics~\eqref{eq_LLN}.
\end{remark}
\subsubsection{Correlation term in~\eqref{eq_adjoint_expansion_rd_part}}
Let us write all leading-order contributions to two-point correlations. 
In view of~\eqref{eq_adjoint_expansion_rd_part} and the identity $\bar\eta^2_\cdot =\sigma=1/4$, 
two point correlations may arise from the last line~\eqref{eq_adjoint_expansion_rd_part} only if at least two indices in the expansion of $[\frac{1}{n}\sum_{j\neq i}\bar\eta_j g_{i,j}]^p$ match ($p\in\{3,4\}$). 
The corresponding term cannot be of leading order due to the $1/n$ prefactor. 
Leading-order contributions to two-point correlation therefore only come from the first three lines of~\eqref{eq_adjoint_expansion_rd_part} and read:
\begin{align}
\sqrt{n}\, \text{Corr}^G_{n}(\eta)
&=
8\, \gamma \, a\, \sqrt{n}\, \sum_{i\in\T_n}\bar\eta_i\sum_{j\sim i}\bar\eta_{j} 
-\frac{2\, a}{\sqrt{n}\, }\sum_{i\in\T_n}\sum_{j\neq i} \bar\eta_i\bar\eta_j g_{i,j} 
\label{eq_corr_r_term_0_line2}\\
&\quad
- \frac{2\, \gamma \, a}{\sqrt{n} }\sum_{i\in\T_n}\Big(\sum_{\ell\sim i}\bar\eta_\ell\Big)\sum_{j\neq i,\ell} \bar\eta_j g_{i,j}
\label{eq_corr_r_term_0_line3}\\
&\quad 
+ \frac{a}{2\, \sqrt{n}}\sum_{j\neq \ell\in\T_n}\bar\eta_j\bar\eta_\ell\Big(\frac{1}{n}\sum_{i\notin\{j,\ell\}}g_{i,j}g_{i,\ell}\Big)
\label{eq_corr_r_term_0_line4}
.
\end{align}
Note that the sum in~\eqref{eq_corr_r_term_0_line3} does not include the constant, $j=\ell$ terms which is of sub-leading order by Lemma~\ref{lemm_subleading}.

Using translation invariance, 
the first term in~\eqref{eq_corr_r_term_0_line2} simplifies into:
\begin{align}
8\, \gamma \, a\, \sqrt{n}\, \sum_{i\in\T_n}\bar\eta_i\sum_{j\sim i}\bar\eta_{j}
&= 
16\, \gamma\, a\, \sqrt{n}\, \sum_{i\in\T_n}\bar\eta_i\bar\eta_{i+1}
.
\end{align}
This is the term~\eqref{eq_term_to_cancel_nonpert} that we are trying to cancel through the introduction of $g$.

Using the smoothness of $g$ which implies $\max_{i,j}|g_{i+1,j}+g_{i-1,j}-2g_{i,j}|=O(n^{-2})$, 
the second line~\eqref{eq_corr_r_term_0_line3} reads:
\begin{align}
& -\frac{2\, a\, \gamma}{\sqrt{n}}\sum_{i\in\T_n}\Big(\sum_{\ell\sim i}\bar\eta_\ell\Big)\sum_{j\neq i,\ell}\bar\eta_j g_{i,j}
=
 -\frac{4\, a\, \gamma}{\sqrt{n}}\sum_{i\neq j\in\T_n}\bar\eta_i \bar\eta_j g_{i,j}
+\sqrt{n}\, \epsilon^{G,1}_n(\eta)
,
\end{align}
with $\epsilon^{G,1}_n(\eta)$ an error term given by:
\begin{align}
\epsilon^{G,1}_n(\eta)
&=
-\frac{2\, a\, \gamma}{n^2}\sum_{i}\sum_{j\neq i,i+1}\bar\eta_i\bar\eta_j\partial^n_1 g_{i,j}
+ \frac{4\, a\, \gamma}{n}\sum_{i\in\T_n}\bar\eta_i\bar\eta_{i+1}g_{i,i+1}
.
\end{align}
In particular $\sqrt{n}\, \epsilon^{G,1}_n$ is of sub-leading order, 
and altogether the correlation term reads, with $O(\sqrt{n}\, )$ uniform in the configuration:
\begin{align}
\sqrt{n}\, \text{Corr}^G_n(\eta)
&=
16\, \gamma\, a\, \sqrt{n}\, \sum_{i\in\T_n}\bar\eta_i\bar\eta_{i+1}
-\frac{2\, a}{\sqrt{n} }\big(1+2\gamma\big)\sum_{i\in\T_n}\sum_{j\neq i}\bar\eta_i\bar\eta_j g_{i,j}
\nonumber\\
&\quad
+ \frac{a}{2\, \sqrt{n}}\sum_{j\neq \ell\in\T_n}\bar\eta_j\bar\eta_\ell\Big(\frac{1}{n}\sum_{i\notin\{j,\ell\}}g_{i,j}g_{i,\ell}\Big)
+O(\sqrt{n}\, )
.
\label{eq_expr_CorrG_interm}
\end{align}
As for the exclusion part in~\eqref{eq_adjoint_exclusion_final}, 
we use the regulariy of $g$ to replace discrete sums by integrals up to an error of order $1/n$.  
Since all terms in~\eqref{eq_expr_CorrG_interm} are bounded by $O(n^{3/2})$ uniformly in the configuration, 
the replacement yields uniform $O(\sqrt{n}\, )$ error terms and we find:
\begin{align}
\sqrt{n}\, \text{Corr}^G_n(\eta)
&=
4\, \sqrt{n}\ \Pi^n\Big((x,y)\mapsto -2\, a\, \big(1+2\, \gamma\big)+\frac{a}{2}\int g(z-x)g(z-y)\, dz\Big)
\nnb
&\quad+
16\, \gamma \, a\, \sqrt{n}\, \sum_{i\in\T_n}\bar\eta_i\bar\eta_{i+1}
+O(\sqrt{n}\, )
\label{eq_Corr_r_final}
.
\end{align}
\subsubsection{Four point correlations and higher in~\eqref{eq_adjoint_expansion_rd_part}} 
In~\eqref{eq_adjoint_expansion_rd_part}, four-point correlations arise
from the linear, quadratic, and cubic term in $g$ as well as the
remainder.  In total:
\begin{align}
\sqrt{n}\, C^{G,4+}_n(\eta) 
&=
-8\, a\, \gamma^2\sqrt{n}\, \sum_{i\in\T_n} \bar\eta_{i-1}\bar\eta_i\bar\eta_{i+1}\, \frac{1}{n}\sum_{j\neq i}\bar\eta_j g_{i,j}
\nonumber\\
&\quad
+\frac{a\, \sqrt{n}}{2}\sum_{i\in\T_n}\Big[4\gamma \bar\eta_i\sum_{j\sim i}\bar\eta_j + 4\gamma^2 \bar\eta_{i-1}\bar\eta_{i+1}\Big]\Big[\frac{1}{n}\sum_{\ell\neq i}\bar\eta_{\ell}g_{i,\ell}\Big]^2
\nonumber\\
&\quad
+\frac{a\, \sqrt{n} }{6}\sum_{i\in\T_n}\Big(-2\bar\eta_i-2\gamma\sum_{j\sim i}\bar\eta_j\Big)\Big[\frac{1}{n}\sum_{\ell\neq i}\bar\eta_\ell g_{i,\ell}\Big]^3
\nnb
&\quad
-\frac{a\,\sqrt{n} }{6}\sum_{i\in\T_n}8\gamma^2\bar\eta_{i-1}\bar\eta_i\bar\eta_{i+1}\Big[\frac{1}{n}\sum_{\ell\neq i}\bar\eta_\ell g_{i,\ell}\Big]^3
+ \sqrt{n}\, R^{G,4}_n(\eta)
,
\label{eq_def_C4}
\end{align}
with the remainder $R^{G,4}_n$ corresponding to four point correlations in $R^G_n$ and bounded as in~\eqref{eq_bound_remainder}:
\begin{equation}
|R^{G,4}_n|
\leq 
a\, C(\|g\|_\infty)\bigg[(\cY^n)^4 +\frac{1}{n}\sum_{i\in\T_n}\frac{1}{n^3}\Big(\sum_{j\neq i}\bar\eta_jg^0_{i,j}\Big)^4\bigg]
.
\label{eq_bound_RG4}
\end{equation}
\subsection{Bounds on the adjoint}\label{subsec_bound_adjoint}
In this section, 
we use the expression of the adjoint established in the last two sections to prove Proposition~\ref{prop_bound_adjoint}. 
\subsubsection{Choosing $g$}
Recall that $g$ is chosen in Proposition~\ref{prop_g} as the solution of a certain ordinary differential equation $P(g)=0$ on $(0,1)$, with boundary conditions $g'(1_-)-g'(0_+)=16\, \gamma\, a$. 
Putting together the the contributions~\eqref{eq_adjoint_exclusion_final}--\eqref{eq_def_C4} of the exclusion and Glauber part, 
we arrive at the following expression for the adjoint:
\begin{align}
\sqrt{n}\, L^*_n{\bf 1}
&=
16\, \gamma\, a\, \sqrt{n}\, \sum_{i\in\T_n}\bar\eta_i\bar\eta_{i+1} + (g'(0_+)-g'(1_-))\sqrt{n}\, \sum_{i\in\T_n}\bar\eta_i\bar\eta_{i+1} \nnb
&\quad 
+4\, \sqrt{n}\,  \Pi^n(P(g)) +\sqrt{n}\, C^{G,4+}_n + \sqrt{n}\, C_n^{\rm ex,4}+ O(\sqrt{n}\, )
,
\end{align}
with the $O(\sqrt{n}\, )$ uniform on the configuration. 
As $P(g)=0$ and the boundary condition on $g$ precisely cancels the first term in the right-hand side above (corresponding to~\eqref{eq_term_to_cancel_nonpert}), 
\begin{align}
\sqrt{n}\, L^*_n{\bf 1}
&=
\sqrt{n}\, C^{4+}_n + O(\sqrt{n}\, ),
\qquad 
C^{4+}_n 
:=
C^{G,4+}_n + C_n^{\rm ex,4}
.
\label{eq_adjoint_only4point}
\end{align}
We estimate this term next and explain how to define $Q^{4+,\delta}_n$ satisfying items (i) and (ii) of Proposition~\ref{prop_bound_adjoint} from $C^{4+}_n$ and a renormalisation step. 
\subsubsection{Decomposition into slow and fast modes}
Before renormalising $C^{4+}_n$,  
we separate the contribution of the magnetisation from the rest, decomposing as before: 
\begin{equation}
\label{n-03}
\bar\eta_i 
=
\bar\eta^m_i + \ms M^n
,\qquad
\ms M^n
:=
\frac{1}{n}\sum_{i\in\T_n}\bar\eta_i 
=
\frac{1}{n^{1/4}}\mathcal Y^n
\in [-1/2, 1/2]
.
\end{equation}
Each term in $C^{4+}_n$ can then be rewritten as a function of $\mathcal Y^n$ and $\bar\eta^m$ using the following elementary identity: 
for any $p\geq 1$ and $\phi:\T_n^p\to\R$, 
\begin{equation}
\sum_{i_1,...,i_{p}}\phi_{i_1,...,i_p}\prod_{j=1}^p\bar\eta_{i_j}
=
\sum_{P\subset\{1,...,p\}}\big(\ms M^n\big)^{p-|P|}\sum_{i_1,...,i_p}\phi_{i_1,...,i_p}\prod_{p\in P}\bar\eta^m_{i_p} 
.
\label{eq_recenterting_WJphip}
\end{equation}
The next lemma identifies the leading order terms when using~\eqref{eq_recenterting_WJphip} on $C^{4+}_n$. 
We first fix some notations. 
For $J\subset\Z$ a finite set, $p\in\N$ and $\phi:\T^p\to\R$, 
define (recall the notation $\phi_{i_1,...,i_p}$ from~\eqref{eq_notation_indices}):
\begin{equation}
W^{J,\phi}_p(\eta)
:=
\frac{1}{n^{p-1}}\sum_{\substack{i_1,...,i_p\in\T_n \\ \text{all different}}} \bar\eta^m_{i_1+J}\bar\eta^m_{i_2}...\bar\eta^m_{i_p}\phi_{i_1,...,i_p}
.
\label{eq_def_WJphi_p_sec_relent}
\end{equation}
In this notation, $W^{J,\phi}_p$ is a homogeneous polynomial in the $\bar\eta^m$ of degree $p+|J|-1$ and with $p$ independent indices. 
Let:
\begin{equation}
\mathcal J 
:= 
\{\{0\}, \{0,1\},\{-1,1\},\{-1,0,1\}\}
.
\end{equation}
Define the set $\cI$ as:
\begin{align}
\cI
:=
\left\{ (q,J,p)\in\{0,...,5\}\times \cJ\times\{1,...,4\} 
:
\begin{cases} q+p+|J|-1\geq 4, \\
|J|+p-1\geq 2, \\
{\bf 1}_{q=0}(p+|J|-1)\geq 4
\end{cases}\right\}
.
\label{eq_def_cI} 
\end{align}
\begin{lemma}\label{lemm_projection_Q4}
Recall that $C^{4+}_n$ in~\eqref{eq_adjoint_only4point} is given by
the sum of~\eqref{eq_def_epsilon_ex4}--\eqref{eq_def_C4} and recall
from \eqref{n-03} the definition of $\ms M_n $.
There are $\alpha_0>0$ and functions $\phi^{J,p}:\T^p\to\R$ ($J\in\mathcal J, 1\leq p\leq 4$), 
all defined by linear combinations of products of $g$ and its derivatives and in particular continuous on the closure of the set $\{(x_1,...,x_p)\in[0,1]^p\text{ all different}\}$, 
such that:
\begin{equation}
\sqrt{n}\, C^{4+}_n 
=
\alpha_0 \sqrt{n}\, (\cY^n)^4
+\sqrt{n}\, C^{4+,m}_n
,
\label{eq_Q4+_as_mean4_plus_Qm}
\end{equation}
with, abbreviating $W^{J,\phi^{J,p}}_p$ into $W^{J,\phi}_p$:
\begin{align}
\sqrt{n}\, C^{4+,m}_n
&\leq
\sqrt{n}\, \sum_{(q,J,p)\in\cI}(\ms M^n)^q \cdot W^{J,\phi}_p
+ O(\sqrt{n}\, )
.
\label{eq_expression_Q_4_as_Wjp}
\end{align}
Above, 
the $O(\sqrt{n}\, )$ is uniform in the configuration and the inequality in~\eqref{eq_expression_Q_4_as_Wjp} is only due to the inequality in~\eqref{eq_bound_RG4}. 
\end{lemma}
\begin{remark}\label{rmk_continuity_phi}
The functions $\phi$ are in fact all continuous except the ones that arise from $C^{\rm ex,4}_n$ in~\eqref{eq_def_epsilon_ex4}, 
since this term involves the derivative of $g$ which has a jump. 
We make the continuity properties of the $\phi$'s explicit to invoke large deviation bounds in the next section. 
\end{remark}
\begin{proof}
The quantity $q+p+|J|-1$ corresponds to the number of factors of $\bar\eta$ in $\sqrt{n}\, (\ms M^n)^q W_p^{J,\phi}$, 
thus the restrictions to $q+p+|J|-1\geq 4$, ${\bf 1}_{q=0}(p+|J|-1)\geq 4$ are straightforward from the definition~\eqref{eq_def_epsilon_ex4}--\eqref{eq_def_C4} of $C^{4+}_n$. 
We therefore only need to explain why only terms with $|J|+p-1\geq 2$ are of leading order, 
or equivalently why terms containing a single factor of $\bar\eta^m$ are not of leading order. 
We give an argument for two terms in $C^{4+}_n$, 
others being treated in the same way as one of the two. 
Consider first:
\begin{equation}
-8\, a\, \gamma^2\, \sqrt{n}\, \sum_{i\in\T_n}\bar\eta_{i-1}\bar\eta_i\bar\eta_{i+1}\frac{1}{n}\sum_{j\neq i}\bar\eta_j g_{i,j}
.
\end{equation}
Turning $\bar\eta$ into $\bar\eta^m$ using~\eqref{eq_recenterting_WJphip}, 
the resulting sums involving a single $\bar\eta^m$ read:
\begin{equation}
\sqrt{n}\, (\ms M^n)^3\sum_{i\in\T_n}\big[\bar\eta^m_{i-1}+\bar\eta^m_i +\bar\eta^m_{i+1}\big]\frac{1}{n}\sum_{j:j\neq i}g_{i,j}
+ \sqrt{n}\, (\ms M^n)^3 \cdot \frac{1}{n}\sum_{j\in\T_n}\bar\eta^m_j \sum_{i:i\neq j}g_{i,j}
. 
\label{eq_sum_singlebar_etam_vanishes}
\end{equation}
As $g$ is translation-invariant, the sums on $g$ do not depend on $i$:
\begin{equation}
\sum_{j:j\neq i}g_{i,j}
=
\sum_{k=1}^{n-1}g_{k}
=
\sum_{j:j\neq i}g_{j,i}
\end{equation}
Since $\sum_i\bar\eta^m_i =0$ by definition, 
it follows that~\eqref{eq_sum_singlebar_etam_vanishes} vanishes identically.

Consider now $\sqrt{n}\, C^{\rm ex,4}_n$:
\begin{equation}
\sqrt{n}\, C^{\rm ex,4}_n(\eta)
=
-\frac{1}{n^{3/2}}\sum_{j\neq \ell}\sum_{i\notin\{j-1,j,\ell-1,\ell\}}\bar\eta_i\bar\eta_{i+1}\bar\eta_j\bar\eta_\ell\partial^n_1 g_{i,j}\partial^n_1 g_{i,\ell}
.
\end{equation}
Terms with a single remaining $\bar\eta^m$ read:
\begin{align}
-\frac{(n^{-1/4}\cY^n)^3}{n^{3/2}}\sum_{j\neq \ell}\sum_{i\notin\{j-1,j,\ell-1,\ell\}}\big[\bar\eta^m_i+\bar\eta^m_{i+1}+\bar\eta^m_j+\bar\eta^m_\ell]\partial^n_1 g_{i,j}\partial^n_1 g_{i,\ell}
=
O(n^{-1/2})
,
\end{align}
with the $O(n^{-1/2})$ uniform in the configuration, 
where we again used the translation invariance of $g$, 
for instance in the form:
\begin{equation}
\partial^n_1 g_{i,j}
=
n[g_{i+1-j}-g_{i,j}]\quad
\Rightarrow\quad
\max_{i\in\T_n}\sum_{j \notin\{i,i+1\}}\partial^n_1 g_{i,j}
=
O_n(1)
.
\end{equation}
\end{proof}

\subsubsection{Renormalisation and bounds}\label{subsec_renorm_sec6}
In this section we define the function $\sqrt{n}\, Q^{4+,\delta}_n$ of Proposition~\ref{prop_bound_adjoint} and prove the corresponding bounds. 
This is the key technical estimate as explained in Section~\ref{sec_sketch_free_energy}, 
where we use the fact that the adjoint only involves correlations between strictly more than two $\bar\eta$'s and large deviation bounds to control the adjoint whatever the size of the parameter $a$. 

Equation~\eqref{eq_adjoint_only4point} and Lemma~\ref{lemm_projection_Q4} give:
\begin{align}
\sqrt{n}\, L_n^*{\mb 1}
\leq
\alpha_0 \sqrt{n}\, (\cY^n)^4+\sqrt{n}\, \sum_{(q,J,p)\in\cI}(\ms M^n)^q \cdot W^{J,\phi}_p
+ \sqrt{n}\, \epsilon_n
,\qquad 
\sup_{n,\eta}|\epsilon_n(\eta)|
<
\infty
.
\label{eq_recap_Lstar_before_bounds}
\end{align}
The bound of Proposition~\ref{prop_bound_adjoint} on the adjoint will therefore follow from corresponding bounds on each $W^{J,\phi}_p$ in~\eqref{eq_expression_Q_4_as_Wjp}. 
It is tempting to define $\sqrt{n}\, Q^{4+,\delta}_n$ as the right-hand side above without the quartic term, 
but some of the $W^{J,\phi}_p$ can only be bounded after an intermediate renormalisation step. 
This is stated in the next lemma, 
the technical heart of the free energy estimate.  

\begin{lemma}\label{lemm_estimate_WphiJp}
Let $(p,J,q)\in\cI$, where this set is defined in~\eqref{eq_def_cI}. 
Let $\phi=\phi^{J,p}:\T^p\to\R$ be bounded and continuous on the closure of $\{x_1,...,x_p\text{ all different}\}$. 

\begin{itemize}
	\item[(i)] (Renormalisation step).  
	For each $\delta>0$, 
	there is a non-negative function $E^\delta_n=E_n^{J,\phi,p,q,\delta}:\Omega_n\to\R_+$ such that, 
	for any density $f$ for $\nu^n_g$:
\begin{equation}
\big|\, \nu^{n}_g\big(f\, \sqrt{n}\, (\ms M^n)^{q}\, W^{J,\phi}_p\big)\, \big| 
\leq 
\frac{\delta n^{5/2}}{2} \, 
\Gamma^{\rm ex}_n(f;\nu^n_g)
+
\nu^n_g\big( f \, \sqrt{n}\, E^\delta_n\big)
\end{equation}
There are constants $C(\delta),C'(\delta)>0$ such that the function $E^\delta_n$ satisfies:
	\begin{equation}
\forall \eta\in\Omega_n,\qquad 
E^\delta_n(\eta)
\leq 
C(\delta)\|\phi\|_\infty\, n
,
\label{eq_size_En}
\end{equation}
	and:
	\begin{equation}
\nu^{n}_g\big(f\, \sqrt{n}\,  E^\delta_n\, \big)
\leq 
C'(\delta)\, \sqrt{n}\big( H_n(f \, |\, \nu^n_g) + 1\big)
.
\label{eq_small_time_bound_lemmWj}
\end{equation}
	\item[(ii)] (Long time bound).  
	There is a time $\tau$ depending only on $\|\phi\|_\infty,\delta$ and constants $C_T=C(\delta,T,\|\phi\|_\infty)>0$ such that, 
for all $t\in[\tau\, n^{-1/2},T]$:
\begin{align}
\nu^{n}_g\big(f_t\, \sqrt{n}\,  E^\delta_n\, \big)
&\leq 
\frac{\delta n^{5/2}}{2} \,
\Gamma^{\rm ex}_n(f_t;\nu^n_g)
+C_T\, \sqrt{n} 
.
\label{eq_bound_Wjp_adjoint}
\end{align}
\end{itemize}
\end{lemma}
With this lemma we can now define the quantity $Q^{4+,\delta}_n\geq 0$ appearing in Proposition~\ref{prop_bound_adjoint} as:
\begin{equation}
Q^{4+,\delta}_n 
:=
\sum_{(p,J,q)\in\cI} E^{\delta/|\cI|}_n + |\epsilon_n|
,
\end{equation}
where $\epsilon_n$ is the same as in~\eqref{eq_recap_Lstar_before_bounds}. 
Then $Q^{4+,\delta}_n$ satisfies the bounds of Proposition~\ref{prop_bound_adjoint}, 
the proof of which will therefore be complete once we prove Lemma~\ref{lemm_estimate_WphiJp}. 
\begin{remark}
Note that points $(i)$ and $(ii)$ are qualitatively different. Point $(i)$ is a statement that applies whatever the function $f$, in particular more generally than if $f=f_t$ is the law of the dynamics. 
On the contrary, $(ii)$ makes use of additional properties of the dynamics (such as large deviation bounds) and is therefore only stated for $f=f_t$. 
The time $\tau \, n^{-1/2}$ corresponds to the moment when large deviations bounds are applicable regardless of the initial condition of the dynamics. 
\end{remark}
\begin{proof}[Proof of Lemma~\ref{lemm_estimate_WphiJp}]
We start by decomposing the measure $f_t\nu^n_g$ in terms of slow and fast modes as in~\eqref{eq_recentering_sec_product}: 
\begin{align}
\nu^{n}_g\big(f_t\sqrt{n}\, (\ms M^n)^{q}W^{J,\phi}_p\big)
=
\sum_{m=0}^n \Big[\frac{m}{n}-\frac{1}{2}\Big]^{q} \nu^n_g\big(f_t{\bf 1}_{\sum_i\eta_i=m}\big) \nu^{n,m}_g\big(f_t^m \,  \sqrt{n}\, W^{J,\phi}_p\big)
.
\label{eq_magnetisation_splitting_renorm_non_pert}
\end{align}
We estimate this last average for fixed $m$. 
To do so, 
we can use the entropy inequality \eqref{n-35}, 
a log-Sobolev inequality for $\nu^{n,m}_g$ (see Assumption~\ref{ass_LSI}) 
and concentration results for the $W^{J,\phi}_p$, 
getting something of the form:
\begin{align}
\nu^{n,m}_g\big(f_t^m\, \sqrt{n}\, W^{J,\phi}_p\big)
&\leq 
\frac{\mf c_{\rm LS}\, n^{5/2}}{\lambda}\, \Gamma^{\rm ex}_n(f_t^m;\nu^{n,m}_g)
+\frac{\sqrt{n}\, }{\gamma}\log \nu^{n,m}_g\big[\exp(\lambda\, W^{J,\phi}_p)\big]
\nnb
&\leq 
\frac{\mf c_{\rm LS}\, n^{5/2}}{\lambda}\, \Gamma^{\rm ex}_n(f_t^m;\nu^{n,m}_g)
+O(\sqrt{n}\, )
,
\end{align}
where the second inequality holds provided $\lambda$ is small enough to
apply the concentration results in
Proposition~\ref{prop_concentration}.  The issue is that this means
$\lambda^{-1}$ may be big, whereas the parameter in front of
$\Gamma^{\rm ex}_n(f_t^m;\nu^{n}_g)$ in the
statement~\eqref{eq_bound_Wjp_adjoint} of the lemma should be at most
$\delta>0$.  This means the choice of $\lambda$ is fixed and, as
explained in Section~\ref{sec_sketch_free_energy}, 
we need to restrict the exponential moment to configurations where $W^{J,\phi}_p$ cannot be too large so that concentration
bounds apply.

If $q\geq 1$, i.e. a factor of $[m/n-1/2]$ appears in~\eqref{eq_magnetisation_splitting_renorm_non_pert}, 
then we can use this factor as a small parameter by restricting to magnetisations $|m/n-1/2|\leq \epsilon$ for a small $\epsilon>0$, up to a controllable error as we explain afterwards.

If $q=0$ there is no $[m/n-1/2]$ term and the estimate is more subtle.
We will again use a suitable truncation, but the fact that it does
give us a small parameter is only true because the $W^{J,\phi}_p$ to
estimate when $q=0$ involve four-point correlations or more (in fact
three-point correlations or more is the necessary assumption),
see Proposition~\ref{prop_concentration}. 

In each situation the small parameters are created by large deviation bounds, 
which only apply after a time of order $n^{-1/2}$ due to the initial condition (see Appendix~\ref{app_LD}). \\

To define $E_n^\delta$ and prove~\eqref{eq_bound_Wjp_adjoint}, 
we consider four different cases: 
$q\geq1$, $p\geq 2$; $q\geq1$, $p=1$; $q=0$, $p\ge 3$; and $q=0$,
$p=2$. The second and fourth cases require a renormalisation
argument. \\

\noindent\textbf{Case $q\geq 1,p\geq 2$.} 
With no need for renormalisation we can define $E_n^\delta$ independently of $\delta$:
\begin{equation}
E_n^\delta
:=
\big|\, (\ms M^n)^q \, W^{J,\phi}_p\, \big|
,
\end{equation}
Proposition~\ref{prop_concentration} and Corollary~\ref{coro_concentration_nun_g} imply the following bound on exponential moments of $W^{J,\phi}_p$ for $p\geq 2$: 
for some $\lambda_p,C_p>0$,
\begin{align}
\forall \lambda\in[0,\lambda_p],\qquad 
\sup_{|m/n-1/2|\leq 1/4}\log \nu^{n,m}_g\Big[\exp \big[\lambda\, |W^{J,\phi}_p|/\|\phi\|_\infty\, \big]\Big]
&\leq 
\frac{C_p\, \lambda}{n^{p-2}}
\nnb
\log \nu^{n}_g\Big[\exp \big[\lambda\, |W^{J,\phi}_p|/\|\phi\|_\infty\, \big]\Big]
&\leq 
\frac{C_p\, \lambda}{n^{p-2}}
.
\label{eq_concentration_WJphi_interm}
\end{align}
The trivial bound $E_n^\delta\leq \|\phi\|_\infty\, n$ and the second line above yield~\eqref{eq_small_time_bound_lemmWj} upon applying the entropy inequality:
\begin{align}
\nu^n_g\big(f\sqrt{n}\, E_n^\delta\big)
&\leq 
\frac{\|\phi_\infty\|\, \sqrt{n}}{\lambda_p}\Big(H_n(f\, |\, \nu^n_g) 
+
\log\nu^{n}_g\Big[\exp \big[\lambda_p\, |W^{J,\phi}_p|/\|\phi\|_\infty\, \big]\Big]
\Big)
\nnb
&\leq 
\frac{\|\phi_\infty\|\, \sqrt{n}}{\lambda_p}\Big(H_n(f\, |\, \nu^n_g) 
+ C_p\Big)
.
\label{eq_proof_bound_smalltime_lemmeWj}
\end{align}
Let us now prove~\eqref{eq_bound_Wjp_adjoint}. 
Let $\epsilon\leq 1/4$ and let $\tau_\epsilon = (2\epsilon^2)^{-1}$ so that Proposition~\ref{prop_LD_bound} gives the existence of $c(\epsilon,T)>0$ such that:
\begin{equation}
\sup_{t\in[\tau_\epsilon\, n^{-1/2},T]}\nu^n_g \big(f_t{\bf 1}_{|\sum_i\bar\eta_i|\geq n\epsilon}\big)
\leq 
c(\epsilon,T)^{-1}\, e^{-c(\epsilon,T)\, n}
.
\label{eq_LD_bound_magnetisation_lemm_WJphi}
\end{equation}
The splitting~\eqref{eq_magnetisation_splitting_renorm_non_pert} then becomes:
\begin{equation}
\nu^n_g(f_t E_n^\delta)
\leq 
c(\epsilon,T)^{-1}\, e^{-c(\epsilon,T)\, n}
+
\sum_{m:|m-n/2|\leq \epsilon n} 
\nu^n_g\big(f_t{\bf 1}_{\sum_i\eta_i=m}\big) 
\epsilon^q\, \nu^{n,m}_g(f^m_t |W^{J,\phi}_p|)
.
\label{eq_magnetisation_splitting_withepsilon}
\end{equation}
The entropy and log-Sobolev inequalities at fixed $m$ give:
\begin{align}
\epsilon^{q}\, \Big|\nu^{n,m}_g\big(f_t^mW^{J,\phi}_p\big)\Big|
&\leq 
\epsilon^{q/2}\, H_n(f_t^m|\nu^{n,m}_g) + \epsilon^{q/2}\, \log \nu^{n,m}_g\Big[\exp \big[\epsilon^{q/2}|W^{J,\phi}_p|\, \big]\Big]
\nnb
&\leq 
\epsilon^{q/2}\, {\mf c}_{\rm LS}\, n^2\, \Gamma^{\rm ex}_n(f_t^m;\nu^{n,m}_g) + \epsilon^{q/2} \log \nu^{n,m}_g\Big[\exp \big[\epsilon^{q/2}|W^{J,\phi}_p|\, \big]\Big]
\nnb
&\leq 
\delta\,  n^2\, \Gamma^{\rm ex}_n(f_t^m;\nu^{n,m}_g)+ \frac{\epsilon^{q/2}C_p\|\phi\|_\infty}{n^{p-2}}
,
\label{eq_use_entropy_LSI_WJphip}
\end{align}
where to get the last line we also took $\epsilon$ smaller than $\delta/{\mf c}_{\rm LS}$ and than $\lambda_p/\|\phi\|_\infty$, with $\lambda_p$ as in~\eqref{eq_concentration_WJphi_interm}.  
Multiplying by $\sqrt{n}\, $, 
summing on $m$ and using:
\begin{equation}
\sum_{m}\nu^n_g\big(f_t{\bf 1}_{\sum_i\eta_i=m}\big) 
\Gamma^{\rm ex}_n(f_t^m;\nu^{n,m}_g)
=
\Gamma^{\rm ex}_n(f_t;\nu^{n}_g)
,
\end{equation}
we obtain the claim~\eqref{eq_bound_Wjp_adjoint} in the $q\geq 1, p\geq 2$ case.\\

\noindent\textbf{Case $q\geq 1,p=1$.} 
In this case a direct application of the entropy inequality gives a worse bound than~\eqref{eq_bound_Wjp_adjoint}. 
However, 
due to $p+|J|-1\geq 2$ by definition~\eqref{eq_def_cI} of the set $\cI$ of allowed $(q,J,p)$, 
the $W^{J,\phi}_1$ we need to estimate have $|J|>1$. 
This enables us to use a renormalisation step as in the proof of Theorem~\ref{n-s02}, 
paying a little bit of the carré du champ to turn $W^{J,\phi}_1$ into $W^{J,\phi}_p$ with $p\geq 2$. 
For definiteness we look for an upper bound on:
\begin{equation}
\big|\nu^n_g\big(f_t\, \sqrt{n}\,  \ms M^n W^{\{-1,0,1\},\phi}_1\big)\big|
.
\end{equation}
In order for the entropy inequality to give a useful bound, 
we first replace the local term $\bar\eta_{i-1}^m\bar\eta_{i}^m\bar\eta_{i+1}^m$ by a space average. 
To do so, let $\ell_n\in[n/3,n/2-2]$ and set $\bar\eta_{j} := \bar\eta_{j\mod n}$ for $j\in\Z$. 
Recall the elementary identity:
\begin{equation}
\bar\eta^m_{i+1}
=
\frac{1}{\ell_n}\sum_{j=i+1}^{i+\ell_n}\bar\eta^m_{j}
+
\sum_{j\in\T_n} (\bar\eta^m_{j}-\bar\eta^m_{j+1}) I_{\ell_n}(j-i-1)
,\quad 
I_{\ell_n}(p)
=
\frac{\ell_n-1-p}{\ell_n}{\mb 1}_{[0,\ell_n-1]}(p)
.
\label{eq_def_I_elln}
\end{equation}
Integration by parts (Lemma~\ref{lemm_IBP})  applied for each $j\in\T_n$ with the $h$ in Lemma~\ref{lemm_IBP} given by $h=\sum_{i\in\T_n}\bar\eta^m_{i-1}\bar\eta^m_{i}I_{\ell_n}(j-i-1)$ then yields, for some $C=C(g)>0$ and each $\delta>0$:
\begin{align}
& \sqrt{n}\, 
\nu^{n}_g\big(f_t\ms M^n\, \sum_{i\in\T_n}\bar\eta^m_{i-1}\bar\eta^m_i\bar\eta^m_{i+1}\phi_i\Big)
\nnb
&\ \leq 
\sqrt{n}\, \nu^{n}_g\big(f_t\, \ms M^n\, \frac{1}{\ell_n }\sum_{i,j}\bar\eta^m_{i-1}\bar\eta^m_{i}\bar\eta^m_{j}\phi_i{\bf 1}_{[i+1,i+\ell_n]}(j)\Big)
+
\frac{\delta\, n^{5/2}}{2}\, \Gamma^{\rm ex}_n(f_t;\nu^{n}_g) 
\nonumber\\
&\quad 
+\frac{C \, \sqrt{n}\, }{\delta}\frac{1}{n}\sum_j \nu^{n}_g\Big[f_t \, (\ms M^n)^2\, \Big(\frac{1}{\sqrt{n}}\sum_{i}\bar\eta^{m}_{i-1}\bar\eta^m_{i}\phi_i I_{\ell_n}(j-i-1)\Big)^2\Big]
\label{eq_renormalisation_adjoint_middle}\\
&\quad 
+\frac{C \, \sqrt{n}\, }{n}\sum_j \nu^{n}_g\Big[f_t\, \ms M^n\,  \frac{1}{n}\sum_i\sum_{\ell\neq {j,j+1}}\partial^n_1 g_{j,\ell}\bar\eta^m_\ell\bar\eta^m_{i-1}\bar\eta^m_{i}\phi_i I_{\ell_n}(j-i-1) \Big]
\label{eq_renormalisation_adjoint}
\nnb
&\quad
+  C\|\phi\|_\infty \sqrt{n}\, \nu^n_g(f_t|\ms M^n|)
,
\end{align}
where the $C\|\phi\|_\infty\, \sqrt{n}$ in the last line is a crude bound on the error in Lemma~\ref{lemm_IBP} summed over $j\in\T_n$. 
Let $E_n^\delta$ denote the sum of the absolute value of all terms being averaged under $f_t\nu^n_g$ in~\eqref{eq_renormalisation_adjoint_middle}, 
so that~\eqref{eq_renormalisation_adjoint_middle} can be bounded from above by:
\begin{equation}
\nu^{n}_{g}(f_tE_n^\delta) + \frac{\delta\, n^{5/2}}{2}\, \Gamma^{\rm ex}_n(f_t;\nu^{n}_g) 
.
\end{equation}
Note that $E^\delta_n$ is then given by $C\|\phi\|_\infty \, |\ms M^n|$ plus a sum of terms of the form $|(\ms M^n)^{q'} W^{J,\phi}_{p'}|$ with $q'\geq 1$ and $p'\geq 2$. 
These terms have already been shown to satisfy~\eqref{eq_size_En}--\eqref{eq_small_time_bound_lemmWj}--\eqref{eq_bound_Wjp_adjoint}, 
so $E^\delta_n$ also does, 
concluding the proof in the $q\geq 1,p=1$ case. 
\\

\noindent\textbf{Case $q=0$, $p\geq 3$.} 
In this case also no renormalisation will be needed, 
so $E_n^\delta$ is independent of $\delta$ and simply given by $|\, W^{J,\phi}_p\, |$. 
By definition the uniform bound~\eqref{eq_size_En} holds and the exact same argument as in~\eqref{eq_proof_bound_smalltime_lemmeWj} also gives the bound~\eqref{eq_small_time_bound_lemmWj} in terms of the relative entropy.

We therefore focus on the long-time bound~\eqref{eq_bound_Wjp_adjoint}. 
Since $q=0$ restricting to small magnetisations is not enough to estimate exponential moments of $W^{J,\phi}_p$.
Instead, we directly ask for $|W^{J,\phi}_p|$ to be at most $\epsilon'\, n$ for a sufficiently small $\epsilon'>0$ to be chosen later, 
arguing below that this is indeed a large deviation event. 
That this restriction on the size of $W^{J,\phi}_p$ is indeed sufficient to control the exponential moment was explained in Section~\ref{sec_sketch_free_energy}. 
For concreteness let us bound:
\begin{equation}
W^{\{0,1\},\phi}_3
=
\frac{1}{n^2}\sum_{i,j,k\in\T_n} \bar\eta^m_{i}\bar\eta^m_{i+1}\bar\eta^m_j\bar\eta^m_k \, \phi_{i,j,k}
.
\end{equation}
We start by defining the good event. 
Recall that $\pi^n(dx)=\frac{1}{n}\sum_{i\in\T_n}\eta_i\delta_{i/n}$ denotes the empirical measure. 
Proposition~\ref{prop_LD_bound} controls the probability that $(\pi^n_s)_{s\geq 0}$ differs from $dx/2$. 
Observe that sums involving the $\bar\eta^m$ are also functions of the empirical measure. 
Indeed, if $i,j\in\T_n$, 
\begin{align}
\frac{1}{n}\sum_{k\in\T_n}\phi_{i,j,k}\bar\eta^m_k 
&=
\pi^n(\phi_{i,j,\cdot}) - \frac{\pi^n(1)}{n}\sum_{k\in\T_n}\phi_{i,j,k}
\nnb
&=
\bigg(\pi^n(\phi_{i,j,\cdot}) - \frac{1}{2}\int \phi_{i,j}(z)\, dz\bigg)- \Big(\pi^n(1)-\frac{1}{2}\Big)\cdot\frac{1}{n}\sum_{k\in\T_n}\phi_{i,j,k} 
\nnb
&\quad - \bigg(\frac{1}{2n}\sum_{k\in\T_n}\phi_{i,j,k} - \frac{1}{2}\int \phi_{i,j}(z)\, dz\bigg)
.
\end{align}
The last difference is constant and equal to $o_n(1)$ uniformly in $i,j$, 
while the continuity of $\phi_{i,j}(\cdot)$ on the closure of $[0,1]\setminus\{i/n,j/n\}$ (with uniform norm independent of $n$) is enough to apply the large deviation bounds of Proposition~\ref{prop_LD_bound} can be applied as explained in Remark~\ref{rmk_no_global_continuity_LD}. 
For each $\epsilon'>0$, there is thus $c(\epsilon',T)>0$ such that:
\begin{equation}
\forall n\geq 1,
\qquad 
\sup_{t\in[\tau_{\epsilon'}\, n^{-1/2},T]}\nu^n_g\bigg(f_t \, {\bf 1}\Big\{\Big|\sum_{k\in\T_n} \bar\eta^m_k\phi_{i,j,k}\Big| \geq \epsilon' n\Big\}\bigg)
\leq 
c(\epsilon',T)^{-1}\, e^{-n c(\epsilon',T)}
.
\end{equation}
The compactness of $[0,1]^3$ and continuity of $\phi$ on the closure of $\{x_1,x_2,x_3 \text{ all different}\}$ imply that the bound is in fact valid uniformly on $i,j$: 
for a different constant $c(\epsilon',T)$,
\begin{equation}
\forall n\geq 1,
\qquad 
\sup_{t\in[\tau_{\epsilon'}\, n^{-1/2},T]}\nu^n_g\bigg(f_t\,  {\bf 1}\Big\{\exists i,j\in\T_n, \ \Big|\sum_{k\in\T_n} \bar\eta^m_k\phi_{i,j,k}\Big| \geq \epsilon' n\Big\}\bigg)
\leq 
c(\epsilon',T)^{-1}\, e^{-n c(\epsilon',T)}
.
\end{equation}
Restricting also to magnetisation at most $1/4$ using the large deviation bound~\eqref{eq_LD_bound_magnetisation_lemm_WJphi} and taking $\epsilon'\leq 1/4$ so that $\tau_{\epsilon'}\geq \tau_{1/4}$, 
the splitting~\eqref{eq_magnetisation_splitting_renorm_non_pert} then becomes, for any $t\geq \tau_{\epsilon'}\, n^{-1/2}$:
\begin{align}
\big|\nu^{n}_g\big(f_t\,\sqrt{n}\, W^{J,\phi}_p \big)\big|
&\leq 
n^{3/2}\big[c(1/4,T)e^{-n/c(1/4,T)} + c(\epsilon',T)e^{-n/c(\epsilon',T)}\big]
\\
&\qquad +\sum_{|m-n/2|\leq n/4}\nu^n_g\big(f_t{\bf 1}_{\sum_i\eta_i = m}\big)\nu^{n,m}_g\big(f^m_t\, \sqrt{n}\, \big|W^{J,\phi}_p \big|{\bf 1}_{|W^{J,\phi}_p|\leq \epsilon' n}\big)
\nonumber
.
\end{align}
Applying the entropy- and log-Sobolev inequalities as in~\eqref{eq_use_entropy_LSI_WJphip} and invoking Corollary~\ref{coro_better_concentration} 
to control exponential moments of $W^{J,\phi}_p {\bf 1}_{|W^{J,\phi}_p|\leq \epsilon' n}$ uniformly in $m$ with $|m/n-1/2|\leq 1/4$, 
we find that, 
for $\epsilon'$ small enough depending on $\|\phi\|_\infty,\delta$, 
$t\geq \tau_{\epsilon'}$ and some $C(\epsilon',T)>0$:
\begin{equation}
\big|\nu^n_g(f_t \, \sqrt{n}\, W^{J,\phi}_p)\big|
\leq 
\delta \, n^{5/2}\, \Gamma_n^{\rm ex}(f_t;\nu^n_g) + C(\delta)\, \sqrt{n}\,  + C(\epsilon',T)^{-1} e^{-nC(\epsilon',T)}
.
\label{eq_bound_WJphi_q0}
\end{equation}
\noindent{\bf Case $q=0$, $p=2$.} 
For $p=2$ and $q=0$, 
any $(p,q,J)\in\cI$ satisfies $J=\{-1,0,1\}$, 
in other words $W^{\{-1,0,1\},\phi}_2$ contains products of four $\bar\eta^m$'s. 
However, at the level of exponential moments, 
$W^{J,\phi}_2$ does not satisfy better bounds than the square of a Gaussian, 
which we only knew how to bound upon multiplication with a small parameter as explained in Section~\ref{sec_sketch_free_energy}. 
In particular the truncation argument of the previous $q=0$, $p\geq 3$ case does not give better bounds on exponential moments of $W^{\{-1,0,1\},\phi}_2$. 
We therefore first use a renormalisation step to turn $W^{\{-1,0,1\},\phi}_2$ into a sum of $W^{J,\phi}_{p'}$ with $p'\geq 3$. 
For each of these terms the previous $q=0$, $p\geq 3$ case then applies.

Let $\chi:\R\to[0,1]$ be a smooth function, compactly supported in
$[1/2,1]$.  Write for short $W$ for $W^{\{-1,0,1\},\phi}_2$:
\begin{equation}
W 
:=
\frac{1}{n}\sum_{i\in\T_n} \bar\eta^m_{i-1}\bar\eta^m_i\bar\eta^m_{i+1}\sum_{|j-i|>1}\phi_{i,j}\bar\eta^m_j 
.
\end{equation}
Decompose $\phi=\phi^r+\phi^\ell$ and correspondingly $W=W^r+W^\ell$, 
where:
\begin{equation}
\phi^r(x,y)
=
\phi(x,y)\chi(y-x),
\quad
\phi^\ell(x,y)
=
\phi(x,y)(1-\chi(y-x)),
\qquad 
(x,y)\in\T^2
.
\end{equation}
In this way the sum on $j$ in $W^{r}$ does not involve the
$\lfloor n/2\rfloor$ indices to the right of $i$, i.e. indices between
$i+1$ and $i+\lfloor n/2\rfloor$; while $j$ does not take
any of the $\lfloor n/2\rfloor$ indices to the left of $i$ in
$W^\ell$.  We can therefore take $\ell_n = \lfloor n/4\rfloor$, say,
and turn $\bar\eta^m_{i+1}$ into
$\ell_n^{-1}\sum_{k=i+1}^{i+\ell_n}\bar\eta^m_k$ for $W^r$ through a
series of Kawasaki moves that never change the state of one of the
$\bar\eta_j$.  Similarly, we can average to the left of $i-1$ by
turning $\bar\eta^m_{i-1}$ into
$\ell_n^{-1}\sum_{k=i-\ell_n}^{i-1}\bar\eta^m_k$ for $W^\ell$.  Since the estimate in both cases is identical, we only work with $W^r$.

Integration by parts (Lemma~\ref{lemm_IBP}) gives, 
for each $\delta>0$, 
the existence of $C=C(g)>0$ such that, 
for any density $f$ for $\nu^n_g$:
\begin{align}
&\sqrt{n}\, \nu^{n}_g(f  W^r)
\nnb
&\ \leq 
\nu^{n}_g\bigg(f \frac{\sqrt{n}\, }{\ell_n n}\sum_{i,k\in\T_n}\sum_{|j-i|>1}\bar\eta^m_{i-1}\bar\eta^m_{i}\bar\eta^m_{j}\bar\eta^m_k\phi^r_{i,j}\,  {\bf 1}_{[i+1,i+\ell_n]}(k)\bigg)
+\delta \, n^{5/2}\,  \Gamma^{\rm ex}_n(f;\nu^{n}_g)
\nonumber\\
&\quad 
+\frac{C\, \sqrt{n}\, }{\delta}\frac{1}{n}\sum_{k\in\T_n} \nu^{n}_g\Big[f \Big(\frac{1}{n^{3/2}}\sum_{i}\bar\eta^{m}_{i-1}\bar\eta^m_{i}I_{\ell_n}(k-i-1)\sum_{|j-i|>1}\phi^r_{i,j} \bar\eta^m_j \Big)^2\Big]
\label{eq_renormalisation_adjoint_middle_q0term}\\
&\quad 
+\frac{C\sqrt{n}\, }{n}\sum_{k\in\T_n} \nu^{n}_g\Big[f\frac{1}{n^2}\sum_{i, j:|j-i|>1}\sum_{\ell\neq {j,j+1}}\partial^n_1 g_{j,\ell}\bar\eta^m_\ell\bar\eta^m_{i-1}\bar\eta^m_{i}\bar\eta^m_j\, \phi^r_{i,j}\, I_{\ell_n}(k-i-1)\Big]
\label{eq_renormalisation_adjoint_q0term}
\\
&\quad 
+C\|\phi\|_\infty \sqrt{n}\, \nu^n_g(f_t|\ms M^n|)
\end{align}
For the middle line~\eqref{eq_renormalisation_adjoint_middle_q0term}, 
apply Jensen inequality on the sum on $i$ to find:
\begin{align}
\nu^{n}_g\Big[f \Big(\frac{1}{n^{3/2}}&\sum_{i}\bar\eta^{m}_{i-1}\bar\eta^m_{i}I_{\ell_n}(k-i-1)\sum_{|j-i|>1}\phi^r_{i,j} \bar\eta^m_j \Big)^2\Big]
\nnb
&\qquad
\leq 
\frac{1}{n}\sum_{i}\nu^{n}_g\Big[f \Big(\frac{1}{\sqrt{n}\, }\sum_{|j-i|>1}\phi^r_{i,j} \bar\eta^m_j \Big)^2\Big]
,
\end{align}
which is then of the form $n^{-1/2}\, W^{\{0\},\psi^i}_{1} \leq \sqrt{n}\, \|\phi\|_\infty$ for each $i\in\T_n$. 
On the other hand, 
the observables in the right-hand side of~\eqref{eq_renormalisation_adjoint_q0term} in the first and last line are now of the form $W^{J,\psi}_p$ for $p\geq 3$ with $\psi$ bounded and continuous on the closure of $\{x_1,x_2,x_3\text{ all different}\}$. 
The case $q=0$, $p\geq 3$ thus applies, 
and $E_n^\delta$ in the present case is given by the sum of the $E_n^\delta$ associated with the $W^{J,\psi}_p$ arising in~\eqref{eq_renormalisation_adjoint_q0term} as well as those coming from the renormalisation of $W^\ell$. 
This concludes the proof. 
\end{proof}
\subsection{Moment bound}\label{sec_bound_moment}
In this section we prove Proposition~\ref{prop_bound_moment}.  Let
$t\geq 0$.  The semi-martingale decomposition of $(\cY^n_t)^2$ implies
(the first identity is a tautology):
\begin{align}
\sqrt{n}\, \partial_t \E\big[(\cY^n_t)^2\big]
&=
\sqrt{n}\, \partial_t \nu^n_g\big(f_t (\cY^n)^2\big)
=
a\, n\, \nu^n_g\big(f_t L^G_n(\cY^n)^2\big)
\nnb
&=
\nu^n_g\Big(f_t\Big[-4\,a\, n^{1/4}\, \cY^n \sum_{i\in\T_n} \bar\eta_i c_i(\eta) + \frac{1}{\sqrt{n}}\sum_{i\in\T_n} c(\tau_i\eta)\Big]\Big)
.
\end{align}
The last term in the bracket is bounded by $O(\sqrt{n}\, )$ uniformly in the configuration. 
Moreover, recalling the expression~\eqref{eq_c_i_sec_free_energy} of the jump rate $c$\,:
\begin{equation}
\bar\eta_i c(\tau_i\eta)
=
\bar\eta_i\big[1+4\gamma^2\bar\eta_{i-1}\bar\eta_{i+1}\big] - \gamma(\bar\eta_{i-1}+\bar\eta_{i+1})
.
\end{equation}
As a result and as $1-2\gamma=\theta \, n^{-1/2}$,
\begin{align}
\sqrt{n}\, \partial_t \E\big[(\cY^n_t)^2\big]
&=
-4\, a\, n^{1/4}\, \nu^n_g\Big(f_t \cY^n \sum_i \bar\eta_i[1-2\gamma]\Big) 
\nnb
&\qquad
-16\, a\, \gamma^2\, n^{1/4}\,  \nu^n_g\Big(f_t \cY^n \sum_i\bar\eta_{i-1}\bar\eta_i\bar\eta_{i+1}\Big) + O(\sqrt{n}\, )
\nnb
&=
-4\, a\, \theta\, \sqrt{n}\, \nu^n_g\big(f_t (\cY^n)^2 \big) 
-16\, a\, \gamma^2 \, \sqrt{n}\, \nu^n_g\big(f_t (\cY^n)^4\big) 
\nnb
&\qquad 
+\sqrt{n}\, \nu^n_g\big(f_t R^m\big)+ O(\sqrt{n}\, )
,
\label{eq_Ito_Y4}
\end{align}
where $O(\sqrt{n}\, )$ is uniform in time and the remainder $R^m$ is obtained by recentering $\bar\eta$ as $\bar\eta^m+\ms M^n $ as usual:
\begin{align}
R^m(\eta)
:=
-16\, a\, \gamma^2\, \ms M^n\sum_{i\in\T_n} \bar\eta^m_{i-1}\bar\eta^m_i\bar\eta^m_{i+1}
-16\, a\, \gamma^2\, (\ms M^n)^2\sum_{i\in\T_n} \big[\bar\eta^m_{i-1}\bar\eta^m_{i+1}+2\bar\eta^m_i\bar\eta^m_{i+1}\big]
.
\end{align}
According to Lemma~\ref{lemm_estimate_WphiJp}, 
this remainder satisfies, for some $\tau>0$, 
each $\delta>0$ and each $t\in[0,T]$:
\begin{equation}
|\nu^n_g(f_tR^m)|
\leq 
\delta\, n^{5/2}\, \Gamma_n^{\rm ex}(f_t;\nu^n_g) + C_1(\delta,T)\sqrt{n}
+{\bf 1}_{[0,\tau\, n^{-1/2}]}(t)\, C_2(\delta,T)\, \sqrt{n}\, H_n(f_t\, |\, \nu^n_g)
.
\label{eq_def_tau_moment_bound_sec_65}
\end{equation}
Absorbing also the $(\cY^n)^2$ term in~\eqref{eq_Ito_Y4} into $-(\cY^n)^4$ using the elementary identity:
\begin{equation}
4\theta x^2
\leq 
\frac{\theta^2}{\gamma^2}  + 8\gamma^2x^4,
\qquad 
x\in\R,
\end{equation}
we find that 
for each $\delta>0$, 
there is $C'(\delta,T)>0$ such that, for any $t\in[0,T]$:
\begin{align}
\sqrt{n}\, \partial_t \E\big[(\cY^n_t)^2\big]
&\leq 
\delta\, n^{5/2}\, \Gamma^{\rm ex}_n(f_t;\nu^n_g)
-8\, a\, \gamma^2\, \sqrt{n}\,  \nu^n_g\big(f_t (\cY^n)^4\big) 
\nnb
&\quad 
+ C'(\delta,T)\, \sqrt{n} 
+{\bf 1}_{[0,\tau\, n^{-1/2}]}(t)\, C_2(\delta,T)\, \sqrt{n}\,  H_n(f_t\, |\, \nu^n_g)
.
\label{eq_bound_derivativeY2}
\end{align}
This proves Proposition~\ref{prop_bound_moment}. $\hfill\square$\\

The next corollary is useful in the proof of tightness. 
\begin{corollary}\label{coro_4th_moment}
Let $T>0$. 
Then:
\begin{equation}
\sup_{n\geq 1}\E\Big[\int_0^T (\cY^n_t)^4\, dt\Big]
<
\infty
.
\end{equation}
\end{corollary}
\begin{proof}
Using Equation~\eqref{eq_bound_derivativeY2}, 
we can express $\int_0^T(\cY^n_s)^4\, ds$ in terms of $(\cY^n_T)^2$, $(\cY^n_0)^2$. 
Indeed,  
for each $\delta>0$, 
there is $C(\delta,T)>0$ such that:
\begin{align}
8\, a\, \gamma^2\, \E\Big[\int_0^T (\cY^n_t)^4\, dt\Big]
&\leq 
\delta \int_0^T n^2\, \Gamma^{\rm ex}_n(f_t;\nu^n_g)\, dt
+
\E\Big[\int_0^T R^m(\eta_t)\, dt\Big] 
\nnb
&\quad - \E[(\cY^n_t)^2] + \E[(\cY^n_0)^2] + C(\delta,T)
.
\end{align}
Theorem~\ref{theo_free_energy} implies that $\E[(\cY^n_t)^2]$ is bounded uniformly in $n$ and in $t\in[0,T]$. 
Moreover, 
by Lemma~\ref{lemm_estimate_WphiJp} $R^m$ satisfies, 
for each $\delta'>0$, some $\tau>0$ and each $t\geq 0$:
\begin{equation}
\big|\E\big[R^m(\eta_t)\big]\big|
\leq 
\delta' n^{2}\, \Gamma^{\rm ex}_n(f_t,\nu^n_g) + C(\delta,T) 
+{\bf 1}_{[0,\tau\, n^{-1/2}]}(t)\,  H_n(f_t\, |\, \nu^n_g)
. 
\end{equation}
The proof is concluded by integrating the last equation on $[0,T]$ and using the following bound on the carré du champ, 
consequence of Theorem~\ref{theo_free_energy} (and Proposition~\ref{prop_free_energy_time0} to control the free energy at initial time appearing in Theorem~\ref{theo_free_energy}):
\begin{equation}
\sup_{n\geq 1}\int_0^T n^{2}\, \Gamma^{\rm ex}_n(f_t;\nu^n_g)\, dt
<
\infty
.
\end{equation}
\end{proof}

\section{Proof of Theorem \ref{theo_convergence_magnetisation}}
\label{sec3}

In this section, we prove Theorem~\ref{theo_convergence_magnetisation}
using the bounds of Theorem \ref{n-s05} and
Theorem~\ref{theo_free_energy} on relative entropy and free energy,
respectively.

Throughout this section, 
the dynamics is assumed to start from $\mu_n$, 
with $(\mu_n)_{n\ge 1}$ a sequence of
probability measures such that
\begin{equation}
\label{n-26}
H_n(\mu_n \,|\, \nu^n_{1/2}) \,\le C_0\, \sqrt{n}
\end{equation}
for some finite constant $C_0$, 
and:
\begin{equation}
\mu_n\big(\cY^n\in \cdot\big) \quad \text{ converges weakly}.
\end{equation}
A time $T>0$ is also fixed. 
Recall that ${\tt Y}^n(H)=\frac{1}{n^{3/4}}\sum_{i\in\T_n}\bar\eta_i H_i$ for $H\in C^\infty(\T)$. 
We prove below at the end of Section~\ref{sec_martingale_pb} (see~\eqref{eq_proof_other_modes_negligible_n34}) that:
\begin{equation}
\E_{\mu_n}\Big[ \, \int_0^T \big|\, {\tt Y}^n_t(H)-\cY^n_t \<H\>\,\big|\, dt\, \Big]
=
o_n(1)
,\qquad 
\< H\> := \int_{\T}H(x)\, dx
.
\label{eq_other_modes_negligible_n34}
\end{equation}
We therefore focus on proving the part of Theorem~\ref{theo_convergence_magnetisation} relative to the convergence in law of $(\cY^n_t)_{t\in[0,T]}$. 
For $y:[0,T]\to\R$ such that $\int_0^Ty_s^4\, ds<\infty$, define:
\begin{equation}
\label{n-04}
W_t(y_0,y)
:=
y_t -y_0 + 2\, a\, \int_0^t\big [\, \theta\, y_s + y_s^3\, \big]\, ds
,\qquad 
t\in[0,T]
.
\end{equation}
It is convenient above to single out the initial condition in the
notation $W_t(y_0,y)$ as we will control trajectories only in a suitable $\bbL^p$ space.

Recall the following standard well-posedness result (Theorem 5.5.15
in \cite{ks}). 
\begin{lemma}\label{lemma_martingale_problem_well_posed}
There is a unique probability measure supported on continuous
trajectories $(\cY_t)_{t\in[0,T]}$ on $[0,T]$ such that:
\begin{equation}
\big(W_t(\cY_0,\cY)\big)_{t\in[0,T]}
=
(\sqrt{a}\, B_t)_{t\in[0,T]}
\quad
\text{in distribution,}
\end{equation}
where $B_\cdot$ is a standard Brownian motion. 
\end{lemma}
We will more precisely use the following corollary, 
proven at the end of Section~\ref{sec_continuity}.
\begin{corollary}\label{coro_martingale_problem_well_posed}
Let $p\geq 1$. 
There is a unique probability measure $\bb Q$ on $\R\times \bb L^p([0,T],\R)$ such that:
\begin{itemize}
	\item[(i)] $\bb Q( \{ (y_0,y): y\text{ has a continuous representative starting at $y_0$}\}) = 1$. 
	\item[(ii)]  ${\rm Law}_{\bb Q}\big((W_t(\cY_0,\cY))_{t\in[0,T]}\big)={\rm Law}((\sqrt{a}\, B_t)_{t\in[0,T]})$ as measures on $\bb L^p([0,T],\R)$. 
\end{itemize}
\end{corollary}
The proof of Theorem~\ref{theo_convergence_magnetisation} is structured as follows. 
\begin{itemize}
	\item We show that
the sequence of laws of $(\cY^n_0,(\cY^n_t)_{t\in[0,T]})$ is tight in the space of probability
measures on $\R\times \bbL^p([0,T],\R)$ for $p\in(1,4/3)$ (Section~\ref{sec_tightness}). 
	\item The process $(W_t(\cY^n_0,\cY^n))_{t\in[0,T]}$ is shown to
        converge in distribution in $\bbL^p([0,T],\R)$ for
        $p\in(1,4/3)$ to $(\sqrt{a}\, B_t)_{t\in[0,T]}$,  a multiple of Brownian motion (Section~\ref{sec_martingale_pb}). 
        Recall that a probability measure $m$ on $C([0,T],\R)$ can be canonically lifted to a probability measure $\tilde m$ on $\bb L^p([0,T],\R)$ through $\tilde m(A) = m(j^{-1}(A))$ for a Borel set $A$ in $\bb L^p([0,T],\R)$, 
where $j:C([0,T],\R)\to\bb L^p([0,T],\R)$ is the measurable mapping of a continuous function to its $\bb L^p$ representative. 
The law of $(\sqrt{a}\, B_t)_{t\in[0,T]}$ can therefore unambiguously be defined as a probability measure on $\bb L^p([0,T],\R)$. 
	\item Given $y_0\in\R$, the function 
        $(y_t)_{t\in [0,T]} \mapsto (W_t(y_0,y))_{t\in [0,T]}$ is not
        continuous in $\bb L^p([0,T],\R)$ for $p\in(1,4/3)$. The previous point thus does not prove that limit points of the law of
        $\cY^n_\cdot$ are concentrated on trajectories $\cY$ such that
        $W_\cdot(\cY_0, \cY)=\sqrt{a}\, B_\cdot$ in law.  We prove that it is
        nonetheless the case in Section~\ref{sec_continuity}, 
        that limit points are concentrated on couples $(y_0,y)$ where $y\in \bb L^p([0,T],\R)$ has a continuous representative starting at $y_0$ and therefore
        conclude by Corollary~\ref{coro_martingale_problem_well_posed}.
\end{itemize}
\subsection{Tightness}\label{sec_tightness}
Recall the semi-martingale decomposition of the magnetisation:
\begin{equation}
\cY^n_t
=
\cY^n_0 + \int_0^t \sqrt{n}\, L_n\cY^n_s\, ds
+ M^n_t
,
\label{eq_semimart_decomp_Y}
\end{equation}
where $M^n_\cdot$ is a martingale. 
We wish to prove that, 
for some $p>1$, 
the laws of $(\cY^n_0,(\cY^n_t))_{t\in[0,T]}$ $(n\geq 1)$ are tight as probability measures on $\R\times\bbL^p([0,T],\R)$. 
Since $\cY^n_0$ converges weakly by assumption, 
it is enough to prove tightness in $\bbL^p([0,T])$ of the other two terms in~\eqref{eq_semimart_decomp_Y}.

The next lemma shows that the laws of the martingale term $M^n_\cdot$ are in fact tight probability measures on the space $\cD([0,T],\R)$ of left-limited, right-continuous processes. 
\begin{lemma}[Tightness of martingale term]
The process $(M^n_t)_{t\in[0,T]}$ satisfies:
\begin{itemize}
	\item[(i)] $\sup_{n\geq 1}\sup_{t\in[0,T]}\E_{\mu_n}[(M^n_t)^2]<\infty$.
	\item[(ii)] (Aldous criterion). If $\cS_T$ denotes the set of stopping times bounded by $T$, 
	\begin{equation}
	\limsup_{\delta\to 0}\sup_{\eta\leq \delta}\sup_{\tau\in \cS_T}\limsup_{n\to\infty} \E^n_{\mu_n}\big[\, | M^n_{(\tau+\eta)\wedge T}-M^n_\tau\, |\big]
	=
	0
	.
	\end{equation}
\end{itemize}
In particular the sequence of laws of $(M^n_t)_{t\in[0,T]}$ is tight in the space of probability
measures on $\cD([0,T],\R)$. 
\end{lemma}
\begin{proof}
Let $t\in[0,T]$. 
Recall that the quadratic variation of $M^n_\cdot$ reads:
\begin{equation}
\<M^n\>_t \,=\, \int_0^t  \frac{a}{n} \sum_{i\in\T_n} c(\tau_i
\eta^n(s)) \, ds 
\leq 
a\, (1+\gamma^2)\, t
.
\end{equation}
Since $((M^n_s)^2-\<M^n\>_s)_{s\geq 0}$ is a martingale, 
\begin{align}
\E^n_{\mu_n}[\, |M^n_{t}|^2\, ]
=
\E^n_{\mu_n}[\, \<M^n\>_{t}\, ]
\leq
a\, (1+\gamma^2)\, T
.
\label{eq_BDG}
\end{align}
This proves item (i).  Since for $\tau\in\cS_T$ the process
$M^n_{\tau+\cdot}-M^n_{\tau}$ is also a martingale with quadratic
variation $\<M^n\>_{\tau+\cdot}-\<M^n\>_{\tau}$, the same argument
also gives item (ii).
\end{proof}

We now prove tightness of the drift term
in~\eqref{eq_semimart_decomp_Y}. 
Proving tightness in $\cD([0,T],\R)$ would require a uniform control on increments of the drift term.  
We only have controls on a weaker modulus of continuity due to the delicate replacement estimates involved. 
This is the
reason why tightness of $\cY^n_\cdot$ is proven in
$\bbL^p([0,T],\R)$ only.

For $p\ge 1$ and $0<\epsilon<1$, define the fractional Sobolev space $\bbW^{\epsilon,p}([0,T],\R)$ as 
the subset of $\bbL^p([0,T],\R)$ for which the following
norm is finite:
\begin{equation}
\|u\|_{\bbW^{\epsilon,p}([0,T],\R)}
:=
\|u\|_{\bbL^p([0,T])} + [u]_{\epsilon,p,T},
\qquad 
[u]_{\epsilon,p,T}
=
\bigg(\int_0^T\int_0^T \frac{|u(t)-u(s)|^p}{|t-s|^{\epsilon p+1}}\,
ds\, dt\bigg)^{1/p} 
.
\end{equation}
It is a standard result (see e.g.~\cite[Theorem
7.1]{FractionalSobolev}) that the embedding of
$\bbW^{\epsilon,p}([0,T],\R)$ into $\bbL^{p}([0,T],\R)$ is compact for
any $0<\epsilon<1$ and $p\geq 1$.  Thus, tightness of the drift term
$v(t) = \int_{[0,t]} \sqrt{n}\, L_n\cY^n_r\, dr$ in
$\bbL^{p}([0,T],\R)$ follows from the estimate:
\begin{equation}
\sup_{s,t\in[0,T]}\sup_{n\geq 1}\E^n_{\mu_n}
\bigg[\frac{1}{|t-s|^{\epsilon p}}\, \Big| \int_s^t \sqrt{n}\,
L_n\cY^n_u\, du\, \Big|^p\, \bigg]
<
\infty
\label{eq_tightness_Wsp}
\end{equation}	
for some $\epsilon\in(0,1)$, because in this case
$\sup_{n\geq 1}\E^n_{\mu_n} [\|v\|_{\bbW^{\epsilon',p}([0,T],\R)}] <
\infty$ holds for any $\epsilon'\in(0,\epsilon)$.
         
We first show in the next lemma that one can replace the integral term
in~\eqref{eq_tightness_Wsp} by the drift of the
SDE~\eqref{eq_SDE_magnetisation_in_thm}.

\begin{lemma}\label{lemm_replacement_tightness}
Define:
\begin{equation}
\cR^n_t 
:=
\int_0^t\big\{\, \sqrt{n}\,  L_n\mathcal Y^n_s\, \,+\, 2\, \theta \, a\, \mathcal Y^n \,+\,
2 \, a \, (\mc Y^n_s)^3 \big\} \, ds 
,\qquad 
t\geq 0
.
\end{equation}
Then $\cR^n_\cdot$ satisfies~\eqref{eq_tightness_Wsp} for $p\in(1,4/3)$ and $\epsilon = \frac{p-1}{p}$.  
\end{lemma}
Assuming Lemma~\ref{lemm_replacement_tightness}, the proof of
tightness of the drift term is concluded if we
prove~\eqref{eq_tightness_Wsp} for $\int_0^\cdot\cY^n_s\, ds$,
$\int_0^\cdot(\cY^n_s)^3\, ds$.  Recall
Corollaries~\ref{n-s11}--\ref{coro_4th_moment}:
\begin{equation}
\sup_{n\geq 1}\E^n_{\mu_n}\Big[\int_{0}^{T}|\cY^n_s|^4\, ds\Big]
<
\infty
.
\end{equation}
H\"older inequality then yields the claim with $p=4/3$ and
$\epsilon p =1/3$: 
e.g. for $(\cY^n)^3$, 
\begin{equation}
\E^n_{\mu_n}\Big[ \Big( \int_{s}^{t}|\cY^n_s|^3\, ds\Big)^{4/3}\,  \Big]
\leq 
|t-s|^{1/3} \, \E^n_{\mu_n}\Big[\int_{0}^{T}|\cY^n_s|^4\, ds\Big]
,\qquad 
s,t\in[0,T]
.
\end{equation}
\begin{proof}[Proof of Lemma~\ref{lemm_replacement_tightness}]
Computations similar to those carried out to estimate $\sqrt{n}\, L_n(\cY^n)^2$ in Section~\ref{sec_bound_moment} give:
\begin{align}
\sqrt{n}\, L_n\mathcal Y^n
=
-2\, \theta\, a\, \cY^n -8\, a\, \gamma^2\, (\cY^n)^3
-8\, a\, \gamma^2\, R^m,
\end{align}
where we recall $\gamma=\frac{1}{2}(1-\theta n^{-1/2})$
and with $R^m$ the replacement term obtained by projecting on mean-0 modes: 
for each $\eta\in\Omega_n$, 
writing $\bar\eta^m_i = \bar\eta_i - \frac{1}{n}\sum_{j\in\T_n}\bar\eta_j$,
\begin{equation}
R^m(\eta)
=
\frac{1}{n^{1/4}}\sum_{i\in\T_n}\bar\eta^m_{i-1}\bar\eta^m_{i-1}\bar\eta^m_{i+1} + \frac{\cY^n}{\sqrt{n}\, }\sum_{i\in\T_n}\big[2\bar\eta^m_{i}\bar\eta^m_{i+1} + \bar\eta^m_{i-1}\bar\eta^m_{i+1}\big]
.
\label{eq_def_Rm}
\end{equation}
A renormalisation step as carried out e.g. in Lemma \ref{n-l01} or in
the proof of Lemma~\ref{lemm_estimate_WphiJp}
(see~\eqref{eq_renormalisation_adjoint_middle}) shows that there
exists $S^{m}\geq 0$ (given by $n^{-1/4}E^{1}_n$ there), with
$\|S^{m}\|_\infty= O(n^{3/4})$, such that:
\begin{equation}
n^{3/4}|\nu^n_{\rm ref}(fR^m)|
\leq 
\frac{n^{5/2}}{2} \, \Gamma^{\rm ex}_n(f;\nu^n_{\rm ref}) + 
n^{3/4}\, \nu^n_{\rm ref}(f\, S^m)
.
\label{eq_def_Sm}
\end{equation}
With the notations~\eqref{eq_def_WJphi_p}, 
$n^{3/4}S^m$ is a sum of terms of the form $\sqrt{n}\, W^{J,\phi}_2$ for bounded $\phi$ and $J\in\{ \{0\},\{0,1\}\}$. 
In particular, 
Theorem~\ref{n-s02} (for $a\leq {\mf a}_0$) and Lemma~\ref{lemm_estimate_WphiJp} and the $O(\sqrt{n}\, )$relative entropy bound in Theorem~\ref{theo_free_energy} ($a\geq {\mf a}_0$)
show that there exists $\tau>0$ and $c_1,c_2(T)>0$ such that, for all $u\in[0,T]$:
\begin{equation}
\nu^n_{\rm ref}\big(\, f_u\, |S^m|\, \big)
\leq 
c_1\, n^{5/2-3/4}\, \Gamma^{\rm ex}_n(f_u;\nu^n_{\rm ref})
+
c_2(T)\, n^{-1/4}\, \big(1
+{\bf 1}_{[0,\tau\, n^{-1/2}]}(t)\, \sqrt{n}\, \big)
.
\label{eq_bound_Sm}
\end{equation}

A similar bound holds for $L^*_n \mb 1$ as proven in~\eqref{n-09}
and Theorem~\ref{n-s02} if $a\leq \mf a_0$, or in
Proposition~\ref{prop_bound_adjoint},
and~\eqref{eq_bounds_error_prop_free_energy} for $a\geq {\mf a}_0$:
for each $\delta>0$,
\begin{equation}
|\nu^n_{\rm ref}(f\, \sqrt{n}\, L^*_n{\mb 1})|
\leq 
\frac{\delta\, n^{5/2}}{2}\, \Gamma^{\rm ex}_n(f;\nu^n_{\rm ref}) + 
\nu^n_{\rm ref}(f \, \sqrt{n}\, Q_n^{4+,\delta})
,
\label{eq_def_ellm}
\end{equation}
where for $u\in[0,T]$:
\begin{align}
\nu^n_{\rm ref}\big(\, f_u|Q_n^{4+,\delta}|\,\big)
&\leq 
\frac{\delta\, n^{2}}{2}\, \Gamma^{\rm ex}_n(f;\nu^n_g) + {\mb 1}_{a>{\mf a_0}}\alpha_0 \, \nu^n_{\rm ref}(f_u(\cY^n)^4)
\nnb
&\quad 
+c(T,\delta)\, \big( 1+\sqrt{n}\, {\bf 1}_{[0,\tau\, n^{-1/2}]}(t)\big)
.
\label{eq_bound_ellm}
\end{align}
Fix $0\leq s\leq t\leq T$.  We now proceed to bound the expectation of
$|\cR^n_t-\cR^n_s|$.  Let $q>0$ to be chosen later.  As $R^m$ is
bounded by $4\, n^{3/4}$, the
expectation of $|\cR^n_t-\cR^n_s|^{1+q}$ satisfies:
\begin{align}
\E^n_{\mu_n}\big[\, |\cR^n_t-\cR^n_s|^{1+q}\, \big]
\leq 
(32\, \gamma^2\, a\,   n^{3/4}\, |t-s|)^{q} \E^n_{\mu_n}
\big[\, |\cR^n_t-\cR^n_s|\, \big]
.
\end{align}
Let us show the existence of $c(T)>0$ such that:
\begin{equation}
\E^n_{\mu_n}\big[\, |\cR^n_t-\cR^n_s|\, \big]
\leq 
\frac{c(T)\log n}{n^{1/4}}
,\qquad 
n\geq 1
.
\label{eq_expectation_cR_tightness}
\end{equation}
If this holds then item (ii) holds for any $p=1+q$, $\epsilon>0$ with
$q\in(0,1/3)$ and $\epsilon p = q$, that is, 
$p\in(1,4/3)$ and $\epsilon=(p-1)/p$ as
claimed.  

 To bound the expectation in~\eqref{eq_expectation_cR_tightness}, 
write:
\begin{align}
\E^n_{\mu_n}\big[\, |\cR^n_t-\cR^n_s|\, \big]
&=
8\, a\, \gamma^2\, 
\E^n_{\mu_n}\Big[\, \Big|\int_s^t R^m(\eta_u)\, du\, \Big|\, \Big]
\nnb
&=
8\, a\, \gamma^2\, \int_{0}^{4\, n^{3/4}(t-s)}
\, \P^n_{\mu_n}\Big(\, \Big|\int^{t}_{s} R^m(\eta_u)\, du\, \Big|> \lambda\, \Big)\, d\lambda
.
\label{eq_tightness_Aldous_interm0}
\end{align}
Note that~\eqref{eq_tightness_Aldous_interm0} cannot directly be bounded by
putting the absolute value inside the time integral as $R^m$ needs to be
renormalised in terms of the carr\'e du champ as in~\eqref{eq_def_Sm}.
Instead, we subtract the variable $S^m$ obtained from $R^m$ after
replacement and separately estimate the better behaved $S^m$ and the
replacement cost.  To estimate the latter it turns out to be
convenient to also subtract $Q_n^{4+,\delta}$, so we write 
($c_1$ is the
constant appearing in~\eqref{eq_bound_Sm}):
\begin{align}
\E^n_{\mu_n}&\Big[\, \Big|\int_s^t R^m(\eta_u)\, du\, \Big|\, \Big]
\nnb
&\quad \leq 
\int_0^{4n^{3/4}|t-s|}\, \P^n_{\mu_n}\Big(\int^{t}_{s}\big[R^m-S^m - n^{-3/4}(1+c_1)\, Q_n^{4+,\delta}\big](\eta_u)\, du\, > \frac{\lambda}{2} \Big)\, d\lambda
\nnb
&\qquad
+\int_0^{4n^{3/4}|t-s|}\, \P^n_{\mu_n}\Big(\int^{t}_{s}\big[-R^m-S^m - n^{-3/4}(1+c_1)\, Q_n^{4+,\delta}\big](\eta_u)\, du\, > \frac{\lambda}{2} \Big)\, d\lambda
\nnb
&\qquad +
4\, \E^n_{\mu_n}\bigg[\, \Big|\int^{t}_{s}\Big[S^m +n^{-3/4}(1+c_1)\, Q_n^{4+,\delta}\big](\eta_u)\, du\, \Big|\, \bigg]
.
\label{eq_tightness_Aldous_interm1}
\end{align}
Consider first the first two lines of~\eqref{eq_tightness_Aldous_interm1}. 
The integral of each probability will be estimated similarly, 
so we
only estimate the first line. 
Recall that the reference measure $\nu^n_{\rm ref}$ is defined as:
\begin{equation}
\nu^n_{\rm ref}
=
\begin{cases}
\nu^n_U\quad &\text{if }a\leq\mf a_0,\\
\nu^n_g\quad &\text{if }a> \mf a_0,
\end{cases}
\end{equation}
with $\mf a_0$ given by Theorem~\ref{n-s05}. 
Fix $\lambda>0$.  Since
$\sup_{u\leq T}H_n(f_u\, |\, \nu^n_{\rm ref})\leq c(T)\, \sqrt{n}$, the entropy
inequality and Markov property give:
\begin{align}
\P^n_{\mu_n}&\Big(\int_s^{t}\big[R^m-S^m - n^{-3/4}(1+c_1) \, Q_n^{4+,\delta}\big] (\eta_u)\, du\, > \lambda \Big)
\nnb
&\qquad 
\leq 
\frac{c(T)\, \sqrt{n} +\log 2}{-\log\P^n_{\nu^n_{\rm ref}}\Big(\int_0^{t-s} [R^m-S^m - n^{-3/4}(1+c_1)\, Q_n^{4+,\delta}] (\eta_u)\, du\, > \lambda\Big)}
.
\end{align}
The exponential Chebychev inequality and Feynman-Kac formula 
(\cite[Lemma B.1]{jl}) then give the following estimate of the
denominator:
\begin{align}
\log&\, \P^n_{\nu^n_{\rm ref}}\Big(\int_0^{t-s} \big[R^m-S^m - n^{-3/4}(1+c_1) \,Q_n^{4+,\delta}\big] (\eta_u)\, du\, > \lambda \Big)
\nnb
&\quad
\leq 
-\frac{\lambda\,  n^{3/4}}{2(1+c_1)}
+ \sup_{\substack{f\geq 0 \\ \nu^n_{\rm ref}(f)=1}}\Big\{ 
\nu^n_{\rm ref}\Big(f\Big[\frac{n^{3/4}}{2(1+c_1)}\, (R^m-S^m)
+ \frac{1}{2}\big(L^*_n{\mb 1} - Q_n^{4+,\delta}\big)\Big]\Big)
\nnb
&\hspace{7cm}
- \frac{n^{5/2}}{2}\, \Gamma^{\rm ex}_n(f;\nu^n_{\rm ref}) \Big\}
.
\label{eq_bound_FK_tightness}
\end{align}
Choose $\delta\leq 1$.  The quantities $S^m,Q_n^{4+,\delta}$
in~\eqref{eq_def_Sm}--\eqref{eq_def_ellm} were defined precisely so
that the above variational principle is at most $0$.  Thus, bounding
the next probability by $1$ if $\lambda<n^{-1/4}|t-s|$ and using the
bound~\eqref{eq_bound_FK_tightness} otherwise:
\begin{align}
&\int_{ 0}^{4\, n^{3/4}(t-s)}
\P_{\mu_n}\Big(\int_s^t\big[R^m-S^m - n^{-3/4}(1+c_1)\, Q_n^{4+,\delta}\big] (\eta_u)\, du\, > \lambda \Big)\, d\lambda
\nnb
&\qquad 
\leq 
\frac{|t-s|}{n^{1/4}} + 
2\, \int_{n^{-1/4}|t-s|}^{4\, n^{3/4}|t-s|}
\frac{c(T)\, \sqrt{n} +\log 2}{\lambda\, n^{3/4}}\, d\lambda
\leq 
\frac{c'(T)\log n}{n^{1/4}}
.
\end{align}
This bounds the contribution to~\eqref{eq_expectation_cR_tightness} of the first two lines of~\eqref{eq_tightness_Aldous_interm1}.

Consider next the last line of~\eqref{eq_tightness_Aldous_interm1}. 
Recall the entropy estimate of
Theorems~\ref{n-s05}--\ref{theo_free_energy}:
\begin{equation}
\sup_{s\leq T}H_n(f_s\, |\, \nu^n_{\rm ref}) 
+ \int_0^Tn^{5/2}\, \Gamma^{\rm ex}_n(f_s;\nu^n_{\rm ref})\, ds
\leq 
c(T)\, \sqrt{n} 
.
\end{equation}
This bound and the bounds~\eqref{eq_bound_Sm}--\eqref{eq_bound_ellm} on $S^m,Q_n^{4+,\delta}$ give, for a different $c(T)>0$:
\begin{align}
\E_{\mu_n} \Big[\, \int_{0}^{T}|S^m +n^{-3/4}(1+c_1)& \, Q_n^{4+,\delta}|(\eta_u)\, du\, \Big]
\nnb
&\qquad
\leq 
c(T)\, n^{-1/4} + \alpha_0\, {\bf 1}_{a>\mf a_0}\, n^{-1/4}\int_0^T\E_{\mu_n}\big[(\cY^n_t)^4\big]\, dt
.
\end{align}
The last term is bounded by $O(n^{-1/4})$ according to
Corollary~\ref{coro_4th_moment}. 
This concludes the proof of~\eqref{eq_expectation_cR_tightness}, 
thus of Lemma~\ref{lemm_replacement_tightness}. 
\end{proof}
\noindent{\bf Conclusion of the argument:} Up to this point, we proved
that the martingale part of the decomposition
\eqref{eq_semimart_decomp_Y} is tight in $\cD([0,T], \bb R)$, and the
drift part is tight in $\bb L^p([0,T], \bb R)$, $1<p<4/3$.  Observe
that the topology of $\cD([0,T], \bb R)$ is stronger than the one of
$\bb L^p([0,T], \bb R)$ for any $p\geq 1$: 
real-valued functions continuous on
$\bb L^p([0,T],\R)$ are also continuous on $\cD([0,T],\R)$.  A
converging subsequence of the laws of $(M^n_t)_{t\in[0,T]}$ in
$\cD([0,T],\R)$ therefore also converges in $\bb L^p([0,T],\R)$
($p\geq 1$).  We conclude that the process $(\mc Y^n_t)_{t\in[0,T]}$
is tight in $\bb L^p([0,T],\R)$ for any $1<p<4/3$.

\subsection{Convergence of the martingale term}\label{sec_martingale_pb}
Recall the definition of the martingale $M^n_\cdot$ and its quadratic variation:
\begin{align}
M^n_t 
\, &=\, 
\cY^n_t - \cY^n_0 - \int_0^t a\, \sqrt{n}\, L_n\cY^n_s\, ds
,\nnb
\<M^n\>_t \,&=\, \int_0^t  \frac{a}{n} \sum_{i\in\T_n} c(\tau_i
\eta^n(s)) \, ds 
,\qquad 
t\geq 0
.
\label{nn-01}
\end{align}
\begin{proposition}\label{prop_cv_martingale}
The martingale $(M^n_t)_{t\in[0,T]}$ converges weakly in $\cD([0,T],\R)$ to $(\sqrt{a}\, B_t)_{t\in[0,T]}$, 
where $(B_t)_{t\in[0,T]}$ is a standard Brownian motion. 
\end{proposition}
Recall from \eqref{n-04} the definition of $W_t(y_0,y)$  for $y\colon
[0,T]\to\R$ such that $\int_{[0,T]} y_s^4 \, ds <\infty$ and $t>0$. 
Lemma~\ref{lemm_replacement_tightness} gives:
\begin{equation}
\lim_{n\to\infty}\E_{\mu_n}\Big[ \, \|M^n_\cdot - W_\cdot(\cY^n_0,\cY^n)\|_{\bbL^{p}([0,T])}\, \Big]
=
0
,\qquad 
p\in(1,4/3)
.
\end{equation}
Proposition~\ref{prop_cv_martingale} thus has the following immediate corollary. 
\begin{corollary}\label{coro_cv_W_t(Y)}
For each $p\in(1,4/3)$, 
\begin{equation}
\lim_{n\to\infty} (W_t(\cY^n_0,\cY^n) )_{t\in[0,T]}
=
(\sqrt{a}\, B_t)_{t\in[0,T]}\quad
\text{in law in }\bb L^{p}([0,T],\R)
.
\end{equation}
\end{corollary}
\begin{proof}[Proof of Proposition~\ref{prop_cv_martingale}]
By
~\cite[Theorem VIII.3.11]{js}, since the jumps of $\mc Y^n_t$ (and
therefore those of $M^n_t$) are bounded by $n^{-3/4}$ and since $(\<M^n\>_t)_{t\in[0,T]}$ is uniformly bounded, 
it is enough to
show that the quadratic variation $\<M^n\>_t$ converges in probability
to $a\, t$ for each $t\geq 0$. 
We prove the following slightly stronger claim:
\begin{equation}
\label{nn-02}
\forall t\in[0,T],\qquad 
\lim_{n\to\infty} \bb E_{\mu^n} \big[ \, \big|\, \<M^n\>_t \,-\,  a
\,t\, \big|\,\Big ] \,=\, 0\;.
\end{equation}
Recall from \eqref{nn-01} the formula for the quadratic variation $\<M^n\>_t$.
Splitting $\bar\eta=\bar\eta^m+\ms M^n$ into fast and slow modes as usual (recall $\ms M^n=n^{-1/4}\cY^n$), 
the jump rates satisfy:
\begin{align}
\frac{1}{n}\sum_{i\in\T_n}\big[c((\tau_i\eta))
-1\big]
&=
-\frac{8\, \gamma}{n}\, \sum_{i\in\T_n}\bar\eta_i\bar\eta_{i+1}
+\frac{4\, \gamma^2}{n}\, \sum_{i\in\T_n}\bar\eta_{i-1}\bar\eta_{i+1}
\nnb
&= 
(4\gamma^2-8\gamma)\, (\ms M^n)^2 
-\frac{8\, \gamma}{n}\, \sum_{i\in\T_n}\bar\eta^m_i\bar\eta^m_{i+1}
+\frac{4\, \gamma^2}{n}\, \sum_{i\in\T_n}\bar\eta^m_{i-1}\bar\eta^m_{i+1}
.
\end{align}
To prove \eqref{nn-02}, it is enough to prove that the first moment of the time integral of each of the above terms vanishes. 
The energy estimates of 
Corollary~\ref{n-s09} (if $a\leq \mf a_0$) and
Theorem~\ref{theo_free_energy} (if $a>\mf a_0$) show that $(\ms
M^n_t)^2 = n^{-1/2}(\cY^n_t)^2$ has expectation bounded by
$O(n^{-1/2})$ at each time.

It remains to show that
\begin{equation}
\lim_{n\to\infty} \int_{0}^t
\E_{\mu_n} \big[\,|L^m_s |\, \big]\, ds
=
0
\label{02}
\end{equation}
for all $t\in[0,T]$, where
\begin{equation}
L^{m}_t
:=
-\frac{8\,a\, \gamma}{n}\, \sum_{i\in\T_n}\bar\eta^m_i(t)\bar\eta^m_{i+1}(t)
+\frac{4\,a\,  \gamma^2}{n}\, \sum_{i\in\T_n}\bar\eta^m_{i-1}(t)\bar\eta^m_{i+1}(t)
.
\end{equation}
As $L^m$ is bounded, to prove \eqref{02} it is enough to show that
\begin{equation}
\lim_{n\to\infty} \int_{\tau\, n^{-1/2}}^t
\E_{\mu_n} \big[\,|L^m_s |\, \big]\, ds
=
0
\end{equation}
for some $\tau>0$. 
More generally this is true for any observable bounded by $o(\sqrt{n}\, )$, 
an observation used in
\eqref{eq_proof_other_modes_negligible_n34} to estimate 
terms of the form $n^{-3/4} \sum_{i\in\T_n}H_i \, \bar\eta^m_i (s)$, which are only
$O(n^{1/4})$.

Let $A_n = \Omega_n$ if $a\le \mf a_0$,
$A_n = \{|\sum_i\bar\eta_i|\leq n/4\}$, otherwise.  By Proposition
\ref{prop_LD_bound}, there is $\tau=\tau_{1/4}$ and a finite constant
$c(T)$ such that $\P^n_{\mu_n}(\eta_t\in A_n)\le c(T)^{-1} \exp\{-nc(T)\}$
for all $t\in[\tau \,n^{-1/2},T]$. Therefore, to prove \eqref{02} it
is enough to show that, for all $t\in[0,T]$:
\begin{equation}
\lim_{n\to\infty} \int_{\tau\, n^{-1/2}}^t
\E_{\mu_n} \big[\, {\mb 1}_{A_n}(\eta_s)\, |L^m_s |\, \big]\, ds
=
0
.
\label{eq_reduction_Lm_to_good_event}
\end{equation}

Write as usual
$f_t\nu^n_{\rm ref}$ for the law of the dynamics at time $t\geq 0$
and $f_t^m\nu^{n,m}_{\rm ref}$ for the conditional law on the
subset of configurations with $\sum_i\eta_i=m$.  For each $\lambda>0$,
the entropy- and log-Sobolev inequalities give:
\begin{align}
\E_{\mu_n}\big[\,|L^m_t|{\mb 1}_{A_n}(\eta_t)\, \big]
&=
\sum_{m=0}^n \nu^n_{\rm ref}(f_t{\bf 1}_{\sum_i\eta_i=m}) \nu^{n,m}_{\rm ref}(f^m_t {\mb 1}_{A_n}|L^m|)
\nnb
&\leq 
\frac{1}{\lambda}\sum_{m=0}^n\nu^n_{\rm ref}(f_t{\bf 1}_{\sum_i\eta_i=m}) \Big[\, H_n(f^m_t|\nu^{n,m}_{\rm ref}) + \log \nu^{n,m}_{\rm ref}\big(e^{\lambda {\mb 1}_{A_n} |L^m|}\big)\, \Big]
\nnb
&\leq 
\frac{1}{\lambda}\sum_{m=0}^n \nu^n_{\rm ref}(f_t{\bf 1}_{\sum_i\eta_i=m}) \Big[{\mf c}_{\rm LS} \, n^2\, \Gamma_n^{\rm ex}(f^m_t;\nu^{n,m}_{\rm ref}) + \log \nu^{n,m}_{\rm ref}\big(e^{\lambda {\mb 1}_{A_n} |L^m|}\big)\, \Big]
\nnb
&\leq 
\frac{1}{\lambda}\, {\mf c}_{\rm LS}\,  n^2\, \Gamma_n^{\rm ex}(f_t;\nu^{n}_{\rm ref}) + \frac{1}{\lambda}\max_{0\leq m\leq n}\log \nu^{n,m}_{\rm ref}\big(e^{\lambda {\mb 1}_{A_n} |L^m|}\big)\, 
,
\end{align}
where ${\mf c}_{\rm LS}$ denotes the log-Sobolev constant for $\nu^{n,m}_{\rm ref}$. 
It is known to be independent of $n,m$ if $a\leq \mf a_0$ as in this case $\nu^{n,m}_{\rm ref}$ is the uniform measure~\cite[Theorem 1.4]{LeeYau_LSI_RW1998},  
and is bounded uniformly in $n,m$ for $a>\mf a_0$ under Assumption~\ref{ass_LSI}. 

Consider first the case $a\leq \mf a_0$. 
Proposition~\ref{n-l03} (see also the remark following the proposition)
shows that the last exponential moment is bounded by
$C\, \exp\{C\lambda^2/n\}$ for some $C>0$.  By Theorem~\ref{n-s05},
the time integral of the carr\'e du champ is bounded by a constant.
Therefore, taking $\lambda=\sqrt{n}$ gives the claim.

Consider next the case $a> \mf a_0$.  By
Theorem~\ref{theo_free_energy}, the time integral of the carr\'e du
champ is again bounded by a constant.  By
Proposition~\ref{prop_concentration}, the exponential moments are also
bounded by $C\, \exp\{C\lambda^2/n\}$, uniformly in $n,m$ with the
restriction $|m-n/2|\leq n/4$ on the magnetisation.  This completes
the proof.
\end{proof}
\begin{proof}[Proof of~\eqref{eq_other_modes_negligible_n34}]
Let $H\in C^\infty(\T)$ and note that $\frac{1}{n}\sum_i H_i=\<H\> +O(n^{-1})$. 
As a result, uniformly on the configuration:
\begin{equation}
{\tt Y}^n(H)-\cY^n\<H\>
=
\frac{1}{n^{3/4}} \sum_{i\in\T_n}\bar\eta^m_i H_i
+
o_n(1)
.
\label{eq_proof_other_modes_negligible_n34}
\end{equation}
The proof of Proposition~\ref{prop_cv_martingale} then also
yields~\eqref{eq_other_modes_negligible_n34}.  Indeed, since
$|{\tt Y}^n(H)|\leq n^{1/4}\|H\|_\infty$, it is enough to
prove~\eqref{eq_reduction_Lm_to_good_event} with $L^m_s$ replaced by
$n^{-3/4} \sum_{i\in\T_n}H_i \, \bar\eta^m_i (s)$.  The same concentration
estimates apply.
\end{proof}

\subsection{Conclusion of the proof}\label{sec_continuity}

Fix $p\in(1,4/3)$ and a limit point
$\bb Q\in\ms M_1(\R\times\bb L^{p}([0,T],\R))$ of the laws
$(\bb P^n_{\mu_n}(\cY^n_0\in\cdot,\cY^n\in \cdot))_n$  throughout this section.  We will
abuse notations and write
$\lim_n\bb P^n_{\mu_n}(\cY^n_0\in\cdot,\cY^n_\cdot\in \cdot)=\bb Q$ even though
convergence only holds up to a subsequence. 
We will shorten $\bb L^{p}([0,T],\R),C([0,T],\R)$ to $\bb L^{p}([0,T]),C([0,T])$ respectively.  \\

By Corollary~\ref{coro_martingale_problem_well_posed}, 
to conclude the proof of
Theorem~\ref{theo_convergence_magnetisation} we need to show that
$\bb Q$ is supported on couples $(\cY_0,\cY)\in\R\times \bb L^p([0,T],\R)$ such that $\cY$ has a continuous representative starting at $\cY_0$, 
and such that
\begin{equation}
W_t(\cY_0, \cY)
=
\cY_t - \cY_0 + 2\, a\, \int_0^t\big[\, \theta\, \cY_s+\cY_s^3\, \big]\, ds
,\qquad 
t\in[0,T]
,
\end{equation}
has the law of $\sqrt{a}$ times a standard Brownian motion.  

Corollary~\ref{coro_cv_W_t(Y)} already established the weak
convergence of $W_\cdot(\cY^n_0,\cY^n)$.  This however does not say anything
about $W_\cdot(\cY_0,\cY)$ under $\bb Q$ since the map
$y \mapsto W_\cdot (\mc Y_0,y)$ is not continuous on
$\bb L^{p}([0,T])$, nor have we proven that trajectories under $\bb Q$
have finite third moment or are continuous.  We obtain the claim as a
consequence of the following lemmas.

For $A>0$, 
let $\chi_A:\R\to[0,1]$ denote a Lipschitz approximation of ${\bf 1}_{[-A,A]}$ equal to $1$ on $[-A,A]$, supported on $[-2A,2A]$ and such that $\chi_A$ increases to $1$ as $A\to\infty$. 
Write also $W^A_\cdot$ for the functional acting on $y:[0,T]\to\R$ as:
\begin{align}
W^A_t(y_0,y)
&=
W_t(y_0,y)
-
2\, a\, \int_0^t y^3(s) \, [1- \chi_A(y(s))]\, ds
\nnb
&=
y_t - y_0 + 2\, a\, \int_0^t\big[\, \theta\, y_s+y_s^3\, \chi_A(y(s))\, \big]\, ds
,\qquad 
t\in[0,T]
.
\label{eq_def_WA}
\end{align}
\begin{lemma}\label{lemm_continuity_WA}
For each $A>0$, 
$(y_0,y)\mapsto (W^A_t(y_0,y))_{t\in[0,T]}$ is a continuous mapping from $\R\times\bb
L^{p}([0,T])$ to $\bbL^{p}([0,T])$. 
In addition, 
if $y\in\bb L^4([0,T])$, 
then:
\begin{equation}
\sup_{t\in[0,T]}\big|\, W^A_t(y_0,y)-
W_t(y_0,y)\, \big| 
\leq 
\frac{2\, a\, }{A}\, \|y\|^4_{\bb L^{4}([0,T])}
.
\label{eq_diff_WA_W}
\end{equation}
\end{lemma}
\begin{lemma}\label{lemm_higher_moments}
There is $\mf c_4>0$ such that:
\begin{equation}
\sup_{n\geq 1} \E_{\mu_n}\bigg[\, \int_0^T (\cY^n_t)^4\, dt\, \bigg]
\leq 
\mf c_4
,\qquad 
\E_{\bb Q}\bigg[\, \int_0^T \cY^4_t\, dt\, \bigg]
\leq 
\mf c_4
.
\end{equation}
\end{lemma}
Let us prove Theorem~\ref{theo_convergence_magnetisation} assuming Lemmas~\ref{lemm_continuity_WA}--\ref{lemm_higher_moments}. 
\begin{proof}[Proof of Theorem~\ref{theo_convergence_magnetisation}]
Let $A>0$ and $F:\bb L^{p}([0,T]) \to \bb R$ be
Lipschitz and bounded. 
The continuity of $W^A$ in Lemma~\ref{lemm_continuity_WA} and weak
convergence imply:
\begin{equation}
\lim_{n\to\infty}\E_{\mu_n}\Big[\, F\big((W^A_t(\cY^n_0,\cY^n))_{t\in[0,T]}\big)\, \Big]
=
\E_{\bb Q}\Big[\, F\big((W^A_t(\cY_0,\cY))_{t\in[0,T]}\big)\, \Big]
.
\end{equation}
Together with the moment bounds of Lemma~\ref{lemm_higher_moments} and~\eqref{eq_diff_WA_W}, 
we obtain:
\begin{align}
\E_{\bb Q}\Big[\, F\big((W_t(\cY_0,\cY))_{t\in[0,T]}\big)\, \Big]
&=
\E_{\bb Q}\Big[\, F\big((W^A_t(\cY_0,\cY))_{t\in[0,T]}\big)\, \Big]
+O(1/A)
\nnb
&=
\lim_{n\to\infty}\E_{\mu_n}\Big[\, F\big((W^A_t(\cY_0^n,\cY^n))_{t\in[0,T]}\big)\, \Big]
+O(1/A)
.
\end{align}
Another use of~\eqref{eq_diff_WA_W} gives:
\begin{align}
\sup_{n\geq 1}\E_{\mu_n}\Big[\, \Big| &F\big(W^A_\cdot(\cY^n_0,\cY^n)\big) -  F\big(W_\cdot(\cY^n_0,\cY^n)\big)\,\Big| \Big]
\nnb
&\quad\leq 
\|F'\|_\infty \, \sup_{n\geq 1}\E_{\mu_n}\Big[ \, \big\|W^A_\cdot(\cY^n_0,\cY^n)-W_\cdot(\cY^n_0,\cY^n)\big\|_{\bbL^{p}([0,T])}\, \Big]
\nnb
&\quad\leq 
\frac{T^{\frac{1}{p}}{\mf c}_4 \|F'\|_\infty}{A}
,
\end{align}
whence:
\begin{equation}
\E_{\bb Q}\Big[\, F\big((W_t(\cY_0,\cY))_{t\in[0,T]}\big)\, \Big]
=
\lim_{n\to\infty}\E_{\mu_n}\Big[\, F\big((W_t(\cY_0^n,\cY^n))_{t\in[0,T]}\big)\, \Big]
.
\end{equation}
 Since 
$W_\cdot(\cY^n_0,\cY^n))$ converges weakly to $\sqrt{a}\, B_\cdot$ in $\bb L^{p}([0,T]
,\bb R)$ by Corollary~\ref{coro_cv_W_t(Y)}, 
with $B_\cdot$ a standard Brownian motion, 
we have shown:
\begin{equation}
{\rm Law}_{\bb Q}\big(\big(W_t(\cY_0,\cY)\big)_{t\in[0,T]}\big)
=
{\rm Law}((\sqrt{a}\, B_t)_{t\in[0,T]})
\quad 
\text{on }\bb L^p([0,T])
.
\label{eq_WT_equal_BM_in_law}
\end{equation}
To conclude the proof of Theorem~\ref{theo_convergence_magnetisation}
using the well-posedness result of
Corollary~\ref{coro_martingale_problem_well_posed} (proven at the end of this section), 
it remains to prove
that $\bb Q$ is supported on elements $(y_0,y)\in \R\times \bb L^p([0,T])$ with $y$ having a continuous representative starting at $y_0$. 
Equation~\eqref{eq_WT_equal_BM_in_law} implies that there is a set 
$\Omega_W\subset \bb L^p([0,T])$ such that $\bb Q(W(\cY^0,\cY)\in \Omega_W)=1$ and elements of $\Omega_W$ have a continuous representative. 
Equivalently $\bb Q(\Omega)=1$ if $\Omega= W(\cdot,\cdot)^{-1}(\Omega_W)$. 
For $y\in\Omega$, let $(\tilde W^y_t)_{t\in[0,T]}$ denote the continuous representative of $(W_t(y_0,y))_{t\in[0,T]}$ and define $\tilde y:[0,T]\to\R$ through:
\begin{equation}
\tilde y_t 
= 
y_0 
+ 
2\, a\, \int_0^t\big[\, \theta\, y_s+y_s^3\, \big]\, ds
+
\tilde W^y_t
\qquad 
t\in[0,T]
.
\end{equation}
Then $\tilde y=y$ almost everywhere, 
$\tilde y_0 = y_0$ and $\tilde y$ is a continuous function as, by H\"older inequality, there is $c>0$ such that, for each $\delta\in(0,1)$:
\begin{equation}
\sup_{s,t:|t-s|\leq \delta}|\tilde y_t-\tilde y_s|
\leq 
c \delta^{1/4}\, \Big(1+\int_0^T y_u^4\, du\Big)^{3/4}
+
\sup_{s,t:|t-s|\leq \delta}|W_t(y_0,y)-W_s(y_0,y)|
=
o_\delta(1)
.
\end{equation}
This concludes the proof.
\end{proof}
\begin{proof}[Proof of Lemma~\ref{lemm_continuity_WA}]
Let $A>0$. 
Recall the definition of $\chi_A$, $W^A_\cdot$ from~\eqref{eq_def_WA}. 
Let us first show that $y\mapsto (W^A_t(y))_{t\in[0,T]}$ is continuous on $\bb L^{p}([0,T])$.  
Define, for $p\in \N$:
\begin{equation}
T^A_p: y\in \bb L^{p}([0,T],\R) 
\mapsto 
\int_0^T y^p(t)\chi_A(y(t))\, dt 
.
\label{eq_def_TA}
\end{equation}
It is enough to prove that $T^A_3$ is continuous on $\bb L^{p}([0,T])$. 
In fact, if $f_q(x) = x^q\chi_A(x)$ for $q\geq 0$, 
then $f_q$ is a Lipschitz function due to $\chi_A$ being compactly supported. 
As a result, if $y^n\in \bb L^{p}([0,T])$ ($n\geq 1$) converges to $y$, 
then:
\begin{equation}
|T^A_p(y^n)-T^A_p(y)|
\leq 
\|f'_p\|_{\infty}\|y^n-y\|_{\bb L^1([0,T])}
.
\label{eq_TA_continuous}
\end{equation}
This proves the continuity of $y\in \bb L^{p}([0,T])\mapsto W^A_\cdot(y)$. 
The moment bound~\eqref{eq_diff_WA_W} on the other hand is a direct consequence of H\"older inequality: 
if $(y_0,y)\in \R\times\bb L^4([0,T])$ and $t\in[0,T]$,
\begin{align}
\big|\, W^A_t(y_0,y)-
W_t(y_0,y)\, \big| 
&\leq 
2\, a\, \int_0^t |y^3(s)| [1-\chi_A(y(s))]\, ds
\leq 
2\, a\, \int_0^t |y^3(s)|{\bf 1}_{|y(s)|\geq A}\, ds
\nnb
&\leq 
\frac{2\, a}{A}\, \|y\|_{\bb L^4([0,T])}^{4}
.
\end{align}
\end{proof}
\begin{proof}[Proof of Lemma~\ref{lemm_higher_moments}]
Corollaries~\ref{n-s09}--\ref{coro_4th_moment} respectively show for $a\leq \mf a_0$ and $a>\mf a_0$ 
that:
\begin{equation}
\sup_{n\geq 1}\E_{\mu_n}\Big[ \int_0^T(\cY^n_t)^4\, dt\Big]
<
\infty
.
\end{equation}
Let $A>0$. 
The continuity of the bounded function $T^A_4$~\eqref{eq_def_TA}, 
proven in~\eqref{eq_TA_continuous}, gives:
\begin{align}
\E_{\bb Q}\Big[\, \int_0^T \cY^4_t\, \chi_A(\cY_t)\, dt \, \Big]
&=
\lim_{n\to\infty}\E_{\mu_n}\Big[\, \int_0^T (\cY^n_t)^4 \, \chi_A(\cY^n_t)\, dt\, \Big]
\nnb
&\leq 
\limsup_{n\to\infty}\E_{\mu_n}\Big[\, \int_0^T (\cY^n_t)^4 \, dt\, \Big]
<
\infty
,
\end{align}
where we used $\chi_A\leq 1$ in the last line. 
As $\chi_A$ increases with $A$, 
the monotone convergence theorem gives the claim:
\begin{equation}
\E_{\bb Q}\Big[\,\int_0^T \cY^4_t \, dt \, \Big]
<
\infty
.
\end{equation}
\end{proof}
\begin{proof}[Proof of Corollary~\ref{coro_martingale_problem_well_posed}]
Let $\bb Q$ be a probability measure on $\R\times\bb L^p([0,T])$ satisfying items $(i)$ and $(ii)$ of Corollary~\ref{coro_martingale_problem_well_posed}.  
Let us prove that there is a one-to-one mapping between $\bb Q$ and a probability measure $\tilde{\bb Q}^{c}$ on $C([0,T],\R)$, 
equipped with its usual $\sigma$-algebra, 
which satisfies the assumptions of Lemma~\ref{lemma_martingale_problem_well_posed}. 
Since Lemma~\ref{lemma_martingale_problem_well_posed} guarantees uniqueness of such a measure, 
this will imply the claim of Corollary~\ref{coro_martingale_problem_well_posed}. \\

Denote by $j:C([0,T])\to \bb L^p([0,T])$ the canonical inclusion and 
let $\cB(\bb L^p)$ denote the Borel $\sigma$-algebra in $\bb L^p([0,T])$. 
Then there is a one-to-one mapping between $\bb Q$ and a probability measure $\tilde {\bb Q}$ 
on the measurable space $(C([0,T]),j^{-1}(\cB(\bb L^p)))$ given by:
\begin{equation}
\tilde{\bb Q}(A) 
=
\bb Q( j(A)),\qquad
A\in j^{-1}(\cB(\bb L^p))
.
\label{eq_link_bbQ_tildebbQ}
\end{equation}
Since $W_\cdot(\cY_0,\cY)= W_\cdot(\cY_0,\tilde\cY)$ in $\bb L^p([0,T])$ for any $\cY\in\bb L^p([0,T])$ having a continuous modification $\tilde \cY$ starting at $\cY_0$,
\begin{align}
\tilde{\bb Q}\Big( \tilde \cY_t 
&=
 \tilde \cY_0 - 2\, a\, \int_0^t [\theta\, \tilde\cY_s + \tilde\cY_s^3]\, ds + W_t(\tilde \cY_0,\tilde\cY_t)\text{ for all } t\in[0,T]\Big)
=
1
\nnb
&\quad
{\rm Law}_{\tilde {\bb Q}}(W_\cdot(\cY_0,\tilde\cY))
=
{\rm Law}(\sqrt{a}\, B_\cdot)\quad 
\text{on }j^{-1}(\cB(\bb L^p))
.
\label{eq_carac_law_tildeQ}
\end{align}
We next prove that $\tilde{\bb Q}$ can be uniquely extended to a probability measure $\tilde {\bb Q}^c$ on $C([0,T],\R)$ equipped with its usual $\sigma$-algebra $\cB(C)$. 
We first define an extension on cylinder functions.

For $\epsilon\in(0,T)$ and $t\in[0,T]$, 
write $I_\epsilon(t):= [(t-\epsilon)\vee 0, (t+\epsilon)\wedge T]$ and $\ell_\epsilon(t)$ for its length, and define:
\begin{equation}
K_{\epsilon,t} : y\in C([0,T])\mapsto
\frac{1}{\ell_\epsilon(t)}\int_{I_\epsilon(t)} y_s\, ds
.
\end{equation}
Let $\cC\cF_T$ denote the set of all cylindrical functions, 
i.e. functions of the form $G(K_{\epsilon,t_1},...,K_{\epsilon,t_N})$ for $\epsilon\in(0,T)$, $N\geq 1$, a bounded continuous $G:\R^N\to\R$ and $t_1,...,t_N\in[0,T]$. 
Define then:
\begin{equation}
\E_{\tilde{\bb Q}^c}[G(\tilde\cY_{t_1},...,\tilde\cY_{t_N})]
:=
\lim_{\epsilon\to 0}\E_{\tilde {\bb Q}^c}[G(K_{\epsilon,t_1}\tilde \cY,...,K_{\epsilon,t_N}\tilde \cY)]
,
\label{eq_def_tildebbQc}
\end{equation}
noting that the limit exists by dominated convergence since $\lim_{\epsilon\to 0}K_{\epsilon,t_1}y = y(t_1)$ for $y\in C([0,T])$. 
Set also $\tilde {\bb Q}^c(A) = \tilde{\bb Q}(A)$ for $A\in j^{-1}(\cB(\bb L^p))$.  

As $[0,T]$ is separable, the $\sigma$-algebra generated by cylinder functions is equal to $\cB(C)$ (a fortiori this is the case for the $\sigma$-algebra generated by cylinder functions and $j^{-1}(\cB(\bb L^p))$). 
The Carath\'eodory extension theorem thus ensures the uniqueness of an extension on $\cB(C)$ that we still denote by $\tilde {\bb Q}^c$. 
Under $\tilde {\bb Q}^c$, the fact that $K_{t_i,\epsilon}$ is bounded and continuous on $\bb L^p([0,T])$ for each $t_i\in(0,T]$, 
the identities~\eqref{eq_carac_law_tildeQ} relating $\cY_t$ and $W_\cdot(cY_0,\cY)$ and characterising the law of $W_\cdot(cY_0,\cY)$, and \eqref{eq_def_tildebbQc} imply: 
\begin{equation}
\E_{\tilde {\bb Q}^c}\big[G\big(W_{t_1}(\cY_0,\cY\big),...,G\big(W_{t_N}(\cY_0,\cY)\big)\big]
=
\E\big[ G\big(\sqrt{a} \, B_{t_1}),...,G(\sqrt{a} \, B_{t_N})\big]
.
\end{equation}
It follows that ${\rm Law}_{\tilde {\bb Q}^c}(W_\cdot(\cY_0,\cY))={\rm Law}(\sqrt{a}\, B_\cdot)$ on $\cB(C)$.  
By Lemma~\ref{lemma_martingale_problem_well_posed} there is a unique such probability measure. 
As there is a one-to-one mapping between $\bb Q$ and $\tilde {\bb Q}^c$, 
Corollary~\ref{coro_martingale_problem_well_posed} is proven. 
\end{proof}

\section{The fast modes}
\label{sec8}
In this section we prove Theorem~\ref{theo_fastmodes},
i.e. convergence in finite-dimensional distributions of
$(y^n_t)_{t\in(0,T]}$ to a Gaussian process that is white in time and
with covariance $\frac{\id}{4} + \frac{(-\Delta^{-1})}{2}$ in space.
We only treat the case of general $a$ under Assumption~\ref{ass_LSI}. 
The same computations apply also to the case $a<\mf a_0$, 
where they are in fact simpler.  
The proof below however makes use of the full force of Lemma~\ref{lemm_estimate_WphiJp}, 
which was only written for general $a$ as it was not needed in the $a<\mf a_0$ case, so we focus on the general $a$ case.

Let $T>0$ be fixed throughout the section. 
Notice that, by density, it is enough to prove
Theorem~\ref{theo_fastmodes} for $y^n$ acting on
${\color{blue} \cC_0} := C^\infty(\T)\cap\{H: \<H\>=\int_{\T}H(x)\, dx
= 0\}$, 
where we recall that $y^n$ acts on $H\in C^\infty(\T)$ according to:
\begin{equation}
y^n(H)
=
\frac{1}{\sqrt{n}\, }\sum_{i\in\T_n}\bar\eta_i H_i
,\qquad 
\eta\in\Omega_n
.
\label{eq_def_yn_sec8}
\end{equation}
In addition, linearity of $y^n$ and the Markov property imply
that it is enough to characterise the law of $y^n_t(H)$ and prove that
it is asymptotically independent from the $\sigma$-algebra $\cF^n_0$
generated by $y^n_0$, for each $H\in \cC_0$ and $t\in(0,T]$.  This is
the content of the next proposition.
\begin{proposition}\label{prop_fastmode_sec_fastmode}
Let $H\in\cC_0$ and $t\in(0,T]$.  
Let $(\mu_n)$ be a sequence of initial conditions as in Theorem~\ref{theo_fastmodes} and let $\cF^n_t$ denote the $\sigma$-algebra $\sigma(\eta_s:s\leq t)$.   
Then:
\begin{itemize}
	\item[(i)] (Independence).
	\begin{equation}
	\lim_{n\to\infty}\E_{\mu_n}\big[ e^{i y^n_t(H)} | \cF^n_0\big]
	=
	\lim_{n\to\infty}\E_{\mu_n}\big[ e^{i y^n_t(H)} \big]
	\end{equation}
	\item[(ii)] (Convergence to a Gaussian field).
	\begin{equation}
	\lim_{n\to\infty}\E_{\mu_n}\big[ e^{i y^n_t(H)}  \big]
	=
	\exp\Big[\, -\frac{1}{8}\|H\|^2_{\bb L^2(\T)}+\frac{a}{4}(H,\Delta^{-1}H)\, \Big]
	.
	\end{equation}
\end{itemize}
\end{proposition}
The proof of Proposition~\ref{prop_fastmode_sec_fastmode} spans the
rest of the section.  Throughout, a sequence $(\mu_n)$ of initial
conditions as in Theorem~\ref{theo_fastmodes}, a test function
$H\in\cC_0$ and $t\in(0,T]$ are fixed.

The starting point is the following semi-martingale decomposition,
using Duhamel formula. For $u\geq 0$, write $H_u:= e^{u\Delta}H$. For
each $0\le s\leq t$,
\begin{align}
y^n_s(H_{\sqrt{n}(t-s)})
&=
y^n_0(H_{\sqrt{n} \, t}) 
-\int_0^s 2\, a\, \theta \, y^n_u(H_{\sqrt{n} \, (t-u)})\, du
+ M^{n,t}_{s}(H)
+O(1/n)
\nonumber\\
&\quad
- 8\, a\, \gamma^2\, \int_0^s \sum_{i\in\T_n}(H_{\sqrt{n}\, (t-u)})_i\Big[(\ms M^n)^3 + \bar\eta^m_{i-1}\bar\eta^m_{i}\bar\eta^m_{i+1} 
+ \bar\eta^m_{i-1}\bar\eta^m_{i+1}\ms M^n 
\label{eq_Ito_fast_modes}\\
&\hspace{3cm}
+ \bar\eta^m_{i}(\bar\eta^m_{i+1}+\bar\eta^m_{i-1}) \, \ms M^n + (\bar\eta^m_{i-1}+\bar\eta^m_{i}+\bar\eta^m_{i+1})(\ms M^n)^2 \Big](u)\, du
,
\nonumber
\end{align}
with $(M^{n,t}_{s}(H))_{s\in[0,t]}$ a martingale independent from
$\cF^n_0$ and $O(1/n)$ an error term uniform in the configuration
that comes from replacing discrete derivatives of $H$ by continuous
ones. Note that the time derivative of $H_{\sqrt{n}\, (t-u)}$ is used
to cancel the discrete Laplacian coming from the exclusion part of the
dynamics.

An important point used several times throughout the proof is that the heat kernel
$e^{\sqrt{n}\, s\Delta}$ reduces the test function $H$ to nearly $0$
at times $s>0$.  This for instance means the time integrals above may
only contribute to order $1$ in $n$ in the range $[t-O(n^{-1/2}),t]$, 
but also that the initial condition will not play a role. 
The former claim is hard to make use of when renormalisation is needed since we only control time integrals of the carré du champ on the full time interval $[0,t]$, see Proposition~\ref{prop_nonlinear_drift_fastmodes}. 
This observation can however be used for the
initial condition: 
since $-\Delta\geq 4\pi^2 \id$ on
$\{h\in \bbL^2(\T):\<h\>=0\}$, 
there is a constant $C>0$ such that:
\begin{equation}
|y^n_0(H_{\sqrt{n}\, t})| 
\leq 
C\sqrt{n}\, e^{-4 \pi^2 \sqrt{n}\,  t}\|H\|_{\bbL^2(\T)}
\underset{n\to\infty}{\longrightarrow}
0
.
\label{eq_y0fast_vanishes}
\end{equation}

We show in Section~\ref{sec_bound_drift_martingale_fastmodes} that the
drift term in~\eqref{eq_Ito_fast_modes} also vanishes in probability
and deduce item $(i)$ of Proposition~\ref{prop_fastmode_sec_fastmode}.
Section~\ref{sec_convergence_fastmodes} studies the martingale term,
proving the convergence of $(y^n_t(H))$ stated in item $(ii)$.
\subsection{Control of the drift: proof of $(i)$}\label{sec_bound_drift_martingale_fastmodes}
The main result of this section is the following claim.
\begin{proposition}\label{prop_yfast_closeto_mart}
Let $\delta>0$.  Then:
\begin{equation}
\lim_{n\to\infty}\bb P_{\mu_n}\Big( \big|y^n_t(H)- M^{n,t}_{t}(H)\big|>\delta\Big)
=
0
.
\label{eq_yfast_closeto_mart}
\end{equation}
for all $H\in\cC_0$ and $t\in(0,T]$. 
In particular item $(i)$ of Proposition~\ref{prop_fastmode_sec_fastmode} holds. 
\end{proposition}
In view of the estimate~\eqref{eq_y0fast_vanishes} of the initial condition, 
only the drift in~\eqref{eq_Ito_fast_modes} needs to be  bounded.  
This is done in the next subsections. 
Assuming~\eqref{eq_yfast_closeto_mart}, let us already show that it implies item $(i)$ of Proposition~\ref{prop_fastmode_sec_fastmode}. 
Let $G\in\cC_0$. 
Since $e^{i \cdot}$ is bounded,
\begin{align}
\E_{\mu_n}\big[\, e^{iy^n_0(G)}\, e^{i y^n_t(H)}\, \big]
&=
\E_{\mu_n}\big[\, e^{iy^n_0(G)}\, e^{i M^{n,t}_{t}(H)}\, \big] +o_n(1)
\nnb
&=
\E_{\mu_n}\big[\, e^{iy^n_0(G)}\,\big] \, \E_{\mu_n}
\big[\, e^{i M^{n,t}_{t}(H)}\, \big]  +o_n(1)
\nnb
&=
\E_{\mu_n}\big[\, e^{iy^n_0(G)}\,\big] \, \E_{\mu_n}
\big[\, e^{i y^{n}_{t}(H)}\, \big]  +o_n(1)
,
\end{align}
where the middle line uses that $M^{n,t}_{t}(H)$ is asymptotically independent from $\cF^n_0$. 
This is the claim $(i)$. 
\subsubsection{Non-linear drift term}
Here we prove that the second line of~\eqref{eq_Ito_fast_modes} converges to $0$ in probability for any $s\leq t$. 
This estimate only makes use of the special form of the test function $H_{\sqrt{n}\, (t-s)}$ to avoid having to deal with an initial time interval of length $O(n^{-1/2})$ at the end of which large deviation bounds become available. 
In this sense the bound below is cruder than what one could hope to get, 
although sufficient for our needs.
\begin{proposition}\label{prop_nonlinear_drift_fastmodes}
Let $J\subset\Z$ be finite and $q\geq 0$.  
Then, if either $q\geq 1$ or $|J|\geq 3$:
\begin{equation}
\forall r>0,\qquad 
\lim_{n\to\infty}\P_{\mu_n}\bigg( \, \Big|\,\int_0^t (\ms M^n_s)^q\, \sum_{i\in\T_n}\bar\eta^m_{i+J} (H_{\sqrt{n}\, (t-s)})_i\, ds \,\Big|\, > r\bigg)
=
0
.
\end{equation}
\end{proposition}
Let $s\in(0,t)$. Since $\|H_{\sqrt{n}\, (t-u)}\|_\infty\leq C(H)e^{-c\sqrt{n}(t-s)}$ for some $C(H),c>0$ and $u\in[0,s]$, 
it is enough to prove Proposition~\ref{prop_nonlinear_drift_fastmodes} on the time interval $[s,t]$. 
This is done in the following lemma, a
generalisation of Lemma~\ref{lemm_estimate_WphiJp}. Recall
from~\eqref{eq_def_WJphi_p_sec_relent} the notation $W^{J,\phi}_1$.
\begin{lemma}\label{lemm_FK_nonlinear_drift}
Let $s\in(0,t)$. 
Let $J\subset\Z$ be finite, $q\geq 0$, $\phi:[s,t]\times\T\to\R$ be a
bounded function, and $\delta,\delta'\in(0,1/2)$. Then,
\begin{itemize}
	\item[(i)] (Renormalisation step).   
	There is a non-negative function $F^{\delta,\delta'}_n=F_n^{J,\phi,q,\delta,\delta'}:\Omega_n\to\R_+$ such that, for any density $f$ for $\nu^n_g$:
\begin{equation}
\big|\, \nu^{n}_g\big(f\, (\ms M^n)^{q}\, W^{J,\phi}_1\big) \big|
\leq 
\frac{\delta n^2}{2} \, 
\Gamma^{\rm ex}_n(f;\nu^n_g)
+
\nu^n_g\big( f \, F^{\delta,\delta'}_n\big)
.
\label{eq_renorm_step_lemm_nonlin_drift}
\end{equation}
	\item[(ii)] (Expectation bound). 
	There is $C_\delta>0$ such that, for each 
$u\in[s,t]$:
\begin{align}
\nu^{n}_g\big(f_u\, F^{\delta,\delta'}_n\big)
&\leq 
\frac{\delta n^{2}}{2} \,
\Gamma^{\rm ex}_n(f_u;\nu^n_g)
+C_\delta (\delta'+o_n(1))
,
\label{eq_bound_drift_fastmodes}
\end{align}
where the $o_n(1)$ is uniform in time. 
\end{itemize}
\end{lemma}
\begin{proof}[Proof of Proposition~\ref{prop_nonlinear_drift_fastmodes} assuming Lemma~\ref{lemm_FK_nonlinear_drift}] 
We proceed as in the proof of Lemma~\ref{lemm_replacement_tightness}, using Feynman-Kac formula to estimate the probability. 
Recall the definition of the quantity $Q_n^{4+,\delta}$ renormalising the adjoint in~\eqref{eq_def_ellm}--\eqref{eq_bound_ellm}. 
Write for short $\phi_u:= H_{\sqrt{n}\, (t-u)}$ so that  the quantity to estimate is simply the time integral of $(\ms M^n)^q W^{J,\phi_u}_1$. 
Then, for each $\epsilon,\delta,\delta'>0$:
\begin{align}
\P_{\mu_n}\bigg( \, &\Big|\,\int_s^t \, (\ms M^n_u)^q\, W^{J,\phi_u}_1 (\eta_u) \, du \,\Big|\, > r\bigg)
\leq 
2\max_{\mf s\in\{+,-\}}\P_{\mu_n}\bigg( \, \int_s^t \mf s\, (\ms M^n_u)^q\, W^{J,\phi_u}_1 (\eta_u) \, du  > r\bigg)
\nnb
&\quad \leq 
2\max_{\mf s\in\{+,-\}}\P_{\mu_n}\bigg( \,\int_s^t \big[\, \mf s\, (\ms M^n_u)^q\, W^{J,\phi_u}_1 (\eta_u)  - F_n^{\delta,\delta'}(\eta_u)- \frac{\epsilon}{\sqrt{n}}\, Q^{4+,\delta}_n(\eta_u)\,\big]\, du  >r/2\bigg) 
\nnb
&\qquad +\frac{4\, }{r }\, \E_{\mu_n}\Big[\, \int_s^t \big[\, F_n^{\delta,\delta'}(\eta_u) + \frac{\epsilon}{\sqrt{n}\,}\,\big|Q^{4+,\delta}_n(\eta_u)\big|\,\big]\, du  \, \Big]
.
\label{eq_splitting_prop83}
\end{align}
The last term is bounded by $c(\epsilon+\delta+C_\delta\delta')/r$ for some $c>0$ by property~\eqref{eq_bound_ellm} of $Q^{4+,\delta}_n$,
~\eqref{eq_bound_drift_fastmodes} of $F_n^{\delta,\delta'}$ and the fact that $\int_0^t n^2\,\Gamma^{\rm ex}_n(f_s;\nu^n_g)\, ds$ is bounded. 
Choosing $\delta'$ as a function of $\delta$ so that $\lim_{\delta\to 0}C_{\delta}\delta'=0$, 
this term therefore vanishes in the limit $n\to\infty$, 
then $\delta\to0$, then $\epsilon\to 0$. 
Consider now the first term in the right-hand side of~\eqref{eq_splitting_prop83}. 
Let $\mf s\in \{+,-\}$. 
The entropy inequality gives:
\begin{align}
&\P^n_{\mu_n}\bigg( \,\int_s^t \big[\, \mf s\, (\ms M^n_u)^q\, W^{J,\phi_u}_1 (\eta_u)  - F_n^{\delta,\delta'}(\eta_u)- \frac{\epsilon}{\sqrt{n}}\, Q^{4+,\delta}_n(\eta_u)\,\big]\, du  >r/2\bigg) 
\nnb
&\qquad
\leq 
\frac{H_n(\mu_n\, |\, \nu^n_g)+\log 2}{-\log \P^n_{\nu^n_g}\big( \,\int_s^t\,  [\, \mf s\, (\ms M^n_u)^q\, W^{J,\phi_u}_1 (\eta_u)  - F_n^{\delta,\delta'}(\eta_u)- \frac{\epsilon}{\sqrt{n}}\, Q^{4+,\delta}_n(\eta_u)\,] \, du  >r/2\big) }
.
\end{align}
Using Feynman-Kac formula to estimate the denominator as in Lemma~\ref{lemm_replacement_tightness} (see~\eqref{eq_bound_FK_tightness}), 
as long as $\epsilon^{-1}\delta<1/4$, 
the carré du champ coming out of the renormalisation step in~\eqref{eq_renorm_step_lemm_nonlin_drift} can be absorbed and we obtain that the probability in the denominator is bounded by $e^{-\sqrt{n}\, r/(2\epsilon)}$. 
Taking $n$ to infinity, then $\delta$ to $0$, then $\epsilon$ to $0$ therefore concludes the proof. 
\end{proof}
\begin{proof}[Proof of Lemma~\ref{lemm_FK_nonlinear_drift}]
The case $q\geq 1$ was already proven in Lemma~\ref{lemm_estimate_WphiJp}.   
We refer to~\eqref{eq_magnetisation_splitting_withepsilon}--\eqref{eq_use_entropy_LSI_WJphip}--\eqref{eq_renormalisation_adjoint_middle} and give no further detail as also similar computations will be performed below.   \\

Consider now the case where $q=0$. 
In this case $|J|\geq 3$. 
For definiteness we only work with
$J=\{-1,0,1\}$, other sets being similar. 
The proof is similar to that of Lemma~\ref{lemm_estimate_WphiJp} but in order to get the small term in the right-hand side of~\eqref{eq_bound_drift_fastmodes} we will need not just one, 
but two renormalisation steps.

Let $u\in[s,t]$. Fix a density $f$ for
$\nu^n_g$. 
Let $\zeta\in(0,1/2)$ that will later be taken to be small and $\ell_n = \lfloor\zeta n\rfloor$. 
Replace $\bar\eta^m_{i+1}$ by $\frac{1}{\ell_n}\sum_{j=i+1}^{i+\ell_n}$ to find, 
using integration by parts (Lemma~\ref{lemm_IBP}) as in e.g.~\eqref{eq_renormalisation_adjoint_middle}, 
that there is $C>0$ such that:
\begin{align}
\nu^{n}_g\big(f\, W^{J,\phi_u}_1\Big)
&\leq 
\nu^n_g\Big( f\, \frac{1}{\ell_n}\sum_{i,j}\bar\eta^m_{i-1}\bar\eta^m_i\bar\eta^m_j {\bf 1}_{[i+1,i+\ell_n]}(j)(\phi_u)_i\Big)
+
\frac{\delta\, n^{2}}{4}\, \Gamma^{\rm ex}_n(f;\nu^{n}_g)
\label{eq_term_to_renorm_again_fastmodes}\\
&\quad 
+\frac{C }{\delta}\frac{1}{n}\sum_j \nu^{n}_g\Big[f\Big(\frac{1}{\sqrt{n}}\sum_{i}\bar\eta^{m}_{i-1}\bar\eta^m_{i}(\phi_u)_i I_{\ell_n}(j-i-1)\Big)^2\, \Big]
\label{eq_squareterm_1st_renorm_fastmodes}
\\
&\quad
+\frac{C }{n}\sum_j \nu^{n}_g\Big[f \frac{1}{n}\sum_i\sum_{\ell\neq {j,j+1}}\partial^n_1 g_{j,\ell}\bar\eta^m_\ell\bar\eta^m_{i-1}\bar\eta^m_{i}(\phi_u)_i I_{\ell_n}(j-i-1) \Big]
\label{eq_gterm_1st_renorm_fastmodes}
\\
&\quad 
+ C\, \nu^n_g\big(\, f\, |\ms M^n|\, \big) \|\phi_u\|_{\infty} 
.
\label{eq_Mn_term_1st_renorm_fastmodes}
\end{align}
The above decomposition does not yet allow to define $F_n^{\delta,\delta'}$ in terms of the integrands in the right-hand side above as the first term in the right-hand side of line~\eqref{eq_term_to_renorm_again_fastmodes} will require another renormalisation step. 
Let us however show that other terms already satisfy the bound~\eqref{eq_bound_drift_fastmodes}, 
so that the absolute value of their sum will constitute part of $F^{\delta,\delta'}_n$. 
The moment bound of Theorem~\ref{theo_free_energy}  already gives for~\eqref{eq_Mn_term_1st_renorm_fastmodes} (recall $\ms M^n = n^{-1/4}\cY^n$):
\begin{equation}
C\, \nu^n_g\big(\, f\, |\ms M^n|\, \big) \|\phi_u\|_{\infty} 
=
O(n^{-1/4})
.
\end{equation}
For~\eqref{eq_squareterm_1st_renorm_fastmodes}--\eqref{eq_gterm_1st_renorm_fastmodes}, 
we start with the large deviation bound on the magnetisation of Proposition~\ref{prop_LD_bound} to write, for any $X:\Omega_n\to\R$ with $n^{-1}\, \|X\|_\infty = O_n(1)$:
\begin{equation}
\nu^n_g(f_u\, X) 
=
\sum_{m: |m-n/2|\leq n/4} \nu^n_g\big(f_u {\bf 1}_{\sum_i \eta_i=m}\big) 
\nu^{n,m}_g\big( f_u^m X\, \big)
+o_n(1)
,
\label{eq_splitting_along_m_sec8}
\end{equation}
with the $o_n(1)$ uniform in $u\in[s,t]$. 
It is therefore enough to focus on bounding the averages in~\eqref{eq_squareterm_1st_renorm_fastmodes}--\eqref{eq_gterm_1st_renorm_fastmodes} uniformly on $m$ with $|m-n/2|\leq n/4$. 
Fix such an $m$. 

For the square term~\eqref{eq_squareterm_1st_renorm_fastmodes}, 
we find for each $j$ and each $\lambda>0$, 
noting that the sum on $i$ only contains $\ell_n = \lfloor \zeta n\rfloor$ terms:
\begin{align}
&\nu^{n,m}_g\Big[f^m_u\Big(\frac{1}{\sqrt{n}}\sum_{i}\bar\eta^{m}_{i-1}\bar\eta^m_{i}(\phi_u)_i I_{\ell_n}(j-i-1)\Big)^2\Big]
\nnb
&\quad 
\leq 
\lambda{\mf c}_{\rm LS} \, n^2\, \Gamma^{\rm ex}_{n}(f^m_u;\nu^{n,m}_g) 
+\lambda \log \nu^{n,m}_g\bigg[ \exp\Big[\lambda^{-1} \Big(\frac{1}{\sqrt{n}}\sum_{i}\bar\eta^{m}_{i-1}\bar\eta^m_{i}(\phi_u)_i I_{\ell_n}(j-i-1)\Big)^2\, \Big]\, \bigg]
\nnb
&\quad\leq 
\lambda{\mf c}_{\rm LS} \, n^2\, \Gamma^{\rm ex}_{n}(f^m_u;\nu^{n}_g) + 
\lambda^{-1} \zeta \|\phi_u\|_\infty
,
\label{eq_lsi_plus_entropy_sec8}
\end{align}
where in the last line we used Remark~\ref{rmk_better_concentration_1pt} to pick up suitable factors of $\zeta$. 
Taking $\lambda\leq \delta/(2 \mf c_{\rm LS})$, 
$\delta':= \zeta^{1/2}=\lambda$, 
and recalling~\eqref{eq_splitting_along_m_sec8}, we obtain~\eqref{eq_bound_drift_fastmodes} for the square term.

Consider next the expectation containing $g$~\eqref{eq_gterm_1st_renorm_fastmodes} at fixed $m$, 
and rewrite it as:
\begin{align}
&\frac{C}{n}\sum_j \nu^{n,m}_g\Big[f_u^m \frac{1}{n}\sum_i\sum_{\ell\neq {j,j+1}}\partial^n_1 g_{j,\ell}\bar\eta^m_\ell\bar\eta^m_{i-1}\bar\eta^m_{i}(\phi_u)_i I_{\ell_n}(j-i-1) \Big]
\nnb
&\quad =
\zeta\, \nu^{n,m}_g\Big[f_u^m \frac{1}{n}\sum_{i\in\T_n}\sum_{\ell\in\T_n}\bar\eta^m_\ell\, \bar\eta^m_{i-1}\, \bar\eta^m_{i}\, \Big(\frac{C}{\zeta n}\, \sum_{j\notin\{\ell-1,\ell\}}\partial^n_1 g_{j,\ell}(\phi_u)_i I_{\ell_n}(j-i-1) \Big)\, \Big]
.
\end{align}
The parenthesis, call it $\psi_{i,j}$, 
defines $\psi:\T^2_n\to\R$ bounded uniformly in $n,\zeta$. 
The last expression is therefore of the form $\zeta \nu^{n,m}_{g}[f^m_u \, W^{\{-1,0\},\psi}_{2}]$.  
The concentration results of Proposition~\ref{prop_concentration} and the entropy- and log-Sobolev inequalities therefore imply as usual, 
for some $C'>0$:
\begin{align}
&\frac{C}{n}\sum_j \nu^{n,m}_g\Big[f_u^m \frac{1}{n}\sum_i\sum_{\ell\neq {j,j+1}}\partial^n_1 g_{j,\ell}\bar\eta^m_\ell\bar\eta^m_{i-1}\bar\eta^m_{i}(\phi_u)_i I_{\ell_n}(j-i-1) \Big]
\nnb
&\quad 
\leq 
\zeta\, {\mf c}_{\rm LS} \, n^2\, \Gamma^{\rm ex}_{n}(f^m_u;\nu^{n,m}_g) + 
C'\zeta \|\phi_u\|_\infty
.
\end{align}
Thus the term containing $g$~\eqref{eq_gterm_1st_renorm_fastmodes} also satisfies~\eqref{eq_bound_drift_fastmodes} upon taking $\delta'=\zeta\leq \delta/(2 {\mf c}_{\rm LS})$. 

We conclude the proof by performing a second renormalisation step on:
\begin{equation}
\nu^n_g\Big( f_u\, \frac{1}{\ell_n}\sum_{i,j}\bar\eta^m_{i-1}\bar\eta^m_i\bar\eta^m_j {\bf 1}_{[i+1,i+\ell_n]}(j)(\phi_u)_i\Big)
,
\end{equation}
turning it into:
\begin{equation}
\nu^{n}_g\Big(f_u\, \frac{1}{\ell_n^2 }\sum_{i,j,k}\bar\eta^m_{k}\bar\eta^m_{i}\bar\eta^m_{j}(\phi_u)_i{\bf 1}_{[i+1,i+\ell_n]}(j){\bf 1}_{[i-\ell_n,i-1]}(k)\Big)
.
\label{eq_twice_renormalised_term}
\end{equation}
One can similarly check that all terms coming out of the integration by parts satisfy the bound~\eqref{eq_bound_drift_fastmodes}. 
We claim that this is also the case for the integrand in the last display. 
Indeed, 
in the notations of~\eqref{eq_def_WJphi_p},  
this term is of the form $W^{\{0\},\psi}_3$ for $\psi:\T^3_n\to\R$ bounded uniformly in $n,\zeta$. 
As explained in Section~\ref{sec_sketch_free_energy}, 
large deviation bounds can be used to bound terms of the form $W^{\{0\},\phi}_p$ for any $p\geq 3$ and bounded tensor $\phi:\T^p_n\to\R$. 
Introduce the event $A:= \{\max_i| \frac{1}{\ell_n}\sum_k \bar\eta^m_k{\bf 1}_{[i-\ell_n,i-1]}(k)|\leq \zeta'\}$, 
for $\zeta'$ that we are free to take arbitrarily small and note that large deviation bounds of Proposition~\ref{prop_LD_bound} imply:
\begin{equation}
\nu^{n}_g\big(f_u\,W^{\{0\},\psi}_3\, \big)
=
\nu^{n}_g\big({\bf 1}_{A}\, f_u\, W^{\{0\},\psi}_3\, \big)
+
o_n(1)
.
\end{equation}
Since $W^{\{0\},\psi}_3$ is smaller than $\zeta'\|\psi\|_\infty n$ on this event, 
using log-Sobolev and entropy inequalities as in~\eqref{eq_lsi_plus_entropy_sec8}, 
then the concentration estimates of Corollary~\ref{coro_better_concentration} yields~\eqref{eq_bound_drift_fastmodes}, 
taking $\zeta'$ suitably small as a function of $\zeta$ and defining $\delta'$ depending on $\zeta$. 
This concludes the proof.
\end{proof}
\subsubsection{Linear drift term}
In this section we conclude the estimate of the drift in~\eqref{eq_Ito_fast_modes}. 
We prove that, for each $r>0$:
\begin{equation}
\lim_{n\to\infty}\P_{\mu_n}\Big(\, \Big|\, \int_0^t y^n_s(H_{\sqrt{n}\, (t-s)})\, ds\, \Big| >r\Big)
=
0
.
\end{equation}
At first glance the time integral seems to be of order $1$ in $n$. 
Due to the time averaging it is in fact very close to its mean $0$. 
To see it, 
introduce a family of functions $(G^s)_{s\in[0,t]}$ in $\cC_0$ solving the inhomogeneous heat equation:
\begin{equation}
\partial_s G^{s} = \Delta G^s 
-
H_{s}
,\qquad 
G^0 \in\cC_0
.
\end{equation}
Note that we put the $s$ as an exponent to differentiate from the previous convention $H_s:= e^{s\Delta}H$. 
The semimartingale decomposition~\eqref{eq_Ito_fast_modes} then gives, 
for each $s\leq t$:
\begin{align}
y^n_t(G^{0})
& =
y^n_0(G^{\sqrt{n}\, t}) 
-\int_0^t 2\, a\, \theta \, y^n_s(G^{\sqrt{n}\, (t-s)})\, ds
 +O(n^{-1}) + M^{n,t}_{t}(G)
\nonumber\\
& \quad
- 8\, a\, \gamma^2\, \int_0^t \sum_{i\in\T_n}(G^{\sqrt{n}\, (t-s)})_i
\Big\{ \bar\eta^m_{i-1}\bar\eta^m_{i}\bar\eta^m_{i+1}
+ \ms M^n \, \mf h_i(\eta) \,
+ (\ms M^n)^2 \mf h'_i(\eta) + (\ms M^n)^3\Big\} (s)\, ds
\nnb
&\quad
+\sqrt{n}\int_0^t y^n_s\big(-\partial_s G^{\sqrt{n}\, (t-s)}+\Delta G^{\sqrt{n}\, (t-s)}\big)\, ds
,
\end{align}
where $\mf h_i (\eta)  = \bar\eta^m_{i-1}\bar\eta^m_{i+1}
+ \bar\eta^m_{i}\bar\eta^m_{i+1}  + \bar\eta^m_{i-1}\bar\eta^m_{i}$ and $\mf h'_i(\eta)=\bar\eta^m_{i-1}+\bar\eta^m_{i}+\bar\eta^m_{i+1}$. 
The integral in the last line is precisely
$\sqrt{n}$ times $\int_0^t y^n_s(H_{\sqrt{n}\, (t-s)})\, ds$. 
This integral can
therefore be estimated by bounding $n^{-1/2}$ times the other terms.  We
already know (Proposition~\ref{prop_nonlinear_drift_fastmodes}) that
the non-linear drift, 
i.e. the middle line, 
vanishes even without multiplying by $n^{-1/2}$. 
The
explicit expression for the quadratic variation of the martingale
$(M^{n,t}_{s}(G))_{s\in[0,t]}$ (see
Equation~\eqref{eq_rmk_quadvar_bounded_fastmode} below) gives a
trivial bound $O(\sqrt{n}\, )$.  In particular,
$n^{-1/2}\, M^{n,t}_{s}(G)$ also converges to $0$ in probability.
It remains to prove the convergence of:
\begin{equation}
\frac{1}{\sqrt{n}}|y^n_t(G^{0})|
,\quad
\frac{1}{\sqrt{n}}|y^n_0(G^{\sqrt{n}\, t})|,
\quad 
\frac{1}{\sqrt{n}}\,\int_0^t |y^n_s(G^{\sqrt{n}\, (t-s)})|\, ds
.
\end{equation}
To do so, note that the Duhamel formula gives for $G$:
\begin{equation}
G^{\sqrt{n}\, s}
=
e^{\sqrt{n}\, s\Delta}G^0 + \int_0^{\sqrt{n}\, s} e^{\sqrt{n}\, (s-u)\Delta}H_u\, du
=
e^{\sqrt{n}\, s\Delta}G^0 + \sqrt{n}\, s\, e^{\sqrt{n}\, s \Delta}H
,\qquad 
s\leq t
.
\end{equation}
Thus $\|G^{\sqrt{n}\, s}\|_{\infty}\leq c(G^0,H)\, e^{-c\sqrt{n}\, s}$ for some $c>0$ and each $s\leq t$ which immediately gives $\sup_{\eta\in\Omega_n}n^{-1/2}\, |y^n(G^{\sqrt{n}\, t})|
=
o_n(1)$
, 
and: 
\begin{align}
\frac{1}{\sqrt{n}}\,\int_0^t |y^n_s(G^{\sqrt{n}\, (t-s)})|\, ds
& \,\leq \,
\int_{0}^t \|G^{\sqrt{n}\, s}\|_{\infty} \, ds
\,\leq \,
\frac{c(G^0,H)\,}{\sqrt{n}}
.
\end{align}
For the remaining term $n^{-1/2}\, y^n_t(G^0)$, 
observe for instance that its being larger than $r>0$ is a large deviation event due to the extra $n^{-1/2}$ factor. 
Proposition~\ref{prop_LD_bound} thus ensures that this term also converges to $0$ in probability. 
\subsection{Convergence of $(y^n_t)_{t\in[0,T]}$: proof of item $(ii)$}\label{sec_convergence_fastmodes}
In view of Proposition~\ref{prop_yfast_closeto_mart}, 
we just need to prove:
\begin{equation}
\lim_{n\to\infty}\E_{\mu_n}\big[\, e^{i M^{n,t}_{t}(H)}\, \big]
=
\exp\Big[\, -\frac{1}{8}\|H\|^2_{\bb
L^2(\T)}-\frac{a}{4}(H,(-\Delta)^{-1} H) \, \Big]
.
\label{eq_cv_characfunct_fastmode}
\end{equation}
We start by computing the quadratic variation of $M^{n,t}_{\cdot}(H)$ in Section~\ref{sec_quadvar_Mfast}, 
then conclude in Section~\ref{sec_ccl_fastmode}.
\subsubsection{Convergence of quadratic variation}\label{sec_quadvar_Mfast}
Here we prove:
\begin{equation}
\lim_{n\to\infty} \E_{\mu_n}\bigg[\, \Big|\big<M^{n,t}_{\cdot}(H)\big>_t
-\frac{1}{4}\| H\|^2_{\bbL^2(\T)}
-
\frac{a}{2}(H, (-\Delta)^{-1}H)\, \Big|\, \bigg]
=
0
.
\label{eq_limit_quadvar_fastmode}
\end{equation}
Writing $\partial^n f(i):= n[f_{i+1}-f_i]$ for $f:\T\to\R$, 
the quadratic variation reads:
\begin{align}
\big<M^{n,t}_{\cdot}(H)\big>_s
&=
\int_0^s \frac{a\, \sqrt{n}\,}{n} \sum_{i\in\T_n} c(\tau_i\eta(u))\big(H_{\sqrt{n}\, (t-u)}\big)_i^2 \,du
\nnb
&\quad
+\int_0^s \frac{\sqrt{n}\,}{n} \sum_{i\in\T_n} |\eta_{i+1}(u)-\eta_i(u)|^2\, \big(\partial^n  H_{\sqrt{n}\, (t-u)}\big)_i^2 \,du 
.
\end{align}
Note that the exponential decay in $H_{\sqrt{n}\, (t-u)}$ for $u<t$ implies:
\begin{equation}
\forall s<t,\qquad 
\lim_{n\to\infty}\big<M^{n,t}_{\cdot}(H)\big>_s
=
0
.
\end{equation}
Note in addition that, as the jump rates are bounded above by $(1+\gamma^2)\leq 4$ (recall their definition~\eqref{01}),
one has:
\begin{equation}
\big<M^{n,t}_{\cdot}(H)\big>_t
\leq 
4\int_0^t \frac{\sqrt{n}\, }{n}\sum_{i\in\T_n}\Big[ a\, \big(H_{\sqrt{n}\, (t-s)}\big)_i^2 +\frac{1}{4}\big(H'_{\sqrt{n}\, (t-s)}\big)_i^2\, \Big]\,ds
.
\label{eq_rmk_quadvar_bounded_fastmode}
\end{equation}
This expression is clearly bounded by $O(\sqrt{n})$, 
but in fact is also bounded uniformly in $n$ as follows from:
\begin{equation}
\|(\partial^n H_{\sqrt{n}\, (t-s)})^2\|_\infty, \|H_{\sqrt{n}\, (t-s)}\|\leq c(H)e^{-c\sqrt{n}(t-s)}
.
\label{eq_bound_Hu_derH_u}
\end{equation} 
Let us now focus on proving~\eqref{eq_limit_quadvar_fastmode}.

Recall the expression~\eqref{01} of the jump rate $c$ and the elementary
formula
$|\eta_{i+1}-\eta_i|^2 = (1/2) -
2\bar\eta_i\bar\eta_{i+1}$.  
Isolating the constant part in the jump
rates and replacing discrete derivatives by continuous ones, one
finds, with a $o_n(1)$ error term uniform in the configuration:
\begin{align}
&\big<M^{n,t}_{\cdot}(H)\big>_t
=
\int_0^t \frac{\sqrt{n}\, }{n}\sum_{i\in\T_n}\Big[ a\, \big(H_{\sqrt{n}\, (t-s)}\big)_i^2 +\frac{1}{2}\big(\nabla H_{\sqrt{n}\, (t-s)}\big)_i^2\, \Big]\,ds
+ o_n(1)+\epsilon_n(t)
,
\label{eq_expression_quadvar_sec8}
\end{align}
where:
\begin{align}
&\epsilon_n(t)
:= 
\int_0^t \frac{\sqrt{n}\, }{n} \sum_{i\in\T_n}\Big[
-2\bar\eta_i\bar\eta_{i+1} \big(\partial^n H_{\sqrt{n}\,
(t-s)}\big)_i^2  
\nnb
&\hspace{5cm}+ 4\, \gamma\, a\, \big[ \, - 2 \bar\eta_{i}\bar\eta_{i+1}
+\gamma \bar\eta_{i-1}\bar\eta_{i+1}\,\big]
(H_{\sqrt{n}\, (t-s)})_i^2\Big]\, ds
.
\end{align}
Writing $\bar\eta=\bar\eta^m+\ms M^n$, 
and recalling the bound~\eqref{eq_bound_Hu_derH_u}, 
there is $C>0$ such that $\epsilon_n(t)$ can be bounded by:
\begin{align}
|\epsilon_n(t)|
&\leq
C\int_0^t  \sqrt{n}\, \big[(\ms M^n_s)^2 + |\ms M^n_s|\, \big] e^{-c\sqrt{n}(t-s)}\,  ds
+|\epsilon^m_n(t)|
,
\label{eq_bound_epsilon_n_sec8}
\end{align}
with $\epsilon^m_n(t)$ defined as $\epsilon_n(t)$ but with $\bar\eta^m$'s instead of $\bar\eta$'s. 
The fact that $\sup_{t\in[0,T]}\E_{\mu_n}[(\cY^n_t)^2]=O_n(1)$ by Theorem~\ref{theo_free_energy} ensures that the expectation of the first term in the right-hand side of~\eqref{eq_bound_epsilon_n_sec8} is bounded by $O(n^{-1/4})$.  
On the other hand, 
computations identical to those of Lemma~\ref{lemm_replacement_tightness} show that
$|\epsilon^m_n(t)|$ converges to $0$ in expectation, 
since $\epsilon^m_n(t)$ is a factor
$n^{-1/4}$ smaller than the quantity estimated in
Lemma~\ref{lemm_replacement_tightness}, see equation
\eqref{eq_def_Rm}).  \\

To conclude the proof of~\eqref{eq_limit_quadvar_fastmode}, 
it only remains to prove convergence of the constant term in~\eqref{eq_expression_quadvar_sec8} to the right-hand side of~\eqref{eq_limit_quadvar_fastmode}. 
A change of time variable gives:
\begin{align}
&\int_0^t \frac{\sqrt{n}\, }{n}\sum_{i\in\T_n}\Big[ a\, \big(H_{\sqrt{n}\, (t-s)}\big)_i^2 +\frac{1}{2}\big(\partial^n H_{\sqrt{n}\, (t-s)}\big)_i^2\, \Big]\,ds
\nnb
&\qquad
=
\frac{1}{n}\sum_{i\in\T_n} \int_0^{\sqrt{n}\, t} \Big[a\, (H_u)^2_i + \frac{1}{2}\big(\partial^n H_{u}\big)_i^2\Big]\, du
.
\end{align}
Recalling the bound~\eqref{eq_bound_Hu_derH_u}, 
the first term is equal to
\begin{align}
a\, \int_0^{\sqrt{n}\, t} \Vert H_u\Vert_2^2\, du \,+\, o_n(1)
\,=\, a\, \int_0^{\sqrt{n}\, t} (e^{2u\Delta} H , H) \, du \,+\, o_n(1)\;.
\end{align}
As $n\to\infty$, this expression converges to
\begin{align}
a\,  \int_0^{\infty} (e^{2u\Delta} H , H) \, du \,=\,
\frac{a}{2}\, (H, (-\Delta)^{-1}H)\;.
\end{align}
Similarly,
\begin{align}
\frac{1}{2n}\sum_{i\in\T_n} \int_0^{\sqrt{n}\, t} 
\big(\partial^n H_{u}\big)_i^2\, du \,=\,
\frac{1}{2} \int_0^{\sqrt{n}\, t} 
\big\Vert \nabla  H_{u}\big \Vert_2^2\, du \,+\, o_n(1)\;, 
\end{align}
and this expression converges to
\begin{align}
\frac{1}{4}(H,(-\Delta)^{-1} (-\Delta) H) \;=\;
\frac{1}{4}(H, H)\;.
\end{align}
\subsubsection{Conclusion of the proof}\label{sec_ccl_fastmode}
We now prove convergence of the characteristic function~\eqref{eq_cv_characfunct_fastmode}. 
Recall first that the following is a mean-1 martingale (see \cite[p336 Appendix A.7]{kl}):
for $s\in[0,T]$,
\begin{equation}
\cE^n_s(H) 
:=
\exp\Big[ iy^n_t(H) - iy^n_0(H_{\sqrt{n}\, t}) - \int_0^t e^{-iy^n_s(H_{\sqrt{n}\, (t-s)})}(\partial_s + \sqrt{n}\, L_n)\, e^{iy^n_s(H_{\sqrt{n}\, (t-s)})}\, ds
\Big]
.
\end{equation}
A direct computation of the action of the generator gives, 
recalling the semi-martingale decomposition~\eqref{eq_Ito_fast_modes} of $y^n_s(H_{\sqrt{n}\, (t-s)})$:
\begin{align}
\cE^n_s(H)
=
\exp\Big[i M^{n,t}_t(H)+\frac{1}{2}\<M^{n,t}_\cdot(H) \>_t + \int_0^t \omega^n_s\, ds\Big], 
\end{align}
where the remainder $\omega^n$ in the Taylor expansion of the exponentials is bounded for some $C>0$ by 
(the last line uses the bound~\eqref{eq_bound_Hu_derH_u}):
\begin{align}
\int_0^t |\omega^n_s|\, ds
&\leq 
\frac{1}{6}\int_0^t \sum_{i\in\T_n} \Big[ n^{5/2}\, \big|\nabla_{i,i+1}y^n_{s}(H_{\sqrt{n}\, (t-s)})\big|^3 
+ a\sqrt{n}\, c(\tau_i\eta_s)\, \big|\nabla_{i}y^n_{s}(H_{\sqrt{n}\, (t-s)})\big|^3\Big] \, ds
\nnb
&\leq 
C\int_0^t \max\big\{\|H_{\sqrt{n}\, (t-s)}\|_\infty, n^{-1}\|\partial^n H_{\sqrt{n}\, (t-s)})\|_\infty\big\}\, ds
=
O(n^{-1/2})
.
\end{align}
Since the quadratic variation is bounded uniformly in $n$ and the configuration (recall~\eqref{eq_rmk_quadvar_bounded_fastmode}),  
we have shown:
\begin{equation}
1 
=
\E_{\mu_n}\big[\cE^n_t(H)\big]
=
\E_{\mu_n}\bigg[ \exp\Big[i M^{n,t}_t(H)+\frac{1}{2}\<M^{n,t}_\cdot(H) \>_t \Big]\bigg]
+o_n(1)
.
\end{equation}
The convergence~\eqref{eq_limit_quadvar_fastmode} of the quadratic variation concludes the proof: 
\begin{align}
1
&= \E_{\mu_n}\bigg[ \exp\Big[i M^{n,t}_t(H)+\frac{1}{2}\<M^{n,t}_\cdot(H) \>_t \Big]\bigg] +o_n(1)
\nnb
&=
\E_{\mu_n}\big[ e^{i M^{n,t}_t(H)}\big] \exp\Big[\frac{1}{8}\|H\|_{\bb L^2(\T)}^2+\frac{a}{4}(H,(-\Delta)^{-1}H)\Big]
+o_n(1)
.
\end{align}

\newpage 

\appendix 

\section{Large deviation bounds}\label{app_LD}
In this section, 
we adapt to our present needs dynamical large deviation bounds from~\cite{LandimTsunoda}. 

Recall that $\ms M(\T)$ denotes the set of positive measures on the torus with mass bounded by $1$, 
equipped with the following distance $d$ which is a metric for weak convergence:
\begin{align}
{\rm d}(\mu_1,\mu_2)
&=
\sum_{k\in\N} \frac{1}{2^{k+1}}\Big[\big|\mu_1(\cos(2\pi k\cdot))-\mu_2(\cos(2\pi k\cdot))\big|
\nnb
&\hspace{4cm}+\big|\mu_1(\sin(2\pi k\cdot))-\mu_2(\sin(2\pi k\cdot))\big|\Big]
.
\label{eq_def_dweak}
\end{align}
For future reference, note that if $\mu_i=\rho_i\, dx$ for $i\in\{1,2\}$, 
then:
\begin{align}
{\rm d}(\mu_1,\mu_2)^2
&\leq 
\sum_{k\in\N} \frac{1}{2^k}\bigg[\Big|\int_{\T}(\rho_1-\rho_2)(x)\cos(2\pi k x))\, dx\Big|^2
\nnb
&\hspace{3cm}
+\Big|\int_{\T}(\rho_1-\rho_2)(x)\sin(2\pi k x))\, dx\Big|^2\bigg]
\nnb
&\leq 
\|\rho_1-\rho_2\|_2^2
.
\label{eq_bound_dw_by_L2}
\end{align}
Write $\pi^n=\frac{1}{n}\sum_i\eta_i \delta_{i/n}$ for the empirical measure associated with a configuration $\eta\in\Omega_n$ and $\pi^n_t=\pi^n(\eta^n_t)$ for the random measure induced by the dynamics with generator $\sqrt{n}\, L_n$ at time $t\geq 0$ (recall~\eqref{eq_def_generator}).   

To recover the Glauber + Kawasaki dynamics in~\cite{LandimTsunoda} in which the Kawasaki part is
diffusively rescaled (i.e. with generator $L_n$ rather than $\sqrt{n}\, L_n$), we need to consider the process $\eta^n_\cdot$
on the time-scale $tn^{-1/2}$. Recall that $V'(\rho) = 2(\rho-1/2)^3$
for $\gamma=1/2+o_n(1)$.

\begin{proposition}[Theorem 2.5 in~\cite{LandimTsunoda}]\label{prop_large_devs}
Let $\tau>0$, let $\rho_0:\T\to[0,1]$ be measurable and let $\eta^n\in\Omega_n$ ($n\in\N$) be associated with $\rho_0$ in the sense:
\begin{equation}
\lim_{n\to\infty} {\rm d}(\pi^n,\rho_0)
=
0
.
\end{equation}
The process $(\pi^n_{tn^{-1/2}})_{t\in[0,\tau]}$ with initial condition $\pi^n_0 = \pi^n(\eta^n)$ satisfies a large deviation principle on $\cD([0,\tau],\ms M(\T))$ with good rate function $I_\tau(\cdot|\rho_0)$. 

The quantity $I(\rho|\rho_0)$ may be finite only if $\rho_0$ is absolutely continuous with respect to Lebesgue measure and if $\rho$ is continuous in time. 
Moreover, 
the rate function $I(\cdot|\rho_0)$ vanishes at a trajectory $(\rho_t)_{t\in[0,\tau]}$ if and only if $\rho_t(dx)$ has a density with respect to Lebesgue measure for each $t\in[0,\tau]$, 
still denoted $\rho_t$, 
that is a weak solution of the hydrodynamic equation:
\begin{equation}
\partial_t\rho_t 
=
\Delta \rho_t - V'(\rho_t)
=
\Delta\rho_t - 2\Big(\rho_t-\frac{1}{2}\Big)^3,\quad
t\in(0,\tau),\qquad 
\rho_{t=0}=\rho_0
.
\label{eq_LLN_appendix}
\end{equation}
\end{proposition}
We make use of a more general statement with unconstrained initial condition, 
which directly follows from the proof of Proposition~\ref{prop_large_devs} in~\cite{LandimTsunoda} and properties of the rate function. 
\begin{corollary}\label{coro_LD}
Let $\tau>0$. 
For any closed sets $C\subset \ms M(\T)$ and $\ms C\subset \cD([0,\tau],\ms M(\T))$, 
\begin{equation}
\limsup_{n\to\infty}\frac{1}{n}\log 
\sup_{\substack{\eta\in \Omega_n \\ \pi^n \in C}}
\P^n_\eta\big(\, (\pi^n_{tn^{-1/2}})_{t\in[0,\tau]}\in \ms C\big)
\leq 
-\inf_{\rho_0\in C}\inf_{(\rho_s)_{s\leq \tau}\in \ms C} I_\tau(\rho|\rho_0)
.
\end{equation}
\end{corollary}
In the next proposition, we use the large deviation principle in diffusive time (i.e. times of order $n^{-1/2}$ for the dynamics with generator $\sqrt{n}\, L_n$) 
to deduce information on time-scales of order $1$. 
To do so, we first describe properties of the hydrodynamic equation~\eqref{eq_LLN_appendix}. 
\begin{lemma}\label{lemm_properties_LLN}
Let $\epsilon>0$ and $\tau_\epsilon:= 2\epsilon^{-2}$. 
Then, for any initial condition $\rho_0(x)\, dx \in \ms M(\T)$:
\begin{equation}
\sup_{t\geq \tau_\epsilon} \|\rho_t-1/2\|_2
\leq
\frac{\epsilon}{2}
.
\end{equation}
\end{lemma}
\begin{proof}
By general arguments (see e.g.~\cite[Proposition A.3]{LandimTsunoda}
and~\cite[Proposition 2.1]{Mottoni}), there is a unique weak solution
$\rho_t$ of the hydrodynamic equation~\eqref{eq_LLN_appendix}, and
this solution is infinitely differentiable on $(0,\infty)\times\T$.
In particular the hydrodynamic equation is satisfied pointwise.
Writing $\bar\rho_t:=\rho_t-1/2$ we get for each $t>0$:
\begin{align}
\partial_t \|\bar\rho_t\|_2^2
&=
-\int_{\T} |\nabla\bar\rho_t(x)|^2\, dx - 4 \int_{\T} \bar\rho_t(x)^4\, dx 
\leq
- 4\|\bar\rho_t\|_2^4
.
\end{align}
This concludes the proof: 
for each $t\ge 2 s$,
\begin{equation}
\|\bar\rho_t\|_2^2 
\leq 
\frac{1}{4(t-s)+\|\bar\rho_s\|_2^{-2}}
\leq 
\frac{1}{2t}
\,\cdot
\end{equation}
\end{proof}
\begin{proposition}\label{prop_LD_bound}
Let $T>0$. 
For each $\epsilon>0$, there is $c(\epsilon,T)>0$ such that:
\begin{equation}
\forall n\geq 1,\forall t\in[\tau_\epsilon \, n^{-1/2},T],
\qquad
\sup_{\eta\in\Omega_n}\P_{\eta}^n\bigg({\rm d}\Big(\pi^n_{t},\frac{dx}{2}\Big)\geq \epsilon\bigg)
\leq 
c(\epsilon,T)^{-1}\, e^{-n c(\epsilon,T)}
,
\end{equation}
with $\tau_\epsilon$ the time of Lemma~\ref{lemm_properties_LLN}. 
\end{proposition}
\begin{proof}
Let $\epsilon>0$ and $t\geq \tau_\epsilon n^{-1/2}$. 
Write $\nu^\eta_t$ for the law of the dynamics at time $t$ starting from $\eta\in\Omega_n$. 
Split the interval $[0,t]$ in $p_{n,\epsilon}:=\lfloor t\sqrt{n}/\tau_\epsilon\rfloor$ intervals of length $\tau_\epsilon n^{-1/2}$ of the form $[k\tau_\epsilon n^{-1/2},(k+1)\tau_{\epsilon}n^{-1/2}]$, 
except possibly the first one that may have length up to $2\tau_\epsilon$. 
The Markov property gives: 
\begin{align}
&\P_{\eta}^n\bigg({\rm d}\Big(\pi^n_{t},\frac{dx}{2}\Big)\geq \epsilon\bigg)
\nnb
&\quad \leq 
\P_{\eta}^n\bigg({\rm d}\Big(\pi^n_{t-\tau_\epsilon n^{-1/2}},\frac{dx}{2}\Big)\geq \epsilon\bigg)
+\P_{\nu^\eta_{t-\tau_\epsilon n^{-1/2}}}^n\bigg({\rm d}\Big(\pi^n_{\tau_\epsilon n^{-1/2}},\frac{dx}{2}\Big)\geq \epsilon,{\rm d}\Big(\pi^n_{0},\frac{dx}{2}\Big)\leq \epsilon\bigg) 
\nnb
&\quad\leq 
\P_{\eta}^n\bigg({\rm d}\Big(\pi^n_{t-(p_{n,\epsilon}-1)\tau_\epsilon\, n^{-1/2}},\frac{dx}{2}\Big)\geq \epsilon\bigg)
\nnb
&\qquad 
+\Big(\,\frac{ t\sqrt{n}\, }{\tau_\epsilon}-1\Big)\sup_{\substack{\eta'\in\Omega_n \\ {\rm d}(\pi^n,dx/2)\leq \epsilon}}\P_{\eta'}^n \bigg({\rm d}\Big(\pi^n_{\tau_\epsilon n^{-1/2}},\frac{dx}{2}\Big)\geq \epsilon\bigg) 
.
\end{align}
As $t-(p_{n,\epsilon}-1)\tau_\epsilon n^{-1/2}\in[\tau_\epsilon n^{-1/2},2\tau_\epsilon n^{-1/2}]$ and taking a supremum in $\eta$, 
we obtain the bound:
\begin{equation}
\sup_{\eta\in\Omega_n}\P_{\eta}^n\bigg({\rm d}\Big(\pi^n_{t},\frac{dx}{2}\Big)\geq \epsilon\bigg)
\leq 
\frac{ t\sqrt{n}\, }{\tau_\epsilon}\, 
\sup_{\eta\in\Omega_n }\P_{\eta}^n \bigg(\sup_{s\in[\tau_\epsilon,2\tau_\epsilon]}{\rm d}\Big(\pi^n_{s\, n^{-1/2}},\frac{dx}{2}\Big)\geq \epsilon\bigg) 
\end{equation}
The set $S_\epsilon := \{\sup_{s\in[\tau_\epsilon,2\tau_\epsilon]}{\rm d}(\rho_{s},dx/2)\geq \epsilon\}$ is closed in $\cD([0,2\tau_\epsilon],\ms M(\T))$. 
From Lemma~\ref{lemm_properties_LLN} and the explicit expression~\eqref{eq_def_dweak} of the distance ${\rm d}$ together with the bound ${\rm d}(\rho_{s},dx/2)\leq \|\rho_s-1/2\|_2$, see~\eqref{eq_bound_dw_by_L2}, 
$S_\epsilon$ does not contain any hydrodynamic trajectory. 
Since $I_{2\tau_\epsilon}(\cdot|\rho_0)$ is a good rate function for each measurable $\rho_0$, 
\begin{equation}
\inf_{\rho\in S_\epsilon} I_{2\tau_\epsilon}(\rho|\rho_0)
>
0
.
\end{equation}
Invoking Corollary~\ref{coro_LD} completes the proof. 
\end{proof}
\begin{remark}\label{rmk_no_global_continuity_LD}
In Section~\ref{subsec_renorm_sec6} we make use of large deviation bounds on $\pi^n(\phi^n)$ for test functions of the form $\phi^n(\cdot)=\psi(\cdot,i_2/n,...,i_p/n)$ for $p\geq 1$, $i_2,...,i_p\in\T_n$, $n\geq 1$ and a function $\psi$ continuous on the closure of $[0,1]^{p}\setminus\{x_1,...,x_{p} \text{ all different}\}$. 
One readily checks that Proposition~\ref{prop_LD_bound} can be used to quantify deviations of $\pi^n(\phi^n)$. 
Indeed, let $\epsilon,\delta>0$ and take a finite cover of $[0,1]^{p-1}$ by balls of radius $\delta$. 
As $\psi$ is uniformly continuous on the closure of $[0,1]^{p}\setminus\{x_1,...,x_{p} \text{ all different}\}$, 
taking $\delta$ small enough we can find some ${\bf x}=(x_2,...,x_p)\in[0,1]^{p-1}$ such that:
\begin{equation}
\bb P^n\Big(\, \Big|\pi^n(\phi^n)
-\frac{1}{2}\int_{\T}\phi^n(x)\, dx\, \Big| > \epsilon\,\Big)
\leq 
\bb P^n\Big(\, \Big|\pi^n(\psi(\cdot,{\bf x}))
-\frac{1}{2}\int_{\T}\psi(x,{\bf x})\, dx\, \Big| > \frac{\epsilon}{2} \,\Big)
.
\end{equation}
The function $\psi(\cdot,{\bf x})$ can then be approximated uniformly by a continuous function $\tilde\phi^{\bf x}$, 
equal to $\psi(\cdot,{\bf x})$ except at distance at most $\delta$ from $x_2,...,x_p$, 
and linearl on each $[x_i-\delta,x_i+\delta]$. 
Taking $\delta$ smaller than $\epsilon/(4(p-1)\|\psi\|_\infty)$ then gives:
\begin{equation}
\bb P^n\Big(\, \Big|\pi^n(\phi^n)
-\frac{1}{2}\int_{\T}\phi^n(x)\, dx\, \Big| > \epsilon\,\Big)
\leq 
\bb P^n\Big(\, \Big|\pi^n(\tilde\phi^{\bf x})
-\frac{1}{2}\int_{\T}\tilde\phi^{\bf x}\, dx\, \Big| > \frac{\epsilon}{4} \,\Big)
.
\end{equation}
The last event can be expressed in terms of the weak distance of $\pi$ to $dx/2$ (recall~\eqref{eq_def_dweak}) uniformly in ${\bf x}$ since, 
for each integer $k\in\Z$, 
the $k^{\text{th}}$ Fourier coefficient of $\tilde\phi^{\bf x}$ is bounded from below uniformly in ${\bf x}$ by construction. 
\end{remark}

\section{Properties of the reference measures}\label{app_concentration}
In this section, we present estimates involving the measures
$\nu^n_U$, $\nu^n_g$ needed in the proof of Theorem
\ref{theo_convergence_magnetisation}, as well as the initial
conditions $\mu^n_{b,c}$ of Remark~\ref{rm_IC}.

\subsection{The reference measure $\nu^n_U$ and initial conditions}\label{app_refmeas_nuU}

Recall from \eqref{eq_def_nu_V} the definition of the partition
function $Z^n_U$.

\begin{lemma}
\label{n-29}
For all $\theta\in\bb R$,
\begin{equation*}
\lim_{n\to\infty} \frac{1}{n^{1/4}}\, Z^n_U \,=\,
\sqrt{\frac{2}{\pi}} \,  \int_{\bb R} e^{-2\theta x^2 -x^4}\, dx \;.
\end{equation*}
\end{lemma}

\begin{proof}
Since $U_0(0) = U_0(1) < \log 2$, the terms $k=0$, $n$ in the sum
which defines $Z^n_U$ give a contribution of order $e^{-c_0n}$ for
some $c_0>0$. Fix $\kappa>0$ which will increase to $+\infty$ after
$n$, and divide the remaining sum in two pieces, the first one over
the set $\Lambda^1_n = \{k: |k - (n/2)|\le \kappa \, n^{3/4}\}$, and the
second one over
$\Lambda^2_n  = \{1, \dots, n-1\} \setminus \Lambda^1_n$.

Recall from \eqref{n-34}. \eqref{n-38} the values of the derivatives
of $U_0$ at $1/2$ (which coincide with those of $U$ at $0$), and the
ones of $E$. By a fifth order Taylor expansion,
$W(k/n) := E(k/n) - U_0(k/n) = 2(\theta/\sqrt{n}) (\mf m/n)^2 + (\mf
m/n)^4 (1 + o_n(1))$ for all $k\in \Lambda^1_n$. By this estimate
and Stirling's formula,
\begin{equation*}
\begin{aligned}
\sum_{k\in \Lambda^1_n} e^{n U_0(k/n)} { n \choose k} \,
\Big(\frac{1}{2}\Big)^n\,  & =\, (1 + o_n(1))\, \sqrt{\frac{2}{\pi n}}
\sum_{k\in \Lambda^1_n} e^{- 2 (\theta/\sqrt{n}) (\mf m/n)^2  -n (\mf m/n)^4} \\
\, & =\, (1 + o_n(1) + o_\kappa(1) )\, \sqrt{\frac{2}{\pi}} \, n^{1/4}
\int_{\bb R} e^{- 2\theta x^2 -x^4}\, dx  \;.
\end{aligned}
\end{equation*}
In this formula, $o_\kappa(1)$ represents a term that vanishes as
$n\to\infty$ and then $\kappa\to\infty$.  On the other hand, since
$W(x) \ge c\, [x-(1/2)]^4$ for some $c>0$, by the same type of
arguments,
\begin{equation*}
\sum_{k\in \Lambda^2_n} e^{n U_0(k/n)} { n \choose k} \,
\Big(\frac{1}{2}\Big)^n \,=\, o_\kappa(1) \, n^{1/4}\;.
\end{equation*}
Adding together the previous identities completes the proof of the lemma.
\end{proof}

Recall the definition of the measure $\mu^n_{b,c}$, $\mu_b$ introduced
in \eqref{n-30}, \eqref{n-32}, respectively. The same sort of argument
shows that
\begin{equation}
\label{n-31a}
\lim_{n\to\infty} \frac{1}{n^{1/4}}\, Z^n_{b,2} \,=\,
\sqrt{\frac{2}{\pi}} \,  \int_{\bb R} e^{b x^2 - (3/2) x^4}\, dx\;,
\quad 
\lim_{n\to\infty}  E_{\mu^n_{b,2}} \big[ \, F(\mc Y^n) \, \big]  \,=\,
\int_{\bb R} F(x) \, \mu^n_{b} (dx)  
\end{equation}
for all bounded continuous function $F\colon \bb R\to \bb R$.  A
similar analysis can be carried out for the partition function
$Z^n_{b,c}$, and the measure $\mu^n_{b,c}$ for $c<2$, provided the set
$\Lambda^1_n$ is defined as the set
$\{k: |k - (n/2)|\le \kappa\, \sqrt{n}\, \}$. It yields that
\begin{equation}
\label{n-31b}
\lim_{n\to\infty} Z^n_{b,c} \,=\,
\sqrt{\frac{2}{\pi}} \,  \int_{\bb R} e^{-(2-c) x^2}\, dx\;,
\quad 
\lim_{n\to\infty}  E_{\mu^n_{b,c}} \big[ \, F(\mc Y^n) \, \big]  \,=\, F(0)
\end{equation}
for all $c<2$, $b\in \bb R$, and bounded continuous function
$F\colon \bb R\to \bb R$.

The same arguments also permit to estimate the relative entropy of the
measures $\mu^n_{b,c}$, $\mu^n_{b,2}$ with respect to $\nu^n_{1/2}$.
In particular, these measures satisfy the requirements on initial
conditions in Theorem~\ref{theo_convergence_magnetisation}.
\begin{lemma}
\label{n-s18}
For all $b\in \bb R$, $c< 2$, there exists a finite constant $C_0=
C_0(b,c, \theta)$
such that
\begin{gather*}
H_n(\mu^n_{b,2} \,|\, \nu^n_{U}) \,\le\, C_0\;, \quad
H_n(\mu^n_{b,c} \,|\, \nu^n_{U}) \,\le\, C_0 \, \log n \;,
\\
H_n(\mu^n_{b,2} \,|\, \nu^n_{1/2}) \,\le\, C_0 \sqrt{n} \;, \quad
H_n(\mu^n_{b,c} \,|\, \nu^n_{1/2}) \,\le\, C_0 \;,
\end{gather*}
for all $n\ge 1$.
\end{lemma}

\begin{proof}
By definition of the relative entropy and of the measures
$\mu^n_{b,c}$, $\nu^n_{U}$, 
\begin{equation*}
H_n(\mu^n_{b,c} \,|\, \nu^n_{U}) \,=\, \int \log
\frac{\mu^n_{b,c}}{d\nu^n_{U}}\; d\mu^n_{b,c}\, 
=\, \log\frac{Z^n_U}{Z^n_{b,c}} \,+\,
\int \big\{\, c n (\mf m/n)^2 + b n^{1/2} (\mf m/n)^2 - n U(\mf m/n)\,
\big\}\; d\mu^n_{b,c} \;.
\end{equation*}
For $c=2$, by Lemma \ref{n-29}, and \eqref{n-31a}, the first term on
the right-hand side converges to a constant $C(b,\theta)$. Since
$U''(0) = 4$, the arguments presented in the proof of Lemma \ref{n-29}
yield that the second term converges to
\begin{equation*}
\int \big\{\, b x^2 - (1/3) x^4\, \big\}\; d\mu_{b}\;.
\end{equation*}

For $c<2$. by Lemma \ref{n-29} and \eqref{n-31b} the first term on the
right-hand side is bounded by $C \log n$. Computation similar to the
ones used to derive \eqref{n-31b} show that the second term
converges to
\begin{equation*}
\frac{1}{\widehat Z_c}\,  \int (c-2)\, x^2 \, e^{-(2-c) x^2} \, dx \;,  
\end{equation*}
where $\widehat Z_c$ is a normalising constant.  This completes the
proof of the lemma.

We turn to the product measure $\nu^n_{1/2}$.  By definition of the
relative entropy and of the measures $\mu^n_{b,c}$, $\nu^n_{1/2}$,
\begin{equation*}
H_n(\mu^n_{b,c} \,|\, \nu^n_{1/2}) \,=\, \int \log
\frac{\mu^n_{b,c}}{d\nu^n_{1/2}}\; d\mu^n_{b,c}\, 
=\, \log\frac{1}{Z^n_{b,c}} \,+\,
\int \big\{\, c n (\mf m/n)^2 + b n^{1/2} (\mf m/n)^2 \,
\big\}\; d\mu^n_{b,c} \;.
\end{equation*}
For $c=2$, by \eqref{n-31a}, the first term on the right-hand side is
bounded by a constant $C(b)$. The arguments presented in the proof of
Lemma \ref{n-29} yield that the second term is bounded by
$C_0\sqrt{n}$ for some finite constant $C_0$ independent of $n$.

For $c<2$. by\eqref{n-31b} the first term on the right-hand side is
bounded by $C$. Computation similar to the ones used to derive
\eqref{n-31b} show that the second term converges to
\begin{equation*}
\frac{1}{\widehat Z_c}\,  \int c \, x^2 \, e^{-(2-c) x^2} \, dx \;.
\end{equation*}
This completes the proof of the lemma.
\end{proof}

\subsection{The reference measure $\nu^n_g$}
Recall that $\nu^n_g$ is defined by:
\begin{equation}
\nu^n_g(\eta)
=
\frac{1}{Z^n_g}\exp\Big[\frac{1}{2n}\sum_{i\neq j}g_{i,j}\bar\eta_i\bar\eta_j\Big] \nu^n_{1/2}(\eta),
\qquad
\eta\in\Omega_n
,
\end{equation}
where $g_{i,j}=g_{i-j}$ is defined in Proposition~\ref{prop_g} as a function of $a\geq 0$ and $\gamma\in(-1,1)$. 
Throughout the section $a,\gamma$ are fixed. 
\subsubsection{Bound on the partition function}
\begin{lemma}\label{lemm_partition_fct_nu_g}
The partition function $Z^n_g$ satisfies:
\begin{equation}
\sup_{n\geq 1}\frac{\log Z^n_g}{\sqrt{n}\, }
<
\infty
.
\end{equation}
\end{lemma}
\begin{remark}
The bound on $\log Z^n_g$ is much worse than in Lemma~\ref{n-29}. 
This is due to the crude bound~\eqref{eq_crude_bound_Zng} 
that simplifies the argument but is not essential. 
One could also prove $\sup_{n\geq 1} n^{-1/4}Z^n_g<\infty$ with more effort.  
\end{remark}
\begin{proof}
Recall that $g= \min\{8\gamma,4\} {\bf 1} + g^0$ with $g^0$ having mean $0$ and ${\bf 1}$ the function constant equal to $1$.  
It is convenient to compute $Z^n_g$ when $g$ has leading eigenvalue strictly below $4$, so we consider $h$ given by:
\begin{equation}
h = g-\min\{8\gamma,4\}{\bf 1} + 4(1-n^{-1/2}){\bf 1}
,
\end{equation}
and focus on bounding $Z^n_h$ since:
\begin{equation}
Z^n_g
\leq 
e^{c\sqrt{n}\, }Z^n_h
,\qquad 
c>0
.
\label{eq_crude_bound_Zng}
\end{equation}
By definition,
\begin{equation}
Z^n_h 
=
E_{\nu^n_{1/2}}\bigg[ \exp\Big[ \frac{1}{2n}\sum_{i\neq j}\bar\eta_i\bar\eta_j h_{i,j}\Big]\,\bigg]
=
e^{-h(0)/8}\, E_{\nu^n_{1/2}}\bigg[ \, \exp\Big[ \frac{1}{2n}(\bar\eta,h\bar\eta)\Big]\, \bigg]
,
\end{equation}
where we used the identity $\bar\eta^2_\cdot=1/4$ to add the diagonal term and write $(u,v)=\sum_i u_iv_i$ for the usual inner product in $\R^n$.

Introduce a matrix $H$ on $\T_n^2$ as follows:
\begin{equation}
H_{i,j}
:=
n\int_{I_{i,j}} h(x,y)\, dx\, dy
,\qquad
I_{i,j}
:=
\Big[\frac{i}{n},\frac{i+1}{n}\Big)\times\Big[\frac{j}{n},\frac{j+1}{n}\Big)
.
\label{eq_def_H_bound_Z}
\end{equation}
Recall from Proposition~\ref{prop_g} that $g$ is $C^1$ on $[0,1]$, 
therefore $h$ as well and the approximation of $h_{i,j}$ by $n^{2}\int_{I_{i,j}}h$ incurs an error $O(n^{-1})$ at most. 
Thus,
for some constant $C(g)>0$, 
\begin{equation}
Z^n_h
\leq 
C(g)\, E_{\nu^n_{1/2}}\Big[e^{(\bar\eta,H\bar\eta)/2}\Big]
.
\label{eq_Zng_interm}
\end{equation}
Note that for any $u\in\R^n$, 
\begin{equation}
(u,Hu) = \int f(x)h(x-y)f(y)\, dx\, dy,\qquad 
f(x)
=
\sum_{i\in\T_n}{\bf 1}_{[\frac{i}{n},\frac{i+1}{n})}(x)u_i
.
\end{equation}
The fact that $h$ is positive definite as an operator on $\bb L^2(\T)$ and has spectral radius $4-n^{-1/2}$ therefore immediately transfers over to $H$: 
for each $u\in\R^n\setminus\{0\}$ with $\|u\|_2=1$,
\begin{align}
0
<
(u,Hu)
&\leq 
\sup_{f:\|f\|_2=1} \int_{\T^2} f(x) f(y)\, h(x-y)\, dx\, dy
\nnb
&=
\int_{\T^2} h(x-y)\, dx\, dy
\,=\,
4(1-n^{-1/2})
,
\label{eq_inverse_H}
\end{align}
where the second line comes from the explicit knowledge of the Fourier
decomposition of $g$ (thus of $h$) from Proposition~\ref{prop_g}.
Similarly, if $u\in\R^n$ has norm $1$ and satisfies $(u,{\mb 1})=0$:
\begin{equation}
(u,Hu)
\leq 
\sup_{\substack{f:\|f\|_2=1\\ \int_{\T}f(x)\, dx =0}} \int_{\T^2} f(x) f(y)\, h(x-y)\, dx
\leq 
\lambda^-_1
,
\label{eq_gap_H}
\end{equation}
with $(\lambda^-_\ell)_{\ell\geq 0}$ defined in Proposition~\ref{prop_g}. 

To bound the expectation in~\eqref{eq_Zng_interm}, 
denote by $\gamma_H\propto e^{-(\psi,H\psi)/2}\, d\psi$ the centred Gaussian measure on $\R^n$ with covariance $H^{-1}$, 
which is non-degenerate by~\eqref{eq_inverse_H}. 
The formula for the moment generating function of $\gamma_H$ then yields:
\begin{equation}
E_{\nu^n_{1/2}}\Big[e^{(\bar\eta,H\bar\eta)/2}\Big]
=
E_{\gamma_H}\bigg[ E_{\nu^n_{1/2}}\Big[ e^{(\bar\eta,H\psi)}\Big]\bigg]
.
\end{equation}
The quantity to average against $\nu^n_{1/2}$ now factorises between different sites and a direct computation yields:
\begin{equation}
E_{\gamma_H}\bigg[ \, E_{\nu^n_{1/2}}\Big[ e^{(\bar\eta,H\psi)}\Big]\, \bigg]
=
E_{\gamma_H} \bigg[ \exp\Big[ \sum_{i\in\T_n}\log\cosh\big((H\psi)_i/2\big) \Big]\bigg]
.
\label{eq_Zng_interm1}
\end{equation}
The last formula involves the Gaussian expectation of an explicit random variable, call it $-V(\psi)$:
\begin{equation}
V(\psi)
:=
-\sum_{i\in\T_n}\log\cosh\big((H\psi)_i/2\big)
,\qquad 
\psi\in\R^n
.
\end{equation}
This $V$ satisfies $V(0)=0$, $\nabla V(0)=0$ and its Hessian reads: 
\begin{equation}
\He V(\psi)
=
-\frac{1}{4}H (\id - D(\psi))H,\qquad 
D(\psi)
:=
\diag \Big( \tanh((H\psi)_i)^2,\quad i\in\T_n\Big)
.
\end{equation}
Note in particular:
\begin{equation}
\He V(\psi)
\geq 
-\frac{1}{4}H^2
,\qquad 
\psi\in\R^n
,
\label{eq_bound_HessV}
\end{equation}
so that recalling~\eqref{eq_Zng_interm}--\eqref{eq_Zng_interm1} we have now established:
\begin{equation}
Z^n_h
\leq 
C(g)
\int_{\R^n}\frac{\sqrt{\det H}}{\pi^{n/2}} \exp\Big[ - \frac{(\psi,H\psi)}{2} + \frac{(\psi,H^2\psi)}{8}\Big]\, d\psi
.
\end{equation}
By~\eqref{eq_inverse_H} 
the quantity $H-\frac{1}{4}H^2$ is a positive definite matrix. 
The Gaussian integral can therefore be computed and we find:
\begin{equation}
Z^n_h
\leq 
C(g)
\sqrt{\frac{\det(H)}{\det\big(H-\frac{1}{4}H^2\big)}}
=
C(g)
\Big[\det\Big(\id-\frac{1}{4}H\Big)\Big]^{-1/2}
.
\label{eq_bound_Znh_nearend}
\end{equation}
It remains to bound the last determinant from below. 
Recalling Definition~\eqref{eq_def_H_bound_Z} of $H$, 
note first that $g\in C^1([0,1])$ implies, for some $c(g)>0$: 
\begin{equation}
\Big|\tr(H^p)- \int_{\T} h^p(x)\, dx\Big|
\leq 
\frac{c(g)}{n}
,\qquad 
p\in\{1,2\}
.
\label{eq_bound_trace_H}
\end{equation}
Write $\lambda^H_0 = 4[1-n^{-1/2}]\geq \lambda^H_1\geq ...\geq \lambda^H_{n-1}>0$ for the eigenvalues of $H$.  
The elementary bound $\log(1-x)\geq -x-x^2/2$ ($x\in(0,1)$) implies:
\begin{equation}
\log(1-\lambda^H_\ell/4)
\geq 
-\frac{\lambda^H_\ell}{4} 
-\frac{(\lambda^H_\ell)^2}{32},\qquad 
\ell\in\{1,...,n-1\}
.
\end{equation}
We now estimate the determinant using this bound and~\eqref{eq_bound_trace_H}. 
For some $c'(g)>0$,
\begin{align}
\det\Big(\id-\frac{1}{4}H\Big)
&=
\exp\Big[\sum_{\ell=0}^{n-1}\log(1-\lambda^H_\ell/4)
\Big]
=
\sqrt{n}\, 
\exp\Big[\sum_{\ell=1}^{n-1}\log(1-\lambda^H_\ell/4)\Big]
\nnb
&\geq 
\sqrt{n}\, 
\exp\Big[\sum_{\ell=1}^{n-1}\Big(-\frac{\lambda^H_\ell}{4} - \frac{(\lambda^H_{\ell})^2}{32}\Big)\, \Big]
\nnb
&\geq 
c'(g)\, \sqrt{n}\, 
.
\end{align}
Injecting this bound in~\eqref{eq_bound_Znh_nearend} concludes the proof.
\end{proof}

\subsubsection{Concentration under $\nu^{n,m}_g$}\label{sec_concentration_nu_nmg}

Let $\Omega_{n,m}$ be the slice of the hypercube with magnetisation $m$:
\begin{equation}
\Omega_{n,m}
:=
\Big\{\eta\in\Omega_n=\{0,1\}^{\T_n}: \sum_i\eta_i = m\Big\}
.
\end{equation}
Recall that $g$ is defined in Proposition~\ref{prop_g} and $\nu^{n,m}_g=\nu^n_g\big(\cdot|\sum_{i}\eta_i=m\big)$ is the canonical measure associated with the probability measure $\nu^n_g\propto \exp\big[\frac{1}{2n}\sum_{i\neq j}\bar\eta_i\bar\eta_j g_{i,j}\big]\nu^n_{1/2}$ on $\Omega_n$. 
\begin{assumption}[log-Sobolev inequality]\label{ass_LSI}
For any $a>0$, 
there is $\mf c_{\rm LS}=\mf c_{\rm
LS}(a)>0$ such that the measure
$\nu^{n,m}_g$ satisfies a log-Sobolev inequality with respect to the
nearest-neighbour Dirichlet form with constant $\mf c_{\rm LS}\,
n^2$, and with respect to the Bernoulli Laplace Dirichlet form with
constant $\mf c_{\rm LS}\,\log(n^2/m(n-m))$:
\begin{equation}
\ent_{\nu^{n,m}_g}(F)
\leq 
\mf c_{\rm LS}\, n^2\, \Gamma^{\rm ex}_{n}(F;\nu^{n,m}_g)
,
\label{eq_LSI_nn}
\end{equation}
and:
\begin{equation}
\ent_{\nu^{n,m}_g}(F)
\leq 
\mf c_{\rm LS}\, \log\Big(\frac{n^2}{m(n-m)}\Big) \nu^{n,m}_g\Big(\frac{1}{n}\sum_{i,j\in\T_n}\big[\nabla_{i,j}\sqrt{F}\big]^2\Big)
,
\label{eq_LSI_longrange}
\end{equation}
for all $n\ge 1$, $0\le m\le n$, and functions
$F:\Omega_{n,m}\to\R_+$. Here, $\nabla_{i,j} F =
F(\eta^{i,j})-F(\eta)$ and:
\begin{equation}
\ent_{\nu^{n,m}_g}(F)
:=
\nu^{n,m}_g(F\log F) - \nu^{n,m}_g(F)\log \nu^{n,m}_g(F)
.
\end{equation}
\end{assumption}
\begin{remark}
The log-Sobolev inequality~\eqref{eq_LSI_nn} with respect to the nearest-neighbour Dirichlet form is used to bound fast modes in the relative entropy estimate of Section~\ref{sec_free_energy_bounds}. 
The log-Sobolev inequality~\eqref{eq_LSI_longrange} with respect to a long-range Dirichlet form is on the other hand used to obtain concentration results under $\nu^{n,m}_g$ in Proposition~\ref{prop_concentration} below.

In the
$g=0$ case~\eqref{eq_LSI_nn}--\eqref{eq_LSI_longrange} are the Lee-Yau
results~\cite[Theorems 4 and 5]{LeeYau_LSI_RW1998}. 
For $g\neq
0$~\eqref{eq_LSI_longrange} is known for a range of values of
$a$ only if the right-hand side is replaced with the modified Dirichlet
form, 
i.e. with
$[\nabla_{i,j}\sqrt{F}]^2$ replaced by $\nabla_{i,j}F\nabla_{i,j}\log
F$.  This implies~\eqref{eq_LSI_nn} in the same range, but with an additional
$\log n$ prefactor~\cite{RRG}. 
\end{remark}
Recall that $\bar\eta^m_i = \bar\eta_i-\frac{1}{n}\sum_j\bar\eta_j$, with $\bar\eta_i = \eta_i-1/2$ ($i\in\T_n$). 
For a set $J\subset\N_{\geq 1}$, 
write:
\begin{equation}
\bar\eta^m_{i+J}
:=
\prod_{j\in J}\bar\eta^m_{i+j}
,
\end{equation}
For $p\geq 1$, and a bounded function $\phi:\T^p\to\R$, 
let $W^{J,\phi}_p$ be the following observable:
\begin{equation}
W^{J,\phi}_p(\eta)
:=
\frac{1}{n^{p-1}}\sum_{i_1,...,i_p} \bar\eta^m_{i_1+J}\bar\eta^m_{i_2}...\bar\eta^m_{i_p}\phi_{i_1,...,i_p}
.
\label{eq_def_WJphi_p}
\end{equation}
We may assume, without loss of generality, that
$\phi$ is symmetric in its last
$p-1$ indices.  In addition, we may also assume
(see~\cite[Theorem A.3]{Correlations2022}):
\begin{equation}
\phi_{i_1,...,i_p}=0\quad \text{whenever there exists}
\, j\in J \;\;\text{such that}\;\;
|\{i_1+j, i_2 ,\dots ,i_p\}|<p
.
\label{eq_ass_tensor_diagonal}
\end{equation}
\begin{proposition}\label{prop_concentration}
Suppose Assumption~\ref{ass_LSI} holds. 
Let $\epsilon\in(0,1/2)$ and $m\in[\epsilon n,(1-\epsilon)n]$.  
If $p=1$, then:
\begin{equation}
\forall n\geq 1,\forall s>0,
\qquad
\log \nu^{n,m}_g\Big[\exp \Big(\frac{s|W^{J,\phi}_1|}{\sqrt{n}\, \|\phi\|_\infty} \Big)\Big]
\leq 
2s^2 {\mf c}_{\rm LS}\log(2/\epsilon)
.
\end{equation}
Assume now $p\geq 2$, $\epsilon\in(0,1/2)$ and $m\in[\epsilon
n,(1-\epsilon)n]$. Fix a a bounded function
$\phi:\T^p\to\R$ satifying the assumptions
\eqref{eq_ass_tensor_diagonal}.  Then, there are
$\lambda_p,C_p>0$ depending only on $J,p,\epsilon$ and
$g$ such that the following holds: for any
$\lambda\in[0,\lambda_p]$ and any $n\geq 1$,
\begin{equation}
\log \nu^{n,m}_g\Big[\exp \Big(\frac{\lambda |W^{J,\phi}_p|}{\|\phi\|_\infty}\Big)\Big]
\leq 
\frac{C_p \lambda}{n^{(p-2)/2}}\,\cdot
\label{eq_bound_mom_exp_appendix}
\end{equation}
\end{proposition}
\begin{remark}
The claim is valid under Assumption~\ref{ass_LSI}, 
thus in particular holds for the uniform measure $\nu^{n,m}=\nu^{n,m}_{g=0}$. \demo
\end{remark}
\begin{remark}\label{rmk_better_concentration_1pt}
In the $p=1$ case one has in fact the following stronger result. 
Let $\phi:\T\to\R$ and define $\|\phi\|^2_{\rm HS}:= \frac{1}{n}\sum_{i=1}^n \phi_i^2$. 
Then:
\begin{equation}
\forall n\geq 1,\forall s>0,
\qquad
\log \nu^{n,m}_g\Big[\exp \Big(\frac{s|W^{J,\phi}_1|}{\sqrt{n}\,\|\phi\|^2_{\rm HS}} \Big)\Big]
\leq 
2s^2 {\mf c}_{\rm LS}\log(2/\epsilon)
.
\end{equation}
\end{remark}
\begin{proof}[Proof of Proposition~\ref{prop_concentration}]
Let $p\geq 1$, $\epsilon\in(0,1/2)$ and $m\in[\epsilon n,(1-\epsilon)n]$. 
The $p=1$ case is an instance of the well-known Herbst argument (see~\cite[Section 5.2]{BLM13}), 
so we give no more detail and focus on the proof of~\eqref{eq_bound_mom_exp_appendix}. 

Fix $p\geq 2$ henceforth.   
Define the unnormalised version of $W^{\phi,J}_p$:
\begin{equation}
X^{\phi,J}_p
:= 
n^{p-1}W^{\phi,J}_p
.
\end{equation}
We claim that, to prove~\eqref{eq_bound_mom_exp_appendix}, 
it is sufficient to prove the existence of $\zeta_p>0$, 
depending only on $J,p,g$ but not $\phi$, 
such that: 
\begin{equation}
\nu^{n,m}_g\bigg( \exp\Big[\frac{\zeta_p |X^{J,\phi}_p|^{2/p}}{n\|\phi\|^{2/p}_\infty}\Big]\bigg)
\leq 
2
.
\label{eq_Gaussian_exp_moments}
\end{equation}
Indeed, define $\lambda_p:=\zeta_p/2$ and recall the elementary formula
$\E[e^{|X|}]=1+\int_0^\infty e^t\Prob(|X|>t)\, dt$.  Using that
$|W^{J,\phi}_p|\leq n\,\|\phi\|_\infty$, the
bound~\eqref{eq_Gaussian_exp_moments} implies, for each
$\lambda\in[0,\lambda_p]$:  
\begin{align}
\nu^{n,m}_g\Big[\exp \Big(\frac{\lambda |W^{J,\phi}_p|}{\|\phi\|_\infty}\Big)\Big]
&=
1+ \int_0^{\lambda n} e^{t}\nu^{n,m}_g\Big(\frac{\lambda |X^{J,\phi}_p|}{\|\phi\|_\infty}> tn^{p-1}\Big)\, dt
\nnb
&=
1+ \int_0^{\lambda n} e^{t}\nu^{n,m}_g\bigg(\frac{\zeta_p|X^{J,\phi}_p|^{2/p}}{n\|\phi\|^{2/p}_\infty}>\zeta_p(1/\lambda)^{2/p}n^{1-2/p}t^{2/p}\bigg)\, dt
\nnb
&\leq 
1+ 2\int_0^{\lambda n} \exp\Big[t-\zeta_p(1/\lambda)^{2/p}n^{1-2/p}t^{2/p}\Big]\, dt
.
\label{eq_moment_by_tail_bound}
\end{align}
Elementary computations show that the argument of the exponential satisfies:
\begin{equation}
t-\zeta_p(1/\lambda)^{2/p}n^{1-2/p}t^{2/p}
\leq 
-\frac{1}{2}\zeta_p(1/\lambda)^{2/p}n^{1-2/p}t^{2/p}
,\qquad
t\leq \lambda n
.
\end{equation}
In particular the integral is bounded uniformly in $n$, 
and a change of variable gives the decay in~\eqref{eq_bound_mom_exp_appendix} and the following expression for the constant $C_p$:
\begin{equation}
C_p
\leq 
2\zeta_p^{-p/2}\int_0^\infty e^{-u^{2/p}/2}\, du
.
\label{eq_bound_constant_exp_moment}
\end{equation}

We now prove~\eqref{eq_Gaussian_exp_moments}. 
Let $c_p>0$. 
Using Jensen inequality to write:
\begin{equation}
\nu^{n,m}_g\bigg( \exp\Big[\frac{\zeta_p |X^{J,\phi}_p|^{2/p}}{n\|\phi\|^{2/p}_\infty}\Big]\bigg)
\leq 
1+\sum_{q=1}^\infty \frac{1}{q!}\nu^{n,m}_g\Big[ \frac{\zeta^{pq/2}_{p} |X^{J,\phi}_p|^{q}}{n^{pq/2}\|\phi\|^{q}_\infty}\Big]^{2/p}
\end{equation}
and using $q^q\leq q!\, e^q$ for $q\geq 1$, 
the bound~\eqref{eq_Gaussian_exp_moments} holds with constant $\zeta_p=(2e\log(2/\epsilon))^{-1} c_p^{-2/p}$ (and therefore~\eqref{eq_bound_mom_exp_appendix} with constant $\lambda_p=\zeta_p/2$) 
as soon as the following bound on moments is valid:
\begin{equation}
\forall q\geq 1,\qquad
M_q(X^{J,\phi}_p)
\leq 
c_p \log(2/\epsilon)^{p/2}q^{p/2}\, n^{p/2}\, \|\phi\|_\infty
,\qquad
M_q(X^{\phi,J}_p) 
:=
\nu^{n,m}_g \Big[\big|X^{\phi,J}_p\big|^{q}\Big]^{1/q}
.
\label{eq_moment_bound_uniform}
\end{equation}
We therefore focus on proving~\eqref{eq_moment_bound_uniform} for some
$c_p>0$. By Chauchy-Schwarz inequality we can restrict to $q\geq 2$. 

The proof of~\eqref{eq_Gaussian_exp_moments} then relies on the
log-Sobolev inequality of Assumption~\ref{ass_LSI} and is similar
to~\cite[Theorem A.3]{Correlations2022}, itself adapted
from~\cite{Goetze}.  The only difference with~\cite[Theorem
A.3]{Correlations2022} is the type of gradient arising in the
Dirichlet form: there, the gradients where of Glauber type (one spin
flipped) whereas here gradients involve spin exchanges.  The statement
in~\cite{Goetze} already involved general gradients so the following
proof is just an adaptation of their argument with explicit bounds.

Using the log-Sobolev inequality~\eqref{eq_LSI_longrange} as in the proof of~\cite[Theorem A.3]{Correlations2022}, 
one arrives at the following bound on moments (recall that $m\in[\epsilon n,(1-\epsilon)n]$ implies $\log(n^2/(m(n-m))\leq \log(2/\epsilon)$):
\begin{align}
\forall q\geq 2,\qquad
M_q(X^{J,\phi}_p)^2
&\leq 
\mf c_{\rm LS}\, \log(2/\epsilon)\, q\,  M_{q/2}\Big(\frac{1}{n}\sum_{k,\ell\in\T_n}\big[\nabla_{k,\ell}X^{J,\phi}_p\big]^2\Big)
\nonumber\\
&\leq 
\mf c_{\rm LS}\, \log(2/\epsilon)\, q\, \frac{1}{n}\sum_{k,\ell\in\T_n} M_{q}\big(\nabla_{k,\ell}X^{J,\phi}_p\big)^2
\label{eq_moment_bound_LSI_interm}
,
\end{align}
with $\nabla_{k,\ell}F(\eta) := F(\eta^{k,\ell})-F(\eta)$ for any
sites $k,\ell\in\T_n$ and configuration $\eta$.  Since, by hypothesis,
$\phi$ is symmetric on the last $p-1$ variables, the gradient in the
right-hand side reads, for $k,\ell\in \T_n$:
\begin{align}
\nabla_{k,\ell}X^{J,\phi}_p(\eta) \, & =\,
(\eta_{\ell}-\eta_k)\, \sum_{j\in J}
\, \bar\eta^m_{(k-j) + J\setminus\{j\}}
\sum_{\substack{i_2,...,i_p\\ i_a
\neq k,\ell}} \phi(k-j,i_2,...,i_p)
\, \bar\eta^m_{i_2}...\bar\eta^m_{i_p}\,
\nonumber
\\
& -\, (\eta_{\ell}-\eta_k)\, \sum_{j\in J}
\bar\eta^m_{(\ell-j) + J\setminus\{j\}}
\sum_{\substack{i_2,...,i_p\\ i_a \neq k,\ell}} 
\phi(\ell-j,i_2,...,i_p)\, 
\bar\eta^m_{i_2}...\bar\eta^m_{i_p}\,
\label{eq_gradient_X_m}
\\
&
+(p-1)(\eta_{\ell}-\eta_k) \, \sum_{\substack{i_1,...,i_{p-1}\\ i_a
\neq k,\ell}} \,
\big[\, \phi(i_1,i_2,...,k) - \phi(i_1,i_2,...,\ell)\,\big]
\, \bar\eta^m_{i_1+J}\bar\eta^m_{i_2}...\bar\eta^m_{i_{p-1}}
.
\nonumber
\end{align}
Mind that each term on the right-hand side can be expressed as some
$X^{J',\phi'}_{p-1}$. Indeed, let $\phi^{(1)}_{j,k,\ell}$,
$\phi^{(2)}_{j,k,\ell}$ $j\in J$, be given by
$\phi^{(1)}_{j,k,\ell} (i_1, \dots, i_{p-1}) = \phi (k-j, i_1, \dots,
i_{p-1}) \mb 1\{ i_a \neq k, \ell\}$,
$\phi^{(2)}_{j,k,\ell} (i_1, \dots, i_{p-1}) = \phi (\ell-j, i_1,
\dots, i_{p-1}) \mb 1\{ i_a \neq k, \ell\}$. With this notation, the
first two terms on the right-hand side cab be written as
\begin{equation}
\label{05}
(\eta_{\ell}-\eta_k)\, \sum_{j\in J}
\big\{\, \bar\eta^m_{(k-j) + J\setminus\{j\}} \,
X^{\{0\},\phi^{(1)}_{j,k,\ell}}_{p-1} 
\,-\, \bar\eta^m_{(\ell-j) + J\setminus\{j\}}
\, X^{\{0\},\phi^{(2)}_{j,k,\ell}}_{p-1} \,\big\}\;.
\end{equation}
Similarly, define $\phi^{(3)}_{k,\ell} (i_1, \dots, i_{p-1})$ as 
$[\phi (i_1, \dots, i_{p-1}, k) -  \phi (i_1, \dots, i_{p-1}, \ell)]\,
\mb 1\{ i_a \neq k, \ell\}$, to write the last term as
$(p-1)\, (\eta_{\ell}-\eta_k) X^{J,\phi^{(3)}_{k,\ell}}_{p-1}$. Note
that the functions $\phi^{(1)}_{j,k,\ell}$, $\phi^{(2)}_{j,k,\ell}$,
$\phi^{(3)}_{k,\ell}$ satisfy the hypotheses
\eqref{eq_ass_tensor_diagonal} and are bounded by $C_0
\Vert\phi\Vert_\infty$ for some constant $C_0$ independent of
$n$.
This leads to a proof by recursion on $p$.

In the $p=1$ case,~\eqref{eq_moment_bound_uniform} and ~\eqref{eq_gradient_X_m} give:
\begin{align}
M_q(X^{J,\phi}_1)^2
&\leq 
\mf c_{\rm LS}\, \log(2/\epsilon)\,
\frac{q}{n}\sum_{k,\ell}M_{q}
\Big(\sum_{j\in J}\big[\phi(k-j)\bar\eta^m_{(k-j) + J \setminus\{j\}}
- \phi(\ell-j)\bar\eta^m_{(\ell-j)+J\setminus\{j\}}\big]\Big)^2
\nonumber\\
&\leq 
4\, \mf c_{\rm LS}\, \log(2/\epsilon) q\,  n |J|^2\|\phi\|^2_{\infty}
.
\end{align}
This proves~\eqref{eq_moment_bound_uniform} for $p=1$ with constant $c_1 = \sqrt{\mf c_{\rm LS}}\, |J|$.

Assume now that~\eqref{eq_moment_bound_uniform} hold up to rank
$p-1\geq 1$.  By \eqref{05}, for each $k,\ell\in\T_n$:
\begin{align}
\nabla_{k,\ell}X^{J,\phi}_p
=
(\eta_{\ell}-\eta_k)\, \Big\{ \sum_{j\in J}
\big\{\, \bar\eta^m_{(k-j) + J\setminus\{j\}} \,
X^{\{0\},\phi^{(1)}_{j,k,\ell}}_{p-1} 
\,-\, \bar\eta^m_{(\ell-j) + J\setminus\{j\}}
\, X^{\{0\},\phi^{(2)}_{j,k,\ell}}_{p-1} \,\big\}
+
(p-1)\, X^{J,\phi^{(3)}_{k,\ell}}_{p-1}\Big\}\;.
\end{align}
This implies, bounding the $\bar\eta^m$'s by $1$:
\begin{align}
&\big|\nabla_{k,\ell}X^{J,\phi}_p\big|
\leq
(p-1) \big|X^{J,\phi^{(3)}_{k,\ell}}_{p-1}\big| 
+
\sum_{j\in J}\Big\{\, \big|X^{\{0\},\phi^{(1)}_{j,k,\ell}}_{p-1} \big|
+ \big|X^{\{0\},\phi^{(2)}_{j,k,\ell}}_{p-1}\big|\, \Big\}
.
\end{align}
as observed above, the tensors appearing in each of these three
objects still satisfy Assumption~\eqref{eq_ass_tensor_diagonal} by
construction.  The moment bound~\eqref{eq_moment_bound_LSI_interm} and
the recursion hypothesis at rank $p-1$ thus yield:
\begin{align}
M_q(X_p^{J,\phi})^2
&\leq 
\mf c_{\rm LS}\, \log(2/\epsilon)\,  \frac{q}{n}\,
\sum_{k,\ell}\bigg[2(p-1)^2 M_q\big(X^{J,\phi^{(3)}_{k,\ell}}_{p-1}\big)^2 
\nnb
&\hspace{3.5cm}
+ 2\Big(\sum_{j\in J}\big[M_q
\big(X^{\{0\},\phi^{(1)}_{j,k,\ell}}_{p-1}\big)
+ M_q\big(X^{\{0\},\phi^{(2)}_{j,k,\ell}}_{p-1}\big)\big] \Big)^2\bigg]
\nnb
&\leq 2\, \mf c_{\rm LS}\, c_{p-1}\, \log(2/\epsilon)^{p}\,  q^{p}\, n\,  \big(1+2|J|^2\big)\|\phi\|^{p}_{\infty}
.
\end{align}
This proves~\eqref{eq_moment_bound_uniform} at rank $p$ 
with $c_p:= \sqrt{c_{p-1}\mf c_{\rm LS}(1+2|J|^2)}$ and concludes the proof.
\end{proof}
As soon as $p\geq 3$, i.e. with sums of correlations between at least three independent indices, 
it becomes possible to improve the sub-Gaussian concentration bounds implied by the log-Sobolev inequality using the fact that $W^{J,\phi}_p$ is bounded. 
This gives exponential integrability of $\lambda W^{J,\phi}_p$ even for large $\lambda$ provided $|W^{J,\phi}_p|/n$ is small enough as explained next.
\begin{corollary}\label{coro_better_concentration}
With the notations and assumptions of Proposition~\ref{prop_concentration}, 
let $p\geq 3$ and $\delta>0$. 
Then, for any $R\in(0,2\lambda_p(2\delta)^{2/p-1})$,
\begin{equation}
\log \nu^{n,m}_g\Big[ \exp \Big(\frac{R |W^{J,\phi}_p|}{\|\phi\|_\infty} {\bf 1}_{|W^{J,\phi}_p|\leq \delta n\|\phi\|_\infty}\Big)\Big]
\leq 
\frac{C_pR}{n^{(p-2)/2}}
.
\end{equation}
\end{corollary}
\begin{proof}
The proof uses the same estimates~\eqref{eq_moment_by_tail_bound} to~\eqref{eq_bound_constant_exp_moment} with 
the additional input that the integral now stops at $\delta Rn$. 
\end{proof}
\subsubsection{Concentration under $\nu^n_g$}
Concentration results on the $W^{J,\phi}_p$ under $\nu^{n,m}_g$ established in the previous section are also valid under $\nu^n_g$ as we now show.
\begin{corollary}\label{coro_concentration_nun_g}
Let $p\geq 1$, $J\subset\Z$ be finite and $\phi:\T^p\to\R$. 
Then, for some $\lambda'_p,C'_p>0$: 
\begin{align}
\forall n\geq 1,
\qquad \qquad
\log \nu^{n}_g\Big[\exp \Big(\frac{\lambda |W^{J,\phi}_1|}{ \sqrt{n}\, \|\phi\|_\infty}\Big)\Big]
&\leq 
C_1 \lambda^2
,
\qquad 
\lambda\leq \lambda'_1\sqrt{n},
\nnb
\log \nu^{n}_g\Big[\exp \Big(\frac{\lambda |W^{J,\phi}_p|}{\|\phi\|_\infty}\Big)\Big]
&\leq 
\frac{C_p \lambda}{n^{(p-2)/2}}
,\qquad
p\geq 2, \lambda\leq \lambda'_p
.
\end{align}
The claim of Corollary~\ref{coro_better_concentration} similarly also holds under $\nu^n_g$. 
\end{corollary}
\begin{proof}
We just prove the counterpart of Proposition~\ref{prop_concentration}, 
the proof for the other claims being similar. 
Decompose $\nu^n_g$ according to the magnetisation:  
\begin{equation}
\nu^n_g =
\sum_{m=0}^n \nu^n_g\Big(\sum_i\eta_i=m\Big) \nu^{n,m}_g.
\end{equation}
A direct large deviation bound provides the following concentration result under $\nu^n_g$: 
there is $c>0$ such that, for all $n\geq 1$ (recall $\pi^n=\frac{1}{n}\sum_i\eta_i\delta_{i/n}$ is the empirical measure),
\begin{equation}
\nu^n_g\big( {\rm d}(\pi^n, dx/2)\geq \epsilon\big)
\leq 
c^{-1}e^{-nc}
.
\end{equation}
The $W^{J,\phi}_p$ are on the other hand bounded by $\|\phi\|_{\infty} n $ if $p\geq 2$, or by $\|\phi\|_\infty\, \sqrt{n}$ if $p=1$. 
Reducing $\lambda_p$ to ensure $\lambda_p \leq c/2$ for $p\geq 2$, 
we thus find, for each $\lambda\leq \lambda_p$:
\begin{equation}
\nu^n_g\Big[\exp\Big( \frac{\lambda\big|W^{J,\phi}_p\big|}{\|\phi\|_\infty}\Big)\Big]
\leq 
\max_{m:|m-n|\leq 1/4}\nu^{n,m}_g\Big[\exp\Big( \frac{\lambda\big|W^{J,\phi}_p\big|}{\|\phi\|_\infty}\Big)\Big]
+ c^{-1}e^{-cn/2}
,
\end{equation}
which upon taking $\frac{1}{\lambda}\log $ gives the claim when $p\geq 2$. 
The argument is similar when $p=1$, 
with the upper bound on $\lambda$ coming from the need to ensure the contribution of the $|m-n/2|\geq n/4$ is  compensated by $\nu^n_g(|\sum_i\bar\eta_i|\geq n/4)$.  A similar reasoning yields an analogue of Corollary~\ref{coro_better_concentration}. 
\end{proof}
\subsection{Relative entropy with respect to the product measure}\label{app_relative_entropy}
In this section we prove Proposition~\ref{prop_free_energy_time0}.
Let $\mu_n$ be a sequence of probability measures on $\Omega_n$
satisfying:
\begin{equation}
\sup_{n\geq 1}\frac{H_n(\mu_n\, |\, \nu^n_{1/2})}{\sqrt{n}}
<
\infty
.
\end{equation}
\begin{proof}[Proof of Proposition~\ref{prop_free_energy_time0}]
The entropy inequality gives, for each $\epsilon>0$:
\begin{equation}
E_{\mu_n}[(\cY^n)^2]
\leq 
\frac{H_n(\mu_n\, |\, \nu^{n}_{1/2})}{\epsilon \, \sqrt{n}} + \frac{1}{\epsilon \, \sqrt{n}}\log E_{\nu^n_{1/2}}[ e^{\epsilon \sqrt{n} \,(\cY^n)^2}]
.
\end{equation}
Elementary computations show that the exponential moment is bounded with $n$ for all $\epsilon$ smaller than some $\epsilon_0>0$. 
In particular this follows from the proof of Lemma~\ref{n-29} which deals with a more complicated setting.  

We now turn to the relative entropy bound. It reads:
\begin{align}
H_n(\mu_n\, |\, \nu^{n}_{\rm ref})
&=
\sum_{\eta\in\Omega_n} \mu_n(\eta)\log\Big(\frac{\mu_n(\eta)}{\nu^n_{\rm ref}(\eta)}\Big)
\nnb
&=
\sum_{\eta\in\Omega_n} \mu_n(\eta)\log\Big(\frac{\mu_n(\eta)}{\nu^n_{1/2}(\eta)}\Big) 
- \sum_{\eta\in\Omega_n} \mu_n(\eta)\log\Big(\frac{\nu^n_{\rm ref}(\eta)}{\nu^n_{1/2}(\eta)}\Big)
.
\end{align}
The first term is just $H(\mu_n|\nu^n_{1/2})$, 
which is $O(\sqrt{n}\, )$ by assumption. 
Let us prove that the second term is bounded by $O( \sqrt{n}\, )$. \\

\noindent {\bf Case 1: $\nu^n_{\rm ref}=\nu^n_U$.} 
Then:
\begin{equation}
-\sum_{\eta\in\Omega_n} \mu_n(\eta)\log\Big(\frac{\nu^n_{\rm ref}(\eta)}{\nu^n_{1/2}(\eta)}\Big)
=
-E_{\mu_n}\big[ n U(\ms M^n )\big] + \log Z^n_U
,\qquad 
\ms M^n 
=
\frac{1}{n}\sum_{i\in\T_n}\bar\eta_i
.
\end{equation}
The logarithm of the partition function is bounded by $O(\log n)$ by Lemma~\ref{n-29}, 
while $U\geq 0$ by definition (recall~\eqref{n-06}). 
This proves the bound in this case. \\

\noindent {\bf Case 2: $\nu^n_{\rm ref}=\nu^n_g$.} 
One has:
\begin{equation}
-\sum_{\eta\in\Omega_n} \mu_n(\eta)\log\Big(\frac{\nu^n_{\rm ref}(\eta)}{\nu^n_{1/2}(\eta)}\Big)
=
-E_{\mu_n}\Big[ \frac{1}{2n}\sum_{i\neq j}\bar\eta_i\bar\eta_j g_{i,j}\Big] + \log Z^n_g
.
\end{equation}
The second term is bounded by $O(\sqrt{n}\, )$ by Lemma~\ref{lemm_partition_fct_nu_g}. 
The explicit formula for $g$ shows that it is a positive kernel in $\bbL^2(\T)$. 
As $g\in C(\T)\cap C^1([0,1])$ by Proposition~\ref{prop_g}, 
the matrix $(g_{i,j})_{i,j\in\T_n}$ is also positive (this is proven e.g.~\cite[Theorem 2.3]{FerreiraMenegato} or can be seen from~\eqref{eq_def_H_bound_Z}--\eqref{eq_inverse_H} using an explicit approximation argument). 
Thus, adding and subtracting the indices $i=j$ and recalling $\bar\eta_i^2=1/4$:
\begin{equation}
E_{\mu_n}\Big[ \frac{1}{2n}\sum_{i\neq j}\bar\eta_i\bar\eta_j g_{i,j}\Big] 
\geq 
- \frac{g(0)}{8}
.
\end{equation}
This concludes the proof. 
\end{proof}

\section{Integration by parts formula}\label{sec_IBP}
Let $m\in\{0,...,n\}$ and recall that $\Omega_{n,m}\subset\Omega_n$ is the subset of configurations with magnetisation $m$. 
This section contains an integration by parts formula under the fixed-magnetisation measure $\nu^{n,m}_g$, with $g$ given by Proposition~\ref{prop_g}:
\begin{equation}
\nu^{n,m}_g
:=
\nu^n_g\Big(\cdot \Big| \sum_{i\in\T_n}\eta_i = m\Big)
.
\end{equation}
The formula is more generally valid under $\nu^n_{1/2}$, $\nu^{n,m}_{1/2}$ and $\nu^n_U$ (recall $\nu^{n,m}_{1/2}=\nu^{n,m}_{g=0}$ is the uniform measure on $\Omega_{n,m}$).  
For $i\in\T_n$ and $f:\Omega_{n,m}\to\R$, define:
\begin{equation}
\Gamma^{{\rm ex}, i,i+1}_{n}(f;\nu^{n,m}_g)
:=
\nu^{n,m}_{g}\big([\nabla_{i,i+1}f(\eta)]^2\big)
,
\qquad
\nabla_{i,i+1}f(\eta)
=
f(\eta^{i,i+1})-f(\eta)
,\quad
\eta\in\Omega_{n,m}
.
\end{equation}
\begin{lemma}\label{lemm_IBP}
Let $i\in\T_N$ and $f:\Omega_{n,m}\to\R_+$.  
Let also $h:\Omega_{n,m}\to\R$ be invariant under the transformation $\eta\mapsto\eta^{i,i+1}$, i.e. $h(\eta^{i,i+1})=h(\eta)$. 
Then,  
for any $\delta >0$:
\begin{align}
&\nu^{n,m}_{g}\big( fh (\bar\eta^m_{i+1}-\bar\eta^m_i)\big)
\leq
\delta \, \Gamma^{{\rm ex}, i,i+1}_{n}(f;\nu^{n,m}_g)  + \frac{e^{\|g'\|_\infty/n}}{2\delta} \nu^{n,m}_{g}\big(f|h|^2\big)
\nonumber\\
&\qquad 
- \frac{1}{2}\nu^{n,m}_{g}\bigg( f h (\bar\eta^m_{i+1}-\bar\eta^m_i) \Big(\exp \Big[-\frac{(\bar\eta^m_{i+1}-\bar\eta^m_i)}{n^2} \sum_{j\neq i,i+1}\bar\eta^m_j\partial^{n}_1g_{i,j}\Big]-1\Big)\bigg)
\end{align}
and:
\begin{align}
&\nu^{n,m}_{g}\big( fh (\bar\eta^m_{i+1}-\bar\eta^m_i)\big)
\leq
\delta\, \Gamma^{{\rm ex}, i,i+1}_{n}(f;\nu^{n,m}_g)  + \frac{e^{\|g'\|_\infty/n}}{\delta} \nu^{n,m}_{g}\big(f|h|^2\big)
\nonumber\\ 
&\hspace{2cm}
+ \frac{e^{2\|g'\|_\infty/n}}{2n^2}\nu^{n,m}_{g}\bigg( f \Big|h\sum_{j\neq i,i+1}\bar\eta^m_j\partial^{n}_1g_{i,j}\Big|\bigg) +\frac{\|g'\|_\infty e^{2\|g'\|_\infty/n}}{2n^2}\nu^{n,m}_g(f|h \ms M^n|)
.
\end{align}
\end{lemma}
\begin{remark}
Setting $g=0$, the same integration by parts formula holds under $\nu^{n,m}_{1/2}=\nu^{n,m}_{g=0}$, with the same proof. 
The result similarly holds under $\nu^n_{1/2},\nu^n_U,\nu^n_g$ provided $f,h$ take arguments in $\Omega_n$. 
It also holds under all these measures with each $\bar\eta^m$ replaced by $\bar\eta$. 
\demo
\end{remark}
\begin{proof}
A direct computation using the fact that $\eta\mapsto\eta^{i,i+1}$ is bijective gives:
\begin{equation}
\nu^{n,m}_{g}\big( f h (\bar\eta^m_{i+1}-\bar\eta^m_i)\big) 
=
\frac{1}{2}\sum_{\eta\in\Omega_{n,m}}\nu^{n,m}_{g}(\eta)h(\eta) (\bar\eta^m_{i+1}-\bar\eta^m_i)\Big[f(\eta) - \frac{\nu^{n,m}_{g}(\eta^{i,i+1})}{\nu^{n,m}_{g}(\eta)}f(\eta^{i,i+1})\Big]
.
\end{equation}
The ratio of measures reads, for each $\eta\in\Omega_{n,m}$:
\begin{equation}
\frac{\nu^{n,m}_{g}(\eta^{i,i+1})}{\nu^{n,m}_{g}(\eta)}
=
\exp \Big[-\frac{(\bar\eta^m_{i+1}-\bar\eta^m_i)}{n^2} \sum_{j\neq i,i+1}\bar\eta^m_j\partial^n_1 g_{i,j}\Big]
.
\end{equation}
Thus, writing $e^x = 1 + e^x-1$ with $x$ the argument of the above exponential and again changing variables between $\eta$ and $\eta^{i,i+1}$:
\begin{align}
\nu^{n,m}_{g}\big( fh (\bar\eta_{i+1}-\bar\eta_i)\big)
&\leq
\frac{1}{2}\nu^{n,m}_{g}\big( |\nabla_{i,i+1} f||h| \big)
\\
&\ - 
\frac{1}{2}\nu^{n,m}_{g}\bigg( fh(\bar\eta^m_{i+1}-\bar\eta^m_i) \Big(\exp\Big[-\frac{(\bar\eta_{i+1}-\bar\eta_i)}{n^2} \sum_{j\neq i,i+1}\bar\eta_j\partial^n_1 g_{i,j}\Big]-1\Big)\bigg)
\nonumber
.
\end{align}
Note that $\bar\eta_{i+1}-\bar\eta_i = \bar\eta^{m}_{i+1}-\bar\eta^m_i$. 
The definition of $g$ in Proposition~\ref{prop_g} moreover implies:
\begin{equation}
\Big|\sum_{j\neq i,i+1}\partial^n_1 g_{i,j}\Big|
=
n|g_1-g_{-1}|
\leq 
2\|g'\|_\infty
.
\end{equation}
Writing $\bar\eta_j=\bar\eta^m_j + \frac{1}{n}\sum_{\ell}\bar\eta_\ell$ and using the elementary inequality $|uv|\leq \delta u^2 + (4\delta)^{-1}|v|^2$ 
with $u=\nabla_{i,i+1}\sqrt{f}$ and $v=[\sqrt{f}(\eta)+\sqrt{f}(\eta^{i,i+1})]|h|$ therefore gives:
\begin{align}
&\nu^{n,m}_{g}\big( fh (\bar\eta_{i+1}-\bar\eta_i)\big)
\leq
\delta \, \Gamma^{{\rm ex}, i,i+1}_{n}(f;\nu^{n,m}_g) + \frac{1}{8\delta}\nu^{n,m}_g\big([\sqrt{f}(\eta)+\sqrt{f}(\eta^{i,i+1})]^2|h|^2\big)
\nnb
&\hspace{1cm}
-\frac{1}{2}\nu^{n,m}_{g}\bigg( fh(\bar\eta^m_{i+1}-\bar\eta^m_i) 
\\
&\hspace{2cm}\qquad 
\times\Big(\exp\Big[-\frac{(\bar\eta^m_{i+1}-\bar\eta^m_i)}{n^2} \Big(\sum_{j\neq i,i+1}\bar\eta^m_j\partial^n_1 g_{i,j} + \ms M^n [n(g_1-g_{-1})]\Big)\Big]-1\Big)\bigg)
\nonumber
.
\end{align}
Expanding the square, 
changing variables $\eta\mapsto\eta^{i,i+1}$ and noticing that $\nu^{n,m}_g(\eta)/\nu^{n,m}_g(\eta^{i,i+1})$ is bounded by $e^{\|g'\|_\infty/n}$ gives the first claim. 
The second claim follows from $|\bar\eta^m_{i+1}-\bar\eta^m_i|^2\leq 1$, 
the bound $|e^{x}-1|\leq e^{x_0}|x|$ for $|x|\leq x_0$ and any $x_0>0$, 
and $1/n^2\leq 1/n$. 
\end{proof}
\section{Properties of the kernel $g$}\label{app_prop_g}
In this section we prove Proposition~\ref{prop_g}, i.e.:
\begin{enumerate}
	\item The existence of a family $(\delta,b)\in(-1,1)\times \R_+\mapsto g_{\delta,b}$ of solutions of~\eqref{eq_ODE_g} in $C(\T)\cap C^{\infty}([0,1])$, 
	and uniqueness when enforcing that $b\mapsto\|g^0_{\delta,b}\|_2$ and $\delta\mapsto\|g_{\delta,b}\|_2$ be continuous and vanish at $0$ ($g^0_{\delta,b}:=g_{\delta,b}-\int_\T g_{\delta,b}$). 
	\item The explicit Fourier representation~\eqref{eq_Fourier_g}. 
\end{enumerate}
\begin{proof}
We look for a weak solution to~\eqref{eq_ODE_g} in the space $\mathbb L^2(\T)$, 
i.e. we look for $g=g_{\delta,b}\in\mathbb L^2(\T)$ satisfying, for any $f$ in the Sobolev space $\mathbb H^2(\T)$:
\begin{align}
&\int_{\T} g(x) \big[f''(x)-2b(1+2\delta) f(x)\big]\, dx 
+ 4\sigma^{-1}\delta b f(0)
\nnb
&\quad+\int_{\T^2} g(x-z)g(z)\big[- \sigma f''(x) + (b/2) f(x)\big]\, dz\, dx 
=
0
.
\label{eq_weak_ODE_g}
\end{align}
As $(x,y)\mapsto g(x-y)$ is translation invariant, 
a closed equation on its Fourier coefficients can be obtained. 
Define:
\begin{equation}
c_{\ell}(g)
:=
\int_{\T} g(x) e^{2\pi i \ell x}\, dx
,\qquad
\ell\in\Z
.
\end{equation}
Let $\ell\in\Z$ and take $f=e^{2\pi i \cdot}$ as test function to find:
\begin{equation}
c_\ell(g) \big[-4\pi^2\ell^2 - 2b(1+2\delta)\big] 
+4\sigma^{-1}\delta b + (4\sigma\pi^2\ell^2 +b/2) |c_\ell(g)|^2
=
0
. 
\end{equation}
This second-degree polynomial has roots:
\begin{equation}
\lambda^\pm_\ell 
=
\frac{4\pi^2\ell^2 +2b(1+2\delta)}{b+8\sigma \pi^2\ell^2} 
\pm 
\frac{\sqrt{\big(4\pi^2\ell^2 +2b(1+2\delta)\big)^2 - 16\delta b \sigma^{-1}\big(4\sigma\pi^2\ell^2+b/2\big)}}{b+8\sigma \pi^2\ell^2} 
.
\label{eq_def_lambda_pm_ell}
\end{equation}
The requirement $g_{0,b}=0$ forces the mean $c_0(g_{0,b})$ of $g_{0,b}$ to be given by $\lambda^-_0$. 
Since we want continuity of the mean as a function of $\delta$, we must have:
\begin{equation}
\int_{\T} g(x)\, dx
=
\lambda^-_0
.
\end{equation}
Similarly, 
the only way $g^0 = g-\int_{\T}g$ can vanish when $b=0$ is if the Fourier coefficients of $g^0$ are given by the $\lambda^{-}_\ell$, ($\ell\neq 0$). 
Thus, by continuity:
\begin{equation}
g
=
\lambda^-_0 
+ \sum_{\ell\geq 1} \lambda^-_\ell\,  2\cos(2\pi\ell \cdot)
\qquad 
\text{in }\bbL^2(\T)
\label{eq_def_g_appendix}
\end{equation}
Expanding the square and using $\sigma=1/4$, 
the $\lambda^-_\ell$ can be rewritten as:
\begin{align}
\lambda^-_\ell
=
\frac{4\pi^2\ell^2+2b(1+2\delta)-\big|4\pi^2\ell^2 - 2b(1-2\delta)\big|}{b+2\pi^2\ell^2}
\geq 
0
,\quad
\lambda^-_0 
&= 
2(1+2\delta) -2|1-2\delta| 
\\
&=\min\{8\delta,4\}
\nonumber
.
\label{eq_formula_Fourier_coeffsg_appendix}
\end{align}
As a result,~\eqref{eq_def_g_appendix} defines an element of $\mathbb H^1(\T)$, thus also an element of $C^0(\T)$ by Sobolev embedding~\cite[Theorem 4.12]{Adams2003}. 

It remains to prove regularity of $g$ on $[0,1]$. 
Let $f\in C^\infty((0,1))$ be compactly supported. 
Integrating $f''$ terms in~\eqref{eq_weak_ODE_g} by parts and using the translation invariance of $g$ yields:
\begin{align}
\int_{\T} g'(x) f'(x)\, dx 
&= 
-2b(1+2\delta)\int_{\T}f(x)\Big(1+(b/2)\int_{\T} g(x-z)g(z)\, dz\Big)\, dx 
\nnb
&\quad - \sigma\int_{\T^2}g'(x-z)g'(z)f(x)\, dz\, dx
.
\end{align}
The right-hand side is a bounded linear form in $\bbL^2((0,1))$, 
whence $g\in \mathbb H^2((0,1))$.  
Iterating the procedure gives $g\in\mathbb H^k((0,1))$ for every $k\geq 2$, thus $g\in C^\infty([0,1])$ by Sobolev embedding. 
\end{proof}
\section*{Acknowledgements}
We would like to thank Milton Jara for discussions at the beginning of this project, as well as Kenkichi Tsunoda for comments and spotting a
mistake in an earlier draft.
B.D. would like to thank Hendrik Weber for discussions on~\cite{HMW25}. \\
B.D. acknowledges funding from the European Union's Horizon 2020
research and innovation programme under the Marie Sk\l odowska-Curie
grant agreement No 101034255. C. L. has been partially supported by FAPERJ CNE
E-26/201.117/2021, by CNPq Bolsa de Produtividade em Pesquisa PQ
305779/2022-2.

%
\end{document}